\documentclass[11pt,leqno]{report}
\usepackage{amsmath,amsthm,amsfonts,amssymb}
\usepackage[colorlinks=true, pdfstartview=FitV, linkcolor=blue, citecolor=red, urlcolor=blue]{hyperref}
\hypersetup{breaklinks=true}
\usepackage{graphicx}
\usepackage{caption,subfig}
\usepackage{color}
\usepackage{tikz}
\usepackage{calrsfs}
\usetikzlibrary{calc}

\usepackage{geometry}
\newcommand{\N}{{\mathbb N}}
\newcommand{\Z}{{\mathbb Z}}
\newcommand{\R}{{\mathbb R}}
\newcommand{\C}{{\mathbb C}}
\newcommand{\B}{{\mathbb B}}
\newcommand{\E}{{\mathbb E}}
\newcommand{\F}{{\mathbb F}}
\newcommand{\G}{{\mathbb G}}
\newcommand{\eps}{{\varepsilon}}
\newcommand{\bfS}{{\bf S}}

\newcommand\cA{{\mathcal A}}
\newcommand\cB{{\mathcal B}}
\newcommand\cC{{\mathcal C}}
\newcommand\cD{{\mathcal D}}
\newcommand\cE{{\mathcal E}}
\newcommand\cF{{\mathcal F}}
\newcommand\cG{{\mathcal G}}
\newcommand\cH{{\mathcal H}}

\newcommand\cK{{\mathcal K}}
\newcommand\cL{{\mathcal L}}
\newcommand\cM{{\mathcal M}}
\newcommand\cN{{\mathcal N}}
\newcommand\cO{{\mathcal O}}

\newcommand\cU{{\mathcal U}}

\newcommand{\Lag}{{\mathcal L}}
\newcommand{\En}{{\mathcal W}}
\newcommand{\bnu}{{\boldsymbol \nu}}
\newcommand{\bxi}{{\boldsymbol \xi}}
\newcommand{\bfeta}{{\boldsymbol \eta}}

\newcommand{\uxi}{\underline{\xi}}

\newcommand{\uz}{\underline{z}}

\def\D{\partial}
\newcommand\adots{\mathinner{\mkern2mu\raise1pt\hbox{.}
\mkern3mu\raise4pt\hbox{.}\mkern1mu\raise7pt\hbox{.}}}

\newtheorem{theo}{Theorem}[chapter]
\newtheorem{prop}[theo]{Proposition}
\newtheorem{cor}[theo]{Corollary}
\newtheorem{lem}[theo]{Lemma}
\newtheorem{defn}[theo]{Definition}

\newtheorem{rem}[theo]{Remark}

\newtheorem{nota}[theo]{Notations}
\newtheorem{assumption}[theo]{Assumption}

\begin{document}

\title{Geometric optics for surface waves in nonlinear elasticity}

\author{Jean-Fran\c{c}ois {\sc Coulombel}\thanks{CNRS and Universit\'e de Nantes, Laboratoire de math\'ematiques 
Jean Leray (UMR CNRS 6629), 2 rue de la Houssini\`ere, BP 92208, 44322 Nantes Cedex 3, France. Email: 
{\tt jean-francois.coulombel@univ-nantes.fr}. Research of J.-F. C. was supported by ANR project BoND, 
ANR-13-BS01-0009-01.} 
\& Mark {\sc Williams}\thanks{University of North Carolina, Mathematics Department, CB 3250, Phillips Hall, 
Chapel Hill, NC 27599. USA. Email: {\tt williams@email.unc.edu}. Research of M.W. was partially supported by 
NSF grant  DMS-1001616.}}

\maketitle

\begin{abstract}
\emph{\quad} This work is devoted to the analysis of high frequency solutions to the equations of nonlinear elasticity in a 
half-space.   We consider surface waves (or more precisely, Rayleigh waves) arising in the  general class of isotropic hyperelastic models, which includes in particular the Saint Venant-Kirchhoff system.
Work has been done by a number of authors since the 1980s on the formulation and well-posedness of a nonlinear evolution equation whose (exact) solution gives the leading term of an \emph{approximate} Rayleigh wave solution to the underlying elasticity equations.   This evolution equation, which we refer to as ``the amplitude equation", is an integrodifferential equation of nonlocal Burgers type.    We begin by reviewing and providing some extensions of  the theory of the amplitude equation.   The remainder of the paper is devoted to a rigorous proof in 2D that exact, highly oscillatory, Rayleigh wave solutions $u^\eps$ to the nonlinear elasticity equations exist on a fixed time interval independent of the wavelength $\eps$, and that the approximate Rayleigh wave solution provided by the analysis of the amplitude equation is indeed close in a precise sense to $u^\eps$ on a time interval independent of $\eps$.    The paper focuses mainly on the case of Rayleigh waves that are \emph{pulses}, which have profiles with continuous Fourier spectrum, but our method applies equally well to the case of wavetrains, whose Fourier spectrum is discrete.

\end{abstract}

\tableofcontents
\newpage

\chapter{General introduction}

\emph{\quad}This work is devoted to the rigorous justification of weakly nonlinear geometric optics expansions for 
a class of second order hyperbolic initial boundary value problems. These are evolution problems for 
which we ask the two main questions of geometric optics:
\begin{enumerate}
 \item Does an exact solution, say $u_\eps$, exist for $\eps \in (0,1]$ on a fixed time interval $[0,T_0]$ 
 independent of $\eps$, where $\eps$ represents the small wavelength of the data?

 \item Suppose the answer to the first question is yes. If we let $u^{app}_\eps$ denote an approximate 
 solution on $[0,T_0]$ constructed by the methods of nonlinear geometric optics (that is, solving eikonal 
 equations for phases and suitable transport equations for profiles), how well does $u^{app}_\eps$ 
 approximate $u_\eps$ for $\eps$ small? For example, is it true that
\begin{equation*}
\lim_{\eps\to 0} \eps^{-\alpha} \, \| u_\eps-u^{app}_\eps \|_{L^\infty}\to 0 \, ?
\end{equation*}
 where $\alpha$ is a scaling parameter such that both $u_\eps$ and $u^{app}_\eps$ have $O(\eps^\alpha)$ 
 $L^\infty$ norm (in that case $u_\eps-u^{app}_\eps$ is meant to be a remainder of higher order in $\eps$ 
 than $\eps^\alpha$).
\end{enumerate}

Pursuing a long-term program, we deal here with the above two issues in the context of elastodynamics 
in a fixed domain, with boundary conditions giving rise to  Rayleigh waves. These are 
hyperbolic initial boundary value problems for which the so-called uniform Lopatinskii condition\footnote{This is often referred to as the uniform Kreiss-Lopatinkii condition; problems satisfying  this condition are said to be ``strongly" or ``uniformly" stable.}
fails in a manner that is manifested by  the presence of surface waves.\footnote{We use the phrase ``surface waves" as a general term that includes both surface wavetrains and surface pulses.}  
At the linear level such waves were identified long ago,  and they display an exponential decay with respect to the normal variable to the boundary of 
the space domain. Our goal here is to understand the behavior of small amplitude, highly oscillating 
surface waves in the so-called weakly nonlinear regime, where nonlinearity is visible at the leading 
order of the solution, meaning that both $u_\eps$ and $u^{app}_\eps$ are constructed by solving 
nonlinear evolution equations.   More precisely, the phrase ``weakly nonlinear" indicates that the $L^\infty$ norm of $u^{app}_\eps$, say $\eps^\alpha$,  is taken to be as small as it can be while still satisfying the requirement that the evolution equation for the  leading term of $u^{app}_\eps$ be nonlinear.    This particular choice of amplitude $\eps^\alpha$ (for waves oscillating with wavelength $\eps$) is what we mean by the ``weakly nonlinear  scaling".

Before reviewing earlier works on weakly nonlinear surface waves, let us recall several facts on hyperbolic 
initial boundary value problems and geometric optics expansions. Well-posedness of linear, and subsequently 
nonlinear, hyperbolic initial boundary value problems depends on whether or not the  uniform 
Lopatinskii  condition is satisfied, and if not, \emph{how} it fails; see, e.g., \cite{K,sakamoto,CP} and \cite[Chapter 4]{BS}. In the most favorable 
case where this condition is satisfied, the corresponding initial boundary value problems satisfy {\it maximal} 
energy estimates with no loss of derivatives from the data to the solution, including for the trace of the solution 
at the boundary. Nonlinear problems may then be solved by fairly standard Picard iteration techniques. For uniform Lopatinskii problems the 
regime of weakly nonlinear geometric optics corresponds to the same scaling as for the Cauchy problem 
in the whole space. The main features of geometric optics in that case include: propagation at some given 
group velocities inside the space domain, and reflection of (outgoing) oscillating wave packets at the boundary. 
The reflection coefficients which determine incoming oscillating wave packets in terms of outgoing ones are 
{\it finite} as a byproduct of the uniform Lopatinskii condition. A complete justification of the high 
frequency WKB expansions for quasilinear problems is provided in the references \cite{W,CGW1,Hernandez} 
for wavetrains, and in \cite{CW1} for pulses. The difference between the two problems is that for wavetrains, 
the data, exact and approximate solutions depend in a periodic way on the fast variables $\Phi/\eps$ (here $\Phi$ 
represents the phases of the oscillations arising in the solution), while for pulses the fast variables all lie 
in $\R$ with `some' decay property at infinity.  In problems with several interacting phases the pulse setting is slightly more favorable because pulses 
associated with different phases do not interact at leading order, while wavetrains may interact at leading 
order when {\it resonances} occur between several phases.  A consequence is that, for pulses, one can usually 
derive a rate of convergence between exact and approximate solutions \cite{CW1,CW2}, while this is usually out of reach 
for interacting wavetrains
\cite{JMR1,JMR}.   An offsetting difficulty in the pulse setting is that it is usually much harder to construct correctors, and almost never possible to construct many correctors.

In this work we study a situation where the uniform Lopatinskii condition fails, but in a controlled way. 
The degeneracy of this  condition may occur in different ways, and we consider here the case where 
surface waves of finite energy occur, see \cite[Chapter 7]{BS}.   Let us mention right away the work of Marcou \cite{Mar}
for first order nonlinear systems in the  surface wavetrain case.  In \cite{Mar} Marcou provided a complete justification of weakly nonlinear geometric optics expansions for surface wavetrains arising in  first order conservation laws with linear, homogeneous boundary conditions.  
In contrast our main focus will be on the second order hyperbolic systems with fully nonlinear, nonhomogeneous boundary conditions arising in elasticity theory.  We shall study surface pulses (or more precisely, Rayleigh pulses) for the most part, and these require very different methods.    Moreover, as we explain below, it seems that the rigorous justification even of wavetrain solutions in nonlinear elasticity requires methods different from those of \cite{Mar}.



As pointed out  by Serre 
\cite{SerreJFA}, surface waves arise in a rather systematic way in second order hyperbolic problems 
with a variational formulation, and we shall therefore consider this framework, especially in sections \ref{var} and \ref{wna}, before returning to elasticity in section \ref{isoe} and for the rest of the paper.   
Since the uniform Lopatinskii condition fails, well-posedness of the initial boundary value problem is an issue 
because maximal energy estimates are known to fail.  For a class of isotropic, hyperelastic, Neumann type problems in two space dimensions, a class that includes 
the Saint-Venant Kirchhoff system \eqref{a0} as a basic representative, 
 Sabl\'e-Tougeron \cite{S-T} has 
proved precise {\it microlocal} energy estimates which take advantage of the fact that the uniform 
Lopatinskii condition degenerates in the {\it elliptic} region (Definition \ref{lopa}) of the frequency domain.\footnote{As we explain in section \ref{higherD}, 
there is still a serious obstacle to deriving such estimates in dimensions $d\geq 3$.}
The main feature of the energy estimates proved in \cite{S-T} is that there is no loss of regularity from 
the initial data and interior source term to the solution. The only loss with respect to the strongly stable 
case previously mentioned is a loss of one derivative from the boundary data to  the trace of the solution, but one recovers an optimal control in the 
interior of the space domain due to the exponential decay of the surface waves. 
This is consistent with 
one of the main results in \cite{SerreJFA},  which shows that a related class of linear initial boundary value problems 
with \emph{homogeneous} boundary conditions may be solved, at least in cases of time-independent coefficients, by the Hille-Yosida Theorem.  

 The class treated in \cite{SerreJFA} consists of linear problems that arise as in \eqref{eqint}, \eqref{cl}  from a stored energy function $W(\nabla u)$ that is quadratic in $\nabla u$, where $u$ is the displacement; but the Saint Venant-Kirchhoff energy function is fourth order in $\nabla u$, and it leads to more complicated linearized problems.   Only the quadratic part of $W(\nabla u)$ contributes to the linearization at $\nabla u=0$, but to solve the nonlinear problem studied here,  we like \cite{S-T}  must consider the linearizations at nearby nonzero states, where the higher order terms of $W$ do affect the linearized problem.   In the literature the phrase ``linearized elasticity" or ``classical linearized elasticity" is often used to refer to the problem arising from a quadratic $W(\nabla u)$, or from one of the form $W(x,\nabla u)$, where $W$ is quadratic in $\nabla u$ \cite{HM,Ci,CC,SerreJFA}.  These notions of linearized problem are \emph{different} from the one we must work with here; this is clear from a quick inspection of the principal symbol \eqref{u0} of the Saint Venant-Kirchhoff system linearized about about a state $\nabla u\neq 0$.

By combining precise estimates for the linearized problem with a clever iteration scheme based on considering the original nonlinear system together with two auxiliary systems obtained from it by differentiation, \cite{S-T} was able to prove a short-time existence theorem for the so-called ``traction" boundary problem in nonlinear elasticity \eqref{a0} where Rayleigh waves appear.    In spite of the loss of derivatives, an essentially classical iteration scheme did the job.  
 Our work 
builds on this approach, rather than on alternative approaches for nonlinear elasticity, see, e.g., \cite{Kato, SN}.    If one applies the result of \cite{S-T} directly to the nonlinear problem  with highly oscillatory boundary data considered here, one obtains a time of existence $T_\eps$ converging to zero with the wavelength $\eps$.    The main  new challenge addressed in this paper is to obtain estimates that are \emph{uniform} with respect to small $\epsilon$, and which hold on a time interval independent of $\eps$.

We recall  now some of the earlier works devoted to constructing approximate, oscillatory  surface waves. The appropriate 
scaling and leading order amplitude equation were identified in pioneering works going back, at least, 
to Lardner, Parker and others, see for instance 
\cite{Lardner1,Lardner2,ParkerTalbot,Parker,HamiltonIlinskyZabolotskaya}. The analogous theory in the case 
of first order hyperbolic systems was initiated by Hunter \cite{H}. 
In both first and second order problems, the leading amplitude equation takes 
the form of a nonlinear integro-differential evolution equation   of nonlocal Burgers type.    The solution of this equation determines the trace of the leading part of  
the approximate solution $u^{app}_\eps$. Once the trace is known, the leading profile of  $u^{app}_\eps$ is obtained by a simple lifting procedure used for example in \cite{Le}.  
Thus, the main step in  constructing the approximate solution 
 is to prove a well-posedness result for 
such nonlocal versions of the more standard Burgers equation.  This well-posedness problem has 
been solved in \cite{Hunter2006,B}, so constructing the leading profile of $u^{app}_\eps$ will not be our main 
concern here. However, in order to have a functional framework in which we can show that approximate solutions are close to exact solutions, we shall need to extend  somewhat the constructions in \cite{Lardner1,Lardner2} and 
recent generalizations in \cite{BC,BCproc}.

Unlike most earlier works we consider  a WKB ansatz which incorporates 
all possible slow and fast variables.   One needs at least the tangential slow spatial variables and the slow time variable in order to specify exactly along which space-time curve on the boundary a Rayleigh pulse propagates.   In addition we must also include the slow variable normal to the boundary to have profiles that can be used in the rigorous error analysis.  
The analysis of section \ref{wna}, which works in all space dimensions $d\geq 2$,  computes the exact group velocity of Rayleigh waves and shows that the propagation curves are characteristics of the Lopatinski determinant.   In addition we try to unify previous works and clarify the arguments 
that are used to derive the leading order amplitude equation.   In particular, we uncover a cancellation property \eqref{annulation}, \eqref{cancel} that was overlooked by \cite{Lardner1,Lardner2}, and which greatly simplifies the derivation of the amplitude equation.  
This is achieved in Chapter \ref{chapter2} 
below.   Chapter \ref{chapter2} is the most formal one, in which we keep the discussion in a 
rather broad variational context.
Consistently with earlier references from the physics literature, in section \ref{isoe} we identify the
nonlocal amplitude equation for the class of isotropic hyperelastic models, which includes the Saint Venant-Kirchhoff model.
We analyze the WKB ansatz in the case of pulses, that is, when the fast variables lie in $\R$. 
Similar arguments can be applied in the case of wavetrains, and we leave these rather small modifications 
to the interested reader.

Hyperbolic boundary problems in which the uniform Lopatinski condition fails typically exhibit some kind of amplification of solutions \cite{CGW2,CW2,CW4} at the boundary.  
In the second order problem   with first order boundary condition \eqref{a2}-\eqref{a3} considered here, interior forcing of size $\eps$ (in $L^\infty$) with zero boundary and zero initial data gives rise to a surface wave solution of size $O(\eps^2)$; here the response is the same size as  for the forward problem in the whole space, or for a boundary problem of the same order that satisfies the uniform Lopatinski condition.  This is compatible with the absence of any loss of derivatives from interior forcing to solution in the linearized estimates. On the other hand,  boundary forcing of size $O(\eps^2)$ with zero interior forcing and zero initial data gives rise to a surface wave solution  of size $O(\eps^2)$, while for a boundary problem of the same order that satisfies the uniform Lopatinski condition, the solution (which would not be a surface wave!) is of size $O(\eps^3)$.  This discrepancy is compatible with the loss of one derivative from boundary forcing to solution in the linearized estimates.   Thus, there is a sense, perhaps slightly contorted, in which ``amplification" does occur here.

At this stage, it remains unclear whether or not the leading order amplitude, which we know to exist on a time interval that is 
independent of the small wavelength $\eps$, is a reasonable approximation of the exact solution. 
It is also unclear whether or not the exact solution exists on a time interval that is independent of the small 
wavelength $\eps$. This part of the analysis is achieved  for wavetrains in first order systems in \cite{Mar}  by 
using high order approximate solutions to construct nearby exact solutions.
This type of argument dates back to 
\cite{Gues}.   In this case the Fourier transform of profiles with respect to the periodic fast variable $\theta$ is a function of $k\in\Z$ rather than $k\in \R$.   Marcou \cite{Mar} took advantage of this discrete Fourier spectrum to construct arbitrarily many correctors yielding an arbitrarily high order approximate solution, and then added a small remainder to this to obtain a nearby exact solution on a fixed time interval independent of $\eps$.   In particular, her approximate solutions involved profiles that decayed at a rate $e^{-\delta z}$ in the fast variable $z=x_2/\eps$ normal to the boundary for a \emph{fixed} $\delta$ independent of $k\in \mathbb{Z}\setminus \{0\}$.  In the pulse setting the presence of Fourier spectrum arbitrarily close to $k=0$ makes it  impossible to construct such ``strongly evanescent" profiles; indeed, now $\delta=\delta(k)\to 0$ as $k\to 0$. 
In the case of surface waves given by pulses, as in other problems involving pulses \cite{CW1,CW2},  there is no hope of constructing high order approximate solutions.
 Indeed, in the pulse setting it is usually impossible to construct high order correctors  even in problems where the uniform Lopatinski condition is satisfied \cite{CW1}.
In the Rayleigh pulse problem studied here, we are actually  able to construct only a \emph{single} corrector.

The  first-order systems in \cite{Mar} were assumed to be symmetric with maximal dissipative boundary conditions; moreover, the boundary conditions were linear and \emph{homogeneous}, so she was able to use the well-posedness estimates without loss of derivatives which hold for such systems to rigorously justify the high-order expansions.   We see no way to use estimates without loss in a similar way  in  the traction problem of nonlinear elasticity, even if one is trying to justify  high order approximate \emph{wavetrain} solutions.
To sum up,  for several reasons a quite different approach is needed to both of the open questions described above.   Our alternative approach to both questions depends on the study of singular systems.   This is an idea that goes back to \cite{JMR} for problems in free space, and to \cite{W} for problems on domains with boundary. 

As we explain in more detail in the introduction to Chapter \ref{chapter3}, if one looks for an \emph{exact} solution $U^\eps(t,x)$ to the system of 2D nonlinear elasticity \eqref{a2} in the form 
\begin{align}\label{z1}
U^\eps(t,x)=u^\eps(t,x,\theta)|_{\theta=\frac{\beta\cdot (t,x_1)}{\eps}}, \text{ where }\beta=(\beta_0,\beta_1) \text{ and }x=(x_1,x_2),
\end{align}
then by plugging the ansatz \eqref{z1} into \eqref{a2}, one obtains a similar system for the function $u^\eps(t,x,\theta)$, except that the derivatives $\partial_t$, $\partial_{x_1}$ are now replaced by
\begin{align}\label{z2}
\partial_{t,\eps}:=\partial_t+\beta_0\frac{\partial_\theta}{\eps},\;\;\partial_{x_1,\eps}:=\partial_{x_1}+\beta_1\frac{\partial_\theta}{\eps}
\end{align}
wherever they occur.    We refer to this new problem \eqref{a5} as the \emph{singular system} associated with \eqref{a2}.   One gain is immediately apparent: whereas high Sobolev norms of the boundary data $\eps^2 G(t,x_1,\frac{\beta\cdot (t,x_1)}{\eps})$ in \eqref{a2} clearly blow up as $\eps\to 0$, the Sobolev norms of the boundary data $\eps^2 G(t,x_1,\theta)$ in \eqref{a5} \emph{go to zero} as $\eps\to 0$.   A significant  price is also apparent.  The problem \eqref{a5} is now singular in two different senses; first because of the factors $1/\eps$ that appear in the derivatives, and second because derivatives now occur in the linear combinations \eqref{z2}.   Even if $\eps$ is fixed, the second sense still makes the singular system hard to study.   The first sense of singularity implies that when the boundary, $x_2=0$,  is noncharacteristic,  one cannot hope to control norms of normal ($\partial_{x_2}$) derivatives of solutions, uniformly with respect to $\eps$, simply by employing the classical device of first controlling tangential derivatives and then using the equation.
Our decision to build on the approach of \cite{S-T} to the traction problem, rather than that of \cite{Kato} or \cite{SN}, was based on the fact that we saw no way to implement either of the latter approaches in the associated \emph{singular} traction problem.

In Chapter \ref{chapter3} we prove the existence of exact solutions to the original nonlinear Saint Venant-Kirchhoff system \eqref{a2} on a fixed time interval independent of $\eps$ as a consequence of such an existence theorem for the associated singular system \eqref{a5}.   In this approach the exact solution is obtained without any reliance on the approximate solution.    Employing an idea used by \cite{S-T} in the nonsingular setting, we consider not just \eqref{a5} but the trio of coupled singular systems \eqref{a7}-\eqref{a9}, where \eqref{a7} and \eqref{a8} are obtained from the original system by differentiating it, and \eqref{a9} is essentially the same as \eqref{a5}.   In contrast to \cite{S-T},   it turns out that we are not able to use an iteration scheme modeled on the one employed in that paper, or indeed any iteration scheme at all, to prove estimates uniform with respect to $\eps$.   
The uniform estimates depend on knowing the relation $v^\eps=\nabla_\eps u^{\eps}$, which holds (for short times) between the solution $v^\eps$ of \eqref{a7}-\eqref{a8} and $u^\eps$ of \eqref{a9}, and this relation holds only for exact solutions \emph{not} for iterates.\footnote{Here $\nabla_\eps=(\partial_{x_1,\eps},\partial_{x_2})$.} 
Instead, we use a continuous induction argument based on local existence and continuation theorems for singular systems with $\eps$ fixed (Propositions \ref{localex} and \ref{continuation}) and the uniform a priori estimate of Proposition \ref{c5} for the coupled systems.  The proof of the latter proposition, based on \emph{simultaneous} estimation of the trio of modified singular systems \eqref{c1}-\eqref{c3}, is the core of the rigorous analysis of this paper.     The continuous induction argument is summarized in more detail in the introduction to Chapter \ref{chapter3}, and our uniform existence result for  exact solutions is stated in Theorem \ref{uniformexistence}.

The analysis of singular systems in this paper, for example the proof of the basic estimates for the linearized systems corresponding to \eqref{a7}-\eqref{a9} given in Proposition \ref{basicest},  uses two main tools.   The first is the calculus of singular pseudodiffererential operators for pulses  constructed in \cite{CGW}, and the second is the collection of estimates proved in section \ref{nonlinear} of singular norms of nonlinear functions of $u$. 

The  calculus of \cite {CGW} is a calculus for symbols of finite $(t,x,\theta)$ regularity which allows us to compose and take adjoints of operators that are, roughly speaking, pseudodifferential with respect to the vector fields $\partial_{t,\eps}$, $\partial_{x_1,\eps}$ \eqref{z2}.    The operators are defined in \eqref{singularpseudop}.\footnote{In \cite{CGW} a singular calculus for wavetrains was also constructed that was an improvement over the one first constructed in \cite{W}.}  This is a ``first-order" calculus in the sense that only the principal symbols of compositions or adjoints are given; but there are explicit formulas for the error terms  which give a clear picture of how big they are as $\eps\to 0$.     Using this calculus, one can for example construct a Kreiss symmetrizer for a singular problem simply by taking the classical symbol of the Kreiss symmetrizer for the corresponding nonsingular problem and quantizing it in the singular calculus by the process described in \eqref{sect8}.  The main results of this calculus are recalled in Appendix \ref{calculus}.   The calculus was originally created with applications to first order systems in mind.    Here we deal with second order singular systems, so commutators arise, sometimes involving operators of fractional order, which can not be treated using the results of \cite{CGW}.  Thus we have had to extend the calculus of \cite{CGW} in several ways,  and these extensions are given in section \ref{commutator}.

A quick inspection of the estimates of Proposition \ref{basicest} for the linearized singular problems shows that a variety of ``singular norms" occur there.   These are norms 
of the form 
$$
|\Lambda^r_D u|_{H^m(t,x,\theta)},
$$
where $\Lambda^r_D$ is the singular operator associated to the symbol $\langle \xi'+\beta \frac{k}{\eps},\gamma\rangle^r$,
and $r$ and $m$ are (usually) nonnegative constants.  Clearly, to apply these estimates to nonlinear problems we need to be able to estimate singular norms 
$$
|\Lambda^r_D f(u)|_{H^m}
$$ 
of nonlinear functions of $u$ in terms of singular norms of $u$.   Estimates of this kind, which are new,  are proved in section \ref{nonlinear} for analytic functions $f(u)$;  both tame and simpler non-tame estimates are given.\footnote{The tame estimates are needed mainly for the continuation result, Proposition \ref{continuation}.}  In addition, to take advantage of the extra microlocal precision in the estimates provided, for example, by the terms involving the singular pseudodifferential cutoffs $\phi_{j,D}$ in \eqref{b3} and \eqref{b5}, we need to show that in some cases extra singular \emph{microlocal} regularity of $u$ is preserved under nonlinear functions $f(u)$.   A result of this type, which can be viewed as a singular version of the classical Rauch's lemma \cite{R}, is given in Proposition \ref{f5}.




Once we have the leading order approximate solution and the exact solution, the error analysis 
leading to the full justification of geometric optics relies on the construction of an appropriate 
{\it corrector}.   More precisely, we need to add a small corrector to the leading order approximate 
solution in order to be able to control its difference with the exact solution.   In Chapter \ref{chapter2} equations  are derived \eqref{bkwordre2} that one ``would like" the corrector to satisfy.  The amplitude equation (Proposition \ref{propelas}), which determines the trace of the leading order term of the approximate solution, is a solvability condition for the Fourier transform with respect to $\theta$ of these corrector equations \eqref{correcteurv2}.  In Chapter \ref{chapter4} we solve the transformed equations \eqref{correcteurv2} for each $k\neq 0$ (here $k$ is the Fourier transform variable dual to $\theta$) and discover that the corrector is $O(1/k^2)$ near $k=0$.   Roughly, these two factors of $1/k$ reflect the two integrations in $\theta$, each on the unbounded domain $\mathbb{R}$, that are needed to construct a corrector in this second order problem; on the $\theta$ side, each integration in $\theta$ introduces  growth with respect to $\theta$ in the corrector.    This illustrates the difficulty of constructing correctors in pulse problems.  
Our solution of  the transformed equations \eqref{correcteurv2}  behaves too badly at $k=0$ to be inverse transformed,  so we now regard it as a preliminary corrector.  Using an idea we learned from \cite{AR}, we modify this object by multiplying it by a low frequency cutoff, $\chi(k/\eps^b)$, where $\chi(s)$ is a smooth cutoff supported away from $0$ and equal to one on $|s|\geq 1$, and $b>0$ is a constant to be chosen.  This modification introduces new errors of course, but we show in the error analysis of Chapter \ref{chapter5} that the new errors are offset by the presence  of the factor $\eps^3$ on the corrector, provided $b$ is chosen correctly. 

It is natural to wonder if one could construct correctors with better decay properties in $\theta$ if one started with boundary data given by a function $G(t,x_1,\theta)$ with stronger than $H^s$ type decay in $\theta$.  It turns out that even
 if one assumes $G(t,x_1,\theta)$ decays \emph{exponentially} as $|\theta|\to \infty$, the leading term of the approximate solution 
generally exhibits no better than $H^s$ type decay in $\theta$; see Remark \ref{baz}.   This loss of $\theta$-decay  from data to solution is \emph{}{}{}{typical of evanescent pulses}.    It is a linear phenomenon and occurs even in  problems where the uniform Lopatinskii condition is satisfied \cite{Willig}.   Evanescent pulses are generally not even $L^1$ in $\theta$, so unlike evanescent wavetrains they have no well-defined mean.

With the exact and approximate solutions in hand we are ready in Chapter \ref{chapter5} for the error analysis. As in Chapter \ref{chapter3} the estimates, which involve some  norms that cannot be localized in time, must be done on the full half-space $\Omega$.  In particular, we need extensions of the approximate solutions, which at this point are only defined on a short time interval,  to the full half space. 
This extension process has to be done carefully and, in fact, it turns out to require most of  theory of Chapter \ref{chapter3}.   Parallel to the trio of modified singular systems \eqref{c1}-\eqref{c3} that were estimated in the study of the exact solution in Chapter \ref{chapter3}, we define a trio of approximate solution systems \eqref{p1}-\eqref{p3} 
whose solutions provide the needed extensions; a causality argument (Remark \ref{k2y}) shows that the solutions agree with the original approximate solutions for small times.
The exact solution trio \eqref{c1}-\eqref{c3} had solutions $(v^\eps,u^\eps)$ on $\Omega$, while the approximate solution trio \eqref{p1}-\eqref{p3} has solutions $(v_a^\eps,u^\eps_a)$.  Naturally, then, we consider the trio of error equations  \eqref{p4}-\eqref{p6} satisfied by the differences $(w^\eps,z^\eps):=(v^\eps-v^\eps_a,u^\eps-u_a^\eps)$.  As in the earlier cases, this trio must be estimated simultaneously in order to take advantage of the relation $w^\eps=\nabla_\eps z^\eps$ that holds for short times.  These arguments are summarized  in the introduction to Chapter \ref{chapter5}. Ultimately, we are able to derive a rate of convergence to zero for the $E_{m,\gamma}$ norm \eqref{c00} of the difference between the exact and approximate solutions;  this result is stated in Theorem \ref{approxthm}  and Corollary \ref{corapprox}.\footnote{Several papers, for example  \cite{JMR, CGW1, CW1,Hernandez}, in which singular systems were used to rigorously justify approximate solutions for quasilinear problems, used the method of simultaneous Picard iteration, which estimates the difference between exact solution iterates and approximate solution iterates.  That method was not an option here, since we were unable to prove the uniform existence of exact solutions by any kind of iteration scheme.  A variant of the method in section \ref{end}, based on  direct estimation of  the difference between exact and approximate solutions, can be used to avoid the use of simultaneous Picard iteration in the earlier works.}

In Chapter \ref{chapter6} we explain how the main theorems, Theorems \ref{uniformexistence} and \ref{approxthm},  extend to general isotropic hyperelastic materials
governed by an analytic stored energy function, and also how those theorems extend readily to the wavetrain case.  In section \ref{higherD} we discuss the only obstruction that remains to extending these theorems to dimensions $d\geq 3$. We show that in $d\geq 3$ the linearized problem has characteristics of variable multiplicity that fail to be algebraically regular in the sense of \cite{MZ},  and which are at the same time glancing.   Thus, the problem falls outside the scope of existing Kreiss symmetrizer technology.




We have avoided any energy dissipation argument in the construction of exact solutions, 
for this gives us hope to extend our work to free boundary problems for first order hyperbolic problems 
that also give rise to surface waves. Such situations arise indeed in the modeling of liquid-vapor phase 
transitions or in magnetohydrodynamics, see \cite{Benzoni1998,AliHunter}. We believe that the error 
analysis in Chapter \ref{chapter5} is flexible enough so that, provided one has an exact solution 
in such a problem, 
 the 
construction of a corrector and the estimates of the error terms involved in Chapter \ref{chapter5} 
could be adapted in order to yield a full justification of the high frequency asymptotics.

Chapters \ref{chapter2}, \ref{chapter3}, \ref{chapter4} and \ref{chapter5} are meant to be as independent as possible.   Chapter \ref{chapter5} is  really the only chapter that depends substantially on others, those being mainly Chapters \ref{chapter3} and \ref{chapter4}.  
We have included detailed and mostly non-technical introductions to these chapters  in an effort to make the reader familiar with the main ideas before having to plunge into the estimates.   We hope the reader will forgive a certain amount of repetition arising from this.

\paragraph{Notation}

Throughout this work, we let ${\mathcal M}_{n,N}({\mathbb K})$ denote the set of $n \times N$ matrices 
with entries in ${\mathbb K} = \R \text{ or }\C$, and we use the notation ${\mathcal M}_N({\mathbb K})$ 
when $n=N$. The trace of a matrix $M \in {\mathcal M}_N({\mathbb K})$ is denoted $\text{\rm tr } M$. 
The transpose of a matrix (or vector) $M$ is denoted $M^T$. We let $I$ denote the identity matrix, without 
mentioning the dimension. The norm of a (column) vector $X \in \C^N$ is $|X| := (X^* \, X)^{1/2}$, where 
the row vector $X^*$ denotes the conjugate transpose of $X$. If $X,Y$ are two vectors in $\C^N$, we let 
$X \cdot Y$ denote the (bilinear) quantity $\sum_j X_j \, Y_j$, which coincides with the usual scalar product 
in $\R^N$ when $X$ and $Y$ are real. We often use Einstein's summation convention in order to make 
some expressions easier to read.

The Fourier transform of a function $f$ from $\R$ to $\C$ is defined as
$$
\forall \, k \in \R \, ,\quad \widehat{f}(k) := \int_\R {\rm e}^{-i\, k \, \theta} \, f(\theta) \, {\rm d}\theta \, .
$$
The Hilbert transform ${\mathcal H}f$ of $f$ is then defined by
$$
\widehat{{\mathcal H}f}(k) := -i \, \text{\rm sgn } \! \! (k) \, \widehat{f}(k) \, ,
$$
where $\text{\rm sgn}$ denotes the sign function.

The letter $C$ always denotes a positive constant that may vary from line to line or within the same line. 
Dependence of the constant $C$ on various parameters is made precise throughout the text. The sign 
$\lesssim$ means $\le$ up to a multiplicative constant. 

To avoid having expressions like $v^{\eps,s}_{1,T}$ or  $v^{\eps,s}_{T}$ appear repeatedly  in Chapter \ref{chapter3} and later,  we shall often suppress the $\eps$ and $T$ indices and write simply $v^s_1$ or $v^s$ instead.  Here $v^\eps_1$ is the first component of $v^\eps=(v^\eps_1,v^\eps_2)$.  The $s$ and $T$ on $v^{\eps,s}_{1,T}$ indicate that we are taking a Seeley extension (Proposition \ref{c0e}) to all time of $v^\eps_1|_{t<T}$.  Almost every function that occurs  in the study of singular systems has  $\eps$ dependence, so suppressing $\eps$ should cause no trouble.   The presence of the superscript $s$ should \emph{always} be taken to imply the presence of a suppressed subscript $T$.   
The positive number $T$ will always be small, and sometimes will lie in a range that depends on $\eps$, $0<T\leq T_\eps$.   
 The coefficients in the linearized systems we study on the whole half-space, for example those appearing in \eqref{c1}-\eqref{c3},  will usually be functions of terms like $v^s$.  Later reminders of this will be provided.    Further notation is made precise in the body of 
the text.

\chapter{Derivation of the weakly nonlinear amplitude equation}
\label{chapter2}

\emph{\quad} In this Chapter, we revisit the weakly nonlinear asymptotic analysis for second order hyperbolic initial 
boundary value problems that come from a variational principle. Our main goal is to derive an amplitude 
equation that governs the evolution of small amplitude high frequency solutions in the case where the 
so-called WKB ansatz incorporates all possible slow and fast variables. As expected from pioneering 
works devoted to nonlinear elasticity, the amplitude equation we derive takes the form of a scalar 
integro-differential equation taking place on the boundary of the space domain. This equation displays 
two main features: propagation at an appropriate {\it group velocity} according to the slow spatial variables 
along the boundary of the space domain, and nonlinearity due to the presence of a {\it bilinear Fourier 
multiplier} that governs the evolution with respect to the fast variable. We unify previous works arising 
either from the `applied' or more `theoretical' literature and clarify which solvability condition is actually 
needed in order to derive our main amplitude equation.

Taking slow variables into account is crucial in the upcoming Chapters \ref{chapter3}, \ref{chapter4},  and \ref{chapter5} 
for providing a functional framework in which we are able to construct and analyze {\it exact} pulse 
solutions. In the present Chapter, the analysis is mostly {\it formal} and we aim at identifying the leading 
order term in the presumably valid asymptotic expansion of exact solutions. Showing that exact solutions 
are indeed well approximated by this leading order term is the purpose of Chapter \ref{chapter5}.

Though our main concern in this work is the system of nonlinear elasticity, we aim at keeping the discussion 
in a rather general framework whenever possible, which might accelerate the adaptation of the present work 
to related problems with surface waves arising for instance in the modeling of liquid vapor phase transitions, 
magnetohydrodynamics and/or liquid crystals, see, e.g., \cite{Benzoni1998,AliHunter,Saxton,AustriaHunter} 
and further references therein.

\section{The variational setting: assumptions}\label{var}

\emph{\quad} In this Section, we revisit the analysis of \cite{BCproc} and derive the amplitude equation that governs 
the evolution of weakly nonlinear surface (or Rayleigh) waves. Keeping the discussion in a rather general 
framework, we start from a Lagrangian of the form:
\begin{equation*}
\Lag [{\bf u}]:= \int_0^T \! \! \! \int_{\Omega} \left( \dfrac{1}{2} \, |{\bf u}_t|^2 -W(\nabla {\bf u}) \right) \, 
{\rm d}x \, {\rm d}t \, .
\end{equation*}
Here the space domain $\Omega \subset \R^d$ is a half-space, ${\bf u} \in \R^N$ is the (possibly vector-valued) 
unknown function, ${\bf u}_t$ denotes its time derivative and $\nabla {\bf u}$ denotes its spatial Jacobian matrix 
(which we sometimes call its gradient). The function $W$ plays the role of a `stored elastic energy', and the total 
amount of energy over the space domain $\Omega$ is then denoted
\begin{equation*}
\En [{\bf u}] := \int_{\Omega} W(\nabla {\bf u}) \, {\rm d}x \, .
\end{equation*}
Opposite to the 
case considered in \cite{BCproc} we assume here for simplicity that $W$ only depends on the unknown ${\bf u}$ 
through its (spatial) gradient. Hence Assumptions (H1) and (H2) in \cite{BCproc} are trivially satisfied. In the case 
of hyperelastic materials, which is our main concern here, there holds $N=d$, with either $d=2$ or $d=3$.

In the following calculations, Greek letters $\alpha, \beta,\gamma$ usually correspond to indices for the 
coordinates of the vector ${\bf u}$ and therefore run through the set $\{ 1,\dots,N \}$. Roman letters $j,\ell,m$ 
refer to the coordinates of the space variable $x$ and therefore run through $\{ 1,\dots,d \}$. For instance, the 
$(\alpha,j)$ coordinate of the matrix $\nabla {\bf u}$ is denoted $u_{\alpha,j}$, where the subscript `$,j$' is a 
short cut for denoting partial differentiation with respect to $x_j$. We thus consider $W$ as a function from 
$\cM_{N,d} (\R)$ into $\R$ and tacitly assume that it is as smooth as we want (at least $\cC^3$ as far as the 
derivation of the amplitude equation is concerned). We use from now on Einstein's summation convention 
over repeated indices, unless otherwise stated.

Let us write the space domain as $\Omega = \{ x \cdot \bnu >0 \}$, with $\bnu$ the normal vector to $\partial 
\Omega$ pointing inwards. We are then interested in {\it critical points} of the above Lagrangian $\Lag$. These 
correspond to functions ${\bf u}$ that satisfy the interior equations:
\begin{equation}
\label{eqint}
\forall \, \alpha=1,\dots,N \, ,\quad 
\partial_t^2 u_\alpha -\left( \dfrac{\partial W}{\partial u_{\alpha,j}} (\nabla {\bf u}) \right)_{,j} =0 \, ,
\end{equation}
with boundary conditions:
\begin{equation}
\label{cl}
\forall \, \alpha=1,\dots,N \, ,\quad 
\nu_j \, \dfrac{\partial W}{\partial u_{\alpha,j}} (\nabla {\bf u}) \Big|_{\partial \Omega} =0 \, ,
\end{equation}
In what follows, we assume that all constant states $\underline{\bf u} \in \R^N$ are critical points of $\Lag$, 
independently of the choice of the space domain $\Omega$. We also normalize the stored energy so that it 
vanishes at the origin. Equivalently, we make the following assumption:
\begin{itemize}
\item[(H1)] $\qquad W(0)=0$ and $\dfrac{\partial W}{\partial u_{\alpha,j}}(0) =0$ for all $\alpha,j$.
\end{itemize}
Assumption (H1) holds when the stored energy $W$ depends quadratically on $\nabla u$, and in that case the 
equations \eqref{eqint}-\eqref{cl} are linear with respect to ${\bf u}$. Here we shall be interested in solutions ${\bf u}$ 
that are small perturbations of a constant state, say $0$ (up to translating in ${\bf u}$). In the context of nonlinear 
elasticity, ${\bf u}(t,x)$ refers to the displacement with respect to an `equilibrium configuration'. The deformation 
of the equilibrium configuration is given by the mapping $(x \mapsto x +{\bf u}(t,x))$.

For future use, we introduce the coefficients involved in the Taylor expansion of $W$ up to the third order 
at the origin:
\begin{equation}
\label{defcoeff}
c_{\alpha j \beta \ell} := \dfrac{\partial^2 W}{\partial u_{\alpha,j} \, \partial u_{\beta,\ell}}(0) \, ,\quad 
d_{\alpha j \beta \ell \gamma m} := 
\dfrac{\partial^3 W}{\partial u_{\alpha,j} \, \partial u_{\beta,\ell} \, \partial u_{\gamma,m}}(0) \, .
\end{equation}
The linearization of \eqref{eqint}-\eqref{cl} at the constant solution ${\bf u} \equiv 0$ reads
\begin{equation}
\label{eqlin}
\begin{cases}
\partial_t^2 v_\alpha -c_{\alpha j \beta \ell} \, v_{\beta,j\ell} =0 \, ,& 
\forall \, \alpha=1,\dots,N \, ,\quad x \in \Omega \, ,\\
\nu_j \, c_{\alpha j \beta \ell} \, v_{\beta,\ell} \big|_{\partial \Omega} =0 \, ,& \forall \, \alpha=1,\dots,N \, ,
\end{cases}
\end{equation}
and the determination of formal and/or rigorous high frequency weakly nonlinear solutions to 
\eqref{eqint}-\eqref{cl} heavily depends on the stability properties of \eqref{eqlin}. In what follows, we shall 
assume that the linearized problem \eqref{eqlin} admits a one-dimensional space of `surface waves' which, 
in the context of linearized elasticity, correspond to Rayleigh waves. Let us be a little bit more specific. We 
first assume that the linearization of the stored energy $W$ at $0$ is strictly rank one convex (a strong form 
of the Legendre-Hadamard condition), that is:
\begin{itemize}
\item[(H2)] There exists a constant $c>0$ such that for all $\bxi \in \R^d$ and all $v \in \R^N$, there holds
$$
c_{\alpha j \beta \ell} \, v_\alpha \, \xi_j \, v_\beta \, \xi_\ell \ge c \, |v|^2 \, |\bxi|^2 \, .
$$
\end{itemize}
Assumption (H2) ensures that the Cauchy problem
\begin{equation*}
\partial_t^2 v_\alpha -c_{\alpha j \beta \ell} \, v_{\beta,j\ell} =0 \, ,\quad \alpha=1,\dots,N \, ,\quad x \in \R^d \, ,
\end{equation*}
is well-posed in the homogeneous Sobolev space $\dot{H}^1(\R^d;\R^N)$, see \cite{SerreJFA}. Hence \eqref{eqlin} 
is a linear hyperbolic boundary value problem which should be supplemented with some initial data for $v$, which 
we do not write for the moment. The analysis of \eqref{eqlin} follows the general theory of \cite{K,sakamoto} and 
relies on the so-called normal mode analysis. We therefore look for solutions to \eqref{eqlin} of the form
\begin{equation*}
{\bf v}(t,x) = {\rm e}^{i \, (\tau-i\, \gamma) \, t+i \, \bfeta \cdot x} \, V(\bnu \cdot x) \, ,
\end{equation*}
with $\gamma>0$, $\bfeta$ in the cotangent space to $\partial \Omega$ (which reduces here to assuming that 
$\bfeta$ is orthogonal to $\bnu$ because $\Omega$ is a half-space) and a profile $V$ vanishing at $+\infty$. 
Plugging the previous ansatz for ${\bf v}(t,x)$ in \eqref{eqlin}, we are led to determining the set of functions $V$, 
from $\R^+$ into $\C^N$, that vanish at $+\infty$ and satisfy the second-order differential problem:
\begin{equation}
\label{normalmode'}
\begin{cases}
(\tau-i\, \gamma)^2 \, V_\alpha +c_{\alpha j \beta \ell} \, (i\, \eta_j +\nu_j \, \partial_z) \, 
(i\, \eta_\ell +\nu_\ell \, \partial_z) \, V_\beta=0 \, ,& \forall \, \alpha =1,\dots,N \, ,\quad z>0 \, ,\\
\nu_j \, c_{\alpha j \beta \ell} \, (i\, \eta_\ell +\nu_\ell \, \partial_z) \, V_\beta \big|_{z=0} =0 \, ,& 
\forall \, \alpha =1,\dots,N \, .
\end{cases}
\end{equation}
It is convenient to rewrite \eqref{normalmode'} in a more compact form, and we therefore introduce the following 
$N \times N$ matrices for all vector $\bxi \in \R^d$:
\begin{equation}
\label{defASigma}
\Sigma (\bxi) := \Big( c_{\alpha j \beta \ell} \, \xi_j \, \xi_\ell \Big)_{\alpha,\beta=1,\dots,N} \, ,\quad 
A_j(\bxi) := \Big( c_{\alpha j \beta \ell} \, \xi_\ell \Big)_{\alpha,\beta=1,\dots,N} \, ,\quad \forall \, j=1,\dots,d \, .
\end{equation}
With the previous notation, \eqref{normalmode'} reads
\begin{equation}
\label{normalmode}
\begin{cases}
\big( (\tau-i\, \gamma)^2 \, I_N -\Sigma(\bfeta) \big)\, V +i\, \big( \nu_j \, A_j(\bfeta) +\nu_j \, A_j(\bfeta)^T \big) \, 
\partial_z \, V +\Sigma(\bnu) \, \partial_{zz}^2 V =0 & z>0 \, ,\\
i\, \nu_j \, A_j(\bfeta) \, V +\Sigma(\bnu) \, \partial_z \, V \big|_{z=0} =0 \, ,& 
\end{cases}
\end{equation}
where we have used the Schwarz relation which translates into symmetry properties for the coefficients 
$c_{\alpha j \beta \ell}$ (namely, $c_{\alpha j \beta \ell}=c_{\beta \ell \alpha j}$). The strict rank one convexity 
assumption (H2) shows that $\Sigma(\bnu)$ is symmetric positive definite, and therefore the second-order 
boundary value problem \eqref{normalmode} can be equivalently recast as a first-order augmented system
\begin{equation}
\label{hamiltonien}
\begin{cases}
\partial_z \begin{pmatrix}
V \\ W \end{pmatrix} = {\bf J} \, H(\tau-i\, \gamma,\bfeta) \, \begin{pmatrix}
V \\ W \end{pmatrix} \, ,& z>0 \, ,\\
W(0) =0 \, ,&
\end{cases}
\end{equation}
where, using standard block matrix notation, we have set
\begin{equation}
\label{defJH}
{\bf J} :=\begin{pmatrix}
0 & I_N \\
-I_N & 0 \end{pmatrix} \, ,\quad 
H(\zeta,\bfeta) := \begin{pmatrix}
K_0(\zeta,\bfeta) & i\, \nu_j \, A_j(\bfeta)^T \, \Sigma(\bnu)^{-1} \\
-i\, \Sigma(\bnu)^{-1} \, \nu_j \, A_j(\bfeta) & \Sigma(\bnu)^{-1} \end{pmatrix} \, ,
\end{equation}
and the upper-left block $K_0$ of $H$ in \eqref{defJH} is defined by:
\begin{equation}
\label{defK0}
K_0(\zeta,\bfeta) :=\zeta^2 \, I_N -\Sigma(\bfeta) +\big( \nu_j \, A_j(\bfeta)^T \big) \, 
\Sigma(\bnu)^{-1} \, \big( \nu_\ell \, A_\ell(\bfeta) \big) \, .
\end{equation}
The first-order `Hamiltonian' formulation \eqref{hamiltonien} was already introduced in \cite{SerreJFA} and it 
was used in \cite{BC} in order to show structural properties for the amplitude equation which we shall derive 
here in a slightly more general context. The second-order formulation \eqref{normalmode} was adopted in 
\cite{BCproc} and we mainly follow this approach here. However, the equivalent formulation \eqref{hamiltonien} 
will be used at some point, which is why we recall it here with notation similar to those in \cite{BC}.

The aim of the normal mode analysis is to determine the frequency pairs $(\tau-i\, \gamma,\bfeta)$ with 
$\gamma \ge 0$ for which \eqref{normalmode} admits a nontrivial solution. The subtlety lies in the definition 
of which solutions are admissible when $\gamma$ equals zero. However, the strict rank one convexity 
assumption (H2) above yields the classical result that the matrix ${\bf J} \, H(\tau-i\, \gamma,\bfeta)$ is 
{\it hyperbolic} (in the sense of dynamical systems, meaning that it has no purely imaginary eigenvalue) 
as long as $\gamma \neq 0$ or\footnote{Here $\lambda_{\rm min} (S)$ denotes the smallest eigenvalue 
of a real symmetric matrix $S$.}
\begin{equation*}
\gamma=0 \quad \text{\rm and} \quad \tau^2 <\min_{\omega \in \R} \lambda_{\rm min} \, \big( \Sigma 
(\bfeta +\omega \, \bnu) \big) \, .
\end{equation*}
The latter case corresponds to `elliptic frequencies' $(\tau,\bfeta)$ of the cotangent bundle of $\R_t \times 
\partial \Omega$. Observe that by the strict rank one convexity assumption (H2), this set contains a cone 
of the form $\{ |\tau| < c \, |\bfeta| \}$ for some $c>0$. For such elliptic frequencies, $H(\tau,\bfeta)$ in 
\eqref{defJH} is Hermitian, and eigenvalues of ${\bf J} \, H(\tau,\bfeta)$ come in pairs 
$(-\omega,\overline{\omega})$ because the adjoint matrix $({\bf J} \, H(\tau,\bfeta))^*$ is conjugated to 
$-{\bf J} \, H(\tau,\bfeta)$. Here, and from now on, we adopt the convention Re $\omega>0$. The stable 
subspace $\E^s(\tau-i\, \gamma,\bfeta)$ of ${\bf J} \, H(\tau-i\, \gamma,\bfeta)$ is therefore well-defined, 
has dimension $N$ and depends analytically on $(\tau-i\, \gamma,\bfeta)$ on the connected set that is the 
union of $\{ \gamma>0\}$ and of the set of elliptic frequencies. Furthermore, it is known from the general 
theory in \cite{K,sakamoto,Met} that, at least in the case where the hyperbolic operator in \eqref{eqlin} has 
constant multiplicity, the stable subspace $\E^s$ admits a continuous extension up to $\gamma=0$ for all 
$(\tau,\bfeta) \neq (0,0)$. In what follows, $\E^s(\tau,\bfeta)$ denotes this continuous extension for all 
$(\tau,\bfeta)$, which coincides with the `true' stable subspace of ${\bf J} \, H(\tau,\bfeta)$ for elliptic 
frequencies\footnote{The analysis in this Chapter is only concerned with elliptic frequencies, which is 
the reason why we do not emphasize too much the assumptions ensuring that the continuous extension 
of $\E^s$ up to $\gamma=0$ is well-defined. Such assumptions will be enforced in Chapter \ref{chapter3} 
where continuous extension of $\E^s$ will play a major role.}.

Well-posedness of the linear initial boundary value problem \eqref{eqlin}, with prescribed initial data $v|_{t=0}$, 
was investigated thoroughly by Serre in \cite{SerreJFA} where he obtained the striking result that strong 
well-posedness holds {\it if and only if} the energy
\begin{equation*}
\En_2 [{\bf u}]:= \int_{\Omega} W_2(\nabla {\bf u}) \, {\rm d}x \, ,\quad 
W_2(F) := c_{\alpha j \beta \ell} \, F_{\alpha,j} \, F_{\beta,\ell} \, ,\quad \forall \, F \in {\mathcal M}_{N,d}(\R) \, ,
\end{equation*}
is coercive (and thereby strictly convex) over $\dot{H}^1(\Omega)$. The subscript 2 in $\En_2,W_2$ refers to 
the Taylor expansion at order $2$ of the original stored energy $W$ at the origin. When $\En_2$ is coercive, 
the analysis in \cite{SerreJFA} also shows that there exist {\it surface wave} solutions to \eqref{eqlin}, which 
we shall assume here to be of {\it finite energy}. More precisely, in agreement with the results in \cite{SerreJFA}, 
we make the following assumption:
\begin{itemize}
\item[(H3)] For all $\bfeta \neq 0$ in the cotangent space to $\partial \Omega$, there exists a real 
$\tau_{\rm r}(\bfeta)$ satisfying
$$
0 < \tau_{\rm r}(\bfeta) 
< \left( \min_{\omega \in \R} \lambda_{\rm min} \, \big( \Sigma (\bfeta +\omega \, \bnu) \big) \right)^{1/2} \, ,
$$
such that for all $(\gamma,\tau,\bfeta) \neq (0,0,0)$, there holds:
$$
\left\{ \begin{pmatrix}
V \\ W \end{pmatrix} \in \E^s(\tau -i\, \gamma,\bfeta) \, | \, W=0 \right\} \neq \{ 0 \}
$$
if and only if $\gamma=0$ and $\tau =\pm \tau_{\rm r}(\bfeta)$. Furthermore, for all $\underline{\bfeta} \neq 0$, the 
matrix ${\bf J} \, H(\tau,\bfeta)$ is geometrically regular near $(\pm \tau_{\rm r}(\underline{\bfeta}),\underline{\bfeta})$, 
meaning that it admits a basis of eigenvectors $({\bf R}_1,\dots,{\bf R}_{2\, N})(\tau,\bfeta)$ that depends analytically 
on $(\tau,\bfeta)$ near $(\pm \tau_{\rm r}(\underline{\bfeta}),\underline{\bfeta})$ with the convention that ${\bf R}_1, 
\dots,{\bf R}_N$ span the stable subspace $\E^s$, and the $N \times N$ determinant
$$
\Delta (\tau,\bfeta) := \det \big( S_1(\tau,\bfeta) \cdots S_N(\tau,\bfeta) \big) \, ,\quad 
{\bf R}_\alpha :=\begin{pmatrix}
R_\alpha \\ S_\alpha \end{pmatrix} \, ,
$$
has a simple root (with respect to $\tau$) at $(\pm \tau_{\rm r}(\underline{\bfeta}),\underline{\bfeta})$.
\end{itemize}

Let us recall that due to the fact that the so-called Lopatinskii determinant $\Delta$ has a simple root at 
$(\tau_{\rm r}(\underline{\bfeta}),\underline{\bfeta})$, the subspace 
$$
\left\{ \begin{pmatrix}
V \\ W \end{pmatrix} \in \E^s(\tau_{\rm r}(\underline{\bfeta}),\underline{\bfeta}) \, | \, W=0 \right\} 
$$
is one-dimensional. (Observe also that $\E^s(\tau,\bfeta)$ does not depend on the sign of $\tau$, which is why 
we restrict here to positive values of $\tau$.) Consequently, for all tangential wave vector $\bfeta \neq 0$, there 
exists a one-dimensional family of surface waves
$$
{\rm e}^{\pm i \, \tau_{\rm r}(\bfeta) \, t+i \, \bfeta \cdot x} \, V(\bnu \cdot x) \, ,
$$
solution to \eqref{eqlin} with $V(+\infty) =0$ (the decay is actually exponential).

By rescaling, it is readily observed from \eqref{normalmode} that for all $k>0$, $\tau_{\rm r}(k\, \underline{\bfeta}) 
=k\, \tau_{\rm r}(\underline{\bfeta})$ and if
$$
{\rm e}^{\pm i \, \tau_{\rm r}(\bfeta) \, t+i \, \bfeta \cdot x} \, V(\bnu \cdot x) \, ,
$$
is a (nontrivial) surface wave solution to \eqref{eqlin}, then
$$
{\rm e}^{\pm i \, k\, \tau_{\rm r}(\bfeta) \, t+i \, k \, \bfeta \cdot x} \, V(k\, \bnu \cdot x) \, ,
$$
is a (nontrivial) surface wave solution to \eqref{eqlin} for the rescaled frequencies $k\, (\tau_{\rm r}(\bfeta),\bfeta)$.

We have now made all our assumptions on the linearized problem \eqref{eqlin} and turn to the derivation of 
the amplitude equation governing weakly nonlinear asymptotic solutions to the nonlinear problem \eqref{eqint}, 
\eqref{cl}.

\section{Weakly nonlinear asymptotics}\label{wna}

\emph{\quad} In this Section, we show formally that high frequency weakly nonlinear solutions to the nonlinear equations 
\eqref{eqint}, \eqref{cl} are governed by a nonlocal Burgers type equation that is similar to the ones derived 
in \cite{BC} or \cite{BCproc} for second-order equations, or in \cite{H,Mar} for first-order hyperbolic systems. In 
the case of elastodynamics, the derivation of such amplitude equations dates back at least to \cite{Lardner1} for 
two-dimensional elasticity. Part of the analysis in \cite{Lardner1} was later simplified in \cite{ParkerTalbot,Parker} 
though those last two references did not include any dependence of the weakly nonlinear solution on what we 
call below `slow spatial variables'. Those slow variables were considered in \cite{Lardner1}, leading to rather 
complicated calculations and a somehow mysterious relation (Equation (A17) in \cite{Lardner1}) stating that 
weakly nonlinear Rayleigh waves actually propagate at the Rayleigh speed with respect to the slow tangential 
variables along the boundary of the elastic material. The analogous relation in the case of anisotropic materials 
(Equation (51) in \cite{Lardner2}) is explained in a more convincing way, though it seems to rely heavily on the 
fact that one considers two-dimensional materials. Independently of one's ability to verify the accuracy of the 
expressions given in \cite{Lardner1,Lardner2}, there is a puzzling fact arising, which is that the propagation 
of the weakly nonlinear Rayleigh wave in the slow tangential space variables along the boundary seems to 
be governed by the Rayleigh speed, which rather arises as a {\it phase velocity}, while one might reasonably 
expect to see a {\it group velocity} arise. As we show below, the conclusions in \cite{Lardner1,Lardner2} are 
indeed correct but they are linked to the fact that the boundary has only one spatial dimension so the dispersion 
relation between $\tau_{\rm r}$ and $\bfeta$ becomes linear. We also simplify below some of the calculations 
in \cite{Lardner1,Lardner2} by showing that the substitution of normal derivatives in term of tangential ones 
made in those two references is actually unnecessary, since one can directly derive the amplitude equation with 
slow spatial derivatives taking place only along the boundary of $\Omega$. This `cancellation' property, 
which seems to have been unnoticed in all above mentioned references, is our main improvement with 
respect to \cite{BCproc} where only fast spatial variables were taken into account.
\bigskip

From now on, we fix a nonzero wave vector $\bfeta$ in the cotangent space to $\partial \Omega$, and we fix, 
as in \cite{Lardner1}, the frequency $\tau :=-\tau_{\rm r}(\bfeta)>0$ (the case $\tau =\tau_{\rm r}(\bfeta)$ is 
entirely similar). We also fix a nonzero (exponentially decaying at $+\infty$) profile $V$ such that
$$
{\rm e}^{i \, \tau \, t+i \, \bfeta \cdot x} \, V(\bnu \cdot x) \, ,
$$
is a (surface wave) solution to \eqref{eqlin}.

We look for asymptotic solutions ${\bf u}^\eps$ to the nonlinear equations \eqref{eqint} satisfying an 
inhomogeneous version of the boundary conditions \eqref{cl}, namely
\begin{equation}
\label{clinhom}
\forall \, \alpha=1,\dots,N \, ,\quad 
\nu_j \, \dfrac{\partial W}{\partial u_{\alpha,j}} (\nabla {\bf u}^\eps) \Big|_{x \in \partial \Omega} 
=\eps^2 \, G \left( t,x,\dfrac{\tau \, t +\bfeta \cdot x}{\eps} \right) \, ,
\end{equation}
where $G$ is defined on $(-\infty,T] \times \partial \Omega \times \R$ for some given time $T>0$, and decays 
at infinity with respect to its last argument, which we denote $\theta$ from now on. We do not give a precise 
meaning to the `decay at infinity' for $G$, but we tacitly assume at least that $G$ is $L^2$ with respect to 
$\theta$ so that we can apply the Fourier transform\footnote{Unlike several previous works on pulse-like 
solutions, we shall not consider here any polynomial decay with respect to the fast variable $\theta$, though 
such technical considerations can be skipped in the framework of this Chapter and postponed to Chapter 
\ref{chapter3}.}. We expect the exact solution ${\bf u}^\eps$ to \eqref{eqint}, \eqref{clinhom} to have an 
asymptotic expansion of the form
\begin{equation}
\label{wkbu}
{\bf u}^\eps \sim \eps^2 \, {\bf u}^{(1)} \left( t,x,\dfrac{\tau \, t +\bfeta \cdot x}{\eps},\dfrac{\bnu \cdot x}{\eps} \right) 
+\eps^3 \, {\bf u}^{(2)} \left( t,x,\dfrac{\tau \, t +\bfeta \cdot x}{\eps},\dfrac{\bnu \cdot x}{\eps} \right) +\cdots \, ,
\end{equation}
and we wish to determine the `amplitude equation' governing the evolution of the leading profile ${\bf u}^{(1)}$. 
For convenience, we have considered here the regime of weakly nonlinear high frequency waves, which means 
that the `slow' variables are $(t,x)$ and the `fast' variables are $(\theta,z) :=((\tau \, t +\bfeta \cdot x)/\eps,\bnu 
\cdot x/\eps)$, while in \cite{Lardner1,Lardner2,H,BC,BCproc} the authors considered the regime of weakly 
nonlinear modulation on large times where the `slow' variables are $(\eps \, t,\eps \, x)$ and the analogue of 
the `fast' variables are $(\tau \, t +\bfeta \cdot x,\bnu \cdot x)$. The link between the two regimes comes from 
the scale invariance properties of \eqref{eqint}, \eqref{cl}. Namely, since the half-space $\Omega$ is invariant 
by dilation, if ${\bf u}(t,x)$ is a solution to \eqref{eqint}, \eqref{cl}, then $\alpha \, {\bf u}(t/\alpha,x/\alpha)$ is also 
a solution for any $\alpha>0$. This explains why the scaling in \eqref{wkbu} involves solutions of amplitude 
$O(\eps^2)$ while the references \cite{Lardner1,Lardner2,H,BC,BCproc} consider solutions of amplitude 
$O(\eps)$. One advantage of our scaling is to deal with approximate and/or exact solutions to \eqref{eqint}, 
\eqref{clinhom} that are defined on a {\it fixed} time interval independent of the wavelength $\eps$.

We follow the calculations in \cite{BCproc}, with the novelty here that the profiles ${\bf u}^{(1)},{\bf u}^{(2)}$ also 
depend on the slow spatial variables $x$ while only a slow time variable was considered in \cite{BCproc}. Recalling 
the definition \eqref{defASigma}, we introduce the so-called `fast-fast'/`fast-slow' operators in the interior and the 
`fast'/`slow' operators on the boundary:
\begin{align*}
{\mathcal L}_{\rm ff} &:=\big( \tau^2 \, I -\Sigma(\bfeta) \big) \, \partial^2_{\theta \theta} 
-\big( \nu_j \, A_j(\bfeta) +\nu_j \, A_j(\bfeta)^T \big) \, \partial^2_{\theta z} -\Sigma(\bnu) \, \partial^2_{zz} \, ,\\
{\mathcal L}_{\rm fs} &:= 2\, \tau \, \partial^2_{t \theta} -\big( A_j(\bfeta) +A_j(\bfeta)^T \big) \, \partial^2_{j \theta} 
-\big( A_j(\bnu) +A_j(\bnu)^T \big) \, \partial^2_{j z} \, ,\\
\ell_{\rm f} &:= \nu_j \, A_j(\bfeta) \, \partial_\theta +\Sigma (\bnu) \, \partial_z \, ,\\
\ell_{\rm s} &:= A_j(\bnu)^T \, \partial_j \, ,
\end{align*}
as well as the `fast' quadratic operators:
\begin{align*}
{\mathcal Q}[{\bf u}]_\alpha &:= d_{\alpha j \beta \ell \gamma m} \, (\eta_j \, \partial_\theta +\nu_j \, \partial_z) \, 
\Big\{ \big( (\eta_\ell \, \partial_\theta +\nu_\ell \, \partial_z) \, u_\beta \big) \, 
\big( (\eta_m \, \partial_\theta +\nu_m \, \partial_z) \, u_\gamma \big) \Big\} \, ,\\
{\mathcal M}[{\bf u}]_\alpha &:= \nu_j \, d_{\alpha j \beta \ell \gamma m} \, \Big\{ 
\big( (\eta_\ell \, \partial_\theta +\nu_\ell \, \partial_z) \, u_\beta \big) \, 
\big( (\eta_m \, \partial_\theta +\nu_m \, \partial_z) \, u_\gamma \big) \Big\} \, ,
\end{align*}
where $\alpha$ runs through $\{1,\dots,N\}$. The operator ${\mathcal Q}[\cdot ]$, resp. ${\mathcal M}[\cdot ]$, is 
then defined as the vector valued operator in $\R^N$ whose $\alpha$-coordinate is ${\mathcal Q}[\cdot ]_\alpha$, 
resp. ${\mathcal M}[\cdot ]_\alpha$.

Plugging the ansatz \eqref{wkbu} in \eqref{eqint}, \eqref{clinhom} and collecting the powers of $\eps$, we are led to 
constructing profiles ${\bf u}^{(1)},{\bf u}^{(2)}$ solutions to the equations
\begin{equation}
\label{bkwordre1}
\begin{cases}
{\mathcal L}_{\rm ff} \, {\bf u}^{(1)} =0 \, ,& x \in \Omega \, ,\quad z>0 \, ,\\
\ell_{\rm f} \, {\bf u}^{(1)} \Big|_{x \in \partial \Omega,z=0} =0 \, ,& 
\end{cases}
\end{equation}
\begin{equation}
\label{bkwordre2}
\begin{cases}
{\mathcal L}_{\rm ff} \, {\bf u}^{(2)} =-{\mathcal L}_{\rm fs} \, {\bf u}^{(1)} +\dfrac{1}{2} \, {\mathcal Q}[{\bf u}^{(1)}] \, ,& 
x \in \Omega \, ,\quad z>0 \, ,\\
\ell_{\rm f} \, {\bf u}^{(2)} \Big|_{x \in \partial \Omega,z=0} =G -\Big( \ell_{\rm s} \, {\bf u}^{(1)} 
+\dfrac{1}{2} \, {\mathcal M}[{\bf u}^{(1)}] \Big) \Big|_{x \in \partial \Omega,z=0} \, .& 
\end{cases}
\end{equation}
One major difference here with respect to \cite{BCproc} is that the interior equations on ${\mathcal L}_{\rm ff} 
\, {\bf u}^{(1,2)}$ hold for any $x \in \Omega$ and $z>0$ while the boundary conditions in 
\eqref{bkwordre1}-\eqref{bkwordre2} correspond to a `double trace' $x \in \partial \Omega$ and $z=0$. However, 
$x$ enters as a parameter in the interior equations since the fast-fast operator ${\mathcal L}_{\rm ff}$ only involves 
differentiation with respect to $(\theta,z)$. We can therefore take the trace of the interior equation on $\partial 
\Omega$. In particular, the leading profile ${\bf u}^{(1)}$ should satisfy
\begin{equation*}
\begin{cases}
{\mathcal L}_{\rm ff} \, \Big( {\bf u}^{(1)} \big|_{x \in \partial \Omega} \Big) =0 \, ,& z>0 \, ,\\
 & \\
\ell_{\rm f} \, \Big( {\bf u}^{(1)} \big|_{x \in \partial \Omega} \Big) =0 \, ,& z=0 \, ,
\end{cases}
\end{equation*}
which is the fast problem considered in \cite{BCproc} (Equation (P1) there, though with slightly different notation 
for the fast variables). The Fourier transform (with respect to the fast variable $\theta$) of the trace 
${\bf u}^{(1)}|_{\partial \Omega}$ therefore satisfies
\begin{equation}
\label{defw}
{\bf v}^{(1)} (t,x,k,z) := {\mathcal F}_\theta {\bf u}^{(1)}|_{x \in \partial \Omega} =\widehat{w}(t,x,k) \, \widehat{\bf r}(k,z) 
\, ,\quad \widehat{\bf r}(k,z) := \begin{cases}
V(k\, z) \, ,& k>0 \, ,\\
\overline{V(-k\, z)} \, ,& k<0 \, .
\end{cases}
\end{equation}
We emphasize that, though we let $x$ denote one of the arguments of ${\bf v}^{(1)}$, it should be kept in mind 
that $x$ here is restricted to the boundary of $\Omega$. The first and actually main task is to determine the 
unknown scalar function $w$ (or equivalently its Fourier transform with respect to $\theta$), which will determine 
the leading profile ${\bf u}^{(1)}$ at the boundary of $\Omega$. In Chapter \ref{chapter4} we shall 
go further in the construction of the asymptotic expansion of 
${\bf u}^\eps$.

Before going on, let us introduce two operators similar to the ones defined in \cite{BCproc}, and that arise after 
performing the Fourier transform with respect to $\theta$ on ${\mathcal L}_{\rm ff}$ and $\ell_{\rm f}$:
\begin{align*}
{\bf L}^k &:= -(k\, \tau)^2 \, I +\Sigma(k\, \bfeta) 
-i\, \big( \nu_j \, A_j(k\, \bfeta) +\nu_j \, A_j(k\, \bfeta)^T \big) \, \partial_z -\Sigma(\bnu) \, \partial^2_{zz} \, ,\\
{\bf C}^k &:= i\, \nu_j \, A_j(k\, \bfeta) +\Sigma (\bnu) \, \partial_z \, ,
\end{align*}
The first corrector ${\bf u}^{(2)}$ to the leading profile ${\bf u}^{(1)}$ should satisfy the system \eqref{bkwordre2}. 
In particular, the Fourier transform ${\bf v}^{(2)}$ with respect to $\theta$ of the trace ${\bf u}^{(2)}|_{x \in \partial 
\Omega}$ should satisfy
\begin{equation}
\label{correcteurv2}
\forall \, k \neq 0 \, ,\quad \begin{cases}
{\bf L}^k \, {\bf v}^{(2)} =\F (k,z) \, ,& z>0 \, ,\\
{\bf C}^k \, {\bf v}^{(2)} =\G (k) \, ,& z=0 \, ,
\end{cases}
\end{equation}
where the source terms $\F,\G$ are defined by
\begin{equation}
\label{defFG}
\F := {\mathcal F}_\theta \Big( -{\mathcal L}_{\rm fs} \, {\bf u}^{(1)} 
+\dfrac{1}{2} \, {\mathcal Q}[{\bf u}^{(1)}] \Big) \Big|_{x \in \partial \Omega} \, ,\quad 
\G := \widehat{G} -{\mathcal F}_\theta \left( \ell_{\rm s} \, {\bf u}^{(1)} 
+\dfrac{1}{2} \, {\mathcal M}[{\bf u}^{(1)}] \right) \Big|_{x \in \partial \Omega,z=0} \, .
\end{equation}
The main problem at this stage is that both operators ${\mathcal L}_{\rm fs},\ell_{\rm s}$ involve a normal derivative 
with respect to $\partial \Omega$, which led Lardner \cite{Lardner1,Lardner2} to perform substitutions of normal 
derivatives in terms of tangential ones (enforcing a `compatibility' condition between the source terms $\F,\G$ 
which precludes a secular growth phenomenon in $z$) and thereby deriving an amplitude equation for $w$ 
(called $\gamma$ in \cite{Lardner1,Lardner2}) along the boundary $\partial \Omega$. It turns out that these 
manipulations in \cite{Lardner1,Lardner2} are {\it unnecessary} here. The only compatibility condition we need 
between the source terms $\F$ and $\G$ aims at providing with the existence of a corrector ${\bf v}^{(2)}$ 
solution to \eqref{correcteurv2} that should be at least bounded in $z$ for all $k \neq 0$. Deriving accurate 
bounds for such a corrector will be one of the main issues in the error analysis of Chapter \ref{chapter4}.

Let us now recall the following duality relation which was exhibited in \cite{BCproc} (an analogous duality 
relation was proved in \cite{BC} for the first order Hamiltonian formulation \eqref{hamiltonien}):
\begin{equation*}
\int_0^{+\infty} {\bf v} \cdot {\bf L}^k \, {\bf w} \, {\rm d}z - \Big( {\bf v} \cdot {\bf C}^k \, {\bf w} \Big) \Big|_{z=0} 
=\int_0^{+\infty} {\bf L}^{-k} \, {\bf v} \cdot {\bf w} \, {\rm d}z - \Big( {\bf C}^{-k} \, {\bf v} \cdot {\bf w} \Big) \Big|_{z=0} \, ,
\end{equation*}
where we use the notation ${\bf v} \cdot {\bf v}'$ for the quantity $v_\alpha \, v'_\alpha$, and vectors have 
indifferently real or complex coordinates. Actually, by translating in $z$, it is not hard to see that the same 
duality relation holds on any interval $[z,+\infty)$, namely:
\begin{equation}
\label{duality}
\int_z^{+\infty} {\bf v} \cdot {\bf L}^k \, {\bf w} \, {\rm d}Z - \Big( {\bf v} \cdot {\bf C}^k \, {\bf w} \Big) \Big|_{Z=z} 
=\int_z^{+\infty} {\bf L}^{-k} \, {\bf v} \cdot {\bf w} \, {\rm d}Z - \Big( {\bf C}^{-k} \, {\bf v} \cdot {\bf w} \Big) \Big|_{Z=z} 
\, .
\end{equation}

The surface wave profile $\widehat{\bf r}(k,\cdot)$ in \eqref{defw} satisfies
\begin{equation*}
\forall \, k \neq 0 \, ,\quad \begin{cases}
{\bf L}^k \, \widehat{\bf r}(k,\cdot) =0 \, ,& z>0 \, ,\\
{\bf C}^k \, \widehat{\bf r}(k,\cdot) =0 \, ,& z=0 \, ,
\end{cases}
\end{equation*}
and $\widehat{\bf r}(-k,z)=\overline{\widehat{\bf r}(k,z)}$. Applying the duality relation \eqref{duality} with $z=0$, 
we find that for a `reasonable' solution ${\bf v}^{(2)}$ to the problem \eqref{correcteurv2} to exist, the source 
terms $\F,\G$ in \eqref{correcteurv2} should satisfy the Fredholm type condition
\begin{equation}
\label{eqw1}
\forall \, k \neq 0 \, ,\quad 
\int_0^{+\infty} \overline{\widehat{\bf r}(k,z)} \cdot \F (k,z) \, {\rm d}z -\overline{\widehat{\bf r}(k,0)} \cdot \G (k) =0\, .
\end{equation}
We do not make precise for the moment the meaning of `reasonable solution' for the corrector problem 
\eqref{correcteurv2}. Let us merely say that for all fixed $k \neq 0$, the function ${\bf v}$ to which we apply 
the duality relation \eqref{duality} is $\widehat{\bf r}(-k,z)$ and it thus has exponential decay in $z$ (with an 
exponential decay rate that depends on $k$). For the integrals in \eqref{duality} to make sense, the corrector 
${\bf v}^{(2)}$ should be for instance bounded in $z$ (with first two $z$-derivatives also bounded). The main 
problem to avoid is to have a corrector ${\bf v}^{(2)}$ that displays exponential growth in $z$.

In what remains of this Section, we make Equation \eqref{eqw1} more explicit in terms of the scalar function $w$ 
that defines the trace of the leading profile ${\bf u}^{(1)}$ at the boundary $\partial \Omega$. Namely, our goal is 
to prove the following result.

\begin{prop}
\label{propamp}
Under Assumptions (H1), (H2), (H3), consider the source terms $\F$ and $\G$ defined by \eqref{defFG}, where 
the profile ${\bf u}^{(1)}$ has the form \eqref{defw} on the boundary $\partial \Omega$, and satisfies
$$
\forall \, x \in \Omega \, ,\quad \begin{cases}
{\mathcal L}_{\rm ff} \, {\bf u}^{(1)} =0 \, ,\quad z>0 \, ,& \\
\lim_{z \to +\infty} {\bf u}^{(1)} (t,x,\theta,z) =0 \, .&
\end{cases}
$$
Then the compatibility condition \eqref{eqw1} is a closed evolution equation for the scalar function $w$. For 
concreteness, let us assume furthermore that $\Omega$ is the half-space $\{ x_d >0\}$. Then \eqref{eqw1} 
equivalently reads\footnote{Here the sign convention for $\tau$ plays a role. If we had chosen $\tau=\tau_{\rm r}$ 
rather than $\tau=-\tau_{\rm r}$, then the group velocity $\nabla_{\bfeta} \tau_{\rm r}$ in \eqref{eqw} would have 
been changed into its opposite.}
\begin{equation}
\label{eqw}
\partial_t w +\sum_{j=1}^{d-1} \partial_{\eta_j} \tau_{\rm r} (\bfeta) \, \partial_j w 
+{\mathcal H} \, \big( {\mathcal B}(w,w) \big) =g \, ,
\end{equation}
where ${\mathcal H}$ denotes the Hilbert transform with respect to the variable $\theta$, ${\mathcal B}$ 
is the bilinear Fourier multiplier
\begin{equation}
\label{defB}
\widehat{{\mathcal B}(w,w)} (k) :=-\dfrac{1}{4\, \pi \, c_0} \int_\R b(-k,k-\xi,\xi) \, \widehat{w}(k-\xi) \, \widehat{w}(\xi) 
\, {\rm d}\xi \, ,
\end{equation}
where the constant $c_0$ is defined in Equation \eqref{defc0} below and the kernel $b$ is defined in \eqref{defnoyau} 
(the slow variables $(t,x)$ enter as parameters in the definition of ${\mathcal B}$). Eventually the source term $g$ in 
\eqref{eqw} depends linearly on $G$. Its expression is given in \eqref{defg}.
\end{prop}

\begin{proof}
The first main point is to clarify whether \eqref{eqw1} reads as a closed equation on the function $w$. This is 
not obvious at first sight because both source terms $\F$ and $\G$ involve the trace of the normal derivative 
of ${\bf u}^{(1)}$ at the boundary. It is only after taking the scalar product with $\overline{\widehat{\bf r}}$ in 
\eqref{eqw1} that these normal derivatives will actually drop out. We thus first focus on the `linear' terms in 
\eqref{eqw1}, meaning on all those terms in $\F$ and $\G$ where the leading profile ${\bf u}^{(1)}$ appears 
linearly.

$\bullet$ \underline{The slow time derivative}. This term coincides with that obtained in \cite{BCproc}, and 
we find that the only term on the left hand side of \eqref{eqw1} where a $t$-derivative of ${\bf u}^{(1)}$ 
appears is
\begin{align*}
\int_0^{+\infty} \overline{\widehat{\bf r}(k,z)} \cdot {\mathcal F}_\theta \left( -2\, \tau \, \partial^2_{t \theta} {\bf u}^{(1)} 
\right) \Big|_{x \in \partial \Omega} {\rm d}z &=-2\, i \, \tau \, \text{\rm sgn} (k) \int_0^{+\infty} |\widehat{\bf r}(1,z)|^2 \, 
{\rm d}z \, \partial_t \widehat{w} \\
&=: i\, c_0 \, \text{\rm sgn} (k) \, \partial_t \widehat{w} \, ,
\end{align*}
with (recall that $\tau$ is nonzero due to Assumption (H3)):
\begin{equation}
\label{defc0}
c_0 :=-2\, \tau \, \int_0^{+\infty} |\widehat{\bf r}(1,z)|^2 \neq 0 \, .
\end{equation}

$\bullet$ \underline{Slow spatial derivatives}. We now examine those terms in \eqref{eqw1} that involve a 
partial derivative with respect to $x_j$. (Here $j$ is fixed and there is temporarily no summation over $j$ 
even though this index might be repeated.) One contribution comes from $\F$ and another contribution 
comes from $\G$, giving the term
\begin{multline*}
\int_0^{+\infty} \overline{\widehat{\bf r}(k,z)} \cdot {\mathcal F}_\theta \left( 
\big( A_j(\bfeta) +A_j(\bfeta)^T \big) \, \partial^2_{j \theta} {\bf u}^{(1)} 
+\big( A_j(\bnu) +A_j(\bnu)^T \big) \, \partial^2_{j z} {\bf u}^{(1)} \right) \Big|_{x \in \partial \Omega} {\rm d}z \\
+\overline{\widehat{\bf r}(k,0)} \cdot {\mathcal F}_\theta \left( A_j(\bnu)^T \, \partial_j \, {\bf u}^{(1)} 
\right) \Big|_{x \in \partial \Omega, z=0} \, .
\end{multline*}
Letting ${\bf v}_j$ denote the Fourier transform of the trace $\partial_j {\bf u}^{(1)} |_{\partial \Omega}$, the 
latter term reduces to
\begin{multline}
\label{eqw2}
\int_0^{+\infty} \overline{\widehat{\bf r}(k,z)} \cdot \left( i\, k \, \big( A_j(\bfeta) +A_j(\bfeta)^T \big) \, {\bf v}_j 
+\big( A_j(\bnu) +A_j(\bnu)^T \big) \, \partial_z {\bf v}_j \right) \, {\rm d}z 
+\overline{\widehat{\bf r}(k,0)} \cdot A_j(\bnu)^T \, {\bf v}_j |_{z=0} \\
=\int_0^{+\infty} \left( \overline{\widehat{\bf r}(k,z)} \cdot i\, k \, \big( A_j(\bfeta) +A_j(\bfeta)^T \big) \, {\bf v}_j 
+\overline{\widehat{\bf r}(k,z)} \cdot A_j(\bnu) \, \partial_z {\bf v}_j -A_j(\bnu) \, \partial_z \overline{\widehat{\bf r}(k,z)} 
\cdot {\bf v}_j \right) \, {\rm d}z \, .
\end{multline}
The problem of course is that the $x_j$-derivative and the trace operator on $\partial \Omega$ do not 
necessarily commute, so the latter term is not obviously computable in terms of the scalar function $w$. 
More precisely, we can decompose the function $\partial_j {\bf u}^{(1)} |_{\partial \Omega}$ as
\begin{equation}
\label{decompuj}
\partial_j {\bf u}^{(1)} |_{\partial \Omega} =\nu_j \, \partial_{\bf n} {\bf u}^{(1)} |_{\partial \Omega} 
+\partial_{{\rm tan},j} {\bf u}^{(1)} |_{\partial \Omega} \, ,
\end{equation}
with $\partial_{\bf n} {\bf u}^{(1)}$ the normal derivative of ${\bf u}^{(1)}$ with respect to $\partial \Omega$ and 
$\partial_{{\rm tan},j}$ a tangential vector field along $\partial \Omega$. It remains to verify that all the contributions 
from the terms \eqref{eqw2} that involve the normal derivative of ${\bf u}^{(1)}$ at $\partial \Omega$ sum to zero.

Given any point $x \in \Omega$, the profile ${\bf u}^{(1)}$ should satisfy
$$
{\mathcal L}_{\rm ff} \, {\bf u}^{(1)} =0 \, ,\quad z>0 \, ,
$$
and ${\bf u}^{(1)}(z=+\infty) =0$. This means that for all $x \in \Omega$ and all $k \neq 0$, the Fourier transform 
$\widehat{{\bf u}^{(1)}}(k,z)$ has exponential decay in $z$ and satisfies
$$
{\bf L}^k \, \widehat{{\bf u}^{(1)}}(k,z) =0 \, ,\quad z>0 \, .
$$
Taking the normal derivative on $\partial \Omega$, we find that the Fourier transform ${\bf v}_{\bf n}$ of 
$\partial_{\bf n} {\bf u}^{(1)} |_{\partial \Omega}$ satisfies
$$
{\bf L}^k \, {\bf v}_{\bf n}(k,z) =0 \, ,\quad z>0 \, ,
$$
and ${\bf v}_{\bf n}(z=+\infty)=0$. Let us now apply the duality condition \eqref{duality} on an interval $[z,+\infty)$ 
with ${\bf v} :=\overline{\widehat{\bf r}(k,z)}$ and ${\bf w} :={\bf v}_{\bf n}(k,z)$. We get the relation
$$
\forall \, z >0 \, ,\quad \overline{\widehat{\bf r}(k,z)} \cdot {\bf C}^k \, {\bf v}_{\bf n}(k,z) 
={\bf C}^{-k} \, \overline{\widehat{\bf r}(k,z)} \cdot {\bf v}_{\bf n}(k,z) \, ,
$$
that is
\begin{equation}
\label{annulation}
\overline{\widehat{\bf r}(k,z)} \cdot \Big( i\, k \, \big( \nu_\ell \, A_\ell(\bfeta) +\nu_\ell \, A_\ell(\bfeta)^T \big) 
+\Sigma(\bnu) \, \partial_z \big) {\bf v}_{\bf n} 
=\partial_z \overline{\widehat{\bf r}(k,z)} \cdot \Sigma (\bnu) \, {\bf v}_{\bf n} \, .
\end{equation}

We now use the decomposition \eqref{decompuj} in \eqref{eqw2}, and sum with respect to $j=1,\dots,d$ the 
contributions involving the normal derivative ${\bf v}_{\bf n}$, which yields (here the summation convention 
with respect to $j$ is used again):
\begin{multline}\label{cancel}
\int_0^{+\infty} \left( \overline{\widehat{\bf r}(k,z)} \cdot i\, k \, \big( A_j(\bfeta) +A_j(\bfeta)^T \big) \, \nu_j \, {\bf v}_{\bf n} 
+\overline{\widehat{\bf r}(k,z)} \cdot \nu_j \, A_j(\bnu) \, \partial_z {\bf v}_{\bf n} 
-\nu _j \, A_j(\bnu) \, \partial_z \overline{\widehat{\bf r}(k,z)} \cdot {\bf v}_{\bf n} \right) \, {\rm d}z \\
=\int_0^{+\infty} \left( \overline{\widehat{\bf r}(k,z)} \cdot i\, k \, \big( \nu_j \, A_j(\bfeta) +\nu_j \, A_j(\bfeta)^T \big) 
\, {\bf v}_{\bf n} +\overline{\widehat{\bf r}(k,z)} \cdot \Sigma(\bnu) \, \partial_z {\bf v}_{\bf n} 
-\Sigma (\bnu) \, \partial_z \overline{\widehat{\bf r}(k,z)} \cdot {\bf v}_{\bf n} \right) \, {\rm d}z =0 \, ,
\end{multline}
where we have used the relation (see \eqref{defASigma}):
$$
\forall \, \bxi \, ,\quad \xi_j \, A_j(\bxi) =\Sigma(\bxi) \, ,
$$
and the cancellation property \eqref{annulation}. In other words, we have proved that the sum with respect to 
$j$ of the slow spatial derivative terms in \eqref{eqw1} reads
\begin{equation}
\label{eqw3}
\partial_{{\rm tan},j} \widehat{w} \, \int_0^{+\infty} 
\overline{\widehat{\bf r}(k,z)} \cdot \, i\, k \, \big( A_j(\bfeta) +A_j(\bfeta)^T \big) \, \widehat{\bf r}(k,z) 
+\overline{\widehat{\bf r}(k,z)} \cdot A_j(\bnu) \, \partial_z \widehat{\bf r}(k,z) 
-A_j(\bnu) \, \partial_z \overline{\widehat{\bf r}(k,z)} \cdot \widehat{\bf r}(k,z) \, {\rm d}z \, .
\end{equation}

For concreteness, let us assume from now on that the half-space $\Omega$ is $\{ x_d>0\}$ so that for $j=d$, 
one has $\partial_{{\rm tan},j} =0$, and for $j=1,\dots,d-1$, one has $\partial_{{\rm tan},j} =\partial_j$. (There 
also holds that ${\bnu}$ is the last vector in the canonical basis of $\R^d$ though we do not simplify the 
expressions that depend on $\bnu$ accordingly.) Then the slow spatial derivative terms arising in \eqref{eqw1} 
read
$$
i\, \text{\rm sgn} (k) \, \sum_{j=1}^{d-1} c_j \, \partial_j \widehat{w} \, ,
$$
with
\begin{equation}
\label{defcj}
c_j := \int_0^{+\infty} \overline{\widehat{\bf r}(1,z)} \cdot \, \big( A_j(\bfeta) +A_j(\bfeta)^T \big) \, \widehat{\bf r}(1,z) 
\, {\rm d}z +2\, \text{\rm Im} \int_0^{+\infty} \overline{\widehat{\bf r}(1,z)} \cdot A_j(\bnu) \, 
\partial_z \widehat{\bf r}(1,z) \, {\rm d}z\, .
\end{equation}

$\bullet$ \underline{Quadratic terms}. All remaining terms in \eqref{eqw1} correspond to the contributions from the 
quadratic operators ${\mathcal Q}$ in $\{ z>0 \}$ and ${\mathcal M}$ at $z=0$. These terms are exactly identical 
to the ones considered in \cite{BCproc}, and we shall therefore not reproduce the (lengthy) calculations leading to 
the final expression
$$
-\dfrac{1}{4\, \pi} \int_\R b(-k,k-\xi,\xi) \, \widehat{w}(t,x,k-\xi) \, \widehat{w}(t,x,\xi) \, {\rm d}\xi \, ,
$$
with
\begin{align}
b(\xi_1,\xi_2,\xi_3) :=& \int_0^{+\infty} d_{\alpha j \beta \ell \gamma m} \, 
(\nu_j \, \rho_{\alpha,z} +i\, \xi_1 \, \eta_j \, \rho_\alpha) \, (\nu_\ell \, \rho_{\beta,z} +i\, \xi_2 \, \eta_\ell \, \rho_\beta) \, 
(\nu_m \, \rho_{\gamma,z} +i\, \xi_3 \, \eta_m \, \rho_\gamma) \, {\rm d}z \, ,\label{defnoyau}\\
&\rho_\alpha := \widehat{r}_\alpha (\xi_1,z) \, ,\quad \rho_\beta := \widehat{r}_\beta (\xi_2,z) \, ,\quad 
\rho_\gamma := \widehat{r}_\gamma (\xi_3,z) \, ,\notag
\end{align}
and the subscript `$,z$' denotes differentiation with respect to $z$.

At this stage, we have found that the amplitude equation \eqref{eqw1} reads
\begin{equation*}
i\, c_0 \, \text{\rm sgn} (k) \, \partial_t \widehat{w} +i\, \text{\rm sgn} (k) \, \sum_{j=1}^{d-1} c_j \, \partial_j \widehat{w} 
-\dfrac{1}{4\, \pi} \int_\R b(-k,k-\xi,\xi) \, \widehat{w}(t,x,k-\xi) \, \widehat{w}(t,x,\xi) \, {\rm d}\xi 
=\overline{\widehat{\bf r}(k,0)} \cdot \widehat{G} (k)\, , 
\end{equation*}
where the constants $c_0,\dots,c_{d-1}$ are defined in \eqref{defc0}, \eqref{defcj}, and the kernel $b$ is given 
in \eqref{defnoyau}. Multiplying by $-i \, \text{\rm sgn} (k) \, c_0^{-1}$, and taking the inverse Fourier transform 
in $\theta$, we find that weakly nonlinear high frequency solutions to \eqref{eqint}, \eqref{clinhom} are governed, 
at the leading order, by the amplitude equation
\begin{equation}
\label{eqw5}
\partial_t w +\sum_{j=1}^{d-1} \dfrac{c_j}{c_0} \, \partial_j w +{\mathcal H} \, \big( {\mathcal B}(w,w) \big) =g \, ,
\end{equation}
where, as announced in the statement of Proposition \ref{propamp}, ${\mathcal H}$ denotes the Hilbert transform 
with respect to the fast variable $\theta$ (namely, $\widehat{{\mathcal H} \, f}(k) := -i \, \text{\rm sgn} (k) \, \widehat{f} 
(k)$), and ${\mathcal B}$ is the bilinear Fourier multiplier defined in \eqref{defB}. The source term $g$ in \eqref{eqw5} 
is obtained by setting
\begin{equation}
\label{defg}
\widehat{g} (k):=-\dfrac{i\, \text{\rm sgn} (k)}{c_0} \, \overline{\widehat{\bf r}(k,0)} \cdot \widehat{G} (k)\, ,
\end{equation}
which defines a real valued function $g$ provided that $G$ is real valued, which was tacitly assumed in 
\eqref{clinhom}.

$\bullet$ \underline{Simplifying the linear terms}. We shall have therefore proved Proposition \ref{propamp} 
provided that we get the relations
\begin{equation}
\label{vitessegroupe}
\forall \, j=1,\dots,d-1 \, ,\quad \dfrac{c_j}{c_0} =\partial_{\eta_j} \tau_{\rm r} (\bfeta) \, ,
\end{equation}
which express that the propagation of the amplitude $w$ along the boundary $\partial \Omega$ is governed by the 
{\it group velocity} of the variety along which the Lopatinskii determinant $\Delta$ vanishes. The proof of the relations 
\eqref{vitessegroupe} follows some arguments that already appeared in \cite{BC} for the first-order Hamiltonian 
formulation \eqref{hamiltonien}, and which we adapt\footnote{It is also the opportunity to correct some (harmless) 
normalizing mistakes that appeared in \cite{BC}.} here to the second-order formulation \eqref{normalmode}. 
From Assumption (H3), we know that the matrix ${\bf J} \, H(\tau,\bfeta)$ is diagonalizable and hyperbolic, the 
diagonalization being locally analytic in the frequencies $(\tau,\bfeta)$. For clarity, we denote from now on the 
reference frequency in \eqref{wkbu} $(\underline{\tau},\underline{\bfeta})$ and keep the notation $(\tau,\bfeta)$ 
for real frequencies that are close to $(\underline{\tau},\underline{\bfeta})$, though not necessarily linked by the 
dispersion relation $\tau+\tau_{\rm r}(\bfeta)=0$. We can choose a smooth basis $({\bf R}_1,\dots,{\bf R}_{2\, N}) 
(\tau,\bfeta)$ of $\C^{2\, N}$ such that
\begin{equation}
\label{vecteurspropres}
\forall \, \alpha=1,\dots,N \, ,\quad \begin{cases}
{\bf J} \, H(\tau,\bfeta) \, {\bf R}_\alpha (\tau,\bfeta) =-\omega_\alpha (\tau,\bfeta) 
\, {\bf R}_\alpha (\tau,\bfeta) \, ,& \\
{\bf J} \, H(\tau,\bfeta) \, {\bf R}_{N+\alpha} (\tau,\bfeta) 
=\overline{\omega_\alpha (\tau,\bfeta)} \, {\bf R}_{N+\alpha} (\tau,\bfeta) \, ,&
\end{cases}
\end{equation}
with the normalization convention
\begin{equation*}
\begin{pmatrix} 
{\bf R}_1 & \cdots & {\bf R}_{2\, N} \end{pmatrix}^* \, {\bf J} \, \begin{pmatrix} 
{\bf R}_1 & \cdots & {\bf R}_{2\, N} \end{pmatrix} =-{\bf J} \, .
\end{equation*}
Defining some vectors ${\bf L}_\alpha$ by
\begin{equation*}
\forall \, \alpha=1,\dots,2\, N \, ,\quad {\bf L}_\alpha (\tau,\bfeta) :=\begin{cases}
{\bf R}_{N+\alpha} \, ,&\text{\rm if } \alpha \le N \, ,\\
-{\bf R}_{\alpha-N} \, ,&\text{\rm if } \alpha>N \, ,
\end{cases}
\end{equation*}
we get the orthogonality relations
\begin{equation}
\label{orthogonalite}
\forall \, \alpha,\beta =1,\dots,2\, N \, ,\quad {\bf L}_\alpha^* \, {\bf J} \, {\bf R}_\beta =\delta_{\alpha \beta} \, ,
\end{equation}
which hold for all $(\tau,\bfeta)$ close to $(\underline{\tau},\underline{\bfeta})$, though we shall only use 
them at the reference frequency $(\underline{\tau},\underline{\bfeta})$. The Lopatinskii determinant is then 
defined\footnote{Observe that the value of $\Delta$ depends on the basis of $\E^s$ with which it is defined 
but the location of the roots and their multiplicity does not, which is why we can equivalently define $\Delta$ 
with our basis $({\bf R}_\alpha)$.} as
$$
\Delta (\tau,\bfeta) := \det \big( S_1(\tau,\bfeta) \cdots S_N(\tau,\bfeta) \big) \, ,\quad 
{\bf R}_\alpha :=\begin{pmatrix}
R_\alpha \\ S_\alpha \end{pmatrix} \, .
$$
In what follows, we decompose all vectors ${\bf R}_\alpha \in \C^{2\, N}$ as above, and not only those 
corresponding to $\alpha=1,\dots,N$. Underlined quantities refer to evaluation at the frequency 
$(\underline{\tau},\underline{\bfeta})$. Since the vectors $\underline{S}_1,\dots,\underline{S}_N$ span 
an $N-1$-dimensional subspace in $\C^N$, we can fix a nonzero vector $(\lambda_1,\dots,\lambda_N)$ 
in $\C^N$ such that (here we use the convention $\lambda_\gamma=0$ if $\gamma \ge N+1$):
\begin{equation}
\label{lopfaible}
\lambda_\gamma \, \underline{S}_\gamma =0 \, .
\end{equation}
There is no loss of generality in assuming $\lambda_1 \neq 0$ and even in normalizing the $\lambda_\gamma$'s 
by assuming $\lambda_1=1$, which means that $\underline{S}_2,\dots,\underline{S}_N$ are linearly 
independent. In that case, the surface wave $V$ which appears in the definition \eqref{defw} of 
$\widehat{{\bf r}}$ reads\footnote{Or at least we can fix it this way, since $V$ is defined up to a nonzero 
multiplicative constant.}
\begin{equation}
\label{decompV}
V(z) =\lambda_\gamma \, {\rm e}^{-\underline{\omega}_\gamma \, z} \, \underline{R}_\gamma \, .
\end{equation}

We now observe that the two linear forms
\begin{equation*}
X \in \C^N \longmapsto \det \big( X \, \, \underline{S}_2 \cdots \underline{S}_N \big) \, ,\quad 
X \in \C^N \longmapsto (\lambda_\gamma \, \underline{R}_\gamma)^* \, X \, ,
\end{equation*}
are nonzero and vanish on the hyperplane spanned by $\underline{S}_1,\dots,\underline{S}_N$ (use 
\eqref{lopfaible} and the orthogonality relations \eqref{orthogonalite} for the latter). Hence there exists 
a nonzero constant $\kappa \in \C$ such that
\begin{equation*}
\forall \, X \in \C^N \, ,\quad \det \big( X \, \, \underline{S}_2 \cdots \underline{S}_N \big) 
=\kappa \, (\lambda_\gamma \, \underline{R}_\gamma)^* \, X \, .
\end{equation*}

Let us assume for simplicity that the half-space $\Omega$ is $\{ x_d>0\}$, which means that the tangential 
wave vectors $\bfeta$ are parametrized by $\eta_1,\dots,\eta_{d-1}$. Then the partial derivative $\partial \Delta$ 
of $\Delta$ with respect to any of the variables $\tau,\eta_1,\dots,\eta_{d-1}$ is given by
$$
\partial \Delta (\underline{\tau},\underline{\bfeta}) =\det \big( \lambda_\alpha \, 
\partial S_\alpha (\underline{\tau},\underline{\bfeta}) \, \, \underline{S}_2 \cdots \underline{S}_N \big) 
=\kappa \, \lambda_\alpha \, (\lambda_\gamma \, \underline{R}_\gamma)^* \, 
\partial S_\alpha (\underline{\tau},\underline{\bfeta}) \, .
$$
For each partial derivative $\partial$, with respect to either of the variables $\tau,\eta_1,\dots,\eta_{d-1}$, 
we follow \cite{BC} and decompose
\begin{equation}
\label{decompderivee}
\forall \, \alpha=1,\dots,N \, ,\quad 
\partial {\bf R}_\alpha (\underline{\tau},\underline{\bfeta}) =\mu_{\alpha \beta} \, \underline{{\bf R}}_\beta \, ,
\end{equation}
where summation with respect to $\beta$ includes all indices $\beta=1,\dots,2\, N$. We thus get the expression
$$
\partial \Delta (\underline{\tau},\underline{\bfeta}) 
=\kappa \, \lambda_\alpha \, \mu_{\alpha \beta} \, (\lambda_\gamma \, \underline{R}_\gamma)^* 
\, \underline{S}_\beta \, ,
$$
where summation runs over $\alpha,\gamma=1,\dots,N$ (recall $\lambda_\alpha=0$ for $\alpha \ge N+1$) and 
$\beta=N+1,\dots,2\, N$ (because of the orthogonality relation $(\lambda_\gamma \, \underline{R}_\gamma)^* \, 
\underline{S}_\beta=0$ for all $\beta=1,\dots,N$). We thus need to determine the coefficients $\mu_{\alpha \beta}$ 
in \eqref{decompderivee} for $\alpha \le N$ and $\beta \ge N+1$. These coefficients are obtained by differentiating 
\eqref{vecteurspropres} and using the orthogonality relations \eqref{orthogonalite}. We get
$$
\forall \, \alpha =1,\dots,N \, ,\quad \forall \, \beta =N+1,\dots,2\, N \, ,\quad 
\mu_{\alpha \beta} =\dfrac{\underline{{\bf L}}_\beta^* \, \partial H 
(\underline{\tau},\underline{\bfeta}) \, \underline{{\bf R}}_\alpha}
{\underline{\omega}_\alpha +\overline{\underline{\omega}_{\beta-N}}} 
=-\dfrac{\underline{{\bf R}}_{\beta-N}^* \, \partial H 
(\underline{\tau},\underline{\bfeta}) \, \underline{{\bf R}}_\alpha}
{\underline{\omega}_\alpha +\overline{\underline{\omega}_{\beta-N}}} \, ,
$$
and therefore
$$
\partial \Delta (\underline{\tau},\underline{\bfeta}) =-\kappa \, \lambda_\alpha \, 
\dfrac{\underline{{\bf R}}_\beta^* \, \partial H (\underline{\tau},\underline{\bfeta}) \, \underline{{\bf R}}_\alpha}
{\underline{\omega}_\alpha +\overline{\underline{\omega}_\beta}} \, (\lambda_\gamma \, \underline{R}_\gamma)^* 
\, \underline{S}_{\beta+N} \, ,
$$
where summation now runs over $\alpha,\beta,\gamma=1,\dots,N$. It remains to observe that we have the relation 
(use \eqref{lopfaible}):
$$
\forall \, \beta=1,\dots,N \, ,\quad 
(\lambda_\gamma \, \underline{R}_\gamma)^* \, \underline{S}_{\beta+N} 
= (\lambda_\gamma \, {\bf R}_\gamma)^* \, {\bf J} \, {\bf R}_{\beta+N} 
= -(\lambda_\gamma \, {\bf L}_{\gamma+N})^* \, {\bf J} \, {\bf R}_{\beta+N} =-\overline{\lambda_\beta} \, ,
$$
and the partial derivative of the Lopatinskii determinant thus reduces to its final expression
\begin{equation}
\label{deriveedelta}
\partial \Delta (\underline{\tau},\underline{\bfeta}) =\kappa \, \lambda_\alpha \, \overline{\lambda_\beta} \, 
\dfrac{\underline{{\bf R}}_\beta^* \, \partial H (\underline{\tau},\underline{\bfeta}) \, \underline{{\bf R}}_\alpha}
{\underline{\omega}_\alpha +\overline{\underline{\omega}_\beta}} \, .
\end{equation}

We first use the relation \eqref{deriveedelta} for the $\tau$-partial derivative, and recall the expression 
\eqref{defJH} of $H$:
\begin{equation*}
\partial_\tau \Delta (\underline{\tau},\underline{\bfeta}) =2\, \tau \, \kappa \, \lambda_\alpha \, 
\overline{\lambda_\beta} \, \dfrac{\underline{R}_\beta^* \, \underline{R}_\alpha}
{\underline{\omega}_\alpha +\overline{\underline{\omega}_\beta}} 
= 2\, \tau \, \kappa \, \int_0^{+\infty} |\widehat{{\bf r}}(1,z)|^2 \, {\rm d}z 
= -\kappa \, c_0 \, .
\end{equation*}
The computation of $\partial_{\eta_j} \Delta$ is slightly more complicated but not too troublesome either. 
We differentiate $H$ in \eqref{defJH} with respect to $\eta_j$ and use the definitions of $A_j,\Sigma$ in 
\eqref{defASigma} to get
\begin{align*}
\underline{{\bf R}}_\beta^* \, \partial_{\eta_j} H (\underline{\tau},\underline{\bfeta}) \, \underline{{\bf R}}_\alpha 
=&-\underline{R}_\beta^* \, (A_j(\bfeta) +A_j(\bfeta)^T) \, \underline{R}_\alpha \\
&+\underline{R}_\beta^* \, A_j(\bnu) \, \Sigma(\bnu)^{-1} \, (\nu_\ell \, A_\ell(\bfeta)) \, \underline{R}_\alpha 
+\underline{R}_\beta^* \, (\nu_\ell \, A_\ell(\bfeta)^T) \, \Sigma(\bnu)^{-1} \, A_j(\bnu)^T \, \underline{R}_\alpha \\
&+i\, \underline{R}_\beta^* \, A_j(\bnu) \, \Sigma(\bnu)^{-1} \, \underline{S}_\alpha 
-i\, \underline{S}_\beta^* \, \Sigma(\bnu)^{-1} \, A_j(\bnu)^T \, \underline{R}_\alpha \, .
\end{align*}
Recall that we are interested here in indices $\alpha,\beta$ between $1$ and $N$, so we can use the relations 
(which express that ${\bf R}_\gamma$ is an eigenvector of ${\bf J}\, H$ for the eigenvalue $-\omega_\gamma$):
$$
\underline{S}_\alpha =(i\, \nu_\ell \, A_\ell(\bfeta) -\underline{\omega}_\alpha \, \Sigma(\bnu)) \, 
\underline{R}_\alpha \, ,\quad 
\underline{S}_\beta =(i\, \nu_\ell \, A_\ell(\bfeta) -\underline{\omega}_\beta \, \Sigma(\bnu)) \, 
\underline{R}_\beta \, .
$$
This simplifies the above expression of $\underline{{\bf R}}_\beta^* \, \partial_{\eta_j} H 
(\underline{\tau},\underline{\bfeta}) \, \underline{{\bf R}}_\alpha$ accordingly:
\begin{equation*}
\underline{{\bf R}}_\beta^* \, \partial_{\eta_j} H (\underline{\tau},\underline{\bfeta}) \, \underline{{\bf R}}_\alpha 
=-\underline{R}_\beta^* \, (A_j(\bfeta) +A_j(\bfeta)^T) \, \underline{R}_\alpha 
-i\, \underline{\omega}_\alpha \, \underline{R}_\beta^* \, A_j(\bnu) \, \underline{R}_\alpha 
+i\, \overline{\underline{\omega}_\beta} \, \underline{R}_\beta^* \, A_j(\bnu)^T \, \underline{R}_\alpha \, .
\end{equation*}
Plugging this expression in \eqref{deriveedelta}, we obtain
\begin{align*}
\partial_{\eta_j} \Delta (\underline{\tau},\underline{\bfeta}) =& -\kappa \, \lambda_\alpha \, \overline{\lambda_\beta} \, 
\dfrac{\underline{R}_\beta^* \, (A_j(\bfeta) +A_j(\bfeta)^T) \, \underline{R}_\alpha}
{\underline{\omega}_\alpha +\overline{\underline{\omega}_\beta}} \\
&-i\, \kappa \, \lambda_\alpha \, \overline{\lambda_\beta} \, \underline{\omega}_\alpha \, 
\dfrac{\underline{R}_\beta^* \, A_j(\bnu) \, \underline{R}_\alpha}
{\underline{\omega}_\alpha +\overline{\underline{\omega}_\beta}} 
+i\, \kappa \, \lambda_\alpha \, \overline{\lambda_\beta} \, \overline{\underline{\omega}_\beta} \, 
\dfrac{\underline{R}_\beta^* \, A_j(\bnu)^T \, \underline{R}_\alpha}
{\underline{\omega}_\alpha +\overline{\underline{\omega}_\beta}} \\
=&-\kappa \, \int_0^{+\infty} \widehat{{\bf r}}(1,z)^* \, (A_j(\bfeta) +A_j(\bfeta)^T) \, \widehat{{\bf r}}(1,z) \, {\rm d}z 
-2\, \kappa \, \text{\rm Im} \int_0^{+\infty} \widehat{{\bf r}}(1,z)^* \, A_j(\bnu) \, \partial_z \widehat{{\bf r}}(1,z) \, {\rm d}z \, ,
\end{align*}
and we thus find the relation $\partial_{\eta_j} \Delta (\underline{\tau},\underline{\bfeta}) =-\kappa \, c_j$. 
In other words, we have found the relation
$$
\forall \, j=1,\dots,d-1 \, ,\quad \dfrac{c_j}{c_0} =
\dfrac{\partial_{\eta_j} \Delta (\underline{\tau},\underline{\bfeta})}
{\partial_\tau \Delta (\underline{\tau},\underline{\bfeta})} \, ,
$$
and the claim of Proposition \ref{propamp} follows from the factorization (which holds near the simple root 
$(\underline{\tau},\underline{\bfeta})$ of $\Delta$, and here the sign convention for $\tau$ plays a role):
$$
\Delta (\tau,\bfeta) =\vartheta (\tau,\bfeta) \, (\tau+\tau_{\rm r}(\bfeta)) \, ,\quad 
\vartheta (\underline{\tau},\underline{\bfeta}) \neq 0 \, .
$$
We thus get $c_j/c_0 =\partial_{\eta_j} \tau_{\rm r}(\underline{\bfeta})$, which yields the final form of the amplitude 
equation \eqref{eqw}.
\end{proof}

Let us recall that our goal is to construct the asymptotic expansion \eqref{wkbu} of the solution ${\bf u}^\eps$ 
to the high frequency problem \eqref{eqint}, \eqref{clinhom}. At this stage, we have shown that the leading order 
amplitude ${\bf u}^{(1)}$ is given on the boundary $\partial \Omega$ by \eqref{defw} and that for a `reasonable' 
corrector ${\bf u}^{(2)}$ in \eqref{wkbu} to exist, the amplitude function $w$ should solve the equation 
\eqref{eqw}. In order to proceed and construct the first term in the asymptotic expansion \eqref{wkbu}, the main 
question is to investigate the well-posedness properties of the equation \eqref{eqw}. This will rely on earlier 
results by Hunter and Benzoni-Gavage \cite{Hunter2006,B}. Once we have $w$, the leading order term 
${\bf u}^{(1)}$ is defined on the boundary $\partial \Omega$ and, following \cite{Le,Ma,MetivierCL}, one 
possibility for defining ${\bf u}^{(1)}$ in the whole space domain $\Omega$ is simply to let
$$
{\bf u}^{(1)}(t,x,\theta,z) := \chi (x \cdot \bnu) \, {\bf u}^{(1)} (t,y,\theta,z) \, ,
$$
where $y$ denotes the orthogonal projection of $x$ on $\partial \Omega$, and the function $\chi \in 
{\mathcal C}^\infty(\R)$ has compact support and satisfies $\chi(0)=1$. There are two questions then: can 
we construct indeed a corrector ${\bf u}^{(2)}$ that satisfies \eqref{bkwordre2} ? and, provided that the exact 
solution ${\bf u}^\eps$ to \eqref{eqint}, \eqref{clinhom} exists on a fixed time interval independent of $\eps$, 
does \eqref{wkbu} provide with an accurate description of ${\bf u}^\eps$ on this time interval ?  These 
questions are answered positively in Chapters \ref{chapter3}, \ref{chapter4},  and \ref{chapter5}. But before entering 
analytical issues, we are going to focus on how the previous analysis applies to the case of elastodynamics.

\section{Isotropic elastodynamics}\label{isoe}

\emph{\quad} In this section we explain how the previous analysis applies to the system of elastodynamics for hyperelastic 
materials. The results generalize those of \cite{Lardner1} to any space dimension, and put the leading order 
amplitude equation in the form \eqref{eqw} which is more convenient in view of applying the well-posedness 
results of \cite{Hunter2006,B}. We refer to \cite{ciarlet} for an introduction to elasticity and the physical 
background. More specifically, we consider an elastic material in the reference domain $\Omega = \{ x \cdot 
\bnu >0 \} \subset \R^d$. The deformation gradient is $\nabla \varphi =I+\nabla {\bf u}$, where ${\bf u}(t,x) \in 
\R^d$ represents, at a given time $t$, the displacement of the material at a point $x \in \Omega$. The space 
dimension $d$ equals either $2$ or $3$. Due to frame indifference, the elastic energy $W$ is a function of the 
so-called Cauchy-Green strain tensor $C :=\nabla \varphi^T \, \nabla \varphi \in {\mathcal M}_d(\R)$. We may 
equivalently rewrite $W$ in terms of the Green - Saint Venant strain tensor $E :=(C-I)/2$. In terms of the 
displacement gradient, this gives
$$
\forall \, \alpha,j =1,\dots,d \, ,\quad 
E_{\alpha,j} = \dfrac{1}{2} \, (u_{\alpha,j} +u_{j,\alpha} +u_{\alpha,\ell} \, u_{j,\ell}) \, .
$$
When the material is {\it isotropic}, the energy $W$ only depends on the principal invariants of the Cauchy-Green 
strain tensor $C$. Assuming that the energy $W$ is a smooth function of $E$, its Taylor expansion at $E=0$ 
then reads (see \cite[Chapter 1.4]{ciarlet}):
\begin{align}\label{generalenergy}
W(E) =\dfrac{\lambda}{2} \, ({\rm tr } \, E)^2 +\mu \, {\rm tr } \, (E^2) 
+\alpha_1 \, ({\rm tr } \, E)^3 +\alpha_2 \, ({\rm tr } \, E) \, {\rm tr } \, (E^2) +\alpha_3 \, {\rm tr } \, (E^3) 
+O(\| E \|^4) \, .
\end{align}
In the above formula, $\lambda$ and $\mu$ stand for the so-called Lam\'e coefficients of the material, 
and $\alpha_1,\alpha_2,\alpha_3$ are constants. For future use, we introduce the Frobenius norm of a 
matrix $M \in {\mathcal M}_d(\R)$:
$$
\| M \|^2 := \sum_{\alpha,j=1}^d M_{\alpha j}^2 ={\rm tr } \, (M \, M^T) \, .
$$
Using the expression of $E$ in terms of the displacement gradient $\nabla {\bf u}$, we can rewrite the 
above Taylor expansion of $W$ as
\begin{align}
W(\nabla {\bf u}) =& \, \dfrac{\lambda}{2} \, ({\rm tr } \, \nabla {\bf u})^2 
+\dfrac{\mu}{4} \, \left\| \nabla {\bf u} +\nabla {\bf u}^T \right\|^2 \notag \\
&+\beta_1 \, ({\rm tr } \, \nabla {\bf u})^3 
+\beta_2 \, ({\rm tr } \, \nabla {\bf u}) \, \left\| \nabla {\bf u} +\nabla {\bf u}^T \right\|^2 
+\beta_3 \, {\rm tr } \, (\nabla {\bf u} \, \nabla {\bf u}^T \, \nabla {\bf u}) 
+\beta_4 \, {\rm tr } \, (\nabla {\bf u}^3 ) +O(\| \nabla {\bf u} \|^4) \, .\label{taylorW}
\end{align}
For the sake of completeness, the coefficients $\beta_1,\dots,\beta_4$ in \eqref{taylorW} are given by:
\begin{equation*}
\beta_1 := \alpha_1 \, ,\quad \beta_2 := \dfrac{\lambda}{2} +\dfrac{\alpha_2}{4} \, ,\quad 
\beta_3 := \mu +\dfrac{3\, \alpha_3}{4} \, ,\quad \beta_4 := \dfrac{\alpha_3}{4} \, .
\end{equation*}
Their precise expression is not relevant for our purpose though. Let us note however that, even in the 
simplest case of the so-called Saint Venant - Kirchhoff materials, for which one has
\begin{align}\label{SVK}
\forall \, E \, ,\quad W(E) =\dfrac{\lambda}{2} \, ({\rm tr } \, E)^2 +\mu \, {\rm tr } \, (E^2) \, ,
\end{align}
the coefficients $\beta_1,\dots,\beta_4$ in the third order terms of the Taylor expansion of $W(\nabla {\bf u})$ 
are not all zero. As a matter of fact, there is not much simplification in the final expression of the amplitude 
equation if one assumes that one/some of the coefficients $\beta_1,\dots,\beta_4$ are zero, so we keep 
them all in what follows\footnote{The only real simplification occurs when $\beta_2,\beta_3,\beta_4$ are all 
zero, but this is incompatible with the definition of these coefficients since $\mu$ is nonzero and therefore 
$\beta_3,\beta_4$ cannot vanish simultaneously.}.
\bigskip

From now on, we consider the previous derivation of the amplitude equation \eqref{eqw} when $N=d \ge 2$ 
and the energy $W$ satisfies \eqref{taylorW} near $\nabla {\bf u} =0$. Let us first take a look at assumptions 
(H1), (H2), (H3). Assumption (H1) is trivially fulfilled since the Taylor expansion of $W$ starts at the second 
order. Furthermore, we compute\footnote{Some of the expressions below already appeared in \cite{BC} so 
we do not reproduce all the computations but rather give the expressions that are useful for our purpose.}:
$$
c_{\alpha j \beta \ell} =\lambda \, \delta_{\alpha j} \, \delta_{\beta \ell} 
+\mu \, (\delta_{\alpha \beta} \, \delta_{j \ell} +\delta_{\alpha \ell} \, \delta_{\alpha j}) \, ,
$$
with $\delta$ the Kronecker symbol ($\delta_{n_1 n_2}=1$ if $n_1=n_2$, zero otherwise). The above 
expression of the coefficients $c_{\alpha j \beta \ell}$ gives:
$$
\forall \, \bxi,v \in \R^d \, ,\quad c_{\alpha j \beta \ell} \, v_\alpha \, \xi_j \, v_\beta \, \xi_\ell 
=\mu \, |\bxi|^2 \, |v|^2 +(\lambda +\mu) \, (v \cdot \bxi)^2 \, .
$$
The fulfillment of Assumption (H2) is then equivalent to the well-known inequalities
$$
\mu>0 \quad \text{ and } \quad \lambda+2\, \mu>0 \, .
$$
One can then define the velocity of `shear' and `pressure' waves by
$$
c_S := \sqrt{\mu} \, ,\quad c_P := \sqrt{\lambda +2\, \mu} \, .
$$
As has now long been known, see e.g. \cite{Lardner1,SerreJFA}, the fulfillment of Assumption (H3) is equivalent 
to the additional requirement $c_P>c_S$, or equivalently $\lambda+\mu>0$. In that case, one may uniquely 
define a velocity $c_R \in (0,c_S)$ by solving the polynomial equation
\begin{equation}
\label{defcR}
\left( \dfrac{c_R^2}{2\, c_S^2} -1 \right)^4 =\left( 1-\dfrac{c_R^2}{c_S^2} \right) \, \left( 1-\dfrac{c_R^2}{c_P^2} 
\right) \, .
\end{equation}
Assumption (H3) is then satisfied with $\tau_{\rm r}(\bfeta) :=c_R \, |\bfeta|$ for all tangential wave vector 
$\bfeta$ (which obviously satisfies the homogeneity property $\tau_{\rm r}(k\, \bfeta)=k\, \tau_{\rm r}(\bfeta)$ 
for any $k>0$). Given a nonzero tangential wave vector $\bfeta$, the one-dimensional family of surface waves 
solution to \eqref{eqlin} is spanned by
\begin{equation}
\label{Rayleighwaves}
{\rm e}^{\pm i \, \tau_{\rm r}(\bfeta) \, t+i \, \bfeta \cdot x} \, V(\bnu \cdot x) \, ,\quad 
V(z) :={\rm e}^{-\omega_1 \, z} \, \big( -2\, i \, \omega_1 \, \omega_2 \, \bfeta+2\, \omega_2 \, |\bfeta|^2 \, \bnu \big) 
+{\rm e}^{-\omega_2 \, z} \, \big( \omega_1^2 +|\bfeta|^2 \big) \, \big( i \, \bfeta -\omega_2 \, \bnu \big) \, .
\end{equation}
The eigenmodes $\omega_{1,2}$ in \eqref{Rayleighwaves} are defined by:
\begin{align}\label{emodes}
\omega_1 := |\bfeta| \, \sqrt{1-\dfrac{c_R^2}{c_S^2}} \, ,\quad 
\omega_2 := |\bfeta| \, \sqrt{1-\dfrac{c_R^2}{c_P^2}} \, ,
\end{align}
and thus satisfy
\begin{equation}
\label{dispersionRayleigh}
0<\omega_1<\omega_2 \, ,\quad 4\, |\bfeta|^2 \, \omega_1 \, \omega_2 =\big( \omega_1^2 +|\bfeta|^2 \big)^2 \, .
\end{equation}
The latter equality is an equivalent form of \eqref{defcR}.
\bigskip

Our goal now is to identify the amplitude equation \eqref{eqw} with the function $V$ given by \eqref{Rayleighwaves} 
(and the corresponding $\widehat{r}(k,z)$ defined by \eqref{defw}). We first compute the group velocity
$$
\partial_{\eta_j} \tau_{\rm r}(\bfeta) =c_R \, \dfrac{\eta_j}{|\bfeta|} \, .
$$
Specifying from now on to the case $\Omega = \{ x_d>0 \}$, that is $\bnu =(0,\dots,0,1)^T$, the amplitude 
equation \eqref{eqw} reads:
$$
\partial_t w +c_R \, \dfrac{\bfeta}{|\bfeta|} \cdot \nabla_x w +{\mathcal H} \, \big( {\mathcal B}(w,w) \big) =g \, ,
$$
with ${\mathcal B}$ given by \eqref{defB}. Let us recall that the tangential wave vector $\bfeta$ has the form 
$(\eta_1,\dots,\eta_{d-1},0)^T$, and the amplitude $w$, which is defined on the boundary $\partial \Omega$ 
depends on $(t,x_1,\dots,x_{d-1},\theta)$. Let us examine more closely at the expression of the bilinear 
Fourier multiplier ${\mathcal B}$.
\bigskip

$\bullet$ \underline{The constant $c_0$.} It is defined by \eqref{defc0}. With the definition \eqref{Rayleighwaves} 
of $V$, we compute\footnote{Recall that we consider here the case $\tau=-\tau_{\rm r}(\bfeta)=-c_R \, |\bfeta|$.}
\begin{align}
c_0 &=-2\, \tau \, \int_0^{+\infty} |V(z)|^2 \, {\rm d}z \notag \\
&= -2\, \tau \, \left\{ \dfrac{4\, \omega_2^2 \, |\bfeta|^2 \, \big( \omega_1^2 +|\bfeta|^2 \big)}{2\, \omega_1} 
+\dfrac{\big( \omega_1^2 +|\bfeta|^2 \big)^2 \, \big( \omega_2^2 +|\bfeta|^2 \big)}{2\, \omega_2} 
-4\, \omega_2 \, |\bfeta|^2 \, \big( \omega_1^2 +|\bfeta|^2 \big) \right\} \notag \\
&=-4\, \tau \, |\bfeta|^2 \, (\omega_2 -\omega_1) \left\{ |\bfeta|^2 \, \left( \dfrac{\omega_2}{\omega_1} -1 \right) 
+2\, \omega_1 \, \omega_2 \right\} \, ,\label{expressionc0}
\end{align}
where we have used the dispersion relation \eqref{dispersionRayleigh} to obtain the final expression of $c_0$. 
The expression \eqref{expressionc0} gives the constant in the definition \eqref{defB} of ${\mathcal B}$. It 
remains to clarify the expression of the kernel $b$, which is given by \eqref{defnoyau}. The first task is to 
compute the coefficients $d_{\alpha j \beta \ell \gamma m}$.
\bigskip

$\bullet$ \underline{The coefficients $d_{\alpha j \beta \ell \gamma m}$.} They are obtained as the third order 
derivatives of $W$ at $0$, see \eqref{defcoeff}. Using the Taylor expansion \eqref{taylorW} of $W$, we get 
the decomposition
$$
d_{\alpha j \beta \ell \gamma m} =6\, \beta_1 \, d_{\alpha j \beta \ell \gamma m}^1 
+2\, \beta_2 \, d_{\alpha j \beta \ell \gamma m}^2 +\beta_3 \, d_{\alpha j \beta \ell \gamma m}^3 
+3\, \beta_4 \, d_{\alpha j \beta \ell \gamma m}^4 \, ,
$$
with
\begin{align}
d_{\alpha j \beta \ell \gamma m}^1 &= \delta_{\alpha j} \, \delta_{\beta \ell} \, \delta_{\gamma m} 
\, ,\label{noyau1} \\
d_{\alpha j \beta \ell \gamma m}^2 &= \delta_{\alpha j} \, \delta_{\beta \gamma} \, \delta_{\ell m} 
+\delta_{\beta \ell} \, \delta_{\alpha \gamma} \, \delta_{j m} 
+\delta_{\gamma m} \, \delta_{\alpha \beta} \, \delta_{j \ell} \, ,\label{noyau2} \\
d_{\alpha j \beta \ell \gamma m}^3 &= 
\delta_{\alpha \beta} \, (\delta_{\gamma j} \, \delta_{\ell m} +\delta_{\gamma \ell} \, \delta_{j m}) 
+\delta_{\alpha \gamma} \, (\delta_{\beta j} \, \delta_{\ell m} +\delta_{\beta m} \, \delta_{j \ell}) 
+\delta_{\beta \gamma} \, (\delta_{\alpha \ell} \, \delta_{j m} +\delta_{\alpha m} \, \delta_{j \ell}) \, ,\label{noyau3} \\
d_{\alpha j \beta \ell \gamma m}^4 &= \delta_{\alpha \ell} \, \delta_{\beta m} \, \delta_{\gamma j} 
+\delta_{\alpha m} \, \delta_{\beta j} \, \delta_{\gamma \ell} \, .\label{noyau4}
\end{align}
Each of these four coefficients is invariant through the permutations $(\alpha,j) \leftrightarrow (\beta,\ell)$,  
$(\alpha,j) \leftrightarrow (\gamma,m)$ and $(\beta,\ell) \leftrightarrow (\gamma,m)$, and therefore through 
any permutation of $(\alpha,j)$, $(\beta,\ell)$, $(\gamma,m)$. Associated with the above decomposition 
for $d_{\alpha j \beta \ell \gamma m}$, we have a decomposition of the kernel $b(\xi_1,\xi_2,\xi_3)$ in 
\eqref{defnoyau}. We thus find that in the case of isotropic elasticity, the kernel $b(\xi_1,\xi_2,\xi_3)$ is 
a linear combination with real coefficients of the `elementary' kernels
\begin{equation}
\label{defnoyauk}
b^k(\xi_1,\xi_2,\xi_3) := \int_0^{+\infty} d_{\alpha j \beta \ell \gamma m}^k \, 
(\nu_j \, \rho_{\alpha,z} +i\, \xi_1 \, \eta_j \, \rho_\alpha) \, 
(\nu_\ell \, \rho_{\beta,z} +i\, \xi_2 \, \eta_\ell \, \rho_\beta) \, 
(\nu_m \, \rho_{\gamma,z} +i\, \xi_3 \, \eta_m \, \rho_\gamma) \, {\rm d}z \, ,
\end{equation}
where $k$ runs through $\{ 1,\dots,4\}$ and the coefficients $d_{\alpha j \beta \ell \gamma m}^k$ are given 
in \eqref{noyau1}, \eqref{noyau2}, \eqref{noyau3}, \eqref{noyau4}. Let us recall eventually that the functions 
$\rho_\alpha,\rho_\beta,\rho_\gamma$ are given in \eqref{defnoyau}.

In the remaining of this Section, we are going to make the kernels $b^1,\dots,b^4$ explicit, meaning that 
we are going to express them as linear combinations of explicitly computable kernels. This will give the final 
expression of $b$ as a linear combination of explicit kernels. In order to avoid lengthy and somehow useless 
computations, we shall not try to make the coefficients in the linear combinations explicit. The interested reader 
will achieve this by using the expressions below and expanding all trilinear expressions of $\exp(-\omega_{1,2} 
\, |\xi| \, z)$ below explicitly.

From the definition \eqref{defnoyauk}, it appears that we need to compute the quantity
$$
\nu_j \, \dfrac{\partial}{\partial z} \big( \widehat{r}_\alpha (\xi,z) \big) +i\, \xi \, \eta_j \, \widehat{r}_\alpha (\xi,z) \, ,
$$
for all possible values of $\alpha,j$ and $\xi$. Using \eqref{Rayleighwaves}, we get (keeping the subscript 
`$,z$' for denoting partial differentiation with respect to $z$):
\begin{equation*}
\nu_j \, \widehat{r}_{\alpha,z} (\xi,z) +i\, \xi \, \eta_j \, \widehat{r}_\alpha (\xi,z) 
=\begin{cases}
|\xi| \, \eta_\alpha \, \eta_j \, {\mathbb T}_{11} (\xi,z) \, ,& \alpha,j \le d-1 \, ,\\
i\, \xi \, \eta_j \, \omega_2 \, {\mathbb T}_{d1} (\xi,z) \, ,& j \le d-1 \, ,\alpha=d \, ,\\
i\, \xi \, \eta_\alpha \, \omega_2 \, {\mathbb T}_{1d} (\xi,z) \, ,& \alpha \le d-1 \, ,j=d \, ,\\
-|\xi| \, \omega_2 \, {\mathbb T}_{dd} (\xi,z) \, ,& \alpha=j=d \, .
\end{cases}
\end{equation*}
with
\begin{equation}
\label{rhoalphaj}
\begin{cases}
{\mathbb T}_{11} (\xi,z) := 2\, \omega_1 \, \omega_2 \, {\rm e}^{-\omega_1 \, |\xi| \, z} 
-\big( \omega_1^2 +|\bfeta|^2 \big) \, {\rm e}^{-\omega_2 \, |\xi| \, z} \, ,& \\
{\mathbb T}_{d1} (\xi,z) := 2\, |\bfeta|^2 \, {\rm e}^{-\omega_1 \, |\xi| \, z} 
-\big( \omega_1^2 +|\bfeta|^2 \big) \, {\rm e}^{-\omega_2 \, |\xi| \, z} \, ,& \\
{\mathbb T}_{1d} (\xi,z) := 2\, \omega_1^2 \, {\rm e}^{-\omega_1 \, |\xi| \, z} 
-\big( \omega_1^2 +|\bfeta|^2 \big) \, {\rm e}^{-\omega_2 \, |\xi| \, z} \, ,& \\
{\mathbb T}_{dd} (\xi,z) := 2\, \omega_1 \, |\bfeta|^2 \, {\rm e}^{-\omega_1 \, |\xi| \, z} 
-\omega_2 \, \big( \omega_1^2 +|\bfeta|^2 \big) \, {\rm e}^{-\omega_2 \, |\xi| \, z} \, .& 
\end{cases}
\end{equation}
In particular, there holds
\begin{equation}
\label{rhojj}
\nu_j \, \widehat{r}_{j,z} (\xi,z) +i\, \xi \, \eta_j \, \widehat{r}_j (\xi,z) =|\xi| \, \big( \omega_1^2 +|\bfeta|^2 \big) 
\, \big( \omega_2^2 -|\bfeta|^2 \big) \, {\rm e}^{-\omega_2 \, |\xi| \, z} = -|\xi| \, 
\dfrac{c_R^2}{c_P^2} \, |\bfeta|^2 \, \big( \omega_1^2 +|\bfeta|^2 \big) \, {\rm e}^{-\omega_2 \, |\xi| \, z} \, .
\end{equation}
\bigskip

$\bullet$ \underline{The kernel $b^1$.} This is by far the simplest case. We start from the definition 
\eqref{defnoyauk} and use the expression \eqref{noyau1} for the coefficients $d_{\alpha j \beta \ell \gamma m}^1$. 
We get
$$
b^1(\xi_1,\xi_2,\xi_3) = \int_0^{+\infty} (\nu_j \, \rho_{j,z} +i\, \xi_1 \, \eta_j \, \rho_j) \, 
(\nu_\ell \, \rho_{\ell,z} +i\, \xi_2 \, \eta_\ell \, \rho_\ell) \, 
(\nu_m \, \rho_{m,z} +i\, \xi_3 \, \eta_m \, \rho_m) \, {\rm d}z \, ,
$$
where we warn the reader that, in the first term in the integral, $\rho_j$ is a short notation for $\widehat{r}_j(\xi_1,z)$, 
in the second term, $\rho_\ell$ is a short notation for $\widehat{r}_\ell(\xi_2,z)$ and so on. We now use \eqref{rhojj}, 
and get (here the sum of products decouples as the product of sums over $j$, $\ell$ and $m$):
$$
b^1(\xi_1,\xi_2,\xi_3) =\star \, \int_0^{+\infty} |\xi_1| \, |\xi_2| \, |\xi_3| \, 
\exp(-\omega_2 \, (|\xi_1|+|\xi_2|+|\xi_3|) \, z) \, {\rm d}z \, ,
$$
where, here and from now on, $\star$ denotes any {\it real} constant that depends only on $\bfeta$ and that can 
be explicitly computed from \eqref{rhoalphaj} or \eqref{rhojj} (though we shall not keep track of any such constant). 
We thus obtain the expression:
\begin{equation}
\label{expressionb1}
b^1(\xi_1,\xi_2,\xi_3) =\star \, \dfrac{|\xi_1| \, |\xi_2| \, |\xi_3|}{|\xi_1|+|\xi_2|+|\xi_3|} \, ,
\end{equation}
which corresponds to the simplified kernel introduced in \cite{HamiltonIlinskyZabolotskaya} and that also arises 
in incompressible magnetohydrodynamics \cite{AliHunter}. As shown in \cite{AliHunterParker}, this kernel 
corresponds in the physical space to the operator (here we forget about harmless multiplicative real constants):
$$
\dfrac{1}{2} \, \partial_{\theta \theta} \Big( {\mathcal H} \big( ({\mathcal H} \, w)^2 \big) \Big) +({\mathcal H} \, w) 
\, \partial_{\theta \theta} \, w \, .
$$
\bigskip

$\bullet$ \underline{The kernel $b^2$.} We now use the expression \eqref{noyau2} and derive
\begin{align*}
b^2(\xi_1,\xi_2,\xi_3) =& \, \int_0^{+\infty} (\nu_j \, \rho_{j,z} +i\, \xi_1 \, \eta_j \, \rho_j) \, 
(\nu_\ell \, \rho_{\beta,z} +i\, \xi_2 \, \eta_\ell \, \rho_\beta) \, 
(\nu_\ell \, \rho_{\beta,z} +i\, \xi_3 \, \eta_\ell \, \rho_\beta) \, {\rm d}z \\
& +\int_0^{+\infty} (\nu_j \, \rho_{\alpha,z} +i\, \xi_1 \, \eta_j \, \rho_\alpha) \, 
(\nu_\ell \, \rho_{\ell,z} +i\, \xi_2 \, \eta_\ell \, \rho_\ell) \, 
(\nu_j \, \rho_{\alpha,z} +i\, \xi_3 \, \eta_j \, \rho_\alpha) \, {\rm d}z \\
& +\int_0^{+\infty} (\nu_j \, \rho_{\alpha,z} +i\, \xi_1 \, \eta_j \, \rho_\alpha) \, 
(\nu_j \, \rho_{\alpha,z} +i\, \xi_2 \, \eta_j \, \rho_\alpha) \, 
(\nu_m \, \rho_{m,z} +i\, \xi_3 \, \eta_m \, \rho_m) \, {\rm d}z \, .
\end{align*}
Changing indices in the last two integrals, we have written $b^2(\xi_1,\xi_2,\xi_3)$ under the form\footnote{Here, 
and here only, $\varepsilon$ denotes the signature of a permutation.}
\begin{equation}
\label{decompb2}
b^2(\xi_1,\xi_2,\xi_3) ={\mathbb B}^2(\xi_1,\xi_2,\xi_3) +{\mathbb B}^2(\xi_2,\xi_3,\xi_1) 
+{\mathbb B}^2(\xi_3,\xi_1,\xi_2) 
=\sum_{\underset{\varepsilon({\sigma})=1}{\sigma \in \mathfrak{S}_3 \, ,}} 
{\mathbb B}^2 \big( \xi_{\sigma(1)},\xi_{\sigma(2)},\xi_{\sigma(3)} \big) \, ,
\end{equation}
where ${\mathbb B}^2$ is symmetric with respect to its last two arguments (that is, 
${\mathbb B}^2(\xi_1,\xi_3,\xi_2) ={\mathbb B}^2(\xi_1,\xi_2,\xi_3)$ for all $\xi_1,\xi_2,\xi_3$), and the 
formula \eqref{decompb2} shows that one `symmetrizes' the function ${\mathbb B}^2$ over the alternating 
group $\mathfrak{A}_3 \subset \mathfrak{S}_3$. Using \eqref{rhoalphaj} and \eqref{rhojj}, we have
\begin{align*}
{\mathbb B}^2(\xi_1,\xi_2,\xi_3) =& \star \, |\xi_1| \, \int_0^{+\infty} {\rm e}^{-\omega_2 \, |\xi_1| \, z} \, 
(\nu_\ell \, \rho_{\beta,z} +i\, \xi_2 \, \eta_\ell \, \rho_\beta) \, 
(\nu_\ell \, \rho_{\beta,z} +i\, \xi_3 \, \eta_\ell \, \rho_\beta) \, {\rm d}z \\
=& \star \, |\xi_1| \, |\xi_2| \, |\xi_3| \, \int_0^{+\infty} {\rm e}^{-\omega_2 \, |\xi_1| \, z} \, 
{\mathbb T}_{11}(\xi_2,z) \, {\mathbb T}_{11}(\xi_3,z) \, {\rm d}z \\
&+\star \, |\xi_1| \, \xi_2 \, \xi_3 \, \int_0^{+\infty} {\rm e}^{-\omega_2 \, |\xi_1| \, z} \, 
{\mathbb T}_{d1}(\xi_2,z) \, {\mathbb T}_{d1}(\xi_3,z) \, {\rm d}z \\
&+\star \, |\xi_1| \, \xi_2 \, \xi_3 \, \int_0^{+\infty} {\rm e}^{-\omega_2 \, |\xi_1| \, z} \, 
{\mathbb T}_{1d}(\xi_2,z) \, {\mathbb T}_{1d}(\xi_3,z) \, {\rm d}z \\
&+\star \, |\xi_1| \, |\xi_2| \, |\xi_3| \, \int_0^{+\infty} {\rm e}^{-\omega_2 \, |\xi_1| \, z} \, 
{\mathbb T}_{dd}(\xi_2,z) \, {\mathbb T}_{dd}(\xi_3,z) \, {\rm d}z \, .
\end{align*}
Let us examine the first integral in the latter sum of four. We need to compute
$$
\int_0^{+\infty} {\rm e}^{-\omega_2 \, |\xi_1| \, z} \, 
\Big( 2\, \omega_1 \, \omega_2 \, {\rm e}^{-\omega_1 \, |\xi_2| \, z} 
-\big( \omega_1^2 +|\bfeta|^2 \big) \, {\rm e}^{-\omega_2 \, |\xi_2| \, z} \Big) \, 
\Big( 2\, \omega_1 \, \omega_2 \, {\rm e}^{-\omega_1 \, |\xi_3| \, z} 
-\big( \omega_1^2 +|\bfeta|^2 \big) \, {\rm e}^{-\omega_2 \, |\xi_3| \, z} \Big) \, {\rm d}z \, ,
$$
then multiply by $|\xi_1| \, |\xi_2| \, |\xi_3|$ and eventually `symmetrize' with respect to $\xi_1,\xi_2,\xi_3$ 
by summing as in \eqref{decompb2}, and this will give part of the kernel $b^2$. As reported in 
\cite{H,Hunter2006}, this first integral in ${\mathbb B}^2$ gives rise in the expression of $b^2$ to a linear 
combination with real coefficients of the following three kernels:
\begin{align}
& \dfrac{|\xi_1| \, |\xi_2| \, |\xi_3|}{|\xi_1|+|\xi_2|+|\xi_3|} \, ,\label{noyauH1} \\
& |\xi_1| \, |\xi_2| \, |\xi_3| \, \left\{ \dfrac{1}{r \, |\xi_1|+|\xi_2|+|\xi_3|} +\dfrac{1}{|\xi_1|+r\, |\xi_2|+|\xi_3|} 
+\dfrac{1}{|\xi_1|+|\xi_2|+r \, |\xi_3|} \right\} \, ,\label{noyauH2} \\
& |\xi_1| \, |\xi_2| \, |\xi_3| \, \left\{ \dfrac{1}{|\xi_1|+r\, |\xi_2|+r \, |\xi_3|} +\dfrac{1}{r\, |\xi_1|+|\xi_2|+r\, |\xi_3|} 
+\dfrac{1}{r\, |\xi_1|+r \, |\xi_2|+|\xi_3|} \right\} \, ,\label{noyauH3}
\end{align}
where $r :=\omega_1 /\omega_2 \in (0,1)$ can be computed from the Lam\'e coefficients of the material. 
The kernel in \eqref{noyauH1} already appears in \eqref{expressionb1}, so the $b^1$ part of the overall 
kernel $b$ contributes to a term that is similar to one of the many in the decomposition of $b^2$.

Using the expression of ${\mathbb T}_{dd}$ in \eqref{rhoalphaj}, one can check that the last integral in 
the above decomposition of ${\mathbb B}^2$ also gives rise, after multiplication by $|\xi_1| \, |\xi_2| \, |\xi_3|$ 
and symmetrization with respect to $\xi_1,\xi_2,\xi_3$, to a linear combination with real coefficients of the 
three kernels in \eqref{noyauH1}, \eqref{noyauH2}, \eqref{noyauH3}.

We now examine the second and third integrals in the decomposition of ${\mathbb B}^2$, which involve 
the expressions ${\mathbb T}_{d1}$, ${\mathbb T}_{1d}$ in \eqref{rhoalphaj}, and that are multiplied by 
$|\xi_1| \, \xi_2 \, \xi_3$ (no absolute value for $\xi_2 \, \xi_3$ here !). For instance, we compute the integral
$$
\int_0^{+\infty} {\rm e}^{-\omega_2 \, |\xi_1| \, z} \, \Big( 2\, |\bfeta|^2 \, {\rm e}^{-\omega_1 \, |\xi_2| \, z} 
-\big( \omega_1^2 +|\bfeta|^2 \big) \, {\rm e}^{-\omega_2 \, |\xi_2| \, z} \Big) \, \Big( 
2\, |\bfeta|^2 \, {\rm e}^{-\omega_1 \, |\xi_3| \, z} 
-\big( \omega_1^2 +|\bfeta|^2 \big) \, {\rm e}^{-\omega_2 \, |\xi_3| \, z} \Big) \, {\rm d}z \, ,
$$
multiply by $|\xi_1| \, \xi_2 \, \xi_3$ and symmetrize with respect to $\xi_1,\xi_2,\xi_3$ as in \eqref{decompb2}. 
This computation gives rise in the expression of $b^2$ to a linear combination with real coefficients of the 
following kernels:
\begin{align}
& \dfrac{1}{|\xi_1|+|\xi_2|+|\xi_3|} \, \Big\{ |\xi_1| \, \xi_2 \, \xi_3 +\xi_1 \, |\xi_2| \, \xi_3 +\xi_1 \, \xi_2 \, |\xi_3|
\Big\} \, ,\label{noyauH4} \\
& \dfrac{|\xi_1| \, \xi_2 \, \xi_3}{|\xi_1|+r\, |\xi_2|+r \, |\xi_3|} +\dfrac{\xi_1 \, |\xi_2| \, \xi_3}{r\, |\xi_1|+|\xi_2|+r\, |\xi_3|} 
+\dfrac{\xi_1 \, \xi_2 \, |\xi_3|}{r\, |\xi_1|+r \, |\xi_2|+|\xi_3|} \, ,\label{noyauH5} \\
& |\xi_1| \, \xi_2 \, \xi_3 \, \left\{ \dfrac{1}{|\xi_1|+r\, |\xi_2|+|\xi_3|} +\dfrac{1}{|\xi_1|+|\xi_2|+r\, |\xi_3|} \right\} \notag \\
&+\xi_1 \, |\xi_2| \, \xi_3 \, \left\{ \dfrac{1}{r\, |\xi_1|+|\xi_2|+|\xi_3|} +\dfrac{1}{|\xi_1|+|\xi_2|+r\, |\xi_3|} 
\right\} \label{noyauH6} \\
&+\xi_1 \, \xi_2 \, |\xi_3| \, \left\{ \dfrac{1}{r\, |\xi_1|+|\xi_2|+|\xi_3|} +\dfrac{1}{|\xi_1|+r\, |\xi_2|+|\xi_3|} 
\right\} \, .\notag
\end{align}
The reader can check that, in the decomposition of ${\mathbb B}^2$, the integral that involves 
${\mathbb T}_{1d}$ also gives rise, after multiplication by $|\xi_1| \, \xi_2 \, \xi_3$ and symmetrization, 
to a linear combination with real coefficients of the three kernels in \eqref{noyauH4}, \eqref{noyauH5}, 
\eqref{noyauH6}.

At this stage, we have proved that, with the expression \eqref{noyau2} for $d^2_{\alpha j \beta \ell \gamma m}$, 
the corresponding kernel $b^2$ in \eqref{defnoyauk} is a linear combination with real coefficients of the six 
kernels given in \eqref{noyauH1}, \eqref{noyauH2}, \eqref{noyauH3}, \eqref{noyauH4}, \eqref{noyauH5} and 
\eqref{noyauH6}.
\bigskip

$\bullet$ \underline{The kernels $b^3$ and $b^4$.} Using the expression of $d^3_{\alpha j \beta \ell \gamma m}$ 
in \eqref{noyau3}, we can decompose the kernel $b^3$ under the form
\begin{align*}
b^3(\xi_1,\xi_2,\xi_3) 
&= \B^3(\xi_1,\xi_2,\xi_3) +\B^3(\xi_2,\xi_1,\xi_3) +\B^3(\xi_1,\xi_3,\xi_2) +\B^3(\xi_3,\xi_1,\xi_2) 
+\B^3(\xi_2,\xi_3,\xi_1) +\B^3(\xi_3,\xi_2,\xi_1) \\
&= \sum_{\sigma \in \mathfrak{S}_3} \B^3 \big( \xi_{\sigma(1)},\xi_{\sigma(2)},\xi_{\sigma(3)} \big) \, ,
\end{align*}
where the function $\B^3$ is defined by
\begin{equation*}
\B^3(\xi_1,\xi_2,\xi_3) := \int_0^{+\infty} (\nu_j \, \rho_{\alpha,z} +i\, \xi_1 \, \eta_j \, \rho_\alpha) \, 
(\nu_\ell \, \rho_{\alpha,z} +i\, \xi_2 \, \eta_\ell \, \rho_\alpha) \, 
(\nu_\ell \, \rho_{j,z} +i\, \xi_3 \, \eta_\ell \, \rho_j) \, {\rm d}z \, ,
\end{equation*}
and is symmetrized over the symmetric group $\mathfrak{S}_3$ to obtain the kernel $b^3$. Considering 
all possible indices $\alpha,j,\ell$, we end up with the following decomposition of ${\mathbb B}^3$:
\begin{align*}
\B^3(\xi_1,\xi_2,\xi_3) =& \star \, |\xi_1| \, |\xi_2| \, |\xi_3| \, \int_0^{+\infty} {\mathbb T}_{11}(\xi_1,z) \, 
{\mathbb T}_{11}(\xi_2,z) \, {\mathbb T}_{11}(\xi_3,z) \, {\rm d}z \\
& +\star \, |\xi_1| \, |\xi_2| \, |\xi_3| \, \int_0^{+\infty} {\mathbb T}_{dd}(\xi_1,z) \, 
{\mathbb T}_{dd}(\xi_2,z) \, {\mathbb T}_{dd}(\xi_3,z) \, {\rm d}z \\
& +\star \, \xi_1 \, \xi_2 \, |\xi_3| \, \int_0^{+\infty} {\mathbb T}_{d1}(\xi_1,z) \, 
{\mathbb T}_{d1}(\xi_2,z) \, {\mathbb T}_{11}(\xi_3,z) \, {\rm d}z \\
& +\star \, \xi_1 \, \xi_2 \, |\xi_3| \, \int_0^{+\infty} {\mathbb T}_{1d}(\xi_1,z) \, 
{\mathbb T}_{1d}(\xi_2,z) \, {\mathbb T}_{dd}(\xi_3,z) \, {\rm d}z \\
& +\star \, \xi_1 \, |\xi_2| \, \xi_3 \, \int_0^{+\infty} {\mathbb T}_{1d}(\xi_1,z) \, 
{\mathbb T}_{11}(\xi_2,z) \, {\mathbb T}_{d1}(\xi_3,z) \, {\rm d}z \\
& +\star \, \xi_1 \, |\xi_2| \, \xi_3 \, \int_0^{+\infty} {\mathbb T}_{d1}(\xi_1,z) \, 
{\mathbb T}_{dd}(\xi_2,z) \, {\mathbb T}_{1d}(\xi_3,z) \, {\rm d}z \\
& +\star \, |\xi_1| \, \xi_2 \, \xi_3 \, \int_0^{+\infty} {\mathbb T}_{11}(\xi_1,z) \, 
{\mathbb T}_{1d}(\xi_2,z) \, {\mathbb T}_{1d}(\xi_3,z) \, {\rm d}z \\
& +\star \, |\xi_1| \, \xi_2 \, \xi_3 \, \int_0^{+\infty} {\mathbb T}_{dd}(\xi_1,z) \, 
{\mathbb T}_{d1}(\xi_2,z) \, {\mathbb T}_{d1}(\xi_3,z) \, {\rm d}z \, .
\end{align*}

Computing the first two integrals in the latter decomposition of ${\mathbb B}^3$, we get linear combinations 
with real coefficients of the kernels in \eqref{noyauH1}, \eqref{noyauH2}, \eqref{noyauH3} (the first two lines 
are already symmetric with respect to $\xi_1,\xi_2,\xi_3$ so symmetrizing over $\mathfrak{S}_3$ only yields 
the harmless multiplicative factor $6$).

We now compute the third and fourth integrals in the decomposition of ${\mathbb B}^3$, which are entirely 
similar. Expanding the products in the integrals, then multiplying by $\xi_1 \, \xi_2 \, |\xi_3|$ and symmetrizing 
over $\mathfrak{S}_3$ yields\footnote{As a matter of fact, symmetrizing over $\mathfrak{A}_3$ is sufficient 
here because the third and fourth lines are already symmetric with respect to $\xi_1,\xi_2$.} a linear combination 
of the kernels \eqref{noyauH4}, \eqref{noyauH5}, \eqref{noyauH6}, and the (hopefully last!) two kernels:
\begin{align}
& \dfrac{|\xi_1| \, \xi_2 \, \xi_3}{r\, |\xi_1|+|\xi_2|+|\xi_3|} +\dfrac{\xi_1 \, |\xi_2| \, \xi_3}{|\xi_1|+r\, |\xi_2|+|\xi_3|} 
+\dfrac{\xi_1 \, \xi_2 \, |\xi_3|}{|\xi_1|+|\xi_2|+r\, |\xi_3|} \, ,\label{noyauH7} \\
& |\xi_1| \, \xi_2 \, \xi_3 \, 
\left\{ \dfrac{1}{r\, |\xi_1|+r\, |\xi_2|+|\xi_3|} +\dfrac{1}{r\, |\xi_1|+|\xi_2|+r\, |\xi_3|} \right\} \notag \\
&+\xi_1 \, |\xi_2| \, \xi_3 \, \left\{ \dfrac{1}{r\, |\xi_1|+r\, |\xi_2|+|\xi_3|} +\dfrac{1}{|\xi_1|+r\, |\xi_2|+r\, |\xi_3|} 
\right\} \label{noyauH8} \\
&+\xi_1 \, \xi_2 \, |\xi_3| \, \left\{ \dfrac{1}{r\, |\xi_1|+|\xi_2|+r\, |\xi_3|} +\dfrac{1}{|\xi_1|+r\, |\xi_2|+r\, |\xi_3|} 
\right\} \, .\notag
\end{align}
The reader can check that the remaining four integrals in the decomposition of ${\mathbb B}^3$ yield, 
after multiplication by either $\xi_1 \, |\xi_2| \, \xi_3$ or $|\xi_1| \, \xi_2 \, \xi_3$ and symmetrization 
over $\mathfrak{S}_3$, a linear combination of the eight kernels in \eqref{noyauH1}, \eqref{noyauH2}, 
\eqref{noyauH3}, \eqref{noyauH4}, \eqref{noyauH5}, \eqref{noyauH6}, \eqref{noyauH7} and \eqref{noyauH8}.

We leave it as an exercise to the interested reader to verify that with the definition \eqref{noyau4}, 
the corresponding kernel $b^4$ can be also written as a linear combination with real coefficients 
of the eight kernels given in \eqref{noyauH1}, \eqref{noyauH2}, \eqref{noyauH3}, \eqref{noyauH4}, 
\eqref{noyauH5}, \eqref{noyauH6}, \eqref{noyauH7} and \eqref{noyauH8}.
\bigskip

$\bullet$ \underline{Final simplifications.} In \cite{H,Hunter2006}, it is reported that only the kernels 
in \eqref{noyauH1}, \eqref{noyauH2}, \eqref{noyauH3} are necessary for computing the leading order 
amplitude equation. This is not necessarily in contradiction with the above computations since, according 
to \eqref{eqw}, what we really need is the value of $b$ over the triplets $(\xi_1,\xi_2,\xi_3)$ verifying 
$\xi_1+\xi_2+\xi_3=0$. Namely, we now introduce the symmetric kernel
$$
\Lambda (k,k') :=b \big( -(k+k'),k,k' \big) \, ,
$$
which allows us to rewrite ${\mathcal B}$ as
\begin{equation*}
\widehat{{\mathcal B}(w,w)} (k) :=-\dfrac{1}{4\, \pi \, c_0} \int_\R \Lambda (k-\xi,\xi) \, \widehat{w}(k-\xi) \, 
\widehat{w}(\xi) \, {\rm d}\xi \, .
\end{equation*}
It remains to examine whether for the kernels we have found in \eqref{noyauH1}, \eqref{noyauH2}, 
\eqref{noyauH3}, \eqref{noyauH4}, \eqref{noyauH5}, \eqref{noyauH6}, \eqref{noyauH7} and \eqref{noyauH8}, 
the corresponding functions $\Lambda$ are linearly independent.

Let us start with the easiest case of the kernel in \eqref{noyauH4}. Let us first observe that for all $(k,k')$, we 
have the relation
$$
|k+k'| \, \big( |k\, k'|+k\, k' \big) =(k+k') \, \big( |k|\, k' +k\, |k'| \big) \, ,
$$
and therefore
$$
\dfrac{|\xi_1| \, \xi_2 \, \xi_3 +\xi_1 \, |\xi_2| \, \xi_3 +\xi_1 \, \xi_2 \, |\xi_3|}{|\xi_1|+|\xi_2|+|\xi_3|} \, 
\Big|_{(\xi_1,\xi_2,\xi_3)=(-(k+k'),k,k')} =-\dfrac{|\xi_1| \, |\xi_2| \, |\xi_3|}{|\xi_1|+|\xi_2|+|\xi_3|} 
\Big|_{(\xi_1,\xi_2,\xi_3)=(-(k+k'),k,k')} \, ,
$$
which means that the kernel in \eqref{noyauH4} is useless in our list for decomposing the bilinear operator 
\eqref{defB}. In the same spirit, if we let $b_r$, resp. $\widetilde{b}_r$, denote the kernel in \eqref{noyauH3}, 
resp. \eqref{noyauH5}, we then compute
$$
(b_r +\widetilde{b}_r)(-(k+k'),k,k') =\dfrac{4}{1+r} \, \dfrac{|\xi_1| \, |\xi_2| \, |\xi_3|}{|\xi_1|+|\xi_2|+|\xi_3|} 
\Big|_{(\xi_1,\xi_2,\xi_3)=(-(k+k'),k,k')} \, ,
$$
which means that the kernel in \eqref{noyauH5} is also useless. We now leave it as exercise to the 
reader to verify that the kernels in \eqref{noyauH6}, \eqref{noyauH7}, \eqref{noyauH8}, when evaluated at 
$(-k-k',k,k')$, can be written as linear combinations with real coefficients of the expressions in \eqref{noyauH1}, 
\eqref{noyauH2}, \eqref{noyauH3}.

We summarize our findings in the following Proposition. Let us emphasize that our result is independent 
of the space dimension $d \ge 2$.

\begin{prop}
\label{propelas}
Consider a hyperelastic isotropic material with Lam\'e coefficients satisfying $\mu>0$, $\lambda +\mu>0$. 
Then weakly nonlinear Rayleigh waves traveling along the boundary of the half-space $\Omega := \{ x \cdot 
\bnu >0 \}$ are governed on the boundary of $\Omega$ by the amplitude equation
\begin{align}\label{ampeqn}
\partial_t w +c_R \, \dfrac{\bfeta}{|\bfeta|} \cdot \nabla_x w 
+{\mathcal H} \, \big( {\mathcal B}(w,w) \big) =g \, ,
\end{align}
where the bilinear Fourier multiplier ${\mathcal B}$ in the variable $\theta$ has the expression
$$
\forall \, k \in \R \, ,\quad 
\widehat{{\mathcal B}(w,w)} (k) =\int_\R \Lambda(k-k',k') \, \widehat{w}(k-k') \, \widehat{w}(k') \, {\rm d}k' \, ,
$$
and the kernel $\Lambda$ is a linear combination with real coefficients of the expressions in \eqref{noyauH1}, 
\eqref{noyauH2}, \eqref{noyauH3} evaluated at $(-k-k',k,k')$. 
\end{prop}

\section{Well-posedness of the amplitude equation}

In this Section, we show a well-posedness result for the amplitude equation \eqref{eqw}, with the bilinear 
operator ${\mathcal B}$ defined by \eqref{defB}. In the absence of slow spatial variables $y \in \R^{d-1}$, 
well-posedness results for \eqref{eqw} have been obtained in \cite{Hunter2006} when the variable $\theta$ 
lies in the torus $\R/(2\, \pi \, \Z)$ and in \cite{Secchi} when the variable $\theta$ lies in $\R$. Namely, in 
both \cite{Hunter2006} and \cite{Secchi}, well-posedness is proved for an equation of the form
$$
\partial_ t w +\partial_\theta \, {\mathcal A}(w,w) =0 \, ,
$$
that is obtained from \eqref{eqw} by applying a half-derivative in $\theta$ (no slow spatial variable 
$y \in \R^{d-1}$ is considered here). Applying such a half-derivative implies that the bilinear operator 
${\mathcal A}$ is defined by means of a kernel $a$ that is homogeneous degree $1/2$ while the kernel 
$b$ in \eqref{defB} is homogeneous degree $2$, see \cite{BC}. Let us also recall that for piecewise 
smooth kernels homogeneous degree $0$, well-posedness has been proved in \cite{B} under a suitable 
stability condition exhibited in \cite{H}. Our goal here is in some sense to encompass the results from these 
previous works by considering the original form \eqref{eqw} of the amplitude equation and by including the 
slow spatial variables $y \in \R^{d-1}$.

\begin{prop}
\label{propwellposed}
Let ${\bf v} \in \R^{d-1}$ be a fixed velocity vector, and assume that the kernel $b$ in \eqref{defB} is 
symmetric with respect to its arguments and that there exists a constant $C>0$ such that
\begin{equation}
\label{hypotheseb}
\forall \, (\xi_1,\xi_2,\xi_3) \in \R^3 \, ,\quad |b(\xi_1,\xi_2,\xi_3)| \le C \, |\xi_1|^{1/2} \, |\xi_2|^{1/2} \, 
|\xi_3|^{1/2} \, \min \big( |\xi_1|^{1/2},|\xi_2|^{1/2},|\xi_3|^{1/2} \big) \, .
\end{equation}
Then there exists an integer $\overline{m}$ that only depends on the space dimension $d$ such that for all 
$m \in \N$ with $m \ge \overline{m}$ and for all $R>0$, there exists a time $T=T(m,R)$ such that if the initial 
condition $u_0 \in H^m(\R^{d-1}_y \times \R_\theta ;\R)$ satisfies $\| u_0 \|_{H^m} \le R$, then there exists a 
unique $u \in {\mathcal C}([0,T];H^m(\R^{d-1}_y \times \R_\theta;\R))$ solution to the Cauchy problem
\begin{equation}
\label{cauchyamplitude}
\partial_t u +\sum_{j=1}^{d-1} v_j \, \partial_j u +{\mathcal H} \, \big( {\mathcal B}(u,u) \big) =0 \, ,\quad 
u|_{t=0}=u_0 \, .
\end{equation}
\end{prop}

When the initial condition for \eqref{eqw} vanishes but the source term $g$ in \eqref{eqw} is nonzero, one 
solves \eqref{eqw} by using the Duhamel formula. We omit the details here and focus on the solvability of 
the pure Cauchy problem for nonzero initial data and zero forcing term. In isotropic elastodynamics, the 
kernel $b$ whose expression is given in Proposition \ref{propelas} satisfies the bound \eqref{hypotheseb}, 
as shown in \cite{Hunter2006}.

\begin{proof}
Following the previous works \cite{Hunter2006,B,Secchi}, we only show here an a priori estimate for 
the solutions to the Cauchy problem. By standard arguments, see for instance \cite{TaylorIII}, a priori 
estimates can be turned into a well-posedness result as stated in Proposition \ref{propwellposed} by 
using convenient Fourier truncation approximations (recall here that the underlying space domain is 
$\R^{d-1}_y \times \R_\theta$ so Fourier analysis is readily available).

We therefore consider a solution $u \in {\mathcal C}([0,T];H^m(\R^{d-1}_y \times \R_\theta;\R))$ to the 
Cauchy problem \eqref{cauchyamplitude} and try to derive an estimate for the evolution of the $H^m$ 
norm of $u$. We shall make as if $u$ were sufficiently smooth for all manipulations below to be rigorous. 
As a matter of fact, using the Fourier expression of the $H^m$ norm, we even only deal with the $L^2$ 
norm of the functions $u$, $\partial_\theta^m u$, $\partial_y^\alpha u$ with $|\alpha|=m$ ($\alpha \in 
\N^{d-1}$). All other partial derivatives of $u$ can be dealt with by interpolating between such `extreme' 
cases. Once again, we refer to \cite{TaylorIII} for all details on such arguments. Let us first prove the 
following bounds on the operator ${\mathcal B}$.

\begin{lem}
\label{lembornesB}
Under the assumptions of Proposition \ref{propwellposed}, the bilinear operator ${\mathcal B}$ is symmetric. 
It satisfies the Leibniz rule
$$
\partial_\theta {\mathcal B}(u,v) ={\mathcal B}(\partial_\theta u,v) +{\mathcal B}(u,\partial_\theta v) \, ,
$$
and more generally the Leibniz rule at any order of differentiation in $\theta$, as well as the 
bounds\footnote{The bounds obviously extend by continuity to functions in appropriate Sobolev spaces 
and are not restricted to functions in the Schwartz class.}
\begin{align*}
\forall \, u,v,w \in {\mathcal S}(\R^{d-1} \times \R;\R) \, ,\quad 
\left| \int_{\R^{d-1} \times \R} u \, {\mathcal H} \, \big( {\mathcal B}(v,w) \big) \, {\rm d}y \, {\rm d}\theta \right| 
&\le C \, \| u \|_{L^2} \, \| v \|_{H^1} \, \| w \|_{H^{m_0}} \, ,\\
\forall \, u,v \in {\mathcal S}(\R^{d-1} \times \R;\R) \, ,\quad 
\left| \int_{\R^{d-1} \times \R} u \, {\mathcal H} \, \big( {\mathcal B}(u,v) \big)\, {\rm d}y \, {\rm d}\theta \right| 
&\le C \, \| v \|_{H^{m_0+1}} \, \| u \|_{L^2}^2 \, ,
\end{align*}
for a suitable constant $C$ and any integer $m_0$ satisfying $m_0 >(d-1)/2 +2$. (The Sobolev norms refer 
to the space domain $\R^{d-1}_y \times \R_\theta$.)
\end{lem}

\begin{proof}
The fact that ${\mathcal B}$ is symmetric comes from the symmetry of the kernel $b$ with respect to its 
three arguments. We now consider three functions $u,v,w$ in the Schwartz space ${\mathcal S}(\R^{d-1}_y 
\times \R_\theta;\R)$. Below the `hat' notation stands for the partial Fourier transform with respect to 
$\theta \in \R$. Applying Plancherel's Theorem, we get\footnote{In the computations below, we do not use 
the fact that $u,v,w$ are real valued.}
\begin{equation*}
\int_{\R^{d-1} \times \R} u \, {\mathcal H} \, \big( {\mathcal B}(v,w) \big) \, {\rm d}y \, {\rm d}\theta 
=i \int_{\R^{d-1} \times \R \times \R} \overline{\widehat{u}(y,k)} \, \widehat{v}(y,k-k') \, \widehat{w}(y,k') \, 
\text{\rm sgn} (-k) \, b(-k,k-k',k') \, {\rm d}y \, {\rm d}k \, {\rm d}k' \, .
\end{equation*}
We use the bound on $b$ together with the inequality
$$
|k|^{1/2} \le C \, \Big( |k-k'|^{1/2} +|k'|^{1/2} \Big) \, ,
$$
and obtain
\begin{multline*}
\left| \int_{\R^{d-1} \times \R} u \, {\mathcal H} \, \big( {\mathcal B}(v,w) \big) \, {\rm d}y \, {\rm d}\theta \right| 
\le C \, \int_{\R^{d-1} \times \R \times \R} |\widehat{u}(y,k)| \, |\widehat{v}(y,k-k')| \, |\widehat{w}(y,k')| \, \times \\
\qquad \qquad \qquad \Big( 
|k-k'| \, |k'|^{1/2} +|k-k'|^{1/2} \, |k'| \Big) \, \min (|k|^{1/2},|k-k'|^{1/2},|k'|^{1/2}) \, {\rm d}y \, {\rm d}k \, {\rm d}k' \\
\le C \, \int_{\R^{d-1} \times \R \times \R} |\widehat{u}(y,k)| \, |k-k'| \, |\widehat{v}(y,k-k')| \, |k'| \, |\widehat{w}(y,k')| 
\, {\rm d}y \, {\rm d}k \, {\rm d}k' \, .
\end{multline*}
We then apply the classical Young inequality $L^2 * L^1 \rightarrow L^2$, use Plancherel Theorem again and get
\begin{align*}
\left| \int_{\R^{d-1} \times \R} u \, {\mathcal H} \, \big( {\mathcal B}(v,w) \big) \, {\rm d}y \, {\rm d}\theta \right| 
&\le C \, \int_{\R^{d-1}} \| \widehat{u}(y,\cdot) \|_{L^2} \, \| k\, \widehat{v}(y,\cdot) \|_{L^2} \, 
\| k\, \widehat{w}(y,\cdot) \|_{L^1} \, {\rm d}y \\
&\le C \, \int_{\R^{d-1}} \| u(y,\cdot) \|_{L^2} \, \| \partial_\theta v(y,\cdot) \|_{L^2} \, 
\| (1+k^2)^{1/2} \, k\, \widehat{w}(y,\cdot) \|_{L^2} \, {\rm d}y \\
&\le C \, \| u \|_{L^2} \, \| \partial_\theta v \|_{L^2} \, \sup_{y \in \R^{d-1}} \| w(y,\cdot) \|_{H^2(\R)} \, .
\end{align*}
We then apply the Sobolev imbedding Theorem and obtain the first estimate of Lemma \ref{lembornesB}.

Let us now turn to the case $u=v$ where we expect some cancelation arising from the skew-symmetric 
operator ${\mathcal H}$. We compute
\begin{multline*}
\int_{\R^{d-1} \times \R} u \, {\mathcal H} \, \big( {\mathcal B}(u,v) \big) \, {\rm d}y \, {\rm d}\theta 
=i\, \int_{\R^{d-1} \times \R \times \R} \overline{\widehat{u}(y,k)} \, \widehat{u}(y,k-k') \, \widehat{v}(y,k') \, 
\text{\rm sgn} (-k) \, b(-k,k-k',k') \, {\rm d}y \, {\rm d}k \, {\rm d}k' \\
=i\, \int_{\R^{d-1} \times \R \times \R} \widehat{u}(y,-k) \, \widehat{u}(y,k-k') \, \widehat{v}(y,k') \, 
\text{\rm sgn} (-k) \, b(-k,k-k',k') \, {\rm d}y \, {\rm d}k \, {\rm d}k' \\
=\dfrac{i}{2}\, \int_{\R^{d-1} \times \R \times \R} \widehat{u}(y,-k) \, \widehat{u}(y,k-k') \, \widehat{v}(y,k') \, 
\Big( \text{\rm sgn} (-k) +\text{\rm sgn} (k-k') \Big) \, b(-k,k-k',k') \, {\rm d}y \, {\rm d}k \, {\rm d}k' \, ,
\end{multline*}
where we have now used the fact that $u$ is real valued, and the symmetry of $b$. Let us observe 
that if $-k$ and $k-k'$ have opposite signs, then the quantity $\text{\rm sgn} (-k) +\text{\rm sgn} (k-k')$ 
vanishes. If $-k$ and $k-k'$ have the same sign, then the sum of signs is either $2$ or $-2$, and there 
holds
$$
|k| \le |k'| \, ,\quad \text{\rm and } \quad |k-k'| \le |k'| \, .
$$
This yields
\begin{multline*}
\Big| \big( \text{\rm sgn} (-k) +\text{\rm sgn} (k-k') \big) \, b(-k,k-k',k') \Big| \\
\le C \, \Big| \text{\rm sgn} (-k) +\text{\rm sgn} (k-k') \Big| \, |k|^{1/2} \, |k-k'|^{1/2} \, |k'|^{1/2} \, 
\min (|k|^{1/2},|k-k'|^{1/2},|k'|^{1/2}) \le C \, |k'|^2 \, .
\end{multline*}
Using similar inequalities as above (convolution, Cauchy-Schwarz), we can then derive the bound
\begin{align*}
\left| \int_{\R^{d-1} \times \R} u \, {\mathcal H} \, \big( {\mathcal B}(u,v) \big) \, {\rm d}y \, {\rm d}\theta \right| 
&\le C \, \int_{\R^{d-1} \times \R \times \R} | \widehat{u}(y,-k)| \, |\widehat{u}(y,k-k')| \, |k'|^2 \, |\widehat{v}(y,k')| 
\, {\rm d}y \, {\rm d}k \, {\rm d}k' \\
&\le C \, \int_{\R^{d-1}} \| \widehat{u}(y,\cdot) \|_{L^2} \, \| \widehat{u}(y,\cdot) \|_{L^2} \, 
\| k^2\, \widehat{v}(y,\cdot) \|_{L^1} \, {\rm d}y \\
&\le C \, \| u \|_{L^2}^2 \, \sup_{y \in \R^{d-1}} \| v(y,\cdot) \|_{H^3(\R)} \, .
\end{align*}
Applying Sobolev imbedding Theorem completes the proof of Lemma \ref{lembornesB}.
\end{proof}

Let us consider an integer $m \ge m_0+1$ with $m_0$ as in Lemma \ref{lembornesB}. We consider a 
sufficiently smooth solution $u$ to \eqref{cauchyamplitude} and compute (the transport terms with 
respect to the variables $y$ are harmless here):
$$
\dfrac{{\rm d}}{{\rm d}t} \| u(t) \|_{L^2}^2 =-2 \, \int_{\R^{d-1} \times \R} u \, {\mathcal H} \, {\mathcal B}(u,u) 
\, {\rm d}y \, {\rm d}\theta \, .
$$
Applying the first bound in Lemma \ref{lembornesB}, we get
\begin{equation}
\label{apriori1}
\left| \dfrac{{\rm d}}{{\rm d}t} \| u(t) \|_{L^2}^2 \right| \le C \, \| u(t) \|_{L^2} \, \| u(t) \|_{H^1} \, \| u(t) \|_{H^{m_0}} 
\le C \,  \| u(t) \|_{H^m}^3 \, .
\end{equation}

Let us now differentiate $m$ times \eqref{cauchyamplitude} with respect to $\theta$, and get
$$
\partial_t \partial_\theta^m u +\sum_{j=1}^{d-1} v_j \, \partial_j \partial_\theta^m u +2\, {\mathcal H} \, \big( 
{\mathcal B}(\partial_\theta^m u,u) \big) =-\sum_{m'=1}^{m-1} \binom {m} {m'} \, 
{\mathcal H} \, \big( {\mathcal B}(\partial_\theta^{m'} u,\partial_\theta^{m-m'} u) \big) \, .
$$
Taking the $L^2$ scalar product with $\partial_\theta^m u$, we get
\begin{align*}
\dfrac{{\rm d}}{{\rm d}t} \| \partial_\theta^m u(t) \|_{L^2}^2 =&-4 \, \int_{\R^{d-1} \times \R} \partial_\theta^m u \, 
{\mathcal H} \, \big( {\mathcal B}(\partial_\theta^m u,u) \big) \, {\rm d}y \, {\rm d}\theta \\
&-2 \, \sum_{m'=1}^{m-1} \binom {m} {m'} \, \int_{\R^{d-1} \times \R} \partial_\theta^m u \, 
{\mathcal H} \, \big( {\mathcal B}(\partial_\theta^{m'} u,\partial_\theta^{m-m'} u) \big) \, {\rm d}y \, {\rm d}\theta \, .
\end{align*}
For the first integral, we apply the second estimate of Lemma \ref{lembornesB} and get
$$
\left| \int_{\R^{d-1} \times \R} \partial_\theta^m u \, {\mathcal H} \, \big( {\mathcal B}(\partial_\theta^m u,u) 
\big) \, {\rm d}y \, {\rm d}\theta \right| \le C \, \| u(t) \|_{H^{m_0+1}} \, \| u(t) \|_{H^m}^2 
\le C \, \| u(t) \|_{H^m}^3 \, .
$$
For the remaining terms, we apply the first estimate of Lemma \ref{lembornesB}. Assuming without loss of 
generality $m-1 \ge m' \ge m-m' \ge 1$, we get
\begin{align*}
\left| \int_{\R^{d-1} \times \R} \partial_\theta^m u \, {\mathcal H} \, \big( {\mathcal B}(\partial_\theta^{m'} u, 
\partial_\theta^{m-m'} u) \big) \, {\rm d}y \, {\rm d}\theta \right| &\le C \, \| u(t) \|_{H^m} \, \| u(t) \|_{H^{m'+1}} \, 
\| u(t) \|_{H^{m-m'+m_0}} \\
&\le C \, \| u(t) \|_{H^m}^2 \, \| u(t) \|_{H^{m-m'+m_0}} \, .
\end{align*}
We now choose $m \ge \overline{m}$ with $\overline{m} := 2\, m_0$, $m_0$ as in Lemma \ref{lembornesB}. 
Then $m-m' \le m/2$ and $m_0 \le m/2$ so that collecting all previous inequalities we get
\begin{equation}
\label{apriori2}
\left| \dfrac{{\rm d}}{{\rm d}t} \| \partial_\theta^m u(t) \|_{L^2}^2 \right| \le C \, \| u(t) \|_{H^m}^3 \, .
\end{equation}

The $y$-partial derivatives of $u$ are estimated in an entirely similar way, using the Leibniz rule (with respect 
to $y$) for the bilinear term ${\mathcal B}(u,u)$. We omit the details, which are entirely similar to what has been 
done above for the $\theta$ $m$-th derivative. Collecting \eqref{apriori1}, \eqref{apriori2} and the analogous 
estimates for the $y$-derivatives and all cross/intermediate $\theta,y$ derivatives, we end up with
$$
\left| \dfrac{{\rm d}}{{\rm d}t} \| u(t) \|_{H^m}^2 \right| \le C \, \| u(t) \|_{H^m}^3 \, ,
$$
as long as $m \ge \overline{m} =2\, m_0$, and $u$ is a sufficiently smooth solution to \eqref{cauchyamplitude}. 
The latter differential inequality provides with a control of $ \| u(t) \|_{H^m}$ in terms of $ \| u_0 \|_{H^m}$ 
on a time interval that depends on $m$ (because the above constant $C$ does depend on $m$). Applying 
standard regularization procedures as in \cite{TaylorIII}, we can prove well-posedness for 
\eqref{cauchyamplitude} in the Sobolev space $H^m$, $m \ge \overline{m}$, as stated in Proposition 
\ref{propwellposed}.
\end{proof}

\chapter{Existence of exact solutions}
\label{chapter3}

\section{Introduction}\label{intro}

\emph{\quad}Our main focus in the remaining chapters  is to provide rigorous answers to the basic questions of geometric optics for surface \emph{pulses} in isotropic hyperelastic materials.    We will describe how the results extend to wavetrains in section \ref{wavetrains}.   Although many pieces of the argument  work just as well in higher dimensions,  there is one piece that requires us to assume $d=2$ in our main results, Theorems \ref{uniformexistence} and \ref{approxthm}. 
At several points our estimates rely on the use of Kreiss symmetrizers and, as we explain in section \ref{higherD}, there is a serious difficulty with the construction of Kreiss symmetrizers for linearized elasticity in $d\geq 3$.

    We let the unknown  $\phi=(\phi_1,\phi_2)(t,x)$ represent the deformation of  an isotropic, hyperelastic,  Saint Venant-Kirchhoff (SVK) material in the reference configuration $\omega=\{x=(x_1,x_2):x_2>0\}$, which is  subjected to a surface force $g=g(t,x)$.   Here $\phi(t,\cdot):\omega\to \mathbb{R}^2$ and $g(t,\cdot):\partial\omega\to \mathbb{R}^2$.
The equations are  a second-order, nonlinear $2\times 2$ system\footnote{The interior equation is quasilinear, while the boundary equation is fully nonlinear.}
\begin{align}\label{a0}
\begin{split}
&\partial_t^2\phi -\mathrm{Div}(\nabla\phi\;\sigma(\nabla\phi))=0\text{ in }x_2>0\\
& \nabla\phi\;\sigma(\nabla\phi)n=g\text{ on }x_2=0,\; \;\\
&\phi(t,x)=x \text{ and }g=0 \text{ in }t\leq 0,
\end{split}
\end{align}
where  $n=(-1,0)$  is  the outer unit normal to the boundary of $\omega$, $\nabla\phi=(\partial_{x_j}\phi_i)_{i,j=1,2}$ 
is the spatial gradient matrix,  $\sigma$ is the stress $\sigma(\nabla\phi)=\lambda \mathrm{tr} E\cdot I+2\mu E$ with Lam\'e constants $\lambda$ and $\mu$  strictly positive,  and $E$ is the strain $E(\nabla\phi)=\frac{1}{2}({}^t\nabla\phi\cdot\nabla \phi-I)$.  Here ${}^t\nabla\phi$ denotes the transpose of $\nabla\phi$, $\mathrm{tr} E$ is the trace of the matrix $E$, and 
\begin{align}
\mathrm{Div}M=(\sum^2_{j=1}\partial_jm_{i,j})_{i=1,2}\text{ for a matrix }  M=(m_{i,j})_{i,j=1,2}.
\end{align}

We recall that the stored energy function $W(E)$ for an SVK material  \eqref{SVK} is the leading part, quadratic in $E$, of the general isotropic hyperelastic energy given by \eqref{generalenergy}.\footnote{When rewritten in terms of $\nabla\phi$ or $\nabla U=\nabla \phi -I$, the stored energy function is fourth order in those arguments.}  In section \ref{generalisotropic} we explain how our results extend to this more general case.  It will simplify the exposition to work initially with the SVK problem, which contains all the main difficulties.

The system \eqref{a0} has the form
\begin{align}\label{a1}
\begin{split}
&\partial_t^2\phi +\sum_{|\alpha|=2} \mathcal{A}_\alpha(\nabla\phi)\partial_x^\alpha\phi=0\text{ in }x_2>0\\
&B(\nabla\phi)=g\text{ or, equivalently, }\partial _{x_2}\phi=\mathcal{H}(\partial_{x_1}\phi,g)\text{ on }x_2=0\\
&\phi(t,x)=x \text{ and }g=0 \text{ in }t\leq 0,
\end{split}
\end{align}
where the matrices $\mathcal{A}_\alpha$ are symmetric, and the real functions $\mathcal{A}_\alpha(\cdot)$, $B(\cdot)$,  and $\mathcal{H}(\cdot)$ are $C^\infty$ (in fact, analytic) in their arguments.  The second form of the boundary condition is obtained  by an application of the implicit function theorem to the equation $B(\nabla \phi)=g$,  and is valid for $|\nabla\phi-I_{2\times 2}|_{L^\infty}$ small.    Set $x'=(t,x_1)$.

Defining the displacement $U(t,x)=\phi(t,x)-x$, we rewrite \eqref{a1} as 

\begin{align}\label{a2}
\begin{split}
&\partial_t^2 U+\sum_{|\alpha|=2} A_\alpha(\nabla U)\partial_x^\alpha U=0\text{ in }x_2>0\\
&h(\nabla U)=g\text{ or }\partial _{x_2} U=H(\partial_{x_1} U,g)\text{ on }x_2=0\\
&U(t,x)=0 \text{ and }g(t,x_1)=0 \text{ in }t\leq 0,
\end{split}
\end{align}
where the functions $A_\alpha$, $h$,  and $H$ are related to $\mathcal{A}_\alpha$, $B$, and $\mathcal{H}$ in the obvious way.   Here $H$ is defined near $(0,0)$ and satisfies
\begin{align}\label{a2a}
H(0,0)=0.
\end{align}

We take $g(t,x_1)$ to be an $H^s$ pulse with the weakly nonlinear scaling:
\begin{align}\label{a3}
g=g^\eps(t,x_1)=\eps^2 G\left(t,x_1,\frac{\beta\cdot (t,x_1)}{\eps}\right),
\end{align}
where $G(x',\theta)\in H^s(\mathbb{R}^2\times \mathbb{R})$ for some large enough $s$ to be specified, and $\beta\in \mathbb{R}^2\setminus 0$ is a frequency in the elliptic region of  (the linearization at $0$ of) \eqref{a2}, chosen so that the uniform Lopatinskii condition fails at $\beta$.\footnote{See definition \ref{lopa}.   We may take $\beta=(\pm\tau_r(\mathbf{\eta}),\mathbf{\eta})$, for $\tau_r(\mathbf{\eta})$ as in (H3) of chapter \ref{chapter2} and $\bf{\eta}=1$.}
We will refer to $\beta$ as a \emph{Rayleigh frequency}.  We expect the response $U^\eps(t,x)$ to be a Rayleigh wave, or rather, a Rayleigh \emph{pulse}, propagating in the boundary.

The main step in constructing the leading term of a geometric optics approximation to the pulse $U^\eps$  was given in Proposition \ref{propwellposed} of Chapter \ref{chapter2}.  In chapter \ref{chapter4} we complete the construction of the leading term and first corrector,  thereby producing an approximate solution to the SVK system.   The basic questions arise of  
whether an exact pulse solution exists on a fixed time interval independent of $\eps$, and, if so, whether the approximate solution is ``close" in some sense to  the exact solution on such a time interval.   These questions are answered in Theorems \ref{uniformexistence} and \ref{approxthm}. 
In this chapter  we resolve the first question in the affirmative and prove Theorem \ref{uniformexistence}.  In chapter \ref{chapter5} we complete the proof of Theorem \ref{approxthm}, which shows that the approximate solution is close in a precise sense to the exact solution for small $\epsilon$.\footnote{We note that for a fixed $\eps$ the existence of $U^\eps$ on a time interval $(-\infty,T_\eps]$ that may depend on $\eps$ follows from the main result of \cite{S-T}.}

We look for $U^\eps(t,x)$ in the form
\begin{align}\label{aa4}
U^\eps(t,x)=u^\eps(t,x,\theta)|_{\theta=\frac{\beta\cdot (t,x_1)}{\eps}},
\end{align}
where $u^\eps(t,x,\theta)$ satisfies the \emph{singular system} obtained from \eqref{a2} by plugging in the ansatz \eqref{aa4}:

\begin{align}\label{a5}
\begin{split}
&\partial_{t,\eps}^2 u^\eps+\sum_{|\alpha|=2} A_\alpha(\nabla_\eps u^\eps)\partial_{x,\eps}^\alpha u^\eps=0\text{ in }\{(t,x,\theta):x_2>0\}\\
&\partial _{x_2} u^\eps=H(\partial_{x_1,\eps} u^\eps,\eps^2 G(x',\theta))\text{ on }x_2=0\\
&u^\eps(t,x)=0 \text{ and }G(x',\theta)=0 \text{ in }t\leq 0.
\end{split}
\end{align}
With $\beta=(\beta_0,\beta_1)$ and $\alpha=(\alpha_1,\alpha_2)$ the notation in \eqref{a5} is:
\begin{align}\label{a6}
\begin{split}
&\partial_{t,\eps}=\partial_t+\frac{\beta_0}{\eps}\partial_\theta,\;\;\partial_{x_1,\eps}=\partial_{x_1}+\frac{\beta_1}{\eps}\partial_\theta\\
&\nabla_\eps=(\partial_{x_1,\eps},\partial_{x_2}),\;\;\partial_{x,\eps}^\alpha=\partial_{x_1,\eps}^{\alpha_1}\partial_{x_2}^{\alpha_2}.
\end{split}
\end{align}
The main effort of this paper is devoted to proving existence of a solution to the singular system  \eqref{a5} on a fixed time interval independent of $\eps$.   We refer to the system as ``singular" not only because of the factors of $\frac{1}{\eps}$ that appear, but also because $\partial_{x'}$ and $\partial_{\theta}$ derivatives occur in the linear combinations \eqref{a6}. 
To study such systems we require pseudodifferential operators that are singular in the same sense; a calculus of such operators is described in the appendix.

Using an idea employed by \cite{S-T} in the nonsingular case, we set $v^\eps=(v^\eps_1,v^\eps_2)$, formally make the substitutions  
\begin{align}
v_1^\eps=\partial_{x_1,\eps}u^\eps\text{ and }v^\eps_2=\partial_{x_2}u^\eps,
\end{align}
in the system  obtained by differentiating the interior and boundary equations of \eqref{a5} with respect to   $\partial_{x_1,\eps}$:
\begin{align}\label{a7}
\begin{split}
&\partial_{t,\eps}^2 v_1^\eps+\sum_{|\alpha|=2} A_\alpha(v^\eps)\partial_{x,\eps}^\alpha v_1^\eps=-\sum_{|\alpha|=2,\alpha_1\geq 1}\partial_{x_1,\eps}(A_\alpha(v^\eps))\partial_{x_1,\eps}^{\alpha_1-1}\partial_{x_2}^{\alpha_2}v_1^\eps-\partial_{x_1,\eps}(A_{(0,2)}(v^\eps))\partial_{x_2}v_2^\eps\\
&\partial _{x_2} v^\eps_1-d_{v_1}H(v^\eps_1,h(v^\eps))\partial_{x_1,\eps}v^\eps_1=d_gH(v^\eps_1,\eps^2G)\partial_{x_1,\eps}(\eps^2G)\text{ on }x_2=0.
\end{split}
\end{align}
Similarly, differentiating the interior equation of \eqref{a5} with respect to $x_2$ and making the same substitutions we obtain
\begin{align}\label{a8}
\begin{split}
&\partial_{t,\eps}^2 v_2^\eps+\sum_{|\alpha|=2} A_\alpha(v^\eps)\partial_{x,\eps}^\alpha v_2^\eps=-\sum_{|\alpha|=2,\alpha_1\geq 1}\partial_{x_2}(A_\alpha(v^\eps))\partial_{x_1,\eps}^{\alpha_1-1}\partial_{x_2}^{\alpha_2}v^\eps_1-\partial_{x_2}(A_{(0,2)}(v^\eps))\partial_{x_2}v^\eps_2\\
&v^\eps_2=H(v^\eps_1,\eps^2 G)\text{ on }x_2=0.
\end{split}
\end{align}
Equations \eqref{a7} and \eqref{a8} are a coupled nonlinear system for the unknown $v^\eps=(v^\eps_1,v^\eps_2)$.  Once $v^\eps$ is determined, one can solve the following linear system for $u^\eps$.  By Remark \ref{a10} this system is equivalent to  \eqref{a5}:

\begin{align}\label{a9}
\begin{split}
&\partial_{t,\eps}^2 u^\eps+\sum_{|\alpha|=2} A_\alpha(v^\eps)\partial_{x,\eps}^\alpha u^\eps=0\\\
&\partial _{x_2} u^\eps-d_{v_1}H(v^\eps_1,h(v^\eps))\partial_{x_1,\eps}u^\eps=H(v^\eps_1,\eps^2 G(x',\theta))-d_{v_1}H(v^\eps_1,\eps^2 G)v^\eps_1\text{ on }x_2=0.
\end{split}
\end{align}
Recall that $u^\eps$, $v^\eps$ and $G$ all vanish in $t<0$.

\begin{rem}\label{a10}
We show in section \ref{local} that for each fixed $\eps\in (0,1]$, the problems \eqref{a7}, \eqref{a8}, and \eqref{a9} have solutions on a time interval $(-\infty,T_\eps]$ that may depend on $\eps$.
Moreover, as a consequence of  Proposition \ref{localex},  we have $v^\eps_1=\partial_{x_1,\eps}u^\eps$ and $v^\eps_2=\partial_{x_2}u^\eps$ on $(-\infty,T_\eps]$.   The work in sections \ref{main} to \ref{mainestimate} will show that $T_\eps$ can be chosen independent of $\eps$.

\end{rem}

\subsection{Assumptions}\label{assumptions}

 \emph{\quad} We make only one assumption, namely, that we are considering the equations of nonlinear elasticity with ``traction" boundary conditions for a Saint Venant-Kirchhoff material on a 2D half-space.   More precisely, we assume:
 \begin{itemize}
\item[(A1)] We  study  the equations \eqref{a2} with Lam\'e constants satisfying $\mu>0$, $\lambda+\mu>0$.  The interior and boundary operators in these equations are the same as the operators appearing in \eqref{eqint}, \eqref{cl}, where $N=d=2$ and $W(\nabla u)$ is given by  \eqref{taylorW} for  the Saint Venant-Kirchhoff stored energy $W(E)$ as in \eqref{SVK}.   
\end{itemize}

In section \eqref{generalisotropic} we will replace assumption (A1) by the more general assumption: 
  \begin{itemize}
\item[(A1g)] We  study  the equations \eqref{a2} with Lam\'e constants satisfying $\mu>0$, $\lambda+\mu>0$.  The interior and boundary operators in these equations are the operators appearing in \eqref{eqint}, \eqref{cl}, where $N=d=2$ and $W(\nabla u)$ is given by \eqref{taylorW} for $W(E)$ as in \eqref{generalenergy}, the general isotropic hyperelastic stored energy.   Moreover, we assume that $W(E)$ is an analytic function.\footnote{We consider only displacements with $\nabla u$ small, so it is enough to assume that $W(E)$ is analytic near $E=0$.}
\end{itemize}

\begin{rem}
The discussion in section \ref{isoe} shows that under either of these assumptions, the hypotheses (H1), (H2), (H3) of chapter \ref{chapter2} are satisfied.
\end{rem}


 We recall from  \eqref{a2} that the system satisfied by  the displacement $U(t,x)=\phi(t,x)-x$ is:

\begin{align}\label{d1}
\begin{split}
&\partial_t^2 U+\sum_{|\alpha|=2} A_\alpha(\nabla U)\partial_x^\alpha U=0\text{ in }x_2>0\\
&\partial _{x_2} U=H(\partial_{x_1} U,g)\text{ on }x_2=0\\
&U(t,x)=0 \text{ and }g(t,x_1)=0 \text{ in }t\leq 0,
\end{split}
\end{align}

The ``background state" is $\nabla U=0$ and corresponds to the identity deformation $\phi(t,x)=x$.  Let us write the arguments of $H$ as $(v_1,g)$. 
The linearized system at the background state is 
$(P^0,B^0)$, where 
\begin{align}\label{d2}
\begin{split}
&P^0=\partial_t^2+\sum_{|\alpha|=2}A_\alpha(0)\partial_x^\alpha\\
&B^0=\partial_{x_2}-d_{v_1}H(0,0)\partial_{x_1}.
\end{split}
\end{align}
The perturbed linearized system is, with $h(v)$ as in \eqref{a2},
\begin{align}\label{d3}
\begin{split}
&P^v u=\partial_t^2 u+\sum_{|\alpha|=2}A_\alpha(v)\partial_x^\alpha u=f\\
&B^v u =\partial_{x_2}u-d_{v_1}H(v_1,h(v))\partial_{x_1}u = g.
\end{split}
\end{align}

\begin{rem}\label{d4}
For $(v,G)$ near $(0,0)$ we have $h(v_1,v_2)=G\Leftrightarrow v_2=H(v_1,G)$.  The linearized boundary operator at $v$ can be written 
\begin{align}
h_{v_1}(v)\partial_{x_1}+h_{v_2}(v)\partial_{x_2}
\end{align}
By the chain rule applied to $h(v_1,H(v_1,G))=G$ we have 
\begin{align}
(h_{v_2}(v))^{-1}h_{v_1}(v)=-H_{v_1}(v_1,h(v)),
\end{align} 
which explains the form of $B^\mu$.

\end{rem}

\textbf{Lopatinskii determinant.}\footnote{The remainder of this section discusses some consequences of Assumption (A1), and is somewhat more technical.  It may be read after the statement of the main results, section \ref{mainr}, and after the  survey of the proofs, section \ref{survey}.}\;Below we let $(\sigma,\xi_1,\xi_2)$ denote real variables dual to $(t,x_1,x_2)$, set $\tau=\sigma-i\gamma$ for $\gamma\geq 0$, and 
let $\Lambda_\cD$ denote the (nonsingular) operator defined by the Fourier multiplier $\Lambda(\sigma,\xi_1,\gamma):=(\sigma^2+\xi_1^2+\gamma^2)^{1/2}$.

In order to define the Lopatinskii determinant 
we set $U=(e^{\gamma t}\Lambda_{\cD}e^{-\gamma t}u ,D_{x_2}u)=(\Lambda_{\cD,\gamma} u,D_{x_2} u)$, $U^\gamma=e^{-\gamma t}U$,  and rewrite \eqref{d3} as a $4\times 4$  first-order system:
\begin{align}\label{d4a}
D_{x_2}U^\gamma-\cA(v,\cD)U^\gamma=\tilde{f}^\gamma,\;\;\cB(v,\cD)U^\gamma=\tilde{g}^\gamma,
\end{align}
where, with the matrices $A_\alpha$ evaluated at $v$, we have
\begin{align}\label{d4b}
\begin{split}
& \tilde{f}:=\begin{pmatrix}0\\-A_{(0,2)}^{-1}f\end{pmatrix},\;\;\tilde{g}=-ig\\
&\cA(v,\cD)=-\begin{pmatrix}0&\Lambda_{\cD}I_2\\\left(A_{(0,2)}^{-1}(D_{t}-i\gamma)^2+A_{(0,2)}^{-1}A_{(2,0)}D_{x_1}^2\right)\Lambda^{-1}_{\cD}&A_{(0,2)}^{-1}A_{(1,1)}D_{x_1}\end{pmatrix}\\
&\cB(v,\cD)=\begin{pmatrix}-d_{v_1}H(v_1,h(v))D_{x_1}\Lambda^{-1}_{\cD}&I_2\end{pmatrix}.
\end{split}
\end{align}
Denoting the matrix symbols of $\cA(v,\cD)$ and $\cB(v,\cD)$ by $\cA(v,\sigma,\xi_1,\gamma)$, $\cB(v,\sigma,\xi_1,\gamma)$ and setting $z=(\sigma,\xi_1,\gamma)$,   when $\gamma>0$ we let $E^+(v,z)$ denote the direct sum of the generalized eigenspaces of $\cA(v,z)$ associated to the eigenvalues with positive imaginary part.   This two-dimensional space has a continuous extension $E^+(\sigma,\xi_1,0)$ to points $(\sigma,\xi_1,0)\neq 0$.

\begin{defn}\label{lopa}
Let $\mathcal{X}:=\{z=(\sigma,\xi_1,\gamma)\neq 0:\gamma\geq 0\}$.

1)  The operators $(\cA(v,\cD),\cB(v,\cD))$ (or $(P^v,B^v)$) satisfy the Lopatinskii condition at $(v,z)$ provided
\begin{align}
\cB(v,z):E^+(v,z)\to \mathbb{C}^2
\end{align}
is an isomorphism.  

2) One can locally choose a basis  $\{r_1(v,z), r_2(v,z)\}$ for $E^+(v,z)$ that is $C^\infty$ in $\gamma>0$ and which extends continuously to $\gamma=0$.  
We define the $(2\times 2)$ Lopatinskii determinant 
\begin{align}\label{lop}
D(v,z)=\det \left(\cB(v,z)r_1(v,z),\cB(v,z)r_2(v,z)\right):=\det b^+(v,z).
\end{align}
Clearly, $D(v,z)\neq 0$ if and only if the Lopatinskii condition holds at $(v,z)$.

3) Let $\Delta(v,z,\xi_2)$ be the determinant of the principal symbol of $D_{x_2}-\cA(v,\cD)$.  The \emph{elliptic region} at $v$ is the the set of $(\sigma,\xi_1,0)\in \mathcal{X}$ such that $\Delta(v,\sigma,\xi_1,0,\xi_2)$ has no real roots in $\xi_2$.

\end{defn}


Here we list the properties of $(P^v,B^v)$  that are important for proving energy estimates and existence theorems.  

(P1)  The operator $P^0$ is strictly hyperbolic. 

(P2)  The boundary $x_2=0$ is noncharacteristic for $P^0$.  

(P3)  (LU) The  Lopatinskii determinant $D(0,\sigma,\xi_1,\gamma)$  for $(P^0,B^0)$  is nonvanishing at all points $(\sigma,\xi_1,\gamma)$ with $\gamma>0$ and  at all points $(\sigma,\xi_1,0)\in\mathcal{X}$ that lie outside the elliptic region at $v=0$. 

(P4)  (LC) If  $\underline {z}=(\underline \sigma,\underline \xi_1,0)\neq 0$ is such that $D(0,\underline z)=0$, then $\partial_\sigma D(0,\underline z)\neq 0$.  

(P5)   (R$_v$)  Suppose $D(0,\underline{z})=0$.   There exists a $\mu_0>0$ such that if $|v|<\mu_0$ the following holds:  there exist $2\times 2$ matrices $k_i(v,z)$, $i=1,2$, defined in a conic neighborhood $\Gamma$ of $\underline{z}$, which are $C^\infty$, homogeneous  of degree zero in $z$, and elliptic, such that for $z\in\Gamma$
\begin{align}\label{d4c}
b^+(v,z)=k_1(v,z)\begin{pmatrix}q(v,z)\Lambda^{-1}(z)&0\\0&1\end{pmatrix}k_2(v,z), \text{ where }q(v,\sigma,\xi_1,\gamma)=\sigma-i\gamma -\nu(v,\xi_1).
\end{align}
Here, $\nu(v,\xi_1)$ is a \emph{real}-valued $C^\infty$ function homogeneous of degree one in $\xi_1$; we note that $\underline{\xi}_1\neq 0$.

 
The above five properties, which are verified in [S-T], (pp. 268-270 and p. 283), are sufficient for proving the a priori estimates of section \ref{main}.   
 These properties allow us to cover the half sphere $\Sigma=\{z=(\sigma,\xi_1,\gamma)\in \mathcal{X}: |z|^2=1\}$ by a finite number of open sets $\cO_j$, $j\in J_a\cup J_c$ such that for each $j\in J_a$,  $D(0,z)$ is bounded away from $0$  in $\cO_j$,  and for each $j\in J_c$,  $D(0,z)=0$ at some point of $\cO_j$. Moreover, we can choose the $\cO_j$ such that if  $\mu_1>0$ is small enough and $|v|\leq \mu_1$, we have:

(i) $D(v,z)$ is bounded away from zero for $z\in\cO_j$, $j\in J_a$, while for  $j\in J_c$,  we have $D(0,\uz)=0$ for some $\uz\in\cO_j$ and  $b^+(v,z)$ satisfies \eqref{d4c} for $z\in\cO_j$.  

(ii)   We can write  $J_a=J_h\cup J_e$, where   $J_h$ denotes the set of indices  such that $\Delta(0,z,\xi_2)$ has at least one real root  $\xi_2$ for some $z\in\cO_j$,  while  $\Delta(v,z,\xi_2)$ has no real roots for $j\in J_e$, $z\in\cO_j$.

(iii) For $j\in J_c$, $z\in\cO_j$ the symbol $\Delta(v,z,\xi_2)$ has no real roots.

\begin{rem}\label{diag}
1.  Let $\cA(v,z)$ be the symbol of the operator $\cA(v,\cD)$ in \eqref{d4b}.   The operator $P^0$ is strictly hyperbolic; consequently, for $|v|$ small the system \eqref{d4a} can be conjugated microlocally to  the form called ``block structure" in the sense of Kreiss-Majda.\footnote{When $d\geq 3$, the operator $P^0$ fails to be strictly hyperbolic, and for $v$ small $P^v$ exhibits characteristics of variable multiplicity.   As shown in \cite{MZ}, smooth Kreiss symmetrizers can sometimes be constructed in such situations.  Unfortunately, Proposition \ref{u2} shows that the results of \cite{MZ}  do not apply to $(P^v,B^v)$ when $d\geq 3$.}

  In particular, for $|v|$ small there exists for each $\uz\in \mathcal{X}$ a $C^\infty$ invertible matrix $S(v,z)$, homogeneous of degree zero and defined in a conic neighborhood of $\underline{z}$, such that 
\begin{align}\label{conj}
S^{-1}\mathcal{A}S=\mathrm{diag}(a_1,\dots,a_k,a^+,a^-),
\end{align}
where the blocks $a^\pm$ satisfy $\pm\mathrm{Im}\;a^\pm>0$, and the  blocks $a_j$ have Jordan form at $(0,\uz)$ with the single real eigenvalue $\lambda_j$.  We refer to Chapter 7 of \cite{CP} for the full definition.

2.  If $\uz$ lies in the elliptic region at $v=0$,  then for $|v|$ small and $z$ in a conic neighborhood of $\uz$ the $4\times 4$ matrix $S(v,z)=[S^+ \;S^-]$ is chosen so that  the 2 columns of $S^+(v,z)$ give a basis of $E^+(v,z)$.  We define $b^\pm(v,z)=\cB(v,z)S^\pm(v,z)$; thus, in this case the  vectors $r_1,r_2$ in \eqref{lop}satisfy $[r_1,r_2]=S^+$.


 \end{rem}

We let $\{\phi_j(\sigma,\xi_1,\gamma), j\in J_e\cup J_h\cup J_c\}$ be a partition of unity subordinate to the covering $\{\cO_j\}$, consisting of symbols homogeneous of degree zero.    The neighborhoods $\cO_j$ may be  (and are) chosen so that each is a neighborhood in which some matrix $S(v,z)$ as above defined.  In particular, we have
\begin{align}\label{conj2}
S^{-1}\mathcal{A}S=\mathrm{diag}(a^+,a^-) \text{ for }|v|<\mu_1, z\in\cO_j, \text{ when }j\in J_e\cup J_c.
\end{align}
For each $S$ we will also denote by $S$ an elliptic extension of $S$ to all of $\mathcal{X}$.

 The following additional properties, which are verified in [S-T], p. 286,  are used in proving the existence theorems:\footnote{The labels F.G. (``formule de Green") and A.A. (``auto-adjoint"),  like LU (``Lopatinskii uniform") and LC (``controle de Lopatinskii") above, are borrowed from \cite{S-T} and used here to facilitate comparison with that paper.}\\

(P6)  (F.G.) : For $|v|$ small enough, $(P^v,B^v)$ satisfies the following Green's formula:  there exist differential operators $\tilde B$, $A$, $\tilde A$ such that for all $u$, $v$ in  $C^\infty_0(\overline{\mathbb{R}^3_+})$ (we drop the superscript $v$ here):
\begin{align}\label{d5}
(Pu,w)_{L^2(t,x)}=(u,P^*w)_{L^2(t,x)}+\langle Bu(0),\tilde A w(0)\rangle_{L^2(t,x_1)}+\langle Au(0),\tilde B w(0)\rangle_{L^(t,x_1)}.
\end{align}
The operator $\tilde B$ has order $1$, while $A$ and $\tilde A$ have order $0$.

(P7)    (A.A.) The problem $(P^v,B^v)$ satisfies (F.G.) and the operators $P^*$ and $\tilde B$ have the same principal parts as $P^v$ and $B^v$, respectively.




\begin{rem}\label{d6}
1) An essential consequence of (P5) is that the Lopatinski determinant,  when computed at  states $v\neq 0$  but near $0$, can only vanish when $\gamma=0$.   In other words for small enough perturbations $v$ the Lopatinskii determinant of the perturbed operators $(P^v, B^v)$, like that of $(P^0,B^0)$, vanishes only at points where $\gamma=0$. 
In a weakly stable problem, where  one has $D(0,\uz)=0$ for some $\uz$ with $\gamma=0$,  even  if one knows that $D(0,z)$ can only vanish when $\gamma=0$,  for the nonlinear theory one still has to rule out the possibility that $D(v,z)=0$ for some $z$ with $\gamma>0$ (an exponentially growing mode) for states $v$ near $0$.  That is accomplished by verifying (P5) in this case.

Another consequence explained below is that $(R_v)$ allows us to avoid the need for a sharp Garding inequality in the singular calculus.\footnote{This is fortunate, since we have no such inequality.}


2) Caution:  Glancing points as well as points in the hyperbolic and mixed hyperbolic-elliptic regions lie in the support of $\phi_j$ for some $j\in J_h$.


\end{rem}

The paper \cite{S-T} verified  properties (P1)-(P7) for the SVK system with $\lambda>0$, $\mu>0$; in  fact that verification works just as well when $\mu>0$, $\lambda+\mu>0$.  
In section \ref{generalisotropic} we explain why (P1)-(P7) also hold under the more general assumption (A1g).

\subsection{Main results}\label{mainr}

\emph{\quad} Our first main result asserts the existence of exact solutions to the coupled singular problems \eqref{a7}-\eqref{a9} on a fixed time interval independent of $\eps$.  An immediate consequence is the existence of exact solutions to the original (nonsingular) problem \eqref{a2} on such a time interval.

\begin{theo}\label{uniformexistence}

(a)   Assume (A1) and let $\beta$ as in \eqref{a3} be a Rayleigh frequency.   Suppose $m>3d+4+\frac{d+1}{2}$ and 
consider the coupled singular problems \eqref{a7}, \eqref{a8}, \eqref{a9}, where $G(x',\theta)\in H^{m+3}(b\Omega)$ and vanishes in $t<0$.\footnote{The choice $m>3d+4+\frac{d+1}{2}$ is determined by the requirements of the singular calculus, Appendix \ref{calculus}.}  There exist positive constants $\eps_0$ and $T_1$ and  unique $C^2$ solutions  $v^\eps(t,x,\theta)=(v^\eps_1,v^\eps_2)$ and $u^\eps(t,x,\theta)$ to the coupled problems on the time interval $[0,T_1]$ for $\eps\in (0,\eps_0]$; the constant $T_1$ is independent of $\eps\in (0,\eps_0]$.   Moreover,  we have $v^\eps=\nabla_\eps u^\eps$ on $\Omega_{T_1}$.  

 (b) Consequently, for $\eps\in (0,\eps_0]$ the function $u^\eps$  gives the unique $C^2$ solution to the system \eqref{a5} on the time interval $[-\infty,T_1]$, and 
\begin{align}\label{aaa4}
U^\eps(t,x)=u^\eps(t,x,\theta)|_{\theta=\frac{\beta\cdot (t,x_1)}{\eps}}
\end{align}
gives the unique $C^2$ \emph{Rayleigh pulse} solution to the Saint Venant-Kirchhoff system \eqref{a2} with boundary data
\begin{align}\label{aa3}
g=g^\eps(t,x_1)=\eps^2 G\left(t,x_1,\frac{\beta\cdot (t,x_1)}{\eps}\right).
\end{align}

(c) The $E_{m,T_1}(v^\eps)$ norm is uniformly bounded for $\eps\in (0,\eps_0]$; this norm is defined in \eqref{c00}.


\end{theo}

The main steps in the proof of Theorem \ref{uniformexistence} are Propositions \ref{mainprop} and  \ref{c5} (which give a priori estimates uniform with respect to $\eps$) for the singular systems \eqref{a7}-\eqref{a9}, and Propositions \ref{localex} and \ref{continuation}, which are local existence and continuation theorems for a fixed $\eps$.

The main technical tools used in the proof are the calculus of singular pseudodifferential operators for pulses constructed in \cite{CGW} and summarized in Appendix \ref{calculus}, and the new estimates of singular norms of nonlinear functions (including microlocal ``Rauch-lemma"-type estimates for singular norms) given in section \ref{nonlinear}.   Section \ref{commutator} gives several commutator estimates that extend the results of \cite{CGW}.

Our second main result gives a precise sense in which the approximate solution constructed in Chapter \ref{chapter4} is close to the exact solution of Theorem \ref{uniformexistence}, together with a rate of convergence.   The proof of Theorem \ref{approxthm} depends on the results of chapters \ref{chapter3} and \ref{chapter4} and is completed in chapter \ref{chapter5}.

\begin{theo}\label{approxthm}
We make the same assumptions as in Theorem \ref{uniformexistence}, let $\delta>0$,  and let $u^\eps$, $v^\eps=\nabla_\eps u^\eps$ be as in Theorem \ref{uniformexistence}.  
Let 
$$
u^\eps_a(t,x,\theta)=\eps^2u^\eps_\sigma(t,x,\theta)+\eps^3 u^\eps_\tau(t,x,\theta)
$$
be the approximate solution to the singular system \eqref{a5} given in \eqref{p0} and let $v^\eps_a=\nabla_\eps u^\eps_a$.  There exist positive constants  $\eps_1\leq \eps_0$ and $T_2\leq T_1$ such that for $\eps\in (0,\eps_1]$ we have on $\Omega_{T_2}$,
\begin{align}\label{az1}
E_{m,T_2}(v^\eps-v^\eps_a)\leq C_\delta \eps^{\frac{1}{4}-\delta}.
\end{align}

\end{theo}

The norm $E_{m,T_2}$ \eqref{c00} is a sum of 18 terms.
As a corollary of Theorem \ref{approxthm} and the estimates of $u^\eps_\tau$ given in section \ref{bblock} we obtain, for example, the following result for the exact solution $U^\eps(t,x)$ to the original problem \eqref{a2}.

\begin{cor}\label{corapprox}
Let $\delta>0$ and set
$$
U^\eps(t,x)=u^\eps(t,x,\theta)|_{\theta=\frac{\beta\cdot (t,x_1)}{\eps}} \text{ and }U^\eps_\sigma(t,x)=\eps^2u^\eps_\sigma(t,x,\theta)|_{\theta=\frac{\beta\cdot (t,x_1)}{\eps}} .
 $$
There exist positive constants  $\eps_1\leq \eps_0$ and $T_2\leq T_1$ such that for $\eps\in (0,\eps_1]$,
we have on $[-\infty,T_2]\times \mathbb R^2_+$:
 \begin{align}
\begin{split}
& (a)  \frac{1}{\eps}|\nabla_{t,x}(U^\eps(t,x)-\eps^2U^\eps_\sigma(t,x))|_{L^\infty}\leq C_\delta \eps^{\frac{1}{4}-\delta} \text{ and }\\
& (b) |\partial^\alpha_{t,x}(U^\eps(t,x)-\eps^2U^\eps_\sigma(t,x))|_{L^\infty}\leq C_\delta \eps^{\frac{1}{4}-\delta}, \text{ where }\alpha=(\alpha_t,\alpha_{x_1},\alpha_{x_2})\text{ with } |\alpha|=2, \alpha_t\leq 1. 
\end{split}
\end{align}

\end{cor}

\begin{proof}[Proof of Corollary]
The second estimate follows from \eqref{az1} and the presence of the terms $\left|\begin{pmatrix}\Lambda v_j\\D_{x_2} v_j\end{pmatrix}\right|^{(2)}_{\infty,m,T_2}$, $j=1,2$ in the definition of 
the $E_{m,T_2}$ norm \eqref{c0a}.   Similarly, the first estimate follows from the presence of the term 
$\left|\frac{\nabla_\eps u}{\eps}\right|_{\infty,m,T_2}$ in the definition of the $E_{m,T_2}$ norm, together with Remark \ref{extracontrol}.
\end{proof}

\subsection{Survey of the proofs. } \label{survey}  \emph{\quad} In the setting of the original nonsingular problem \eqref{a2} with non-oscillating data,  Sable-Tougeron \cite{S-T} has shown that if sufficiently precise estimates can be obtained for the linearized problems corresponding to \eqref{a7}-\eqref{a9}, then the nonlinear problems can be solved by a standard fixed point iteration applied to (the nonsingular analogues of) the coupled problems \eqref{a7}-\eqref{a9}, even though the linearized estimates for the $v_1$ problem \eqref{a7} exhibit a loss of one derivative on the boundary.\footnote{With some abuse  we shall refer to estimates for the linearized problems  as ``the linearized estimates".}   Roughly, the necessary precision is gained by taking advantage of the fact that the uniform Lopatinskii condition fails for \eqref{a7} in the \emph{elliptic} region to derive estimates that minimize the loss for pieces of the solution microlocalized to that region, and which exhibit no loss for pieces microlocalized away from  the bad set.  At the same time this argument takes advantage of two gains that derive from considering coupled differentiated problems like \eqref{a7}-\eqref{a9}: the fact that the coefficients now depend on the unknown $v$ itself rather that the gradient of the unknown as in \eqref{a5}, and the fact that the problem for $v_2$ \eqref{a8} is a Dirichlet problem, so the associated linearized problem satisfies the uniform Lopatinskii condition and solutions exhibit no loss of derivatives.   In our study of the singular problems in this chapter we shall make use of all these ideas.   Naturally, new difficulties arise due to the singular nature of the problems, and our purpose here is to summarize what is new in our approach to dealing with these.

The first step, carried out in section \ref{main}, is to obtain the basic $L^2$ estimates that are satisfied by solutions of the linearized singular systems \eqref{b1}-\eqref{b2}.    
This is done in Proposition \ref{bb1}, whose proof is a fairly straightforward adaptation to the singular setting of an analogous result of \cite{S-T}.  
To see the loss on the boundary in the estimates \eqref{bb2} for the Neumann-type problem \eqref{b1} relative to the estimates \eqref{bb3} for the Dirichlet problem \eqref{b2}, one must recall that $\cG_1$ is the forcing term in a first-order boundary condition, while $\cG_2$ is the forcing term in a zero-order boundary condition.    Thus, there is a loss of one \emph{singular} derivative $(\Lambda_D)$ on the boundary in the estimate  for $v_1$ relative to the corresponding estimate for $v_2$.\footnote{The singular operator $\Lambda_D$ and the singular norms that appear in  the estimates of this chapter are defined in part (g) of Notations \ref{spaces}.}

The next step is to obtain the ``slightly higher derivative" estimates of Proposition \ref{basicest}.
A serious obstacle arises in the proof of \eqref{b4} and \eqref{b7}, where one seeks to control the $L^\infty(x_2,L^2(t,x_1,\theta))$ norm uniformly with respect to $\eps$.  Such control is essential for later getting $L^\infty(t,x,\theta)$ bounds.    In a nonsingular problem one would use the equation and the fact that the boundary is noncharacteristic to control the $L^2$ norm of one normal derivative $\partial_{x_2}v$ using prior $L^2$ control of first-order tangential derivatives $\partial_{(t,x_1,\theta)}v$.   In a singular problem the factors of $\frac{1}{\eps}$ in the singular tangential derivatives $\partial_{t,\eps}$ and $\partial_{x_1,\eps}$ wreck this argument.   The resolution is to use a singular operator in the extended calculus of section 
\ref{extended}  to partition the solution into one piece with essential support in 
the region $|\sigma,\xi_1,\gamma|\geq\delta |k/\eps|$\footnote{Here $k$ is the Fourier variable dual to $\theta$ and $\delta>0$ is sufficiently small.} where the $\frac{1}{\eps}$ factors are harmless and the above argument works, and another piece concentrated in the complementary region where $(\sigma,\xi_1)+\frac{\beta k}{\eps}$ is nearly parallel to $\beta$.    For the second piece one can use that fact that $\beta$ lies in the elliptic region  to  block-diagonalize the singular system and control the $L^\infty(x_2,L^2(t,x_1,\theta))$ norm by an energy estimate.   This kind of difficulty always arises in singular boundary problems (see \cite{W, CGW2}, but this is the first time we have had to deal with it in a quasilinear weakly stable problem.  

Another difficulty peculiar to singular problems occurs in the proof of the estimate \eqref{b3}, which one might hope to prove simply by commuting the singular operator $\Lambda^{\frac{1}{2}}_D$ through the linearized $v_1$ problem  \eqref{b2} and applying the estimate \eqref{bb2}. It turns out that such an argument can be used only on certain microlocalized pieces of the solution, while other pieces require a separate new energy estimate.  The two cases are distinguished by the size of the commutator errors that arise; those in the first case are controllable uniformly with respect to $\eps$ (and they appear as the final term on the right in \eqref{b3}), while those in the second case are not.  This analysis required a further development of the singular calculus of \cite{CGW} that is given in section \ref{commutator}.  The points discussed in this and the above paragraph are treated mainly in steps \textbf{2} and \textbf{4} of the proof of Proposition \ref{basicest}.    The estimates of Proposition \ref{basicest} provide the foundation for Chapters 3 and 5.  

These estimates introduce a technical problem into subsequent arguments that was already encountered by \cite{S-T} in the nonsingular setting.  The presence of the microlocalized norms, that is the terms involving the singular operators $\phi_{j,D}$, forces us always to apply the linearized estimates to problems posed on the \emph{full} half-space $\Omega$, because there is no way to  time-localize these norms.   However, the other norms that appear in the linearized estimates, even the singular norms involving fractional powers of $\Lambda_D$, can be time-localized by interpolation,  and because of that we are ultimately able to use the global estimates on $\Omega$ to prove a short-time nonlinear existence theorem.

In the proof of short-time existence for the nonlinear, nonsingular problem \eqref{a2} with non-oscillating data, \cite{S-T} used a clever iteration scheme to solve (the nonsingular analogues of) the coupled problems \eqref{a7}-\eqref{a9}.  Calling the iterates $v^k=(v^k_1,v^k_2)$ and $u^k$, Sabl\'e-Tougeron found that the most difficult  step was to get uniform (with respect to $k$) bounds for the $v^k$;  for this she  used her precise estimates and the other ``gains" described in the first paragraph of this section.\footnote{Here we refer to bounds in a Sobolev norm stronger than the $L^\infty$ norm.}   It was possible to control the $v^k$ on a fixed time interval by working just with the two equations \eqref{a7}, \eqref{a8} and ignoring for the moment \eqref{a9}.  
Once the $v^k$ were uniformly controlled, uniform bounds for the $u^k$ were readily obtained,  and the existence of limits $v$ and $u$ satisfying $v=\nabla u$ followed by relatively straightforward arguments.  

We tried at first to prove the uniform (with respect to $\eps$) existence theorem for the nonlinear singular systems, Theorem \ref{uniformexistence}, by applying  the iteration scheme of \cite{S-T} to the coupled problems \eqref{a7}-\eqref{a9}, and attempting to estimate the $v^k$, which now depend on $\eps$, uniformly with respect to \emph{both} $k$ and $\eps$, on a time interval independent of $\eps$.    The iteration scheme is the one given in \eqref{k0a}-\eqref{k0c}; the problem for $v^{k+1}$ is linear so one can apply the estimates of Proposition \ref{basicest} to estimate $v^{k+1}$ provided one has the control of $v^k$ specified by \eqref{b3z}, where $v^k$ plays the role of $w$ in \eqref{b3z}.   This approach failed.  In particular, we 
 we were not able to bound the norm $|v^k/\eps|_{\infty,T}$ of the $k$th iterate uniformly with respect to $\eps$ and $k$ up to a fixed time $T$ independent of $\eps$.  If one knew $v^k=\nabla_\eps u^k$ (or could control the difference  well enough), then this approach could have been made  to work, but we had to abandon this idea.  It seems that the relation 
 $v^\eps=\nabla_\eps u^\eps$, which allows one to use the $u^\eps$ equation \eqref{a9} to control $|v^\eps/\eps|_{\infty,T}$, holds only in the limit $k\to\infty$.

In order to take advantage  of the relation $v^\eps=\nabla_\eps u^\eps$ we use instead a continuous induction argument.  This  argument, given in the proof of Proposition \ref{mainprop}, has three parts.  First, one needs a \emph{short-time existence theorem} for $\eps$ fixed for the nonlinear problems \eqref{a7}-\eqref{a9} on a domain $\Omega_{T_\eps}$ that depends on $\eps$; this is given by Proposition   \ref{localex}.   This result also establishes the relation $v^\eps=\nabla_\eps u^\eps$ on $\Omega_{T_\eps}$.\footnote{The notation $\Omega_T$ is defined in Notations \ref{spaces}(a).}    Second, one needs a \emph{continuation theorem} for $\eps$ fixed, Proposition \ref{continuation}, which states that if one has a solution to \eqref{a7}-\eqref{a9} on some domain $\Omega_{T_{1,\eps}}$  that is sufficiently small in an appropriate norm on that domain, then that solution can be extended with similar bounds to a later time $T_{2,\eps}>T_{1,\eps}$.  The appropriate norm turns out to be the singular energy  norm $E_{m,T_{1,\eps}}(v)$ \eqref{c00} for $m$ sufficiently large.  Finally, the argument requires an \emph{a priori estimate} for the $E_{m,\gamma}$ norm of solutions to the modified systems \eqref{c1}-\eqref{c3} on $\Omega$.   This estimate is given in Proposition \ref{c5}; the constants in the estimate depend only on  $M_0$ as in \eqref{c0h}, a number that is independent of the parameters $\eps$ and  $T$ on which \eqref{c1}-\eqref{c3} depend.  
A detailed explanation of how this estimate can be localized in time and  combined with  the other two results to obtain existence for \eqref{a7}-\eqref{a9} on a domain $\Omega_T$ with $T>0$ independent of $\eps$ is given in section \ref{strategy}.\footnote{We switched to a ``continuous induction" argument after studying the thesis of \cite{Pa},  where such a strategy is used on a quite different problem.}

To understand the form of the modified systems \eqref{c1}-\eqref{c3}, recall that although we are not using an iteration scheme, we still need to work with systems to which we can apply the linearized estimates.  Observe that with one important exception (the right side of the boundary condition in \eqref{c2}; see Remark \ref{c4a}), the coefficients in the modified systems all depend on $v^s$ rather than the unknown $v^\eps$.  Here $v^s$ is our notation with two parameters suppressed for $v^{\eps,s}_T$,  a carefully chosen Seeley extension (Proposition \ref{c0e}) to $t>T$ of $v^\eps|_{\Omega_T}$, where $v^\eps$ is the solution on $\Omega_{T_\eps}$ to the nonlinear problems \eqref{a7}-\eqref{a8} provided by the local existence result, Proposition \ref{localex}, and where $0<T<T_\eps$.    A consequence of the linear existence theorem, Theorem \ref{exist}, is that the linearized singular problems exhibit \emph{causality};  roughly, solutions in $t<T$ are unaffected by changing the forcing terms and coefficients in $t>T$.  Thus, for any choice of $T\in (0,T_\eps)$, we know that the solution $v^\eps$ to the modified problems \eqref{c1}-\eqref{c2} on the full half space $\Omega$ agrees with the already given $v^{s,\eps}_T=v^\eps$ on $\Omega_T$.    When we take into account the relations between the global norm $E_{m,\gamma}(v^\eps)$ and the time-localized norm $E_{m,T}(v^\eps)$ given by Proposition \ref{c0e}, we are then able to estimate solutions of the nonlinear problems \eqref{a7}-\eqref{a8} on $\Omega_T$ by estimating solutions of the linear problems \eqref{c1}-\eqref{c2} on the full space $\Omega$.

The most difficult of the three steps is the proof of the a priori estimate.   This is carried out in section \ref{mainestimate}, which is the heart of the rigorous analysis.  
To take advantage of the relation $v^\eps=\nabla_\eps u^\eps$ on $\Omega_{T_\eps}$, we must estimate solutions to the three systems \eqref{c1}-\eqref{c3} ``simultaneously" in the following sense.   When  we apply the linearized estimates of section \ref{b1a} to the problem \eqref{c1} or \eqref{c2}, we find that terms arise that can only be controlled by estimating solutions to \eqref{c3}.  Similarly, estimates of solutions to \eqref{c3}  lead to terms that can only be controlled by estimating solutions to \eqref{c1} and \eqref{c2}.  An example of this occurs in the first interior commutator estimate \eqref{e24az}, where the terms involving factors of $\frac{1}{\eps}$ on the right can be controlled only by estimating solutions of \eqref{c3} and using the relation 
$v^\eps=\nabla_\eps u^\eps$.  In a sense we have to estimate more than three systems simultaneously, since we estimate not only \eqref{c1}-\eqref{c3}, but also $\frac{1}{\sqrt{\eps}}\eqref{c1}$, $\sqrt{\eps}\eqref{c2}$, $\frac{1}{\eps}\eqref{c3}$, and so on; the complete list is given in the outline of section \ref{outline}.     The norm $E_{m,\gamma}(v)$ that appears in the a priori estimate of Proposition \ref{c5} must incorporate all the time-localizable norms that arise in estimating all these systems; thus, it is not surprising that it contains a rather large number of terms.  In fact each term $E_{m,\gamma}(v_j)$ \eqref{c0a} in the definition of $E_{m,\gamma}$ is a sum of 18 terms. It is perhaps surprising that any \emph{finite} number of terms works!  
We note that the  interior and boundary commutator estimates of section \ref{mainestimate}  can be (and are) applied without change in the error analysis of section \ref{chapter5}.  However, the boundary and interior forcing estimates of chapter \ref{chapter5} are substantially different.

Although we have used the singular calculus of \cite{CGW} in several earlier papers, for example \cite{CGW2,CW1,CW2}, this is the first time we have had to estimate singular norms
$\langle \Lambda^r f(u)\rangle_{m,\gamma}$ (defined below in Notations \ref{spaces}(g)) of nonlinear functions of $u$.  In fact, to take advantage of the extra microlocal precision in the linearized estimates of Proposition \ref{basicest}, we also need to show that in some cases even \emph{microlocal} regularity of $u(t,x,\theta)$ in singular norms is preserved under nonlinear functions.  The microlocal results are a singular analogue of the classical  Rauch's lemma \cite{R}.     Both types of results are new for singular norms and are proved in section \ref{singular}, for example in Propositions \ref{f3} and \ref{f5}.   We use these tools throughout chapters \ref{chapter3} and \ref{chapter5}.  In estimates of $\langle \Lambda^r f(u)\rangle_{m,\gamma}$ where $r>0$, we have to assume that $f$ is real-analytic; we do not know if this can be weakened to, say, $C^\infty$.  This restriction still allows us to handle the SVK system, and is the reason for the condition of analyticity of $W(E)$ in Assumption (A1g).    Both  tame and (simpler) non-tame versions of these estimates are given.  The tame estimates are used mainly in the proof of the continuation result, Proposition \ref{continuation}.

We conclude this survey with a few remarks on the local existence and continuation results for fixed $\eps$ of section \ref{local}.  In this section we take advantage of the fact that when $\eps$ is fixed, one \emph{can} use the equation to control normal derivatives $\partial_{x_2}v$ in terms of singular tangential derivatives $\partial_{t,\eps}v$, $\partial_{x_1,\eps}v$.  However, even when $\eps$ is fixed the nonlinear systems being studied are still singular, because $\partial_{x'}$ and $\partial_\theta$ derivatives occur in the linear combination $\partial_{x'}+\beta\frac{\partial_\theta}{\eps}$.  So the singular calculus is still needed to prove the linearized estimates on which the argument is based.  Given those estimates and the results of section \ref{singular}, the proof of the local existence theorem is a  matter of adapting the argument of \cite{S-T} to singular problems.  For example, the iteration scheme \eqref{k0a}-\eqref{k0c} is the singular analogue of the one used in \cite{S-T}.  This argument is based on the non-tame linearized estimates \eqref{k1}-\eqref{k2}.  The proof of the continuation result is based on a tame version of the linearized estimates given in Proposition \ref{tame2}.  We first prove existence of a low regularity continuation essentially by repeating the local existence proof; we then  use the tame linearized estimates to show that this continuation has the desired higher regularity.


\section{The basic estimates for the linearized singular systems}\label{main}

\emph{\quad}  With apologies to the reader for the following long list,  we begin by gathering in one place most of the notation for spaces and norms that is needed below.

\begin{nota}\label{spaces}
We take $m$ to be a nonnegative integer,  let $(\xi',k)=(\sigma,\xi_1,k)$ be the Fourier transform variables dual to $(x',\theta)=(t,x_1,\theta)$, 
and set $\langle \xi',k\rangle=(|\xi'|^2+k^2+1)^{1/2}$,  $\langle \xi',k,\gamma\rangle=(|\xi'|^2+k^2+\gamma^2)^{1/2}$.

(a)\;Let $\Omega:=\{(t,x_1,x_2,\theta)\in\mathbb{R}^{4}:x_2>0\}$, and for $T>0$ set $\Omega_T:=\Omega\cap\{-\infty<t<T\}$,
$b\Omega:=\{(t,x_1,\theta)\in\mathbb{R}^3$, and $b\Omega_T:=b\Omega\cap \{-\infty<t<T\}$.

(b)\;Let $H^m\equiv H^m(b\Omega)$ be the standard Sobolev space with norm $\langle V(x',\theta)\rangle_m=|\langle \xi',k\rangle^m \hat V(\xi',k)|_{L^2(\xi',k)}$.  
For $\gamma\geq 1$ and $V\in H^m$ we define $|V|_{H^m_\gamma}=|\langle \xi',k,\gamma\rangle^m \hat V(\xi',k)|_{L^2(\xi',k)}$.    For  $V\in e^{\gamma t} \, H^m$ we set $\langle V\rangle_{m,\gamma} := |e^{-\gamma t} \, V |_{H^m_\gamma}$.   We sometimes write $V^\gamma=e^{-\gamma t}V$.

(c)\;$L^2H^m\equiv L^2(\overline{\mathbb{R}}_+,H^m(b\Omega))$ with norm $|U(t,x,\theta)|_{L^2H^m} \equiv
|U|_{0,m}$ given by
\begin{equation*}
|U|_{0,m}^2=\int^\infty_0\langle U(x',x_2,\theta)\rangle_m^2 dx_2.
\end{equation*}
Similarly, for $U\in L^2e^{\gamma t}H^m:=L^2(\overline{\mathbb{R}}_+,e^{\gamma t}H^m(b\Omega))$ we set
\begin{align}
|U|_{0,m,\gamma}^2=\int^\infty_0\langle U(x',x_2,\theta)\rangle_{m,\gamma}^2 dx_2.
\end{align}

(d)\;$CH^m\equiv C(\overline{\mathbb{R}}_+,H^m(b\Omega))$ denotes the space of continuous bounded
functions of $x_2$ with values in $H^m(b\Omega)$, with norm $|U(t,x,\theta)|_{CH^m} =|U|_{\infty,m} :=
\sup_{x_2\geq 0} |U(.,x_2,.)|_{H^m(b\Omega)}$
The corresponding
norm on  $CH^m\equiv C(\overline{\mathbb{R}}_+,e^{\gamma t}H^m(b\Omega))$is denoted $|V|_{\infty,m,\gamma}$.






(e) For a nonnegative integer $M$ and define $C^{0,M} :=C(\overline{\mathbb{R}}_+,C^{M}(b\Omega))$ as the space of
continuous bounded functions of $x_2$ with values in $C^{M}(b\Omega)$, with norm $|U(x',x_2,\theta)|_{C^{0,M}}
:= |U|_{L^\infty W^{M,\infty}}$. Here $L^\infty W^{M,\infty}$ denotes the space $L^\infty(\overline{\mathbb R}_+;
W^{M,\infty}(b\Omega))$.

(f)The corresponding spaces on $\Omega_T$ are denoted $L^2H^m_T$, $L^2e^{\gamma t}H^m_{T}$, $CH^m_T$,
$Ce^{\gamma t}H^m_{T}$ and $C^{0,M_0}_T$ with norms $|U|_{0,m,T}$, $|U|_{0,m,\gamma,T}$, $|U|_{\infty,m,T}$,
$|U|_{\infty,m,\gamma,T}$, and $|U|_{C^{0,M_0}_T}$ respectively. These norms have clear meanings when $m$ is a nonnegative integer; otherwise they are defined as usual by interpolation.  On $b\Omega_T$ we use the spaces
$H^m_T$ and $e^{\gamma t}H^m_{T}$ with norms $\langle U\rangle_{m,T}$ and $\langle U\rangle_{m,\gamma,T}$.

(g) For $r\in\overline{\mathbb{R}}_+$ let $\Lambda^r_D$  be the singular operator associated to the singular symbol $\Lambda^r(\xi'+\frac{\beta k}{\eps},\gamma):=(|\xi'+\frac{\beta k}{\eps}|^2+\gamma^2)^{r/2}$ (see \eqref{singularpseudop}).   With $\xi'=(\sigma,\xi_1)$ for $V=V(x',\theta)$  we  define (with slight abuse of the notation $\langle\cdot\rangle_{m,\gamma}$ defined in (b)\footnote{Observe that for $\langle \Lambda^r V \rangle_{m,\gamma}$ as defined in \eqref{b0} and $\langle w\rangle_{m,\gamma}$ as in (b), we have
\begin{align}\notag
\langle\Lambda^r V\rangle_{m,\gamma}=\langle w\rangle_{m,\gamma}, \text{ where }w=e^{\gamma t}\Lambda^r_D(e^{-\gamma t}V).
\end{align}
}) the singular norms
\begin{align}\label{b0}
\begin{split}
&\langle\Lambda^r V\rangle_{m,\gamma}:=|\Lambda^r_D (e^{-\gamma t}V)|_{H^m_\gamma}=\left|(|\xi'+\frac{\beta k}{\eps}|^2+\gamma^2)^{r/2}\langle \xi',k,\gamma\rangle^m\hat V(\sigma-i\gamma,\xi_1,k)\right|_{L^2(\xi',k)}\\
&\langle\Lambda^r_{1} V\rangle_m:=\left|(|\xi'+\frac{\beta k}{\eps}|^2+1)^{r/2}\langle\xi',k\rangle^m\hat V(\xi',k)\right|_{L^2(\xi',k)}.
\end{split}
\end{align}

(h)When $r$ is a nonnegative integer the norms defined in (g) can be localized to $b\Omega_T$ by replacing $\Lambda_D V$ by $((\partial_{x'}+\beta\frac{\partial_\theta}{\eps})V,\gamma V)$.  Denote the localized norms by $\langle\Lambda^r V\rangle_{m,\gamma,T}$ and $\langle\Lambda^r_{1} V\rangle_{m,T}$.

For non-integer $r>0$, the localized norms are defined by interpolation from the integral case.\footnote{A concise reference for the method of complex interpolation used here is \cite{T}, Chapter 4.}

(i)Having defined these singular norms of functions on $b\Omega$ and $b\Omega_T$, we define the corresponding norms for functions $U(x',x_2,\theta)$ on $\Omega$ and 
$\Omega_T$ by obvious analogy with (c) and (d).   We denote these norms $|\Lambda^r U|_{0,m,\gamma}$, $|\Lambda^r U|_{\infty,m,\gamma}$, $|\Lambda_1^r U|_{0,m}$, $|\Lambda^r_1 U|_{\infty,m}$; the corresponding time-localized norms are $|\Lambda^r U|_{0,m,\gamma,T}$, $|\Lambda^r U|_{\infty,m,\gamma,T}$, $|\Lambda_1^r U|_{0,m,T}$, $|\Lambda^r_1 U|_{\infty,m,T}$.

(j) If $\phi_D$ is a singular pseudodifferential operator associated to the singular symbol $\phi(\eps W(x',\theta), \xi'+\frac{\beta k}{\eps},\gamma)$, we set\begin{align}\label{b0a}
\begin{split}
&\langle\phi V\rangle_{m,\gamma}:=|\phi_D (e^{-\gamma t}V)|_{H^m_\gamma}.
\end{split}
\end{align}
We do not attempt to define a time-localized version of \eqref{b0a}.

(k)We will sometimes (for example, in convolutions) write  $X:=\xi'+\beta\frac{k}{\eps}$, $Y=\eta'+\beta\frac{l}{\eps}$, where $(\eta',l)$ also denotes variables dual to $(x',\theta)$.

(l) All constants appearing in the estimates below are independent of $\eps$, $\gamma$, and $T$ unless such
dependence is explicitly noted.

(m) We use $\partial$ to denote a tangential derivative with respect to one of the variables $(x',\theta)$.

(n) We set $D_{x_2,\eps}=\frac{1}{i}\partial_{x_2}$, $D_{t,\eps}=\frac{1}{i}\partial_{t,\eps}$, $D_{x_1,\eps}=\frac{1}{i}\partial_{x_1,\eps}$,  $D_{x',\eps}=(D_{t,\eps},D_{x_1,\eps})$.

(o) Set $\Lambda^r_{D,\gamma}=e^{\gamma t}\Lambda^r_D e^{-\gamma t}$.


\end{nota}

\subsection{Statement of the estimates. }\label{b1a}

\emph{\quad} Here we give the main estimates for the ``linearized versions" of \eqref{a7} and \eqref{a8}, namely, 

\begin{align}\label{b1}
\begin{split}
&\partial_{t,\eps}^2 v_1^\eps+\sum_{|\alpha|=2} A_\alpha(w^\eps)\partial_{x,\eps}^\alpha v_1^\eps=\mathcal{F}_1\text{ on }\Omega\\
&\partial _{x_2} v^\eps_1-d_{v_1}H(w^\eps_1,h(w^\eps))\partial_{x_1,\eps}v^\eps_1=\mathcal{G}_1\text{ on }x_2=0
\end{split}
\end{align}
and
\begin{align}\label{b2}
\begin{split}
&\partial_{t,\eps}^2 v_2^\eps+\sum_{|\alpha|=2} A_\alpha(w^\eps)\partial_{x,\eps}^\alpha v_2^\eps=\mathcal{F}_2\text{ on }\Omega\\
&v^\eps_2=\mathcal{G}_2\text{ on }x_2=0.
\end{split}
\end{align}
Here the functions  $v^\eps$, $w^\eps$, $\mathcal{F}_i$, $\mathcal{G}_i$ all vanish in $t<0$.

   The operators $\{\phi_{j,D}, j\in J_e\cup J_h\cup J_c\}$ that appear below form a (singular) pseudodifferential partition of the identity operator; they have symbols $\phi_j(X,\gamma)$, where the $\phi_j(\xi',\gamma)$ define the partition of unity chosen in section \ref{assumptions}.




\begin{prop}\label{basicest}
Suppose $n\geq 3d+4$, $s_0>\frac{d+1}{2}+2$, and let $\delta$ and $K$ be positive constants.\footnote{See, for example, Propositions \ref{prop20} and   \ref{commutator4} to understand the restrictions on $n$ and $s_0$.}  Assume that for $\eps\in (0,1]$ the function $w^\eps$ appearing in the coefficients of \eqref{b1}, \eqref{b2} satisfies
\begin{align}\label{b3z}
\left|\frac{w^\eps}{\eps}\right|_{C^{0,n}}+\left|\frac{w^\eps}{\eps}\right|_{CH^{s_0}}+|\partial_{x_2}w^\eps|_{L^\infty(\Omega)}<K , \; |w^\eps|_{L^\infty(\Omega)}<\delta.
\end{align}
The following a priori estimates are valid provided $\delta$ is small enough; the constants $\gamma_0$ and $C$ that appear depend on $K$.  
The first estimates are for the weakly stable Neumann-type problem \eqref{b1} satisfied by $v_1$: for  $\gamma\geq \gamma_0$ we have
%
\begin{align}\label{b3}
\begin{split}
&\gamma \left|\begin{pmatrix}\Lambda^{\frac{3}{2}}v_1\\D_{x_2}\Lambda^{\frac{1}{2}}v_1\end{pmatrix}\right|^2_{0,0,\gamma}+\gamma\left <  \begin{pmatrix}\Lambda v_1\\D_{x_2} v_1\end{pmatrix} \right>^2_{0,\gamma}+\sum_{J_h\cup J_e }\left <  \phi_j \begin{pmatrix}\Lambda^{\frac{3}{2}} v_1\\D_{x_2}\Lambda^{\frac{1}{2}} v_1\end{pmatrix} \right>^2_{0,\gamma}\leq \\
&\quad C \gamma^{-1}\left(|\Lambda^{\frac{1}{2}}\mathcal{F}_1|^2_{0,0,\gamma}+\left< \Lambda \mathcal{G}_1\right>^2_{0,\gamma}\right)+ C\gamma^{-1}\left|\begin{pmatrix}\Lambda v_1/\sqrt{\eps}\\D_{x_2} v_1/\sqrt{\eps}\end{pmatrix}\right|^2_{0,0,\gamma},
\end{split}
\end{align}
\begin{align}\label{b4}
\begin{split}
&\gamma \left|\begin{pmatrix}\Lambda v_1\\D_{x_2} v_1\end{pmatrix}\right|^2_{\infty,0,\gamma}\leq
 C \gamma^{-1}\left(|\Lambda^{\frac{1}{2}}\mathcal{F}_1|^2_{0,0,\gamma}+|\mathcal{F}_1|_{0,1,\gamma}^2+\left<\Lambda\mathcal{G}_1\right >^2_{0,\gamma}+\left < \Lambda^{\frac{1}{2}}\mathcal{G}_1\right>^2_{1,\gamma}\right)+\\
 &\qquad \qquad \qquad \qquad C \gamma^{-1} \left|\begin{pmatrix}\Lambda v_1/\sqrt{\eps}\\D_{x_2} v_1/\sqrt{\eps}\end{pmatrix}\right|^2_{0,0,\gamma}.
\end{split}
\end{align}
\begin{align}\label{b5}
\begin{split}
&\gamma \left|\begin{pmatrix}\Lambda v_1\\D_{x_2} v_1\end{pmatrix}\right|^2_{0,1,\gamma}
+\gamma \left <  \begin{pmatrix}\Lambda^{\frac{1}{2}} v_1\\D_{x_2}\Lambda^{-\frac{1}{2}} v_1\end{pmatrix} \right>^2_{1,\gamma}+ \sum_{J_h \cup J_e}\left <  \phi_j \begin{pmatrix}\Lambda v_1\\D_{x_2} v_1\end{pmatrix} \right>^2_{1,\gamma}\leq\\
&\qquad\qquad\qquad C\gamma^{-1}\left(|\mathcal{F}_1|_{0,1,\gamma}^2+\left < \Lambda^{\frac{1}{2}}\mathcal{G}_1\right>^2_{1,\gamma}\right).
\end{split}
\end{align}

The next estimates are for the uniformly stable Dirichlet  problem \eqref{b2} satisfied by $v_2$: for $\gamma\geq \gamma_0$

\begin{align}\label{b6}
\begin{split}
&\gamma \left|\begin{pmatrix}\Lambda^{\frac{3}{2}}v_2\\D_{x_2}\Lambda^{\frac{1}{2}}v_2\end{pmatrix}\right|^2_{0,0,\gamma}+\gamma\left <  \begin{pmatrix}\Lambda v_2\\D_{x_2} v_2\end{pmatrix} \right>^2_{0,\gamma}
\leq\\
& C \left(\gamma^{-1}|\Lambda^{\frac{1}{2}}\mathcal{F}_2|^2_{0,0,\gamma}+\gamma\left< \Lambda \mathcal{G}_2\right>^2_{0,\gamma}+\sum_{J_h}\left<\phi_j\Lambda^{\frac{3}{2}}\mathcal{G}_2\right>^2_{0,\gamma}\right)+C\gamma^{-1}\left|\begin{pmatrix}\Lambda v_2/\sqrt{\eps}\\D_{x_2} v_2/\sqrt{\eps}\end{pmatrix}\right|^2_{0,0,\gamma}.
\end{split}
\end{align}

\begin{align}\label{b7}
\begin{split}
&\gamma \left|\begin{pmatrix}\Lambda v_2\\D_{x_2} v_2\end{pmatrix}\right|^2_{\infty,0,\gamma}\leq
 C \left(\gamma^{-1}|\mathcal{F}_2|_{0,1,\gamma}^2+\gamma\left< \Lambda \mathcal{G}_2\right>^2_{0,\gamma}+\gamma\left< \Lambda^{\frac{1}{2}}\mathcal{G}_2\right>^2_{1,\gamma}+\sum_{J_h}\left<\phi_j\Lambda\mathcal{G}_2\right>^2_{1,\gamma}\right).
\end{split}
\end{align}

\begin{align}\label{b8}
\begin{split}
&\gamma \left|\begin{pmatrix}\Lambda v_2\\D_{x_2} v_2\end{pmatrix}\right|^2_{0,1,\gamma}+\gamma \left <  \begin{pmatrix}\Lambda^{\frac{1}{2}} v_2\\D_{x_2}\Lambda^{-\frac{1}{2}} v_2\end{pmatrix} \right>^2_{1,\gamma}
\leq
 C \left(\gamma^{-1}|\mathcal{F}_2|_{0,1,\gamma}^2+\gamma\left< \Lambda^{\frac{1}{2}}\mathcal{G}_2\right>^2_{1,\gamma}+\sum_{J_h}\left<\phi_j\Lambda\mathcal{G}_2\right>^2_{1,\gamma}\right).
\end{split}
\end{align}
\end{prop}

\begin{rem}\label{b9}

 The terms on the right in \eqref{b3} and \eqref{b4} involving division by  $\sqrt{\eps}$ are needed to control the commutators with $\Lambda^{1/2}_D$ that appear when that operator is applied to the equation written as a first-order system.   




\end{rem}

\subsection{Proofs of the estimates.}

   \emph{\quad}Estimates \eqref{b3}, \eqref{b5} and \eqref{b6}, \eqref{b8} are proved by analyzing  the singular systems \eqref{b1}, \eqref{b2} using the singular pulse calculus of \cite{CGW}.   A first step is to rewrite the  second-order differential systems  as first order singular pseudo-differential systems to which we can readily apply singular Kreiss symmetrizers. 

 Estimates \eqref{b4} and \eqref{b7} require, in addition, a diagonalization argument using a cutoff $\chi$ \eqref{n31} in the extended calculus similar to that   given in \cite{W}, but modified in the case of \eqref{b4} to take into account the weakly stable nature of the problem \eqref{b1}.  The Dirichlet problem \eqref{b2} is uniformly stable, that is, the Lopatinskii determinant $D(0,z)$ (defined with $\cB_2=\begin{pmatrix}I_2&0\end{pmatrix}$ now) is nonvanishing for all $z\in\mathcal X$, so the set $J_c$ is empty in this case.

\emph{\quad} To prove the estimates for $v_1$ we set $U=(e^{\gamma t}\Lambda_{D}e^{-\gamma t} v_1,D_{x_2}v_1)=(\Lambda_{D,\gamma} v_1,D_{x_2} v_1)$, $U^\gamma=e^{-\gamma t}U$,  and rewrite \eqref{b1} as a $4\times 4$ singular first-order system:
\begin{align}\label{b10a}
D_{x_2}U^\gamma-\cA(w,D_{x',\eps})U^\gamma=\tilde{\cF}_1^\gamma,\;\;\cB_1(w,D_{x',\eps})U^\gamma=\tilde{\cG}_1^\gamma,
\end{align}
where, with the matrices $A_\alpha$ evaluated at $w$, we have
\begin{align}\label{b10b}
\begin{split}
& \tilde{\cF}_1:=\begin{pmatrix}0\\-A_{(0,2)}^{-1}\cF_1\end{pmatrix},\;\;\tilde{\cG}_1=-i\cG_1\\
&\cA(w,D_{x',\eps}):=-\begin{pmatrix}0&\Lambda_{D}I_2\\\left(A_{(0,2)}^{-1}(D_{t,\eps}-i\gamma)^2+A_{(0,2)}^{-1}A_{(2,0)}D_{x_1,\eps}^2\right)\Lambda^{-1}_{D}&A_{(0,2)}^{-1}A_{(1,1)}D_{x_1,\eps}\end{pmatrix}\\
&\cB_1(w,D_{x',\eps}):=\begin{pmatrix}-d_{v_1}H(w_1,h(w))D_{x_1,\eps}\Lambda^{-1}_{D}&I_2\end{pmatrix}.
\end{split}
\end{align}
This is the singular analogue of \eqref{d4b}.

Writing the system \eqref{b2} as a first order system for $U=(e^{\gamma t}\Lambda_{D}e^{-\gamma t} v_2,D_{x_2}v_2)$, we obtain
\begin{align}\label{b10d}
D_{x_2}U^\gamma-\cA(w,D_{x',\eps})U^\gamma=\tilde{\cF}_2^\gamma,\;\;\cB_2U^\gamma=\tilde{\cG}_2^\gamma,
\end{align}
where $\tilde{\cF}_2$ is defined like $\tilde{\cF}_1$, but now
\begin{align}
\cB_2=\begin{pmatrix}I_2&0\end{pmatrix}\text{ and }\tilde{\cG}_2=\Lambda_{D,\gamma}\cG_2.
\end{align}



The following proposition gives the $L^2$ estimates for the singular systems \eqref{b10a}, \eqref{b10d}.
\begin{prop}\label{bb1}
With notation and assumptions as in Proposition \ref{basicest}, we have the following a priori estimates for the linearized singular Neumann and Dirichlet systems,  \eqref{b10a} and  \eqref{b10d}.
For $\gamma\geq \gamma_0$,
\begin{align}\label{bb2}
\begin{split}
&\gamma \left|\begin{pmatrix}\Lambda v_1\\D_{x_2} v_1\end{pmatrix}\right|^2_{0,0,\gamma}
+\gamma \left <  \begin{pmatrix}\Lambda^{\frac{1}{2}} v_1\\D_{x_2}\Lambda^{-\frac{1}{2}} v_1\end{pmatrix} \right>^2_{0,\gamma}+ \sum_{J_h \cup J_e}\left <  \phi_j \begin{pmatrix}\Lambda v_1\\D_{x_2} v_1\end{pmatrix} \right>^2_{0,\gamma}\leq\\
&\qquad\qquad\qquad C\gamma^{-1}\left(|\mathcal{F}_1|_{0,0,\gamma}^2+\left < \Lambda^{\frac{1}{2}}\mathcal{G}_1\right>^2_{0,\gamma}\right)
\end{split}
\end{align}
and 
\begin{align}\label{bb3}
\begin{split}
&\gamma \left|\begin{pmatrix}\Lambda v_2\\D_{x_2} v_2\end{pmatrix}\right|^2_{0,0,\gamma}+\gamma \left <  \begin{pmatrix}\Lambda^{\frac{1}{2}} v_2\\D_{x_2}\Lambda^{-\frac{1}{2}} v_2\end{pmatrix} \right>^2_{0,\gamma}
\leq
 C \left(\gamma^{-1}|\mathcal{F}_2|_{0,0,\gamma}^2+\gamma\left< \Lambda^{\frac{1}{2}}\mathcal{G}_2\right>^2_{0,\gamma}+\sum_{J_h}\left<\phi_j\Lambda\mathcal{G}_2\right>^2_{0,\gamma}\right).
\end{split}
\end{align}

\end{prop}

\begin{proof}

\textbf{1. }The following estimate holds for  $\phi_{j,D}(e^{-\gamma t}U)$ with $j\in J_h\cup J_e$. 
\begin{align}\label{b10}
\begin{split}
&\gamma |\phi_j U|_{0,0,\gamma}^2+\langle\phi_jU\rangle^2_{0,\gamma}\leq \\
&\qquad C\left(\gamma^{-1}|\phi_j\mathcal{F}_1|_{0,0,\gamma}^2+\langle\phi_j\mathcal{G}_1\rangle_{0,\gamma}^2+\gamma^{-1}|U|_{0,0,\gamma}^2+\langle\Lambda^{-1}U(0)\rangle^2_{0,\gamma}\right).
\end{split}
\end{align}
This  Kreiss-type estimate is the singular analogue of estimate (1.6) in \cite{S-T}, p. 268, and it is proved by essentially the same argument; the only difference is that 
all symbols, including the symbol of the Kreiss symmetrizer, are now quantized in the singular calculus by the procedure described in section \ref{sect8}.  For example, if we let $r(v,\xi',\gamma)$ denote the symbol of the (nonsingular) Kreiss symmetrizer used in the proof of (1.6) in \cite{S-T}, the corresponding singular operator that must be used in the proof of \eqref{b10} is the operator associated to $r(w^\eps(x,\theta),X,\gamma)$.\footnote{Full details of how the singular calculus is implemented to prove estimates like \eqref{b10} are given in section 5 of \cite{W}.} 
 In view of the assumption \eqref{b3z} the function $w^\eps(x,\theta)$ can play the role of $\eps V(x,\theta)$ in the integral \eqref{singularpseudop}.   The last two terms on the right in \eqref{b10} arise from commutators by application of Proposition \ref{commutator3}.

\textbf{2. }
The following Kreiss-type estimate is a variant of \eqref{b10} that holds for $\phi_{j,D}U^\gamma$ with $j\in J_e$. 
\begin{align}\label{b11}
\begin{split}
&\gamma |\phi_j U|_{0,0,\gamma}^2+\gamma\langle\phi_j\Lambda^{-\frac{1}{2}}U\rangle^2_{0,\gamma}\leq \\
&\qquad C\left(\gamma^{-1}|\phi_j\mathcal{F}_1|_{0,0,\gamma}^2+\gamma\langle\phi_j\Lambda^{-\frac{1}{2}}\mathcal{G}_1\rangle_{0,\gamma}^2+\gamma^{-1}|U|_{0,0,\gamma}^2+\langle\Lambda^{-1}U(0)\rangle^2_{0,\gamma}\right).
\end{split}
\end{align}
 This is the analogue of estimate (1.7) in \cite{S-T}, p. 268, and is again proved by using singular operators in the argument given there. 

\textbf{3. } For $\phi_{j,D}U^\gamma$ with $j\in J_c$ (the ``bad" set where LU fails) we have
\begin{align}\label{b12}
\begin{split}
&\gamma\left(|\phi_j U|_{0,0,\gamma}^2+\langle\Lambda^{-\frac{1}{2}}\phi_jU\rangle^2_{0,\gamma}\right)\leq \\
&\qquad C\gamma^{-1}\left(|\phi_j\mathcal{F}_1|_{0,0,\gamma}^2+\langle\Lambda^{\frac{1}{2}}\phi_j\mathcal{G}_1\rangle_{0,\gamma}^2+|U|_{0,0,\gamma}^2+\langle\Lambda^{-\frac{1}{2}}U(0)\rangle^2_{0,\gamma}\right).
\end{split}
\end{align}
This is the analogue of estimate (1.10) in \cite{S-T}, p. 272, and, except for an important simplification that we now discuss, is again proved by implementing the argument given there with singular operators.

Consider a zero order symbol $b^+(v,z)$ as in Remark \ref{diag},  defined for $|v|$ small and $z\in\cO_j$, $j\in J_c$. Let $\tilde{b}^+(v,z)$ be an extension to $\mathcal{X}$ defined by taking extensions  $\tilde{k}_1$, $\tilde{k}_2$, and $\tilde{q}$ of the factors in \eqref{d4c}, where  the $\tilde{k}_i$ are elliptic of order zero and we still have $\tilde q$ of order one satisfying
\begin{align}\label{b12z}
\Im \tilde q(v,z)=-\gamma.
\end{align}
Dropping tildes on extensions and letting $q_D$ denote the singular operator with symbol $q(w^\eps,X,\gamma)$, we see that the estimate
\begin{align}\label{b12y}
\Re (iq_D\Lambda^{-1}_D v,v)\geq C\gamma |\Lambda^{-\frac{1}{2}}_D v|^2_{L^2}
\end{align}
follows immediately from \eqref{b12z}; there is no need for a singular sharp Garding inequality to replace the use of the standard sharp Garding inequality in \cite{S-T}. 
With \eqref{b12y} one obtains the key estimate, 
\begin{align}\label{key}
|\Lambda^{\frac{1}{2}} b^+_D w|_{L^2}\geq C\gamma |\Lambda^{-\frac{1}{2}}_Dw|_{L^2},
\end{align}
which is analogous to (1.8) of \cite{S-T}, by arguing as on p. 270 of \cite{S-T}.

\textbf{4. }For later use we observe that since $q_D\Lambda^{-1}_D=(q\Lambda^{-1})_D$, we can apply \eqref{b12z} and Proposition \ref{prop20}(b) to obtain
\begin{align}\label{b12yy}
\Re (iq_D\Lambda^{-1}_D v,\Lambda_Dv)\geq C\gamma |v|^2_{L^2}.
\end{align}
A minor variation of the argument that gave \eqref{key} now yields\footnote{We are not able to prove \eqref{key2} simply by applying \eqref{key} with $w$ replace by $\Lambda^{\frac{1}{2}}_Dw$.}
\begin{align}\label{key2}
|\Lambda_D b^+_D w|_{L^2}\geq C\gamma |w|_{L^2}.
\end{align}

\textbf{5. }Adding the estimates \eqref{b10}-\eqref{b12} and absorbing errors on the left yields for large enough $\gamma$:
\begin{align}\label{b15}
\begin{split}
&\gamma \left|\begin{pmatrix}\Lambda v_1\\D_{x_2} v_1\end{pmatrix}\right|^2_{0,0,\gamma}
+\gamma \left <  \begin{pmatrix}\Lambda^{\frac{1}{2}} v_1\\D_{x_2}\Lambda^{-\frac{1}{2}} v_1\end{pmatrix} \right>^2_{0,\gamma}+ \sum_{J_h \cup J_e}\left <  \phi_j \begin{pmatrix}\Lambda v_1\\D_{x_2} v_1\end{pmatrix} \right>^2_{0,\gamma}\leq\\
&\qquad\qquad\qquad C\gamma^{-1}\left(|\mathcal{F}_1|_{0,0,\gamma}^2+\left < \Lambda^{\frac{1}{2}}\mathcal{G}_1\right>^2_{0,\gamma}\right).
\end{split}
\end{align}

\textbf{6. }To prove the estimates for $v_2$ we use the first-order system \eqref{b10d}.
For the Dirichlet problem \eqref{b2} the bad set $J_c$ is empty, so we can apply the estimates \eqref{b10}, \eqref{b11} with $\tilde{\cF}_2$, $\tilde{\cG}_2$ in place of $\cF_1$, $\cG_1$.  Adding these estimates and absorbing terms yields
\begin{align}\label{b16}
\begin{split}
&\gamma \left|\begin{pmatrix}\Lambda v_2\\D_{x_2} v_2\end{pmatrix}\right|^2_{0,0,\gamma}+\gamma \left <  \begin{pmatrix}\Lambda^{\frac{1}{2}} v_2\\D_{x_2}\Lambda^{-\frac{1}{2}} v_2\end{pmatrix} \right>^2_{0,\gamma}
\leq
 C \left(\gamma^{-1}|\mathcal{F}_2|_{0,0,\gamma}^2+\gamma\left< \Lambda^{\frac{1}{2}}\mathcal{G}_2\right>^2_{0,\gamma}+\sum_{J_h}\left<\phi_j\Lambda\mathcal{G}_2\right>^2_{0,\gamma}\right).
\end{split}
\end{align}

\end{proof}

Next we turn to the proof of Proposition \ref{basicest}.     The $L^2$ estimate of Proposition \eqref{bb1} for $v_1$ can easily be applied (as in step \textbf{1} of the proof below) to yield the estimates \eqref{b3} and \eqref{b5}.  But these estimates by themselves provide no way to estimate $\left|\begin{pmatrix}\Lambda v_1\\D_{x_2} v_1\end{pmatrix}\right|_{\infty,0,\gamma}$; the presence of the factors of $1/\eps$ in the components of $\cA(w,D_{x',\eps})$ prevent us from employing the usual argument that uses the equation to control that norm by controlling $|D_{x_2}U|_{0,0,\gamma}$.   Instead, we must  adapt to this weakly stable problem the type of argument that was  used in \cite{W} to control 
$|U|_{\infty,0,\gamma}$ in the uniformly stable case.   That is done for $v_1$ in steps \textbf{2} and  \textbf{4} of the following proof.

\begin{proof}[Proof of Proposition \ref{basicest}]
 
 \textbf{1.  Proof of \eqref{b5}. }The estimate  \eqref{b5}) follows from \eqref{b15} applied to the problem satisfied by $\partial_{(t,x_1,\theta)} U$. However, we are \emph{not} able to prove \eqref{b3} simply by applying \eqref{b15} to the problem satisfied by $\cU:=\Lambda^{1/2}_{D,\gamma}U$.   The norm of the commutator $\langle\Lambda^{\frac{1}{2}}_D[\cB_1,\Lambda^{\frac{1}{2}}_D]U^\gamma\rangle_{L^2}$, introduces error terms that are unacceptably large.\footnote{These error terms are exhibited in Proposition \ref{commutator4}(b).}   We get around this in the next step by microlocalizing the estimates.  



\textbf{2. Proof of \eqref{b3}. }      The next two estimates are proved by applying \eqref{b10} and \eqref{b11} to the problem satisfied by $\cU$;  the less singular boundary norms in these microlocal estimates, as compared with \eqref{b15}, give rise to acceptable commutator errors.   For $j\in J_h\cup J_e$ we have
\begin{align}\label{b10aa}
\begin{split}
&\gamma |\Lambda^{\frac{1}{2}}\phi_j U|_{0,0,\gamma}^2+\langle\Lambda^{\frac{1}{2}}\phi_jU\rangle^2_{0,\gamma}\lesssim \\
&\qquad \gamma^{-1}|\Lambda^{\frac{1}{2}}\mathcal{F}_1|_{0,0,\gamma}^2+\langle\Lambda^{\frac{1}{2}}\mathcal{G}_1\rangle_{0,\gamma}^2+\gamma^{-1}|\Lambda^{\frac{1}{2}}U|_{0,0,\gamma}^2+\gamma^{-1}|U/\sqrt{\eps}|_{0,0,\gamma}^2+\gamma^{-1}\langle U(0)\rangle^2_{0,\gamma},
\end{split}
\end{align}
while for  $j\in J_e$ we obtain
\begin{align}\label{b11aa}
\begin{split}
&\gamma |\Lambda^{\frac{1}{2}}\phi_j U|_{0,0,\gamma}^2+\gamma\langle\phi_j U\rangle^2_{0,\gamma}\lesssim \\
&\qquad \gamma^{-1}|\Lambda^{\frac{1}{2}}\mathcal{F}_1|_{0,0,\gamma}^2+\gamma\langle \mathcal{G}_1\rangle_{0,\gamma}^2+\gamma^{-1}|\Lambda^{\frac{1}{2}}U|_{0,0,\gamma}^2+\gamma^{-1}|U/\sqrt{\eps}|_{0,0,\gamma}^2+\gamma^{-1}\langle U(0)\rangle^2_{0,\gamma}.
\end{split}
\end{align}
The terms $|U/\sqrt{\eps}|_{0,0,\gamma}^2$ arise from Proposition \ref{commutator4}(a) applied to the interior commutator, while the terms $\gamma^{-1}\langle U(0)\rangle^2_{0,\gamma}$ come from Proposition \ref{commutator4}(d) applied to the boundary commutator.

Next we show that for $j\in J_c$
 \begin{align}\label{b12m}
\begin{split}
&\gamma\left(|\Lambda^{\frac{1}{2}}\phi_j U|_{0,0,\gamma}^2+\langle \phi_jU(x_2)\rangle^2_{0,\gamma}\right)\lesssim\\
&\gamma^{-1}\left(|\Lambda^{\frac{1}{2}}\mathcal{F}_1|_{0,0,\gamma}^2+\langle\Lambda \mathcal{G}_1\rangle_{0,\gamma}^2+|\Lambda^{\frac{1}{2}}U|_{0,0,\gamma}^2+|U/\sqrt{\eps}|_{0,0,\gamma}^2+\langle U(0)\rangle^2_{0,\gamma}\right).
\end{split}
\end{align}

In this case \eqref{conj} becomes $S^{-1}\cA S=\mathrm{diag}(a^+,a^-)$ in $\cO_j$.  Letting $a^\pm$ denote $C^\infty$, homogeneous of degree one extensions to  $\mathcal{X}$ satisfying 
$\pm\mathrm{Im} \;a^\pm>0$, we consider the problems
\begin{align}\label{b12n}
\begin{split}
&D_{x_2}V_--a^-_DV_-=F_-\\
&D_{x_2}V_+-a^+_DV_+=F_+,\;\;  b^+_DV_+(0)=G_+.
\end{split}
\end{align}
Writing
\begin{align}\label{b12o}
\begin{split}
&\mathrm{Im} \int^\infty_{x_2} \langle (D_{x_2}-a^-_D)V_-,\Lambda^2_D V_-\rangle dx_2= \mathrm{Im}\int^\infty_{x_2}\langle F_-,\Lambda_D^2 V_-\rangle dx_2\\
&\mathrm{Im} \int^{x_2}_{0} \langle (D_{x_2}-a^+_D)V_+, V_+\rangle dx_2=\mathrm{Im} \int^{x_2}_{0}\langle F_+, V_+\rangle dx_2,
\end{split}
\end{align}
integrating terms involving $D_{x_2}$  by parts in $x_2$,  and applying the singular Garding's inequality of Theorem \ref{thm11} to terms involving $a^\pm_D$, we obtain
\begin{align}\label{b16z}
\begin{split}
&\langle \Lambda_D V_-(x_2)\rangle_{L^2(x',\theta)}+|\Lambda^{\frac{3}{2}}_DV_-|_{L^2(x,\theta)}\lesssim |\Lambda^{\frac{1}{2}}F_-|_{L^2(x,\theta)}+|V_-/\sqrt{\eps}|_{L^2(x,\theta)}\\
&\sqrt{\gamma} \langle  V_+(x_2)\rangle_{L^2(x',\theta)}+\sqrt{\gamma}|\Lambda^{\frac{1}{2}}_DV_+|_{L^2(x,\theta)}\lesssim    \sqrt{\gamma} \langle  V_+(0)\rangle_{L^2(x',\theta)}+|F_+|_{L^2(x,\theta)}\\
\end{split}
\end{align}
Adding the estimates \eqref{b16z} and using \eqref{key2}, we find
\begin{align}\label{b16p}
\begin{split}
&\sqrt{\gamma} \left(|\Lambda^{\frac{1}{2}}_DV|_{L^2}+\langle V(x_2)\rangle_{L^2}\right)\lesssim \sqrt{\gamma} \langle V_+(0)\rangle_{L^2}+\frac{1}{\sqrt{\gamma}}\left(|\Lambda^{\frac{1}{2}}_DF|_{L^2}+|V_-/\sqrt{\eps}|_{L^2}\right)\\
&\lesssim  \frac{1}{\sqrt{\gamma}}\left(\langle \Lambda_Db^+_DV_+(0)\rangle_{L^2}+|\Lambda^{\frac{1}{2}}_DF|_{L^2}+|V_-/\sqrt{\eps}|_{L^2}\right)\lesssim\\\
&\frac{1}{\sqrt{\gamma}}\left(\langle \Lambda_Db_DV(0)\rangle_{L^2}+\langle \Lambda_Db^-_DV_-(0)\rangle_{L^2}+|\Lambda^{\frac{1}{2}}_DF|_{L^2}+|V_-/\sqrt{\eps}|_{L^2}\right)\lesssim\\
&\frac{1}{\sqrt{\gamma}}\left(\langle \Lambda_Db_DV(0)\rangle_{L^2}+|\Lambda^{\frac{1}{2}}_DF|_{L^2}+|V_-/\sqrt{\eps}|_{L^2}\right)\lesssim\\
\end{split}
\end{align}
By applying \eqref{b16p} to $V=(V_+,V_-):=S^{-1}_D\phi_{j,D}U^\gamma$ and using the commutator estimates of section 7, in particular Proposition \ref{commutator5} and its corollary, we obtain \eqref{b12m}. 

Combining the estimates \eqref{b10aa}, \eqref{b11aa}, and \eqref{b12m}   yields \eqref{b3}.

\textbf{3. } For later use we observe that a similar argument yields for $j\in J_c$ the following extension of \eqref{b12}:
\begin{align}\label{b16q}
\begin{split}
&\gamma\left(|\phi_j U|_{0,0,\gamma}^2+\langle\Lambda^{-\frac{1}{2}}\phi_jU(x_2)\rangle^2_{0,\gamma}\right)\leq \\
&\qquad C\gamma^{-1}\left(|\phi_j\mathcal{F}_1|_{0,0,\gamma}^2+\langle\Lambda^{\frac{1}{2}}\phi_j\mathcal{G}_1\rangle_{0,\gamma}^2+|U|_{0,0,\gamma}^2+\langle\Lambda^{-\frac{1}{2}}U(0)\rangle^2_{0,\gamma}\right).
\end{split}
\end{align}
For this one takes $L^2$ pairings  with $\Lambda_D V_-$ and $\Lambda^{-1}_DV_+$ in \eqref{b12o}, instead of with $\Lambda^2_D V_-$ and $V_+$.
Similarly, one derives the following extension of \eqref{b11} for $j\in J_e$:
\begin{align}\label{b16s}
\begin{split}
&\gamma |\phi_j U|_{0,0,\gamma}^2+\gamma\langle\phi_j\Lambda^{-\frac{1}{2}}U(x_2)\rangle^2_{0,\gamma}\leq \\
&\qquad C\left(\gamma^{-1}|\phi_j\mathcal{F}_1|_{0,0,\gamma}^2+\gamma\langle\phi_j\Lambda^{-\frac{1}{2}}\mathcal{G}_1\rangle_{0,\gamma}^2+\gamma^{-1}|U|_{0,0,\gamma}^2+\langle\Lambda^{-1}U(0)\rangle^2_{0,\gamma}\right).
\end{split}
\end{align}

\textbf{4. Proof of \eqref{b4}. }For $j\in J_c$ the estimate \eqref{b12m} implies
 \begin{align}\label{b12mm}
\begin{split}
&\gamma |\phi_jU|_{\infty,0,\gamma}^2\lesssim
\gamma^{-1}\left(|\Lambda^{\frac{1}{2}}\mathcal{F}_1|_{0,0,\gamma}^2+\langle\Lambda \mathcal{G}_1\rangle_{0,\gamma}^2+|\Lambda^{\frac{1}{2}}U|_{0,0,\gamma}^2+|U/\sqrt{\eps}|_{0,0,\gamma}^2+\langle U(0)\rangle^2_{0,\gamma}\right).
\end{split}
\end{align}

  Let $\chi_D$ denote a cutoff in the extended calculus, with the support property \eqref{n31},  chosen so that $\chi=\sum_{j\in J_c}\phi_j\chi$.  The estimate \eqref{b12mm}  yields
\begin{align}\label{b16r}
\begin{split}
&\gamma |\chi U|_{\infty,0,\gamma}^2\lesssim  \gamma^{-1}\left(|\Lambda^{\frac{1}{2}}\mathcal{F}_1|_{0,0,\gamma}^2+\langle\Lambda \mathcal{G}_1\rangle_{0,\gamma}^2+|\Lambda^{\frac{1}{2}}U|_{0,0,\gamma}^2+|U/\sqrt{\eps}|_{0,0,\gamma}^2+\langle U(0)\rangle^2_{0,\gamma}\right).
\end{split}
\end{align}
On the other hand from the support property of $\chi$ and the equation we obtain
\begin{align}\label{b17}
|(1-\chi)U|_{\infty,0,\gamma}\leq |D_{x_2}(1-\chi)U|_{0,0,\gamma}+|U|_{0,0,\gamma}\leq C(|U|_{0,1,\gamma}+|\mathcal{F}_1|_{0,0,\gamma}).
\end{align}
The estimate \eqref{b4} for $v_1$ is then a consequence of  \eqref{b16r}, \eqref{b3}, and \eqref{b5}. 

\textbf{5. Proof of \eqref{b6} and \eqref{b8}. } The proof of \eqref{b8} is like that of \eqref{b5} but simpler, since boundary commutators are zero.
The estimate \eqref{b6} is proved by applying \eqref{bb3} to the problem satisfied by $\Lambda^{\frac{1}{2}}_DU^\gamma$, where $U^\gamma$ is defined using $v_2$ now.

\textbf{6.  Proof of \eqref{b7}. }The argument is similar to step \textbf{2}, but the set $J_c$ is empty now.   We show first that for $j\in J_e$ we have 
\begin{align} \label{b17z}
|\phi_j U|^2_{\infty,0,\gamma}
\leq  C\left(\gamma^{-1}|\mathcal{F}_2|_{0,0,\gamma}^2+\langle\Lambda\mathcal{G}_2\rangle_{0,\gamma}^2+\gamma^{-1}|U|_{0,0,\gamma}^2+\langle\Lambda^{-1}U(0)\rangle^2_{0,\gamma}\right).
\end{align}
In place of \eqref{b12n}, \eqref{b12o} we have
\begin{align}
\begin{split}
&\mathrm{Im} \int^\infty_{x_2} \langle (D_{x_2}-a^-_D)V_-, V_-\rangle dx_2=\mathrm{Im} \int^\infty_{x_2}\langle F_-, V_-\rangle dx_2\\
&\mathrm{Im} \int^{x_2}_{0} \langle (D_{x_2}-a^+_D)V_+, V_+\rangle dx_2=\mathrm{Im} \int^{x_2}_{0}\langle F_+,V_+\rangle dx_2
\end{split}
\end{align}
and 
\begin{align}
\begin{split}
&\sqrt{\gamma}\langle  V_-(x_2)\rangle_{L^2(x',\theta)}+\gamma|V_-|_{L^2(x,\theta)}\lesssim |F_-|_{L^2(x,\theta)}\\
&\sqrt{\gamma} \langle  V_+(x_2)\rangle_{L^2(x',\theta)}+\gamma|V_+|_{L^2(x,\theta)}\lesssim    \sqrt{\gamma}\langle  V_+(0)\rangle_{L^2(x',\theta)}+|F_+|_{L^2(x,\theta)},
\end{split}
\end{align}
which implies
\begin{align}\label{b12p}
\sqrt{\gamma}\langle V(x_2)\rangle_{L^2}+\gamma |V|_{L^2}\lesssim |F|_{L^2}+\sqrt{\gamma}\langle b_DV(0)\rangle_{L^2}.
\end{align}
Here in place of \eqref{key} we have used
\begin{align}
\langle b^+_Dw\rangle_{L^2}\geq C \langle w\rangle_{L^2}.
\end{align}
Applying \eqref{b12p} to the problem satisfied by $V=S^{-1}_D\phi_{j,D}U^\gamma$, we obtain \eqref{b17z}.

Let $\chi$ denote a cutoff in the extended calculus with the support property \eqref{n31}.    Since $\chi$ can be chosen now   so that $\chi=\sum_{j\in J_e}\phi_j\chi$, we have
\begin{align} 
|\chi U|^2_{\infty,0,\gamma}
\leq  C\left(\gamma^{-1}|\mathcal{F}_2|_{0,0,\gamma}^2+\langle\Lambda\mathcal{G}_2\rangle_{0,\gamma}^2+\gamma^{-1}|U|_{0,0,\gamma}^2+\langle\Lambda^{-1}U(0)\rangle^2_{0,\gamma}\right).
\end{align}
Together with \eqref{b17} (for $U$ corresponding to $v_2$ now) and \eqref{b8} this yields \eqref{b7}.

\end{proof}

\section{Uniform time of existence for the nonlinear singular systems}\label{uniform}

\emph{\quad} In this section we describe our strategy for proving the existence of solutions to the coupled, nonlinear, singular systems  \eqref{a7}-\eqref{a9} on a fixed time interval independent of small $\eps>0$. 
First, we define a ``singular energy norm" $E_{m,\gamma}(v)$, a time-localized variant  of which, namely $E_{m,T}(v)$, will be shown to be uniformly bounded on a time interval independent of $\eps\in (0,1]$.

\subsection{Singular energy norms}\label{singular}

\emph{\quad} Let $m$ be a nonnegative integer.  
Suppressing  the epsilon superscript on $v$ and the subscript $D$ on the singular operator $\Lambda_D$, we
define  for $v=(v_1,v_2)$ the singular energy norm =
\begin{align}\label{c00}
E_{m,\gamma}(v):=E_{m,\gamma}(v_1)+E_{m,\gamma}(v_2),
\end{align}
where $E_{m,\gamma}(v_j)=$ 
\begin{align}\label{c0a}
\begin{split}
&\left|\begin{pmatrix}\Lambda^{\frac{3}{2}}v_j\\D_{x_2}\Lambda^{\frac{1}{2}}v_j\end{pmatrix}\right|^{(2)}_{0,m,\gamma}+ \left|\begin{pmatrix}\Lambda v_j\\D_{x_2} v_j\end{pmatrix}\right|^{(2)}_{\infty,m,\gamma}+\left|\begin{pmatrix}\Lambda v_j\\D_{x_2} v_j\end{pmatrix}\right|^{(2)}_{0,m+1,\gamma}
+\left <  \begin{pmatrix}\Lambda^{\frac{1}{2}} v_j\\D_{x_2}\Lambda^{-\frac{1}{2}} v_j\end{pmatrix} \right>_{m+1,\gamma}\\
&+\left|\begin{pmatrix}\sqrt{\eps}\Lambda^{\frac{3}{2}} v_j\\\sqrt{\eps}\Lambda^{\frac{1}{2}}D_{x_2} v_j\end{pmatrix}\right|^{(2)}_{0,m+1,\gamma}
+\left <  \begin{pmatrix}\sqrt{\eps}\Lambda v_j\\\sqrt{\eps}D_{x_2} v_j\end{pmatrix} \right>_{m+1,\gamma}+\left|\begin{pmatrix}\Lambda v_j/\sqrt{\eps}\\D_{x_2} v_j/\sqrt{\eps}\end{pmatrix}\right|^{(2)}_{0,m,\gamma}+\left \langle  \begin{pmatrix}\Lambda^{\frac{1}{2}} v_j/\sqrt{\eps}\\D_{x_2}\Lambda^{-\frac{1}{2}} v_j/\sqrt{\eps}\end{pmatrix} \right\rangle_{m,\gamma} \\
&+\left|\frac{\nabla_\eps u}{\eps}\right|_{\infty,m,\gamma}+\left|\frac{\nabla_\eps u}{\eps}\right|_{0,m+1,\gamma}+\left|\frac{\Lambda^{\frac{1}{2}}\nabla_\eps u}{{\sqrt{\eps}}}\right|_{0,m+1,\gamma}+\left\langle\frac {\nabla_\eps u}{\sqrt{\eps}}\right\rangle_{m+1,\gamma}+\left|\frac{\Lambda^{\frac{1}{2}}\nabla_\eps u}{\eps}\right|_{0,m,\gamma}.
\end{split}
\end{align}
Here the superscripts $(2)$ on some terms have the following meanings:\footnote{In each case the second term on the right is obtained from the first term by trading $\eps D_{x_2}$ for a tangential derivative; this applies as well to the other three terms  of \eqref{c0a} with $(2)$ as a superscript.}
\begin{align}\label{c0bb}
\begin{split}
&\left|\begin{pmatrix}\Lambda^{\frac{3}{2}}v_j\\D_{x_2}\Lambda^{\frac{1}{2}}v_j\end{pmatrix}\right|^{(2)}_{0,m,\gamma}=\left|\begin{pmatrix}\Lambda^{\frac{3}{2}}v_j\\D_{x_2}\Lambda^{\frac{1}{2}}v_j\end{pmatrix}\right|_{0,m,\gamma}+\left|\begin{pmatrix}\eps D_{x_2}\Lambda^{\frac{3}{2}}v_j\\\eps\Lambda^{\frac{1}{2}}D_{x_2}^2v_j\end{pmatrix}\right|_{0,m-1,\gamma},\\
 &\left|\begin{pmatrix}\sqrt{\eps}\Lambda^{\frac{3}{2}} v_j\\\sqrt{\eps}\Lambda^{\frac{1}{2}}D_{x_2} v_j\end{pmatrix}\right|^{(2)}_{0,m+1,\gamma}=\left|\begin{pmatrix}\sqrt{\eps}\Lambda^{\frac{3}{2}} v_j\\\sqrt{\eps}\Lambda^{\frac{1}{2}}D_{x_2} v_j\end{pmatrix}\right|_{0,m+1,\gamma}+\left|\begin{pmatrix}\eps^{\frac{3}{2}}D_{x_2}\Lambda^{\frac{3}{2}} v_j\\\eps^{\frac{3}{2}}\Lambda^{\frac{1}{2}}D^2_{x_2} v_j\end{pmatrix}\right|_{0,m,\gamma}.
\end{split}
\end{align}

\begin{rem}

1.  It would be more proper to denote the norms  in \eqref{c00} and  \eqref{c0a} by $E_{m,\gamma}(v,\nabla_\eps u)$ and $E_{m,\gamma}(v_j,\nabla_\eps u)$, so there is an abuse of notation here that we have introduced in order to lighten many of the expressions that occur throughout the paper.  The fact that $\nabla_\eps u$ rather than $v$ appears in the third line of \eqref{c0a} is related to the fact that we will later have $v^\eps=\nabla_\eps u^\eps$ \emph{only} on $\Omega_{T_\eps}$. 

2.  Note that all interior norms in the first two lines of the  definition of $E_{m,\gamma}(v)$ have ``total weight" $m+3$ if one assigns a weight of $1$ to  $\eps D_{x_2}$, $\infty$, $\Lambda^{\frac{1}{2}}$,  $\frac{1}{\sqrt{\eps}}$, and  to each tangential derivative $\partial$, and assigns a weight of $2$ to $D_{x_2}$.  All boundary norms in those lines have weight $m+2$.\footnote{In this weighting, $\eps D_{x_2}$ must always be viewed like a tangential derivative of weight $1$, not as the ``composite" of $\eps$ (weight $-2$) and $D_{x_2}$ (weight $2$)} 
The same applies to the third line if $\nabla_\eps u$ is treated like $v$.

3.  The second terms on the right in \eqref{c0bb}  arise in the interior commutator estimates.  See, for example, part (b) of Proposition \ref{e26}.

4. For each $j$, $E_{m,\gamma}(v_j)$ is a sum of $18$ terms.  This is the smallest number of terms for which we have been able to obtain an estimate that ``closes" like the one in Proposition \ref{c5}.
\end{rem}


We also define singular energy norms localized in time  for functions defined on $\Omega_T$.  
Set 
\begin{align}\label{c0c}
E_{m,T}(v)=E_{m,T}(v_1)+E_{m,T}(v_2),
\end{align}
where the norms
$E_{m,T}(v_i)$ are defined by the right  side of \eqref{c0a}, except that norms $|\Lambda_1^r v_i|_{0,m,T}$, $|\Lambda^r_1v_i|_{\infty,m,T}$, etc., are now used in place of $|\Lambda^r v_i|_{0,m,\gamma}$, $|\Lambda^rv_i|_{\infty,m,\gamma}$, etc..\footnote{Recall parts (g)-(j) of Notations \ref{spaces}.}Similarly, norms $E_{m}(v_i)$ are defined using norms $|\Lambda_1^r v_i|_{0,m}$, $|\Lambda^r_1v_i|_{\infty,m}$, etc..

Finally, we define in the obvious way $E_{m,\gamma,T}(v)$, the time-localized version of $E_{m,\gamma}(v)$,  using $|\Lambda^r v_i|_{0,m,\gamma,T}$, $|\Lambda^rv_i|_{\infty,m,\gamma,T}$, etc..

\begin{prop}\label{c0e}
a) Let $T>0$.   For $v$ and $u$ defined on $\Omega$ we have
\begin{align}\label{c0ee}
E_{m,\gamma}(v)\geq E_{m,\gamma,T}(v)\geq C e^{-\gamma T} E_{m,T}(v).
\end{align}
b) For $v$ and $u$  initially defined in $\Omega_T$ with $v=0$ in $t<0$, one can choose a Seeley extension $v^s_T$ to $\Omega$ \cite{CP} such that 
\begin{align}\label{c0g}
\begin{split}
&E_{m,\gamma}(v^s_T)\leq CE_{m,\gamma,T}(v^s_T)\leq C E_{m,T}(v^s_T),\\
&E_m(v^s_T)\leq C E_{m,T}(v^s_T),\text{ and }\\
&\text{the $t$-support of $v^s$ is contained in $[0,\frac{3}{2}T]$}.
\end{split}
\end{align} 
Here all constants $C$ are independent of $\gamma$, $\eps$, and $T$, and our  Seeley extensions $v^s_T$  have $t$-support in a fixed compact subset of $[0,2)$ when $0<T\leq 1$. 

(c)  The exact analogues of (a) and (b) hold for the norms $|w|^*_{m,\gamma}$, $|w|^*_{m,\gamma,T}$, $|w|^*_{m,T}$, and $|w|^*_{m}$ defined in Definition \ref{defnN}.

\end{prop}

\begin{proof}
\textbf{1. }
For $T>0$ the Seeley extension used here is defined first on sufficiently smooth functions $v(t,x,\theta)$, defined in $t\leq T$ and supported in $t\geq 0$, 
by 
\begin{align}\label{c0gg}
v^s(t,x,\theta)=\begin{cases}v(t,x,\theta),\; t\leq T\\ \sum^M_{k=1}\lambda_k  \;v(T+2^k(T-t),x,\theta),\;t>T\end{cases}, 
\end{align}
where $M$ is sufficiently large (depending on $m$) and the $\lambda_k$ are chosen to satisfy $\sum_{k=1}^M \lambda_k  2^{kj}=(-1)^j$ for $j=0,\dots,M-1$.  
Observe that all terms in the sum have $t$-support in $[0,\frac{3}{2}T]$.  The definition is extended by continuity to more general functions.

\textbf{2. } Inequalities corresponding to  \eqref{c0ee} and \eqref{c0g} are proved for the individual norms appearing in the definition of $E_{m,\gamma}(v)$.
Since we can replace $\Lambda_D w$ by 
$((\partial_{x'}+\beta\frac{\partial_\theta}{\eps})w,\gamma w)$, such inequalities are obvious for the  norms that involve only integral powers of $\Lambda_D$.  The inequalities in   the non-integral cases then follow by interpolation. 
\end{proof}

\subsection{Strategy}\label{strategy}

\emph{\quad} By Proposition \ref{localex}  for each fixed $\eps\in (0,1]$ we have  classical solutions $v^\eps$, $u^\eps$  to the quasilinear singular systems \eqref{a7}, \eqref{a8}, and \eqref{a9} on $\Omega_{T_\eps}$ for some $T_\eps>0$.   Moreover, we have $v^\eps=\nabla_\eps u^\eps$ on $\Omega_{T_\eps}$.  Below we often  suppress the superscript $\eps$ on $u$ and $v$.  

Let $T_0<1$ and $M_0$ be positive constants that will be chosen later to be sufficiently small (independently of $\eps$, $\gamma$, and $T$).    For a fixed $m> 3d+4+\frac{d+1}{2}$ define for each $\eps\in (0,1]$
\begin{align}\label{c0h}
T^*_\eps=\sup \{T\in(0,\min(T_0,\frac{T_\eps}{2}]: E_{m,T}(v^\eps)  \leq M_0\},
\end{align}
a number that might converge to $0$ as  $\eps\to 0$, as far as we know now.  
We will use the a priori estimate of Proposition \ref{c5} and a continuation argument to prove the following proposition.

\begin{prop}\label{mainprop}
We make the same assumptions as in Theorem \ref{uniformexistence}, and let  $v^\eps$, $u^\eps$ on $\Omega_{T_\eps}$ satisfying $v^\eps=\nabla_\eps u^\eps$ be solutions  to the quasilinear singular systems \eqref{a7}, \eqref{a8}, and \eqref{a9} provided by Proposition \ref{localex}.  Define $T^*_\eps$ as in \eqref{c0h}.  
There exist $\eps_0>0$ and $T_1$  independent of $\eps\in (0,\eps_0]$ such that
\begin{align}\label{c0i}
T^*_\eps\geq T_1>0\text{ for all }\eps\in (0,\eps_0].  
\end{align} 

\end{prop}

For any $T$ satisfying $0<T\leq T^*_\eps$ we let $v^s_T$ denote a Seeley extension of $v|_{\Omega_T}$ to $\Omega$ chosen as in Proposition \ref{c0e}.  We suppress the  $T$-dependence of $v^s_T$ below  and write $v^s$ for $v^s_T$.   

Consider now the following three \emph{linear} systems for the respective unknowns $v_1$, $v_2$, and $u$.  
We choose $\chi_0(t)\geq 0$ to be a  $C^\infty$ function that  is equal to 1 on a neighborhood of $[0,1]$ and supported in $(-1,2)$. 
\begin{align}\label{c1}
\begin{split}
&\partial_{t,\eps}^2 v_1+\sum_{|\alpha|=2} A_\alpha(v^s)\partial_{x,\eps}^\alpha v_1=-\left[\sum_{|\alpha|=2,\alpha_1\geq 1}\partial_{x_1,\eps}(A_\alpha(v^s))\partial_{x_1,\eps}^{\alpha_1-1}\partial_{x_2}^{\alpha_2}v_1^s-\partial_{x_1,\eps}(A_{(0,2)}(v^s))\partial_{x_2}v_2^s\right]\\
&\partial _{x_2} v_1-d_{v_1}H(v^s_1,h(v^s))\partial_{x_1,\eps}v_1=d_gH(v^s_1,\eps^2G)\partial_{x_1,\eps}(\eps^2G)\text{ on }x_2=0.
\end{split}
\end{align}

\begin{align}\label{c2}
\begin{split}
&\partial_{t,\eps}^2 v_2+\sum_{|\alpha|=2} A_\alpha(v^s)\partial_{x,\eps}^\alpha v_2=-\left[\sum_{|\alpha|=2,\alpha_1\geq 1}\partial_{x_2}(A_\alpha(v^s))\partial_{x_1,\eps}^{\alpha_1-1}\partial_{x_2}^{\alpha_2}v^s_1-\partial_{x_2}(A_{(0,2)}(v^s))\partial_{x_2}v^s_2\right]\\
&v_2=\chi_0(t)H(v_1,\eps^2 G)\text{ on }x_2=0.
\end{split}
\end{align}

\begin{align}\label{c3}
\begin{split}
&\partial_{t,\eps}^2 u+\sum_{|\alpha|=2} A_\alpha(v^s)\partial_{x,\eps}^\alpha u=0\\\
&\partial _{x_2} u-d_{v_1}H(v^s_1,h(v^s))\partial_{x_1,\eps}u=\left[H(v^s_1,\eps^2 G(x',\theta))-d_{v_1}H(v^s_1,\eps^2 G)v^s_1\right]\text{ on }x_2=0.
\end{split}
\end{align}

\begin{rem}\label{c4}
1.  The above three systems are solved on the full domain $\Omega$.  Each system depends on the parameters $\eps$ and $T$.

2.  The definition of $v^s=v^s_T$ and causality  (see Remark  \ref{k2y}) imply that the solutions $v_1$, $v_2$,  and $u$ of \eqref{c1}, \eqref{c2}, and \eqref{c3} are equal to the corresponding solutions of  the  systems \eqref{a7}, \eqref{a8}, and \eqref{a9} on $\Omega_T$.     The right sides of these equations depend on $v^s_T$, so the solutions 
$v$ and $u$ change as $T$ changes, but we suppress this $T$-dependence in the notation.

3.   We recall that the $t-$support of $v^s$ is contained in $[0,\frac{3}{2}T]$.  Since we always have $0\leq 2T\leq T_\eps$ and $v=\nabla_\eps u$ on $\Omega_{T_\eps}$, it follows that $v^\eps=\nabla_\eps u^\eps$ on $\mathrm{supp }\;v^s$ and  $v^s=\nabla_\eps u^s$. 

\end{rem}

    We will estimate solutions to \eqref{c1} and \eqref{c2} by applying the estimates of section \ref{main}, where $\mathcal{F}_1$ and $\mathcal{G}_1$ are given by the right sides of the interior and boundary equations of \eqref{c1}, and where $\mathcal{F}_2$ and $\mathcal{G}_2$ are given by the right sides of \eqref{c2}.  Note that the factor
$\partial_{x_1,\eps}(\eps^2G)$ occurs in $\mathcal{G}_1$, and that the singular derivative here ``uses up" one of the two factors of $\eps$ on $G$.   The remaining factor is used up by $\Lambda_D$ in a term like $\langle \Lambda \mathcal{G}_1\rangle^2_{0,\gamma}$, which occurs on the right side of \eqref{b3}.  On the other hand there is still a factor of $\eps^{1/2}$ ``to spare" in a term like $\langle \phi_j\Lambda^{\frac{3}{2}}\mathcal{G}_2\rangle^2_{0,\gamma}$ which occurs on the right side of \eqref{b6}.

\begin{rem}\label{c4a}
1.  Observe that $v_1$, not $v^s_1$, occurs on the right in the  boundary condition for \eqref{c2}; the reason is the following.  When we apply the estimate \eqref{b6} to estimate 
(say) $|\Lambda^{\frac{3}{2}}v_2 |_{0,m,\gamma}$, we  need to bound $\langle \phi_j\Lambda^{\frac{3}{2}}\mathcal{G}_2\rangle_{m,\gamma}$,  for $j\in J_h$.  Since $\mathcal{G}_2$ depends on $v_1$, Proposition \ref{f7} implies that to estimate  $\langle \phi_j\Lambda^{\frac{3}{2}}\mathcal{G}_2\rangle_{m,\gamma}$ we need control of $\langle\phi_j\Lambda^{\frac{3}{2}}v_1\rangle_{m,\gamma}$, $j\in J_h$.  That control comes from the estimate \eqref{b3}.  If $v^s_1$ appeared as an argument of $\mathcal{G}_2$ in the boundary condition for \eqref{c2}, we would have to estimate   $\langle\phi_j\Lambda^{\frac{3}{2}}v^s_1\rangle_{m,\gamma}$ instead, but we know of no way to do this.  The estimate \eqref{b3} does not apply to $v^s_1$, since $v^s_1$ is not a solution of \eqref{c1} on $\Omega$ (just on $\Omega_T$).  In addition, we cannot use a result like Proposition \ref{c0e} to deduce control of 
$\langle\phi_j\Lambda^{\frac{3}{2}}v^s_1\rangle_{m,\gamma}$ from control of $\langle\phi_j\Lambda^{\frac{3}{2}}v_1\rangle_{m,\gamma}$, since we know of no such result that 
applies to norms involving  $\phi_{j,D}$; we have no way to localize such norms in time.    

On the other hand, since $v=v^s=v^s_T$ on $\Omega_T$, we have $E_{m,T}(v)=E_{m,T}(v^s)$,  so we can use Proposition \ref{c0e} to deduce control of $E_{m,\gamma}(v^s)$ from control of $E_{m,\gamma}(v)$.   There is no version of Proposition \ref{c0e} for norms involving $\phi_{j,D}$, and that is one reason we do not include  such norms in the definition of $E_{m,\gamma}(v)$.

2.  The estimate $E_{m,T}(v^\eps)\leq M_0$, valid for $0<T\leq T^*_\eps$,  and the fact that the coefficients of our systems are functions of $v^s=v^s_T$,  imply that the constants appearing in the estimates will be uniform with respect to $\eps$.  
\end{rem}


We suppose that $G\in H^{m+3}(b\Omega)$ and for $T_0>0$ define the norm
\begin{align}\label{c5a}
N_{m,T_0}(\eps^2 G):= \langle \Lambda^2_1 \eps^2 G\rangle_{m,T_0}+\langle\Lambda_1\eps G\rangle_{m,T_0}+\langle \Lambda_1^{\frac{3}{2}}\eps^{\frac{3}{2}}G\rangle_{m+1,T_0}.
\end{align}
We choose a Seeley extension of  $G|_{\Omega_{T_0}}$ (also denoted $G$) to  $H^{m+3}(b\Omega)$, compactly supported in $t$,  such that the obvious analogue of \eqref{c0g} is satisfied.
Observe that 
\begin{align}\label{MG}
N_{m,T_0}(\eps^2 G)\lesssim\langle G\rangle_{m+3,T_0}:=M_G\text{ for }\eps\in (0,1].
\end{align}


\begin{rem}\label{c5az}
We will eventually need to take $M_G$ small.    This can be arranged for a given choice of $G$ by taking $T_0$ small, since $G$ vanishes in $t<0$.  Alternatively, for a given $T_0$, one can adjust the choice of $G$ to make $M_G$ small.  In this paper we use the first option.     We can and do always suppose that  $M_G\leq 1$. 

\end{rem}
The main step in showing Proposition \ref{mainprop} is the following a priori estimate, whose proof  is concluded at the end of section \ref{mainestimate}.

\begin{prop}\label{c5}
Let $m> 3d+4+\frac{d+1}{2}$ and $G\in H^{m+3}(b\Omega)$. There exist positive constants  $M_0$ as in \eqref{c0h} and $M_G$ as in  \eqref{MG} such that the following is true. 
There exist positive constants $\eps_0$,   $\gamma_0$, and there exist increasing functions $Q_i:\mathbb{R}_+\to\mathbb{R}_+$, $i=1,2$, with 
$Q_i(z)\geq z$ such that for $\eps\in (0,\eps_0]$ and each $T$ with $0<T\leq T^*_\eps$, the solution to \eqref{c1}-\eqref{c3} satisfies
\begin{align}\label{c6}
E_{m,\gamma}(v)\leq \gamma^{-1}E_{m,\gamma}(v)Q_1(E_{m,T}(v^s))+(\gamma^{-\frac{1}{2}}+\sqrt{\eps})Q_2(E_{m,T}(v^s))\text{ for }\gamma\geq \gamma_0.   
\end{align}

\end{prop}

Assuming this proposition, we now prove Proposition \ref{mainprop}. 

\begin{proof}[Proof of Proposition \ref{mainprop}]

 Since $E_{m,T}(v^s)\leq M_0$ we can choose $\gamma_0$  so that  
 \begin{align}
 \gamma^{-1}E_{m,\gamma}(v)Q_1(E_{m,T}(v^s))\leq E_{m,\gamma}(v)/2\text{  for  }\gamma\geq \gamma_0.
 \end{align}
Thus, for some $C>0$ 
\begin{align}\label{c7}
Ce^{-\gamma T}E_{m,T}(v)\leq E_{m,\gamma}(v)\leq 2(\gamma^{-\frac{1}{2}}+\sqrt{\eps})Q_2(E_{m,T}(v)) \text{ for }\gamma\geq \gamma_0,\;\eps\in (0,\eps_0].
\end{align}
For any $T\in (0,T^*_\eps]$ such that $\frac{1}{T}\geq \gamma_0$, we can take $\gamma=1/T$ in \eqref{c7} to obtain
\begin{align}\label{c8}
E_{m,T}(v)\leq 2\frac{e}{C}(\sqrt{T}+\sqrt{\eps})Q_2(M_0)\text{ for }\eps\in (0,\eps_0].
\end{align}
Clearly, for $N\in\mathbb N$ one can choose $\eps_0$ and  $T_1=T_1(M_0,M_G)$ independent of $\eps$  such that the right side of \eqref{c8} is  $<M_0/N$.     However, to finish we need to know that $T_1\leq T^*_\eps$ for all $\eps\in (0,\eps_0]$.  If not, there is some $\eps\in (0,\eps_0]$ such that $T_1>T^*_\eps$.  But then $E_{m,T^{*,-}_\eps}(v^\eps) <M _0/N$, where $T^{*,-}_\eps<T^*_\eps$ is as close as we like to $T^*_\eps$.  For $N$ large enough  we can then  use our local continuation result for fixed $\eps$, Proposition \ref{continuation}, to continue $v^\eps$ to a time $T_c>T^*_\eps$ such that $E_{m,T_c}(v^\eps) < M_0$, a contradiction.\footnote{The proof of Prop. \ref{continuation} shows that the size of $N$ depends on the norm of the Seeley extension operator used there.}

\end{proof}

Clearly, we need to study how the singular norms appearing on the right in \eqref{b3}-\eqref{b8} act on nonlinear functions of $v$, $v_2$,  and $\eps^2G$.  We also need to show that microlocal regularity of functions like $\phi_{j,D}\Lambda_Dv_1^\gamma$ is preserved under nonlinear functions.   This study is carried out in section \ref{nonlinear}, and Proposition \ref{c5} is proved in section \ref{mainestimate}.

\section{Singular norms of nonlinear functions}\label{nonlinear}

\emph{\quad} Here we take $x'=(x_0,x'')=(t,x'')\in\mathbb{R}^d$, $\theta\in \mathbb{R}$ and dual variables $\xi\in\mathbb{R}^d$, $k\in\mathbb{R}$, and (in this section) we drop the prime on $x$.  We also fix $\beta\in\mathbb{R}^d\setminus 0$, set $X=\xi+\beta\frac{k}{\eps}$, and write  
\begin{align}\label{f00}
\partial_\eps:=\partial_x+\beta\frac{\partial_\theta}{\eps},\;\partial_{\eps,\gamma}:=(\partial_t+\gamma,\partial_{x''})+\beta\frac{\partial_\theta}{\eps}.
\end{align}
Observe that 
\begin{align}\label{f00a}
\partial_{\eps,\gamma}(e^{-\gamma t}u)=e^{-\gamma t}\partial_\eps u.
\end{align}
We let $\Lambda$ denote the singular symbol $\Lambda=\sqrt{|X|^2+\gamma^2}$ and $\Lambda_1=\sqrt{|X|^2+1}$.  

Sometimes, for example in convolutions, we will use the variables $y\in\mathbb{R}^d$, $\omega\in \mathbb{R}$, dual variables $\eta\in\mathbb{R}^d$, $l\in\mathbb{R}$, and set   $Y=\eta+\beta\frac{l}{\eps}$.

 In this section we consider functions $u=u(x,\theta)$.  Unless otherwise noted, all constants below are independent of $\gamma\geq 1$ and $\eps\in (0,1]$.    In every estimate we assume that the functions in question are such that the norms appearing on the right are finite.

The following  proposition is used repeatedly in this section.

\begin{prop}[\cite{RR} , Lemma 1.2.2] \label{f0a}
Let $\eps\in (0,1]$ and take $\gamma\geq 1$.
Suppose $G_{\eps,\gamma}:\mathbb{R}^{d+1}\times \mathbb{R}^{d+1}\to\mathbb{C}$ is a locally integrable measurable function that can be decomposed into a finite sum (suppress $\eps$,$\gamma$)
\begin{align}
G(\xi,k,\eta,l)=\sum_{j=1}^K G_j(\xi,k,\eta,l)
\end{align}
such that for each $j$ we have either
\begin{align}
\sup_{\xi,k}\int |G_j(\xi,k,\eta,l)|^2 d\eta dl< C\text{ or }\sup_{\eta,l}\int |G_j(\xi,k,\eta,l)|^2 d\xi dk< C,
\end{align}
where $C$ is independent of $(\eps,\gamma)$.   Then
\begin{align}
(f,g)\to \int G(\xi,k,\eta,l) f(\xi-\eta,k-l)g(\eta,l)d\eta dl
\end{align}
defines a continuous bilinear map of $L^2\times L^2\to L^2$, and 
\begin{align}
 |\int G(\xi,k,\eta,l) f(\xi-\eta,k-l)g(\eta,l)d\eta dl|_{L^2}\leq C |f|_{L^2} |g|_{L^2},
\end{align}
with $C$ independent of $(\eps,\gamma)$.

\end{prop}

\subsection{Non-tame estimates}\label{nontame}

 We begin with the non-tame estimates that are used in the Proof of Proposition \ref{c5}.

\begin{prop}\label{f0}
Let $m>(d+1)/2$, $r\geq 0$, and set $\langle u\rangle_m:=|u|_{H^m}$.  Assume the functions $u$ and $v$ are real-valued.  Then

a) $\langle u v\rangle_m\leq C_1\langle u\rangle_m \langle v\rangle_m.$

b)   $\langle \Lambda^{r}_{1}(uv)\rangle_m\leq C_2[\langle \Lambda^{r}_{1}u\rangle_m\langle v\rangle_m+\langle u\rangle_m\langle \Lambda^{r}_{1}v\rangle_m]$.

c) Let $C_3=\max (C_1,C_2)$. For $n\in\mathbb{N}$, $\langle\Lambda^{r}_{1}(u^n)\rangle_m\leq n\langle\Lambda^{r}_{1} u\rangle_m \cdot (C_3\langle u \rangle_m)^{n-1}$. 

d) Let $f(u)$ be a real-valued, (real-)analytic function satisfying $f(0)=0$ with radius of convergence $R$ at $u=0$.  If $C_3\langle u\rangle_m <R$, then 
\begin{align}
\langle \Lambda^{r}_{1} f(u)\rangle_m\leq \langle \Lambda^{r}_{1}u\rangle_m \cdot h(C_3\langle u \rangle_m), 
\end{align}
where $h$ is an analytic function with radius of convergence $R$ at $0$ with nonnegative coefficients.

\end{prop}

\begin{proof}
\textbf{(a,b)} Part (a) follows directly from Proposition \ref{f0a} by taking 
\begin{align}
G(\xi,k,\eta,l)=\frac{\langle \xi,k\rangle^m}{\langle \xi-\eta,k-l\rangle^m\langle \eta,l\rangle^m}
\end{align}
and consider the cases $|\eta,l|\leq \frac{1}{2}|\xi,k|$ and  $|\eta,l|> \frac{1}{2}|\xi,k|$. 
To prove (b), use Proposition \ref{f0a} in a similar way, together with the inequality
\begin{align}
|X,1|^{r}\leq C (|X-Y,1|^{r}+ |Y,1|^{r}).
\end{align}

\textbf{(c,d)}  Part (c) is a direct consequence of part (b), and (d) is proved by writing $f(z)=\sum^\infty_{n=1}a_nz^n$, applying (c), and taking $h(z)=\sum^\infty_{n=1}n|a_n|z^{n-1}$.

\end{proof}

\begin{prop}\label{f1}

Let $m> (d+1)/2$, $r\geq0$.   Assume that the functions $u$, $v$ are real-valued.  Then

a) $|uv|_{H^m_\gamma}\leq D_1 |u|_{H^m_\gamma}|v|_{H^m}$. 

b) $|\Lambda^{r}_D(uv)|_{H^m_\gamma}\leq D_2\left[| \Lambda^{r}_D u|_{H^m_\gamma}|v|_{H^m}+|u|_{H^m_\gamma}|\Lambda^{r}_{1,D}v|_{H^m}\right]$.

c)  Let $D_3=\max (C_1,D_2)$. Let $f(u)$ be a real-valued, (real-)analytic function satisfying $f(0)=0$ with radius of convergence $R$ at $u=0$.  If $D_3\langle u\rangle_m <R$, then 
\begin{align}
|\Lambda^{r}_{D} f(u)|_{H^m_\gamma}\leq |\Lambda^{r}_{D}u|_{H^m_\gamma}h_1(D_3\langle u\rangle_m)+|u|_{H^m_\gamma} \langle\Lambda^r_{1} u\rangle_m  h_2(D_3\langle u \rangle_m), 
\end{align}
where the $h_i$ are analytic functions with radius of convergence $R$ at $0$ with nonnegative coefficients. 

\end{prop}

\begin{proof}
\textbf{(a,b)} Part  (a) is proved just like Proposition \ref{f0} (a).  To prove (b) use Proposition \ref{f0a} with the inequality
\begin{align}
|X,\gamma|^{r}\leq C (|X-Y,\gamma|^{r}+ |Y,1|^{r}),
\end{align}
and proceed as in the proof of Proposition \ref{f0}, part (b).

\textbf{(c)} Using part (b) we obtain
\begin{align}
\begin{split}
&|\Lambda^{r}_{D} f(u)|_{H^m_\gamma}=|\Lambda^{r}_{D} \left(f'(0)u+g(u)u^2\right)|_{H^m_\gamma}\leq \\
&\qquad\quad C |\Lambda^{r}_{D} u|_{H^m_\gamma}+D_2\left[|\Lambda^{r}_{D} u|_{H^m_\gamma}\langle ug(u)\rangle_m)+|u|_{H^m_\gamma}|\Lambda_{1,D}(ug(u))|_{H^m}\right].
\end{split}
\end{align}
 To finish, apply Proposition \ref{f0}(a),(d) to the $D_2$ term.\footnote{The proof shows that we could replace $D_3$ by $C_1$ in the $h_1$ term.}
\end{proof}

Recalling that $\langle \Lambda^{r}u\rangle_{m,\gamma}:=|\Lambda^{r}_D(e^{-\gamma t}u)|_{H^m_\gamma}$ and $\langle \Lambda^{r}_1u\rangle_{m}:=|\Lambda^{r}_{1,D}u|_{H^m}$ (as in Notations \ref{spaces}), we have the following immediate corollary of Proposition \ref{f1}.

\begin{cor}\label{f2}
Let $m>(d+1)/2$, $r\geq 0$.  Then

a) $\langle \Lambda^{r}(uv)\rangle_{m,\gamma}\leq  C_2\left[ \langle\Lambda^{r} u\rangle_{m,\gamma }\langle v\rangle_m+\langle u\rangle_{m,\gamma}\langle\Lambda^{r}_1v\rangle_{m}\right]$.

(b) For $f$, $h_1$, $h_2$ as in Proposition \ref{f1}: 
\begin{align}
\langle\Lambda^{r} f(u)\rangle_{m,\gamma}\leq \langle\Lambda^{r}u\rangle_{m,\gamma}h_1(D_3\langle u\rangle_m)+\langle u\rangle_{m,\gamma} \langle\Lambda^r_{1} u\rangle_m  h_2(D_3\langle u \rangle_m).
\end{align}
\end{cor}

It is now straightforward to generalize these results to real-valued, real-analytic functions $f(u)=f(u_1,\dots,u_K)=u_1g(u)$ of $K$ real variables converging in a polydisk $\mathbb{D}=\{u:|u_i|<r_i\}$ for some $r_i>0$. 

\begin{prop}\label{f3}
Let $f(u)=f(u_1,\cdots,u_K)$ be as just described.  Let $C_3$, $D_3$ be as in Propositions \ref{f0}, \ref{f1} respectively.    
Then
\begin{align}
\begin{split}
&(a) \langle\Lambda^r_{1}f(u_1,\dots,u_K)\rangle_m\leq \sum^K_{j=1}\langle\Lambda^r_{1}u_j\rangle_m h_j(\langle C_3 u_1\rangle_m,\dots,\langle C_3 u_K\rangle_m),\\
&(b) |\Lambda^{r}_{D} f(u)|_{H^m_\gamma}\leq |\Lambda^{r}_{D}u_1|_{H^m_\gamma}k(\langle D_3u_1\rangle_m,\dots,\langle D_3u_K\rangle_m)+|u_1|_{H^m_\gamma} \sum^K_{j=1}\langle\Lambda^r_{1}u_j\rangle_m k_j(\langle D_3 u_1\rangle_m,\dots,\langle D_3 u_K\rangle_m),\\
&(c) \langle\Lambda^{r} f(u)\rangle_{m,\gamma}\leq \langle\Lambda^{r}u_1\rangle_{m,\gamma}k(\langle D_3u_1\rangle_m,\dots,\langle D_3u_K\rangle_m)+\langle u_1\rangle_{m,\gamma} \sum^K_{j=1}\langle\Lambda^r_{1}u_j\rangle_m k_j(\langle D_3 u_1\rangle_m,\dots,\langle D_3 u_K\rangle_m),
\end{split}
\end{align}
where the functions $h_j$, $k$, $k_j$ are real-analytic functions of $u$ with nonnegative coefficients converging in the polydisk $\mathbb{D}$.

\end{prop}


Next we give microlocal estimates reminiscent of ``Rauch's Lemma" \cite{R}, except that they involve singular norms and microlocalization by singular pseudodifferential operators associated to symbols $\phi=\phi(X,\gamma)$.    In these estimates we continue to assume that the functions in question are such that the norms appearing on the right are finite.  
We also make the following assumption on the $\mathbb{R}$-valued  functions $u$, $v$, or $u_j$ that appear.  
\begin{assumption}\label{f4c}
Let $u(x,\theta)$ be $\mathbb{R}$-valued.  Suppose $r\geq 1$, $m>(d+1)/2$ and let $(\xi_0,\gamma_0)\in\mathbb{R}^{d+1}\setminus 0$.  We suppose that there is a conic neighborhood $\Gamma\subset \mathbb{R}^{d+1}\setminus 0$ of $(\xi_0,\gamma_0)$ such that if $\psi(X,\gamma)$ is any singular symbol of order $0$ (Definition \ref{def4}) with $\mathrm{supp}\; \psi\subset\subset \Gamma$, we have 
\begin{align}\label{f4a}
\psi_D\Lambda^r_D e^{-\gamma t}u\in H^m\text{ and }\Lambda_D^{r-\frac{1}{2}}e^{-\gamma t}u\in H^m.
\end{align}
\end{assumption}

\begin{nota}\label{f4b}
In the microlocal estimates  below we use operators defined by singular symbols  $\psi(X,\gamma)$ and  $\phi(X,\gamma)$ of order $0$ with the following properties.  For $\Gamma$ as in Assumption \ref{f4c} there is a conic neighborhood $\Gamma_1\subset\subset \Gamma$ such that:  

a) $\mathrm{supp}\;\psi\subset\subset \Gamma$ and $\psi=1$ on $\Gamma_1\cap \{|X,\gamma|\geq 1\}$.

b) $\mathrm{supp}\;\phi\subset\subset \Gamma_1$.  

\end{nota}

\begin{prop}\label{f5}

Let $g(u)$ be a real-valued, real-analytic function  with radius of convergence $R$ and set $f(u)=ug(u)$.  Under Assumption \ref{f4c} on $u$ and with $\phi$, $\psi$ as in Notations \ref{f4b}, we have $\phi_D\Lambda^r_D e^{-\gamma t}f(u)\in H^m$ and there exists $C_1>0$ such that (suppress subscripts $D$ on $\phi$, $\psi$, $\Lambda$, $\Lambda_1$):
\begin{align}\label{f6b}
\langle\phi\Lambda^r f(u)\rangle_{m,\gamma} \leq A+B,\;\text{ where }
\end{align}
\begin{align}\label{f6c}
\begin{split}
&A=  \langle\Lambda^{r-\frac{1}{2}}u\rangle_{m,\gamma}\langle \Lambda^{r-\frac{1}{2}}_{1}u\rangle_m k_1\\
&B=\langle\psi\Lambda^r u\rangle_{m,\gamma}k_2+\langle\psi\Lambda u\rangle_{m,\gamma}\langle \Lambda^{r-1}_{1}u\rangle_m k_3,
\end{split}
\end{align}
where the $k_i$ are analytic functions with power series having radius of convergence $R$ at $0$ and nonnegative coefficients, and evaluated at $(C_1\langle u\rangle_m)$.

\end{prop}

\begin{proof}
\textbf{1. } Using \eqref{f00a} we have $\langle\phi\Lambda^r (ug(u))\rangle_{m,\gamma}=$
\begin{align}
|\phi\Lambda^{r-1}\Lambda(e^{-\gamma t} ug(u))|_{H^m_\gamma}\sim |\phi\Lambda^{r-1}\partial_{\eps,\gamma}(e^{-\gamma t} ug(u))|_{H^m_\gamma}= |\phi\Lambda^{r-1}(e^{-\gamma t}\partial_{\eps} (ug(u)))|_{H^m_\gamma},
\end{align}
which is $\leq A_1+A_2$, where 
 \begin{align}
A_1=|\phi\Lambda^{r-1}(e^{-\gamma t}\partial_\eps u\; g(u))|_{H^m_\gamma},\;\;\;A_2=|\phi\Lambda^{r-1} (ug'(u)e^{-\gamma t}\partial_\eps u)|_{H^m_\gamma}.
\end{align}

\textbf{2. }Write $e^{-\gamma t}\partial_\eps u=\psi_D (e^{-\gamma t}\partial_\eps u)+(1-\psi_D)(e^{-\gamma t}\partial_\eps u):=w_1+w_2$. Then $A_1\leq A_{1,1}+A_{1,2}$, where 
\begin{align}
\begin{split}
& A_{1,1}=|\phi\Lambda^{r-1}(w_1\; g(u))|_{H^m_\gamma}\leq |\Lambda^{r-1}(w_1\; g(u))|_{H^m_\gamma}\leq \langle\psi\Lambda^r u\rangle_{m,\gamma}k_1+\langle\psi\Lambda u\rangle_{m,\gamma}\langle\Lambda^{r-1}_1 u\rangle_m k_2
\end{split}
\end{align}
by Proposition \ref{f1}, and $A_{1,2}=|\phi\Lambda^{r-1}(w_2\; g(u))|_{H^m_\gamma}\leq$\footnote{We use $\phi(1-\psi)=0$ in \eqref{faa2}.}
\begin{align}\label{faa2}
 |\phi\Lambda^{r-1}(w_2\; g(0))|_{H^m_\gamma}+|\phi\Lambda^{r-1}(w_2\; uh(u))|_{H^m_\gamma}=|\phi\Lambda^{r-1}(w_2\; uh(u))|_{H^m_\gamma}.
\end{align}
We claim
\begin{align}\label{fa2}
|\phi\Lambda^{r-1}(w_2\; uh(u))|_{H^m_\gamma}\leq C \langle\Lambda_1^{r-\frac{1}{2}}(uh(u))\rangle_m \langle\Lambda^{r-\frac{1}{2}}u\rangle_{m,\gamma}
\end{align}
To see this we apply Proposition \ref{f0a} with $G_{\eps,\gamma}(\xi,k,\eta,l)=$
\begin{align}
\frac{\phi(X,\gamma)(1-\psi(Y,\gamma))\langle X,\gamma\rangle^{r-1}\langle\xi,k,\gamma\rangle^m}{\langle X-Y,1\rangle^{r-\frac{1}{2}}\langle \xi-\eta,k-l,1\rangle^m\langle\eta,l,\gamma\rangle^m|Y,\gamma|^{r-\frac{3}{2}}},
\end{align}
 noting that $r-\frac{3}{2}\geq -\frac{1}{2}$ and that on the support of $G_{\eps,\gamma}$ we have
\begin{align}
|(X,\gamma)-(Y,\gamma)|=|(X-Y,0)|\geq C|X,\gamma|\text{ and } |(X-Y,0)|\geq C|Y,\gamma|.
\end{align}
By Proposition \ref{f0}(d) the right side of \eqref{fa2} is bounded by  $\langle\Lambda^{r-\frac{1}{2}}u\rangle_{m,\gamma}\langle \Lambda^{r-\frac{1}{2}}_{1}u\rangle_m k_3$.

\textbf{3. }The estimate of $A_2$ clearly yields terms of the same form.

\end{proof}

Easy modifications of the previous proof yield the following analogue of Proposition \ref{f5} for $f(u_1,\dots,u_K)$.

\begin{prop}\label{f7}

Let $g(u)=g(u_1,\dots,u_K)$ be a real-valued, real-analytic function of $K$ real variables that converges in a polydisk $\mathbb{D}=\{u:|u_i|<R_i\}$ for some $R_i>0$. 
Let $r\geq1$ and suppose $m>(d+1)/2$.   Set $f(u)=u_1g(u)$.   Under Assumption \ref{f4c} on the  $u_j$ and with $\phi$, $\psi$ as in Notations \ref{f4b}, we have $\phi_D\Lambda^r_D e^{-\gamma t}f(u)\in H^m$ and there exists $C>0$ such that  (suppress subscripts $D$ on $\phi$, $\psi$, $\Lambda_1$, and $\Lambda$)

\begin{align}\label{f9}
&\langle\phi\Lambda^r f(u)\rangle_{m,\gamma}\leq A+B, \text{ where }
\end{align}
\begin{align}
\begin{split}
&A=\langle\Lambda^{r-\frac{1}{2}} u_1\rangle_{m,\gamma}\left(\sum^K_{j=1}\langle\Lambda^{r-\frac{1}{2}}_1 u_j\rangle_m h_j\right)\\
&B=\sum^K_{i=1}\langle\psi\Lambda^r u_i\rangle_{m,\gamma}\;k_i+\sum^K_{i=1}\sum^K_{j=1}\langle\psi\Lambda u_i\rangle_{m,\gamma}\langle\Lambda^{r-1}_1u_j\rangle_mk_{i,j}.
\end{split}
\end{align}
Here the functions $h_j$, $k_i$, and $k_{i,j}$ are real-analytic functions of $u$ with nonnegative coefficients,  converging in the polydisk $\mathbb{D}$, and evaluated at $(\langle Cu_1\rangle_m,\dots,\langle Cu_K\rangle_m)$.

\end{prop}

\begin{rem}\label{f10}
1.)  We shall apply Proposition \ref{f7} with $r=3/2$ or $r=1$.    

2.)  To estimate $\langle \phi \Lambda^r f(u)\rangle_{m,\gamma}$ for any real-analytic function $f(u)$ of $u=(u_1,\dots,u_K)$ such that $f(0)=0$, we write 
\begin{align}
f(u)=\sum^K_{i=1} u_i g_i(u)
\end{align}
and apply Proposition  \ref{f7} to the individual terms in this sum, letting the various $u_i$ consecutively play the role of $u_1$. 

3) Observe that the real-analyticity of the function $f(u)$ in Proposition \ref{f7} is needed only for the parts of the proof that use Proposition \ref{f3}.    If one knew that Proposition \ref{f3} held for (say) $f \in  C^\infty$, then Proposition \ref{f7} would be deducible by the same ``chain rule" proof for such $f$.   

4)   When $r\geq 1$, one can prove Proposition \ref{f3}  for $f\in C^\infty$ inductively by writing $|\Lambda^r_De^{-\gamma t}f(u)|
_{H^m_\gamma}  \sim  |\Lambda^{r-1} (\partial_{\eps,\gamma}e^{-\gamma t}f(u))|_{H^m_\gamma}$.   However, we need Proposition \ref{f3} for $r=1/2$ as well as $r=1$ and $3/2$ in the applications here.  See, for example, the estimate \eqref{b3}. 
\footnote{We have been unable to adapt either Gagliardo-Nirenberg type arguments or arguments involving (double) dyadic decompositions to estimate singular norms of nonlinear $C^\infty$ functions of $u$. The  various estimates of this section work well, though, for singular norms of analytic functions of $u$.}  

\end{rem}

In estimates of commutators we need the following refinement of Proposition \ref{f1}(b).

\begin{prop}\label{f12}
Let $r\geq 0$, and suppose $0\leq \sigma\leq \min(s,t)$, where $s+t-\sigma> \frac{d+1}{2}$.  Then
\begin{align}
\begin{split}
(a)\;|\Lambda^r_D(uv)|_{H^\sigma_\gamma}\lesssim |\Lambda^r_D u|_{H^s_\gamma}|v|_{H^t}+|u|_{H^s_\gamma}|\Lambda^r_{1,D} v|_{H^t}\\
(b)\;\langle\Lambda^r(uv)\rangle_{\sigma,\gamma}\lesssim \langle\Lambda^r u\rangle_{s,\gamma}\langle v\rangle_t+\langle u\rangle_{s,\gamma}\langle\Lambda^r_{1} v\rangle_t.
\end{split}
\end{align}
\end{prop}

\begin{proof}
Using $|X,\gamma|^r\lesssim |X-Y,\gamma|^r+|Y,1|^r$, we have $|\Lambda^r_D(uv)|_{H^\sigma_\gamma}=$
\begin{align}
\begin{split}
&\left||X,\gamma|^r\langle \xi,k,\gamma\rangle^\sigma (\hat u * \hat v)\right|_{L^2(\xi,k)}\lesssim \\
&\left||X-Y,\gamma|^r\langle \xi,k,\gamma\rangle^\sigma (|\hat u| * |\hat v|)\right|_{L^2(\xi,k)}+\left||Y,1|^r\langle \xi,k,\gamma\rangle^\sigma (|\hat u| * |\hat v|)\right|_{L^2(\xi,k)} :=A+B.
\end{split}
\end{align}
We estimate $A$ and $B$ by applying Proposition \ref{f0a} to $G(\xi,k,\eta,l)=$
\begin{align}
\frac{|X-Y,\gamma|^r\langle\xi,k,\gamma\rangle^\sigma}{|X-Y,\gamma|^r\langle\xi-\eta,k-l,\gamma\rangle^s\langle\eta,l\rangle^t}\text{ and }\frac{|Y,1|^r\langle\xi,k,\gamma\rangle^\sigma}{\langle\xi-\eta,k-l,\gamma\rangle^s|Y,1|^r\langle\eta,l\rangle^t}
\end{align}
respectively.
\end{proof}

In section \ref{local} we work with the following norms on the half space that involve higher normal derivatives.
\begin{defn}\label{normal}
Let $\mathbb{R}^{d+1}_+=\{x=(x_0,x_1,\dots,x_d):x_d\geq 0\}$  For functions $u(x,\theta)$ with $(x,\theta)\in \mathbb{R}^{d+1}\times \mathbb{R}$, $r\geq 0$, $m\in \mathbb{N}$,  we define
\begin{align}
\begin{split}
&|\Lambda^r u|_{m,\gamma}:=\sum_{j=0}^m|\Lambda^r\partial^j_{x_d}u|_{0,m-j,\gamma},\\
&|\Lambda^r_1 u|_{m}:=\sum_{j=0}^m|\Lambda^r_1\partial^j_{x_d}u|_{0,m-j}.
\end{split}
\end{align}
Time localized versions of these norms, $|\Lambda^r u|_{m,\gamma,T}$, $|\Lambda^r_1 u|_{m,T},$ are defined in the usual way.
\end{defn}

We will need analogues of some of the above results for these norms.
\begin{prop}\label{normalest}
Let $m>(d+2)/2$, $r\geq 0$.  Then

a) $| \Lambda^{r}(uv)|_{m,\gamma}\lesssim |\Lambda^{r} u|_{m,\gamma }| v|_m+| u|_{m,\gamma}|\Lambda^{r}_1v|_{m}$.

(b) For $f$, $h_1$, $h_2$ as in Proposition \ref{f1} and real-valued $u(x,\theta)$: 
\begin{align}
|\Lambda^{r} f(u)|_{m,\gamma}\leq |\Lambda^{r}u|_{m,\gamma}h_1(C| u|_m)+|u|_{m,\gamma} |\Lambda^r_{1} u|_m  h_2(C| u |_m).
\end{align}

(c) Similarly, the analogue of Proposition \ref{f3} holds for $|\Lambda^rf(u_1,\dots,u_k)|_{m,\gamma}.$

\end{prop}

After taking extensions of $u$ and $v$ into $x_d<0$, one can use the Fourier transform in all variables to prove this Proposition just as before.    

\subsection{Tame estimates} \label{tames}

\emph{\quad}Next we provide the tame estimates that we need for the proof of the continuation argument of section \ref{local}.  In these estimates a lower Sobolev-type norm takes the role sometimes played by an $L^\infty$ or Lipschitz norm in tame estimates.   We prove the first four results below after  the statement of Proposition \ref{j4}.

\begin{prop}\label{j1}

Let $m_0>(d+1)/2$, $m\geq 0$, $r\geq 0$, and set $\langle u\rangle_m:=|u|_{H^m}$.  Assume the functions $u$ and $v$ are real-valued.  Then

a) $\langle u v\rangle_m\leq C_1(\langle u\rangle_m \langle v\rangle_{m_0}+\langle u\rangle_{m_0} \langle v\rangle_m).$

b)   $\langle \Lambda^{r}_{1}(uv)\rangle_m\leq C_2[\langle \Lambda^{r}_{1}u\rangle_m\langle v\rangle_{m_0}+\langle u\rangle_m\langle \Lambda^{r}_{1}v\rangle_{m_0}+\langle \Lambda^{r}_{1}u\rangle_{m_0}\langle v\rangle_m+\langle u\rangle_{m_0}\langle \Lambda^{r}_{1}v\rangle_m]$.

c)  For $n\in\mathbb{N}$, $\langle\Lambda^{r}_{1}(u^n)\rangle_m\leq [\langle\Lambda^{r}_{1} u\rangle_m  \cdot  n(C_3\langle u \rangle_{m_0})^{n-1}+\langle\Lambda^{r}_{1} u\rangle_{m_0}\langle u\rangle_m \cdot n(n-1)(C_3\langle u \rangle_{m_0})^{n-2}]$. 

d) Let $f(u)$ be a real-valued, (real-)analytic function satisfying $f(0)=0$ with radius of convergence $R$ at $u=0$.  If $C_3\langle u\rangle_m <R$, then 
\begin{align}
\langle \Lambda^{r}_{1} f(u)\rangle_m\leq [\langle \Lambda^{r}_{1}u\rangle_m +\langle u\rangle_m\langle \Lambda^{r}_{1}u\rangle_{m_0}]\cdot h(C_3\langle u \rangle_{m_0}), 
\end{align}
where $h$ is an analytic function with nonnegative coefficients and radius of convergence $R$ at $0$ .

\end{prop}

\begin{prop}\label{j2}

Let $m_0>(d+1)/2$, $m\geq 0$, $r\geq0$.   Assume that the functions $u$, $v$ are real-valued.  Then

a) $|uv|_{H^m_\gamma}\leq C [|u|_{H^m_\gamma}\langle v\rangle_{m_0}+\langle u\rangle_{m_0}|v|_{H^m_\gamma}]$. 

b)  $|\Lambda^{r}_{D}(uv)|_{H^m_\gamma}\leq C[| \Lambda^{r}_{D}u|_{H^m_\gamma}\langle v\rangle_{m_0}+| u|_{H^m_\gamma}\langle \Lambda^{r}_{1}v\rangle_{m_0}+|\Lambda^{r}_{D}v|_{H^m_\gamma}\langle u\rangle_{m_0}+|v|_{H^m_\gamma}\langle \Lambda^{r}_{1}u\rangle_{m_0}]$.

c)  Let $f(u)$ be a real-valued, (real-)analytic function satisfying $f(0)=0$ with radius of convergence $R$ at $u=0$.  If $C\langle u\rangle_m <R$, then 
\begin{align}
 |\Lambda^{r}_{D} f(u)|_{H^m_\gamma}\leq [|\Lambda^{r}_{D}u|_{H^m_\gamma} +|u|_{H^m_\gamma}\langle\Lambda^{r}_{1}u\rangle_{m_0}]\cdot h(C\langle u \rangle_{m_0}), 
\end{align}
where $h$ is an analytic function with nonnegative coefficients and radius of convergence $R$ at $0$.
\end{prop}


\begin{cor}\label{j3}
Let $m_0>(d+1)/2$, $m\geq 0$, $r\geq 0$.  Then

a) $\langle \Lambda^{r}(uv)\rangle_{m,\gamma}\leq  C\left[ \langle\Lambda^{r} u\rangle_{m,\gamma }\langle v\rangle_{m_0}+\langle u\rangle_{m,\gamma}\langle\Lambda^{r}_1v\rangle_{m_0}+\langle\Lambda^{r} v\rangle_{m,\gamma }\langle u\rangle_{m_0}+\langle v\rangle_{m,\gamma}\langle\Lambda^{r}_1u\rangle_{m_0}\right]$.

(b) For $f$ and $h$ as in  Proposition \ref{j2}: 
\begin{align}
\langle\Lambda^{r} f(u)\rangle_{m,\gamma}\leq \left [\langle\Lambda^{r}u\rangle_{m,\gamma}+\langle u\rangle_{m,\gamma} \langle\Lambda^r_{1} u\rangle_{m_0}\right]  h(C\langle u \rangle_{m_0}).
\end{align}
\end{cor}


\begin{prop}\label{j4}
Let $m_0>(d+1)/2$, $m\geq 0$, $r\geq 0$. 
Suppose $f(u)=f(u_1,\cdots,u_K)$ is a real-valued, real-analytic function of $K$ real variables converging in a polydisk $\mathbb{D}=\{u:|u_i|<r_i\}$ for some $r_i>0$ and such that $f(0)=0$. 
Then
\begin{align}
\begin{split}
&(a) \langle\Lambda^r_{1}f(u_1,\dots,u_K)\rangle_m\leq \sum^K_{j=1}\langle\Lambda^r_{1}u_j\rangle_m h_j+\sum_{j,k=1}^K\langle u_j\rangle_m\langle\Lambda^r_1u_k\rangle_{m_0}h_{j,k},\\
&(b)  |\Lambda^r_D f(u_1,\dots,u_K)|_{H^m_\gamma}\leq \sum^K_{j=1}|\Lambda^r_D u_j|_{H^m_\gamma} h_j+\sum_{j,k=1}^K | u_j|_{H^m_\gamma}\langle\Lambda^r_1u_k\rangle_{m_0}h_{j,k},\\
&(c)  \langle\Lambda^r f(u_1,\dots,u_K)\rangle_{m,\gamma}\leq \sum^K_{j=1}\langle\Lambda^r u_j\rangle_{m,\gamma} h_j+\sum_{j,k=1}^K\langle u_j\rangle_{m,\gamma}\langle\Lambda^r_1u_k\rangle_{m_0}h_{j,k},\\
\end{split}
\end{align}
where the functions $h_j$, $h_{j,k}$ are real-analytic functions of $u$ with nonnegative coefficients  converging in the polydisk $\mathbb{D}$,
and evaluated at  $(\langle C u_1\rangle_{m_0},\dots,\langle Cu_K\rangle_{m_0})$ for some $C>0$.  They may change from line to line. 
\end{prop}

\begin{proof}
\textbf{1. }The previous four results all follow from Proposition \ref{j2}(b).  Proposition \ref{j2}(a) follows by taking $r=0$, while Proposition \ref{j1} (b) follows by taking $\gamma=1$; moreover, Proposition \ref{j1}(b) implies Proposition \ref{j1}(a).  Corollary \ref{j3})(a) follows from Proposition \ref{j2}(b) because of our freedom to switch $e^{-\gamma t}$ from $u$ to $v$.  The remaining parts of the above propositions, which concern $f(u)$, follow by repeated application of the parts already proved by using analyticity of $f$.

\textbf{2. Proof of Proposition \ref{j2}(b). }We have
\begin{align}
|\Lambda^r(uv)|_{H^m_\gamma}\lesssim\sum^2_{i=1}\left ||X,\gamma|^r\langle\xi,k,\gamma\rangle^m\chi_i\;|\hat{u}(\xi-\eta,k-l)|*|\hat{v}(\eta,l)|\right|_{L^2(\xi,k)}=A_1+A_2,
\end{align}
where $\chi_1(\xi,k,\eta,l)$ (resp.$\chi_2$) is the characteristic function of $\{|\eta,l|\leq \frac{1}{2}|\xi,k|\}$ (resp. $\{|\eta,l|\geq \frac{1}{2}|\xi,k|\}$).
To see that $A_1\lesssim | \Lambda^{r}_{D}u|_{H^m_\gamma}\langle v\rangle_{m_0}+| u|_{H^m_\gamma}\langle \Lambda^{r}_{1}v\rangle_{m_0}$, we write $|X,\gamma|^r\lesssim |X-Y,\gamma|^r+|Y,1|^r$ and apply Proposition \ref{f0a} with $G(\xi,k,\eta,l)$ first equal to 
\begin{align}
\frac{|X-Y,\gamma|^r\langle\xi,k,\gamma\rangle^m\chi_1}{|X-Y,\gamma|^r\langle\xi-\eta,k-l,\gamma\rangle^m\langle\eta,l\rangle^{m_0}},\text{ and then }\frac{|Y,1|^r\langle\xi,k,\gamma\rangle^m\chi_1}{|Y,1|^r\langle\xi-\eta,k-l,\gamma\rangle^m\langle\eta,l\rangle^{m_0}}.
\end{align}
The proof is completed by showing similarly that   $A_2\lesssim  |\Lambda^{r}_{D}v|_{H^m_\gamma}\langle u\rangle_{m_0}+|v|_{H^m_\gamma}\langle \Lambda^{r}_{1}u\rangle_{m_0}$.

\end{proof}

The next estimates will be used in section \ref{local} to get tame estimates of commutators for $\eps$ fixed.

\begin{prop}\label{j7}
Suppose $m_0>\frac{d+1}{2}$,  $0\leq m\leq 2m_0$, and $m_1+m_2\leq m$ with $m_i\geq 0$ for $i=1,2$.   Let $\partial^{k}$ denote any operator of the form
$\partial_{x',\theta}^\alpha$ with $|\alpha|=k$. Then for $r\geq 0$ 
\begin{align}\label{j8}
\begin{split}
&(a) |\Lambda^r (\partial^{m_1}u\cdot \partial^{m_2}v)|_{L^2(x',\theta)}\lesssim | \Lambda^{r}_{}u|_{H^m_\gamma}\langle v\rangle_{m_0}+| u|_{H^m_\gamma}\langle \Lambda^{r}_{1}v\rangle_{m_0}+|\Lambda^{r}_{}v|_{H^m_\gamma}\langle u\rangle_{m_0}+|v|_{H^m_\gamma}\langle \Lambda^{r}_{1}u\rangle_{m_0}\\
&(b)\langle\Lambda^r (\partial^{m_1}u\cdot \partial^{m_2}v)\rangle_{0,\gamma}\lesssim \langle\Lambda^{r} u\rangle_{m,\gamma }\langle v\rangle_{m_0}+\langle u\rangle_{m,\gamma}\langle\Lambda^{r}_1v\rangle_{m_0}+\langle\Lambda^{r} v\rangle_{m,\gamma }\langle u\rangle_{m_0}+\langle v\rangle_{m,\gamma}\langle\Lambda^{r}_1u\rangle_{m_0}\\
&(c) \langle\Lambda^r (\partial^{m_1}u\cdot \partial^{m_2}v)\rangle_{0,\gamma}\lesssim \langle\Lambda^{r} u\rangle_{m,\gamma }\langle v\rangle_{m_0}+\langle u\rangle_{m,\gamma}\langle\Lambda^{r}_1v\rangle_{m_0}+\langle\Lambda^{r}_1 v\rangle_{m}\langle u\rangle_{m_0,\gamma}+\langle v\rangle_{m}\langle\Lambda^{r}u\rangle_{m_0,\gamma}.
\end{split}
\end{align}

\end{prop}

\begin{proof}
At least one of $m_1$, $m_2$ is $\leq m_0$.  If $m_2\leq m_0$, estimate the left side of \eqref{j8}(a) by applying  Proposition \ref{f12}(a) with $\sigma =0$, $s=m-m_1$, $t=m_0-m_2$.  If $m_1\leq m_0$, take $s=m_0-m_1$, $t=m-m_2$. Estimate (b) is proved by the same argument, writing
\begin{align}
e^{-\gamma t}(\partial^{m_1}u\cdot \partial^{m_2}v)=(e^{-\gamma t}\partial^{m_1}u)\cdot \partial^{m_2}v
\end{align}
in the case $m_2\leq m_0$, and switching the exponential to the other factor in the case $m_1\leq m_0$.  The proof of (c) is similar.

\end{proof}

\begin{rem}\label{j9}
1) We also need the exact analogues of Corollary \ref{j3} and Propositions \ref{j2}, \ref{j4}, where the tangential norms are replaced by the norms involving powers of $\partial_{x_2}$ given in Definition \ref{normal}.   The statements and proofs are essentially identical, except that now we must have $m_0>\frac{d+2}{2}$.   

2) There is some redundancy in the results of sections \ref{nontame} and \ref{tames}.  One might think that the nontame estimates simply follow from corresponding tame ones by taking $m=m_0$, but that is not true in all cases. For example, Proposition \ref{j1}(b) implies Proposition \ref{f0}(b), but Proposition \ref{j2}(b) does not imply Proposition \ref{f1}(b).   In section \ref{mainestimate} we use only the simpler nontame estimates, while in section \ref{local} we use both tame and nontame estimates.  

\end{rem}

\section{Uniform higher derivative estimates and proof of Theorem \ref{uniformexistence}}\label{mainestimate}
\emph{\quad} In this section we prove Proposition \ref{c5}.   
First note that for $\eps\in (0,\eps_0]$ and $T\in(0,T^*_\eps]$, $v^s=v^s_T$ satisfies 
\begin{align}\label{e1}
|v^s|_{\infty,m}\leq C\eps M_0,
\end{align}
where $C$ depends on the Seeley extension but is independent of $\eps$ and $T$.   Thus
if $\eps_0$ is small enough, the operators on the left in \eqref{c1}-\eqref{c3} will be admissible, that is, they will satisfy the properties listed in section \ref{assumptions}. 
Moreover,  the bound $\eqref{e1}$ allows us to apply the singular calculus to \eqref{c1}-\eqref{c3} to obtain estimates involving constants that are uniform with respect to $\eps$.

\begin{rem}\label{e1z}
\emph{Throughout this section we make essential use of the fact that $v=\nabla_\eps u$ on the support of $v^s$ and that $v^s=\nabla_\eps u^s=(\nabla_\eps u)^s$}.   This allows us to replace $\nabla_\eps u$ by $v$, or $v$ by $\nabla_\eps u$,  in products where  a factor of $v^s$ is present.   This is important because, for example, we are able to control the norm 
$|\nabla_\eps u/\eps|_{\infty,m,\gamma}$ but not the norm $|v/\eps|_{\infty,m,\gamma}$.  We are able to control the first norm by estimating solutions of $\frac{1}{\eps}$\eqref{c3}, but we are not able to control the second norm by estimating solutions of $\frac{1}{\eps}$\eqref{c1}, since that would require us to commute a negative power of $\Lambda_D$ through the equation, and the singular calculus gives inadequate uniform control of such commutators.

\end{rem}

\subsection{Outline of estimates}\label{outline}
\emph{\quad}
When $j=1$\footnote{This brief outline is filled out in the remainder of this section.} the terms in the first line of \eqref{c0a} will be controlled by applying the estimates \eqref{b3}-\eqref{b5} to the problem \eqref{c1} (or rather, to the problem 
$\partial_{x',\theta}^\alpha$\eqref{c1}, $|\alpha|\leq m$).   The first two terms in the second line are controlled by applying estimate \eqref{b3} to the problem ``$\sqrt{\eps}$\eqref{c1}."   

When $j=2$ the terms in the first line  of \eqref{c0a} and the first two terms in the second line will be controlled by similarly applying the estimates \eqref{b6}-\eqref{b8} to the problem \eqref{c2}. 

The last two terms in the second line of \eqref{c0a} when $j=1$ (resp.  j=2) are estimated by applying \eqref{b15} (resp. \eqref{b16}) to the problem $\frac{1}{\sqrt{\eps}}$\eqref{c1} (resp.  $\frac{1}{\sqrt{\eps}}$\eqref{c2}).

Note that $u$ satisfies a  Neumann-type problem of the same kind as \eqref{c1} for $v_1$.  Thus we will control $|\nabla_\eps u/\eps|_{\infty,m,\gamma}$ and $|\nabla_\eps u/\eps|_{0,m+1,\eps}$ by applying the estimates \eqref{b4} and \eqref{b5}, with $(\Lambda_{D,\gamma} v_1,D_{x_2} v_1)$ replaced by $(\Lambda_{D,\gamma} u,D_{x_2}u)$, to the problem $\frac{1}{\eps}$\eqref{c3}
This will handle the first and second terms in the third line of  \eqref{c0a}.  


The term $\left|\Lambda^{\frac{1}{2}}\nabla_\eps u/{\sqrt{\eps}}\right|_{0,m+1,\gamma}$ is controlled by applying the estimate \eqref{b3}  with $(\Lambda_{D,\gamma} v_1,D_{x_2} v_1)$ replaced by $(\Lambda_{D,\gamma} u,D_{x_2}u)$ to the problem $\frac{1}{\sqrt{\eps}}$\eqref{c3}.   These   estimates
use the control of the first two terms in the second line of \eqref{c0a}.    The estimate of $\left\langle\frac{\nabla_\eps u}{\sqrt{\eps}}\right\rangle_{m+1,\gamma}$ is a byproduct of this estimate.


The term $|\Lambda^{\frac{1}{2}} \nabla_\eps u/\eps|_{0,m,\gamma}$ is controlled by applying \eqref{b3}  to $\frac{1}{\eps}\eqref{c3}$.

\begin{rem}\label{extracontrol}
When the estimates \eqref{b3}-\eqref{b5} are applied to control singular norms of $\nabla_\eps u$, note that these estimates give the same control of 
$\partial_{t,\eps}u$ that they do of $\partial_{x_1,\eps}u$.   Thus, the arguments below that give control of $E_{m,\gamma}(v)$ also give the same control of the norm that is defined by replacing every occurrence of $\nabla_\eps u$ in the third line of \eqref{c0a} by $(\partial_{t,\eps}u,\nabla_\eps u)$.

\end{rem}


\subsection{Preliminaries}

\emph{} \;\; We denote the interior and boundary forcing terms in the $v_1$-problem \eqref{c1} by $\cF_1$ and $\cG_1$ respectively, in the $v_2$-problem \eqref{c2} by $\cF_2$ and $\cG_2$, and in the $u$-problem \eqref{c3} by $0$ and $\cG$.

Letting $\partial=(\partial_{x'},\partial_\theta)=(\partial_t,\partial_{x_1},\partial_\theta)$ and $\partial_\gamma=(\partial_t+\gamma,\partial_{x_1},\partial_\theta)$, 
we have, for example,  $|\Lambda^r v|_{0,m,\gamma}:=$
\begin{align}
|\Lambda^r_D (e^{-\gamma t}v)|_{H^m_\gamma}\sim \sum_{|\alpha|\leq m}|\partial^\alpha_\gamma\Lambda^r_D(e^{-\gamma t}v)|_{L^2(x_2,L^2(x',\theta))}=\sum_{|\alpha|\leq m}|\Lambda^r_D(e^{-\gamma t}\partial^\alpha v)|_{L^2(x_2,L^2(x',\theta))},
\end{align}
so we will prove Proposition \ref{c5} by applying the estimates \eqref{b3}-\eqref{b5}, \eqref{b6}-\eqref{b8}, and \eqref{b15}-\eqref{b16} to the problems satisfied by $\partial^\alpha v$ or $\partial^\alpha u$, where $|\alpha|\leq m$ or $|\alpha|\leq m+1$, in some cases multiplied by a power of $\eps$. 

The functions $\cF_1$ and $\cF_2$ are sums of terms of the form
\begin{align}\label{e2}
 a(v^s)\partial_{x_{i,\eps}}v^s_k\partial_{x_{j,\eps}}v^s_l,
\end{align}
where $a=d_vA_\alpha$ is an analytic function of its arguments, $i,j,k,l$ take values in $\{1,2\}$,  and we set $\partial_{x_{2,\eps}}:=\partial_{x_2}$.  In the estimates below it will be convenient to suppress 
some subscripts and write terms like \eqref{e2} as 
\begin{align}\label{e3}
a(v^s)d_\eps v^sd_\eps v^s
\end{align}
The term $\cG_1$ has the form 
\begin{align}\label{e4}
d_gH(v^s_1,g_\eps)\partial_{x_{1,\eps}}g_\eps,\text{ where }g_\eps:=\eps^2G,
\end{align}
and, with $b(v^s_1,g_\eps)$ denoting an analytic function, we write this as\footnote{As explained in Remark \ref{c4a}, it is important that $v_1$ and not $v^s_1$ occurs in $\cG_2$.   In boundary terms $d_\eps$ represents only $\partial_{x_{1,\eps}}$, never $\partial_{x_2}$ of course.} 
\begin{align}\label{e5}
\cG_1= b(v^s_1,g_\eps)d_\eps g_\eps;\text{ also, }\cG_2= \chi_0(t)H(v_1,g_\eps).
\end{align}
Since $H(0,0)=0$ we compute that 
\begin{align}\label{e6}
\cG=H(v^s_1,g_\eps)-d_{v_1}H(v^s_1,g_\eps)v^s_1=d_gH(0,0)g_\eps+b(v^s_1,g_\eps)\left((v^s_1,g_\eps),(v^s_1,g_\eps)\right),
\end{align}
and we write 
\begin{align}\label{e9}
\cG=C  g_\eps+ b(v^s_1,g_\eps)(v^s_1,g_\eps)^2.
\end{align}

\begin{rem}\label{e9a}
In the estimates of interior and boundary forcing as well as the commutator estimates of this section, which  depend only on  the results of section \ref{nontame} and do  not make use  of the basic energy estimates of section \ref{b1a},  we take $m > \frac{d+1}{2}$, or sometimes slightly larger.  
It is only in Propositions \ref{g1} and \ref{h1} that we require $m>3d+4+\frac{d+1}{2}$.

\end{rem}

\subsection{Estimates of interior forcing}\label{if}

In this section we estimate the interior forcing terms needed to control the terms in  \eqref{c0a}.

\begin{lem}\label{e10}
Assume $m>\frac{d+1}{2}$ ($d=2$ now).  We have
\begin{align}
|\Lambda^{\frac{1}{2}}(d_\eps v^s d_\eps v^s)|_{0,m,\gamma}\lesssim |\Lambda^{\frac{3}{2}}v^s|_{0,m,\gamma}|\Lambda_1 v^s|_{\infty,m}+|\Lambda v^s|_{\infty,m,\gamma}|\Lambda^{\frac{3}{2}}_1v^s|_{0,m}
\end{align}
in the case where both $d_\eps=\partial_{x_1,\eps}$; otherwise, $|\Lambda^{\frac{1}{2}}\partial_{x_2}v^s|_{0,m,\gamma}$, $|\Lambda_1^{\frac{1}{2}}\partial_{x_2}v^s|_{0,m}$, etc. appear in the obvious places on the right.
\end{lem}

\begin{proof}
Consider the case where both $d_\eps=\partial_{x_1,\eps}$.
Treating $x_2$ as a parameter and applying Corollary \ref{f2} ``in the $(t,x_1,\theta)$ variables" we have
\begin{align}
\langle\Lambda^{\frac{1}{2}}(d_\eps v^s d_\eps v^s)\rangle_{m,\gamma}\lesssim \langle\Lambda^{\frac{3}{2}}v^s\rangle_{m,\gamma}\langle\Lambda_1 v^s\rangle_{m}+\langle\Lambda v^s\rangle_{m,\gamma}\langle\Lambda^{\frac{3}{2}}_1v^s\rangle_{m},
\end{align}
and the lemma follows immediately.

\end{proof}

\begin{nota}\label{e10a}
1) In the statements and proofs below we will (with slight abuse) denote analytic functions with nonnegative coefficients like $k(\langle D_3u_1\rangle_m,\dots \langle D_3 u_K\rangle_m)$, where $u=(u_1,\dots,u_K)$, 
simply by $h(\langle u\rangle_m)$, and the function $h$ may change from term to term.   
In addition we will write sums like 
\begin{align}
\langle u_1\rangle_{m,\gamma} \sum^K_{j=1}\langle\Lambda^r_{1}u_j\rangle_m k_j(\langle D_3 u_1\rangle_m,\dots,\langle D_3 u_K\rangle_m)
\end{align}
 as in Proposition \ref{f3} simply as $\langle u\rangle_{m,\gamma}\langle\Lambda^r_1 u\rangle h(\langle u\rangle_m)$.

2)   The symbol $Q$ or $Q_j$, $j\in\mathbb{N}$, will always denote an increasing, continuous function from $\mathbb{R}_+$ to itself  such that $Q(z)\geq z$.  The symbol $Q^o$ or $Q_j^o$ will denote a function of this type such that $Q^o(0)=0$.  The meaning of $Q$ or $Q^o$ may change from term to term.

\end{nota}

\begin{lem}\label{e11}
Let $r>0$ and  $m>\frac{d+1}{2}$.
\begin{align}
\langle \Lambda^r (a(v^s)w)\rangle_{m,\gamma}\lesssim \langle \Lambda^r w\rangle_{m,\gamma}h(\langle v^s\rangle_m)+\langle w\rangle_{m,\gamma}\langle \Lambda^r_1 v^s\rangle_m h(\langle v^s\rangle_m).
\end{align}

\end{lem}

\begin{proof}
Write $a(v^s)=a(0)+b(v^s)v^s$, note that by Corollary \ref{f2}
\begin{align}
\langle\Lambda^r(b(v^s)v^s w)\rangle_{m,\gamma}\lesssim \langle\Lambda^r w\rangle_{m,\gamma}\langle b(v^s)v^s\rangle_m+ \langle w\rangle_{m,\gamma} \langle\Lambda^r_1(b(v^s)v^s)\rangle_m,
\end{align}
and finish by applying Proposition \ref{f3}(a).

\end{proof}

The next proposition gives control of $|\Lambda^{\frac{1}{2}}\cF_j|_{0,m,\gamma}$, $j=1,2$.
\begin{prop}\label{e12}
Let $m>\frac{d+1}{2}$. For small enough $M_0>0$ as in \eqref{c0h} 
we have
\begin{align}\label{ee12}
|\Lambda^{\frac{1}{2}}(a(v^s)d_\eps v^s d_\eps v^s)|_{0,m,\gamma}\lesssim E_{m,\gamma}(v^s)Q^o(E_{m}(v^s))\lesssim Q^o(E_{m,T}(v^s)).
\end{align}
\end{prop}

\begin{proof}
Treating $x_2$ as a parameter, we have by Lemma \ref{e11}
\begin{align}
\langle\Lambda^{\frac{1}{2}}(a(v^s)d_\eps v^s d_\eps v^s)\rangle_{m,\gamma}\lesssim  \langle \Lambda^{\frac{1}{2}} (d_\eps v^s d_\eps v^s)\rangle_{m,\gamma}h(\langle v^s\rangle_m)+\langle d_\eps v^s d_\eps v^s\rangle_{m,\gamma}\langle \Lambda^{\frac{1}{2}}_1 v^s\rangle_m h(\langle v^s\rangle_m)
\end{align}
Since  $|d_\eps v^s d_\eps v^s)|_{0,m,\gamma}\lesssim |\Lambda v^s|_{0,m,\gamma}|\Lambda_1 v^s|_{\infty,m}$ and for $M_0>0$ small enough (depending on the polydisc of convergence of $h$),
\begin{align}\label{e12b}
\sup_{x_2\geq 0}h(\langle v^s\rangle_m)\leq h(E_m(v^s))\leq h(CE_{m,T}(v^s)), \text{ where }E_m(v^s)\leq CE_{m,T}(v^s)  \;\;(\text{recall } \eqref{c0g}),
\end{align}
the lemma now follows from Lemma \ref{e10}.  The presence of the factors involving $\Lambda_1$ allows us to use $Q^o$ instead of $Q$ in \eqref{ee12}. For the last inequality in \eqref{ee12} we use \eqref{c0g}.

\end{proof}

We need the following variant of Lemma \ref{e11}.  
\begin{lem}\label{e12a}
Let $r\geq 0$ and $m>\frac{d+1}{2}$.  
\begin{align}\label{e13}
\begin{split}
&\langle\Lambda^r (a(v^s)w)\rangle_{m+1,\gamma}\lesssim \langle \Lambda^r w\rangle_{m+1,\gamma}h(\langle v^s\rangle_m)+\langle w\rangle_{m+1,\gamma}\langle \Lambda^r_1v^s\rangle_m h(\langle v^s\rangle_m)+\\
& \qquad \langle \Lambda^r w\rangle_{m,\gamma}h(\langle v^s\rangle_{m+1})+\langle w\rangle_{m,\gamma}\langle \Lambda^r_1v^s\rangle_{m+1} h(\langle v^s\rangle_{m+1}).
\end{split}
\end{align}

\end{lem}

\begin{proof}
Write $a(v^s)=a(0)+b(v^s)v^s$ and apply Proposition \ref{j7}(c) with $``u"=w$, $``m"=m+1$, $``m_0"=m$, observing for example that 
\begin{align}
\langle \Lambda^r_1 (v^sb(v^s))\rangle_m\lesssim \langle \Lambda^r_1v^s\rangle_m h(\langle v^s\rangle_m).
\end{align}

\end{proof}

The next proposition gives control of $|\cF_j|_{0,m+1,\gamma}$, $j=1,2$.
\begin{prop}\label{e14}
Let $m>\frac{d+1}{2}$. For small enough $M_0>0$ as in \eqref{c0h}
we have
\begin{align}
|(a(v^s)d_\eps v^s d_\eps v^s)|_{0,m+1,\gamma}\lesssim E_{m,\gamma}(v^s)Q^o(E_{m}(v^s))\lesssim Q^o(E_{m,T}(v^s)).
\end{align}
\end{prop}

\begin{proof}
Treating $x_2$ as a parameter, we have by Lemma \ref{e12a} (with $r=0$)
\begin{align}\label{e14a}
\langle a(v^s)d_\eps v^s d_\eps v^s\rangle_{m+1,\gamma}\lesssim  \langle  d_\eps v^s d_\eps v^s\rangle_{m+1,\gamma}h(\langle v^s\rangle_m)+ \langle  d_\eps v^s d_\eps v^s\rangle_{m,\gamma}h(\langle v^s\rangle_{m+1}).
\end{align}
A standard Moser estimate gives
\begin{align}\label{e14aa}
 \langle  d_\eps v^s d_\eps v^s\rangle_{m+1,\gamma}\lesssim  \langle   d_\eps v^s\rangle_{m+1,\gamma} |d_\eps v^s|_{L^\infty}\lesssim \langle \Lambda v^s\rangle_{m+1,\gamma}|\Lambda_1 v^s|_{\infty,m}.
\end{align}
In addition, 
\begin{align}
|v^s|_{\infty,m+1}\leq \left|\begin{pmatrix}\Lambda_1 v^s\\D_{x_2} v^s\end{pmatrix}\right|_{0,m+1}\text{ implies }\;\; \mathrm{sup}_{x_2\geq 0} h(\langle v^s\rangle_{m+1})\leq h(E_m(v^s)),
\end{align}
so we can finish by taking the $L^2$ norm of \eqref{e14a} in $x_2$.

\end{proof}

To control the terms in the second line of \eqref{c0a}  we must estimate $\sqrt{\eps}|\Lambda^{\frac{1}{2}}\cF_j|_{0,m+1,\gamma}$, $j=1,2$.

\begin{lem}\label{e14b}
Assume $m>\frac{d+1}{2}+1$.  Then $\sqrt{\eps}|\Lambda^{\frac{1}{2}}(d_\eps v^s d_\eps v^s)|_{0,m+1,\gamma}\lesssim$
\begin{align}
\left(\sqrt{\eps}|\Lambda^{\frac{3}{2}}v^s|)_{0,m+1,\gamma}+|\Lambda v^s|_{0,m+1,\gamma}\right)|\Lambda_1v^s|_{\infty,m}+\sqrt{\eps}|\Lambda v^s|_{\infty,m,\gamma}|\Lambda_1^{\frac{3}{2}}v^s|_{0,m}\lesssim E_{m,\gamma}(v_s)E_m(v_s).
\end{align}
\end{lem}

\begin{proof}
Treating $x_2$ as a parameter, we estimate a sum of terms of the form 
\begin{align}
\sqrt{\eps}\langle \Lambda^{\frac{1}{2}}((\partial^{m_1}d_\eps v^s)\cdot (\partial^{m_2}d_\eps v^s))\rangle_{0,0,\gamma}:=A,
\end{align}
where $m_1+m_2=m+1$. In the case where one of $m_1$ or $m_2$ is $m+1$, say $m_1=m+1$, we apply Proposition \ref{f12}(b) with $``u"=\partial^{m+1}d_\eps v^s$ and $``v"=d_\eps v^s$ to obtain
\begin{align}
A\lesssim \sqrt{\eps}\langle \Lambda^{\frac{3}{2}}v^s\rangle_{m+1,\gamma}\langle\Lambda_1  v^s\rangle_m+\langle \Lambda v^s\rangle_{m+1,\gamma}\langle\Lambda_1 v^s\rangle_m.
\end{align}
On the last term we have used $\sqrt{\eps}\langle\Lambda^\frac{3}{2}_1v^s\rangle_{m-\frac{1}{2}}\lesssim \langle \Lambda_1 v^s\rangle_m$. 
In the other case where both $m_1$ and $m_2$ are $\leq m$ we apply Proposition \ref{f12} with $s=m-m_1$ and $t=m-m_2$ to obtain
\begin{align}
A\lesssim \sqrt{\eps}\langle\Lambda^{\frac{3}{2}}v^s\rangle_{m,\gamma}\langle\Lambda_1v^s\rangle_{m}+\sqrt{\eps}\langle\Lambda v^s\rangle_{m,\gamma}\langle\Lambda_1^{\frac{3}{2}}v^s\rangle_{m}.
\end{align}
To finish we take the $L^2(x_d)$ norm in these inequalities.

\end{proof}


\begin{prop}[$\sqrt{\eps}|\Lambda^{\frac{1}{2}}\cF_j|_{0,m+1,\gamma}$]\label{e14c}
Assume $m>\frac{d+1}{2}+1$.  For small enough $M_0$ we have .
\begin{align}
\sqrt{\eps}|\Lambda^{\frac{1}{2}}(a(v^s)d_\eps v^s d_\eps v^s)|_{0,m+1,\gamma}\lesssim E_{m,\gamma}(v^s)Q^o(E_m(v^s))\lesssim Q^o(E_{m,T}(v^s))
\end{align}
\end{prop}

\begin{proof}
Treating $x_2$ as a parameter, we apply Lemma \ref{e11} to get $\sqrt{\eps}\langle\Lambda^{\frac{1}{2}}(a(v^s)d_\eps v^s d_\eps v^s)\rangle_{0,m+1,\gamma}\lesssim$
\begin{align}
\sqrt{\eps}\langle\Lambda^{\frac{1}{2}}(d_\eps v^s d_\eps v^s)\rangle_{m+1,\gamma} h(\langle v^s\rangle_{m+1})+\sqrt{\eps}\langle d_\eps v^s d_\eps v^s\rangle_{m+1,\gamma}\langle\Lambda^{\frac{1}{2}}_1v^s\rangle_{m+1} h(\langle v^s\rangle_{m+1}).
\end{align}

We finish by taking the $L^2(x_d)$ norm and  by applying Lemma \ref{e14b} and \eqref{e14aa}.\footnote{Control of $\sqrt{\eps}|\Lambda^{\frac{1}{2}}v^s|_{\infty,m+1}$ comes from the first term in the second lines of \eqref{c0a}.}

\end{proof}

To control the term  
\begin{align}\label{e14dd}
\left|\begin{pmatrix}\Lambda v_j/\sqrt{\eps}\\D_{x_2} v_j/\sqrt{\eps}\end{pmatrix}\right|^{(2)}_{0,m,\gamma}
\end{align}
we must estimate  $\frac{1}{\sqrt{\eps}}|\cF_j|_{0,m,\gamma}$.

\begin{prop}\label{e14d}
Let $m>\frac{d+1}{2}$. For small enough $M_0>0$  
we have
\begin{align}\label{e14e}
\frac{1}{\sqrt{\eps}}|a(v^s)d_\eps v^s d_\eps v^s|_{0,m,\gamma}\lesssim Q^o(E_{m,T}(v^s)).
\end{align}
\end{prop}

\begin{proof}
Treating $x_2$ as a parameter, we have by Corollary \ref{f2}
\begin{align}
\frac{1}{\sqrt{\eps}}\langle(a(v^s)d_\eps v^s d_\eps v^s)\rangle_{m,\gamma}\lesssim  \frac{1}{\sqrt{\eps}}\langle  (d_\eps v^s d_\eps v^s)\rangle_{m,\gamma}h(\langle v^s\rangle_m).
\end{align}
Since  $|d_\eps v^s d_\eps v^s|_{0,m,\gamma}\lesssim |\Lambda v^s|_{0,m,\gamma}|\Lambda_1 v^s|_{\infty,m}$, the result follows.
\end{proof}

\subsection{Boundary forcing I}\label{bf}

\emph{\quad\quad}Here we estimate  the forcing terms given by $\cG_1$ and $\cG$ for the problems \eqref{c1} and \eqref{c3} respectively. 
The next two Propositions, which are needed to estimate the terms in the first line of \eqref{c0a} when $j=1$, give control of $\langle \Lambda \cG_1\rangle_{m,\gamma}$ and $\langle \Lambda^{\frac{1}{2}} \cG_1\rangle_{m+1,\gamma}$; recall $\cG_1=b(v^s_1,g_\eps)d_\eps g_\eps $. 
\begin{prop}\label{e15}
Let $m>\frac{d+1}{2}$. For small enough constants $\eps_1>0$ and $M_0>0$  
we have for $\eps\in (0,\eps_1]$\footnote{Here $M_G$ is as in \eqref{MG}.}
\begin{align}
\langle\Lambda\cG_1\rangle_{m,\gamma}\lesssim \langle \Lambda d_\eps g_\eps\rangle_{m,\gamma} h(\langle v^s_1\rangle_m)+(\langle \Lambda_1 v^s_1\rangle_m+\eps) h(\langle v^s_1\rangle_m)\lesssim (M_G+\eps)Q(E_{m,T}(v^s)).
\end{align}
\end{prop}

\begin{proof}
Write $b(v^s_1,g_\eps)=b(0,0)+c(v^s_1,g_\eps)$ where $c(0,0)=0$. By Corollary \ref{f2}
\begin{align}
\begin{split}
&\langle\Lambda (c(v^s_1,g_\eps)d_\eps g_\eps)\rangle_{m,\gamma}\lesssim \langle \Lambda d_\eps g_\eps\rangle_{m,\gamma} \langle c(v^s_1,g_\eps)\rangle_m+\langle d_\eps g_\eps\rangle_{m,\gamma}\langle \Lambda_1 c(v^s_1,g_\eps)\rangle_m\leq \\
&\qquad M_G h(\langle v^s_1\rangle_m)+\eps(\langle \Lambda_1 v^s_1\rangle_m+\eps) h(\langle v^s_1\rangle_m).
\end{split}
\end{align}
Here we have used Proposition \ref{f3}(a) and the fact that,  since $g_\eps=\eps^2 G$ with $G\in H^{m+3}(b\Omega)$,  we have for $\eps_1>0$ small enough (depending on the polydisk of convergence of $c$)
\begin{align}
\langle c(v^s_1,g_\eps)\rangle_m\leq h(\langle v^s_1\rangle_m) \text{ and }\langle d_\eps g_\eps\rangle_{m,\gamma}\leq \eps M_G\text{ for }\eps\in (0,\eps_1].
\end{align} 

\end{proof}

\begin{prop}\label{e16}
Let $m>\frac{d+1}{2}$. For small enough $\eps_1>0$ and $M_0>0$ as in \eqref{c0h} 
we have for $\eps\in (0,\eps_1]$
\begin{align}
\langle\Lambda^{\frac{1}{2}}\cG_1\rangle_{m+1,\gamma}\lesssim \sqrt{\eps} h(\langle v^s_1\rangle_{m+1})+\eps (\langle \Lambda^{\frac{1}{2}}_1 v^s_1\rangle_{m+1}+\eps^{\frac{3}{2}}) h(\langle v^s_1\rangle_{m+1})\lesssim \sqrt{\eps}Q(E_{m,T}(v^s)).
\end{align}
\end{prop}

\begin{proof}
By Corollary \ref{f2} we have $\langle\Lambda^{\frac{1}{2}} (b(v^s_1,g_\eps)d_\eps g_\eps)\rangle_{m+1,\gamma}\lesssim$
\begin{align}
\begin{split}
&\langle \Lambda^{\frac{1}{2}} d_\eps g_\eps\rangle_{m+1,\gamma}  h(\langle v^s_1,g_\eps\rangle_{m+1})+\langle d_\eps g_\eps\rangle_{m+1,\gamma}\langle \Lambda^{\frac{1}{2}}_1 (v^s_1,g_\eps)\rangle_{m+1} h(\langle v^s_1,g_\eps\rangle_{m+1})\leq \sqrt{\eps}Q(E_m(v^s)).
\end{split}
\end{align}

\end{proof}


To control the terms in the second line of \eqref{c0a} when $j=1$ we must estimate  $\sqrt{\eps}\langle\Lambda\cG_1\rangle_{m+1,\gamma}$.

\begin{prop}\label{e22a}
Let $m>\frac{d+1}{2}$. For small enough $\eps_1>0$ and $M_0>0$ as in \eqref{c0h} 
we have for $\eps\in (0,\eps_1]$
\begin{align}
\sqrt{\eps}\langle\Lambda\cG_1\rangle_{m+1,\gamma}\lesssim \sqrt{\eps} h(\langle v^s_1\rangle_{m+1})+\eps^{\frac{3}{2}} (\langle \Lambda_1 v^s_1\rangle_{m+1}+\eps) h(\langle v^s_1\rangle_{m+1})\lesssim \sqrt{\eps}Q(E_{m,T}(v^s)).
\end{align}
\end{prop}

\begin{proof}
By Corollary \ref{f2} we have $\sqrt{\eps}\langle\Lambda (b(v^s_1,g_\eps)d_\eps g_\eps)\rangle_{m+1,\gamma}\lesssim$
\begin{align}
\begin{split}
&\sqrt{\eps}\langle \Lambda d_\eps g_\eps\rangle_{m+1,\gamma}  h(\langle v^s_1,g_\eps\rangle_{m+1})+\sqrt{\eps}\langle d_\eps g_\eps\rangle_{m+1,\gamma}\langle \Lambda_1 (v^s_1,g_\eps)\rangle_{m+1} h(\langle v^s_1,g_\eps\rangle_{m+1}).
\end{split}
\end{align}

\end{proof}

To control the terms \eqref{e14dd} with $j=1$ we must estimate $\frac{1}{\sqrt{\eps}}\langle\Lambda^{\frac{1}{2}}\cG_1\rangle_{m,\gamma}$. 
\begin{prop}\label{e22f}
Let $m>\frac{d+1}{2}$. For small enough constants $\eps_1>0$ and $M_0>0$  
we have for $\eps\in (0,\eps_1]$
\begin{align}
\frac{1}{\sqrt{\eps}}\langle\Lambda^{\frac{1}{2}}\cG_1\rangle_{m,\gamma}\lesssim (M_G+\sqrt{\eps})Q(E_{m,T}(v^s)).
\end{align}
\end{prop}

\begin{proof}
Parallel to Proposition \ref{e15} we have
\begin{align}
\begin{split}
&\frac{1}{\sqrt{\eps}}\langle\Lambda^{\frac{1}{2}} (c(v^s_1,g_\eps)d_\eps g_\eps)\rangle_{m,\gamma}\lesssim \frac{1}{\sqrt{\eps}}\langle \Lambda^{\frac{1}{2}} d_\eps g_\eps\rangle_{m,\gamma} \langle c(v^s_1,g_\eps)\rangle_m+\frac{1}{\sqrt{\eps}}\langle d_\eps g_\eps\rangle_{m,\gamma}\langle \Lambda^{\frac{1}{2}}_1 c(v^s_1,g_\eps)\rangle_m\leq \\
&\qquad M_G h(\langle v^s_1\rangle_m)+\sqrt{\eps}(\langle \Lambda^{\frac{1}{2}}_1 v^s_1\rangle_m+\eps^{\frac{3}{2}}) h(\langle v^s_1\rangle_m).
\end{split}
\end{align}

\end{proof}

To control the terms 
\begin{align}\label{e26bb}
\left|\frac{\nabla_\eps u}{\eps}\right|_{\infty,m,\gamma}, \;\left|\frac{\nabla_\eps u}{\eps}\right|_{0,m+1,\gamma},\;\left|\frac{\Lambda^{\frac{1}{2}}\nabla_\eps u}{\eps}\right|_{0,m,\gamma}, \left|\frac{\Lambda^{\frac{1}{2}}\nabla_\eps u}{\sqrt{\eps}}\right|_{0,m+1,\gamma},\text{ and }\left\langle\frac {\nabla_\eps u}{\sqrt{\eps}}\right\rangle_{m+1,\gamma},
\end{align}
 in the third line of \eqref{c0a}, we use that fact that $v=\nabla_\eps u$ on the support of $v^s$ and apply the estimates \eqref{b3}-\eqref{b5} and \eqref{b15} to the problems $\frac{1}{\eps}\eqref{c3}$ and  $\frac{1}{\sqrt{\eps}}\eqref{c3}$.  For this we must estimate the forcing terms
\begin{align}\label{e22bb}
\frac{1}{\eps}\langle\Lambda \cG\rangle_{m,\gamma},\; \frac{1}{\eps}\langle \Lambda^{\frac{1}{2}}\cG\rangle_{m+1,\gamma}, \text{ and }\frac{1}{\sqrt{\eps}}\langle \Lambda \cG\rangle_{m+1,\gamma}, \end{align}
 where $\cG=Cg_\eps+b(v^s_1,g_\eps)(v^s_1,g_\eps)^2$  in the notation of \eqref{e9}.

\begin{prop}\label{e22cc}
Let $m>\frac{d+1}{2}$. For small enough $\eps_1>0$ and $M_0>0$ 
we have for $\eps\in (0,\eps_1]$
\begin{align}
\begin{split}
&\;\frac{1}{\eps}\langle\Lambda \cG\rangle_{m,\gamma}\lesssim M_G+\left(E_m^2(v^s)+\eps^2 M_G^2\right) h(\langle v_s\rangle_m)\lesssim M_GQ(E_{m,T}(v^s))+Q^o(E_{m,T}(v^s))\\
\end{split}
\end{align}
\end{prop}

\begin{proof}
 Clearly, $\frac{1}{\eps}\langle\Lambda g_\eps\rangle_{m,\gamma}\lesssim M_G$.  Writing $v^s$ in place of $v^s_1$ and $b(v^s,g_\eps)=b(0,0)+c(v^s,g_\eps)$ with $c(0,0)=0$, we have by Corollary \ref{f2} that $\frac{1}{\eps}\langle\Lambda [c(v^s,g_\eps)(v^s,g_\eps)^2]\rangle_{m,\gamma}\lesssim $
\begin{align}\label{e22d}
\begin{split}
&\qquad \frac{1}{\eps}\langle\Lambda(v^s,g_\eps)^2\rangle_{m,\gamma}\langle v^s,g_\eps\rangle_m h(\langle v^s\rangle_m)+\frac{1}{\eps}\langle(v^s,g_\eps)^2\rangle_{m,\gamma}\langle\Lambda_1(v^s,g_\eps)\rangle_m h(\langle v^s\rangle_m):=A+B.
\end{split}
\end{align}
The estimate 
\begin{align}\label{e22dd}
\frac{1}{\eps}\langle\Lambda (v^s v ^s)\rangle_{m,\gamma}\lesssim \langle \Lambda v^s\rangle_{m,\gamma}\langle v^s/\eps\rangle_m +\langle v^s/\eps\rangle_{m,\gamma}\langle\Lambda_1 v^s\rangle_m\lesssim E^2_m(v^s)
\end{align}
and other similar ones yields,
\begin{align}
\frac{1}{\eps}\langle\Lambda(v^s,g_\eps)^2\rangle_{m,\gamma}\lesssim E^2_m(v^s)+\eps^2 M_G^2,
\end{align}
so $A\lesssim \left(E_m^2(v^s)+\eps^2 M_G^2\right) h(\langle v_s\rangle_m).$  The term $B$ is treated similarly. 


\end{proof}

\begin{prop}\label{e22e}
Let $m>\frac{d+1}{2}$. For small enough $\eps_1>0$ and $M_0>0$ 
we have for $\eps\in (0,\eps_1]$
\begin{align}
\begin{split}
(a)\;\frac{1}{\eps}\langle\Lambda^{\frac{1}{2}} \cG\rangle_{m+1,\gamma}\lesssim \sqrt{\eps}M_G+\left(E_m^2(v^s)+\eps^{\frac{5}{2}} M_G^2\right) h(\langle v_s\rangle_{m+1})\lesssim Q^o(E_{m,T}(v^s))+\sqrt{\eps}Q(E_{m,T}(v^s))\\
(b)\;\frac{1}{\sqrt{\eps}}\langle\Lambda \cG\rangle_{m+1,\gamma}\lesssim \sqrt{\eps}M_G+\left(E_m^2(v^s)+\eps^{\frac{5}{2}} M_G^2\right) h(\langle v_s\rangle_{m+1})\lesssim  Q^o(E_{m,T}(v^s))+\sqrt{\eps}Q(E_{m,T}(v^s)).
\end{split}
\end{align}
\end{prop}

\begin{proof}
(a) Clearly, $\frac{1}{\eps}\langle\Lambda^{\frac{1}{2}} g_\eps\rangle_{m+1,\gamma}\lesssim \sqrt{\eps} M_G$.  
By Proposition \ref{j7}(c) we have in place of \eqref{e22dd}
\begin{align}
\begin{split}
&\frac{1}{\eps}\langle\Lambda^{\frac{1}{2}} (v^s v ^s)\rangle_{m+1,\gamma}\lesssim \langle \Lambda^{\frac{1}{2}} v^s\rangle_{m+1,\gamma}\langle v^s/\eps\rangle_m +\langle v^s/\sqrt{\eps}\rangle_{m+1,\gamma}\langle\Lambda^{\frac{1}{2}}_1 v^s/\sqrt{\eps}\rangle_m+\\
&\qquad \qquad\langle \Lambda^{\frac{1}{2}} v^s/\sqrt{\eps}\rangle_{m,\gamma}\langle v^s/\sqrt{\eps}\rangle_{m+1} +\langle v^s/\eps\rangle_{m,\gamma}\langle\Lambda^{\frac{1}{2}}_1 v^s\rangle_{m+1}\lesssim E^2_m(v^s).
\end{split}
\end{align}
The rest of the proof is essentially the same as that of Proposition \ref{e22cc}.    The proof of (b) is identical, except that we use control of $\sqrt{\eps}\langle \Lambda v^s\rangle_{m+1,\gamma}$ here.
\end{proof}

\subsection{Interior commutators}\label{ic}

\emph{\quad} 
To estimate terms in the first line of \eqref{c0a} we must estimate  norms involving interior commutators of the form
\begin{align}\label{e23}
|\Lambda^{\frac{1}{2}}[a(v^s)d_\eps^2,\partial^k]v|_{0,0,\gamma}\text{ and }|[a(v^s)d_\eps^2,\partial^{k+1}]v|_{0,0,\gamma},
\end{align}
where $k\leq m$, $a=A_\alpha$, and $\partial^k$ denotes $(\partial_{x'},\partial_\theta)^\alpha$ for some multi-index such that $|\alpha|=k$.  We will use $\partial$ to denote not only the tangential derivatives like $\partial_{t}$, $\partial_{x_1}$, and $\partial_{\theta}$ but  also (with some abuse) $\eps \partial_{x_1}+\beta_1\partial_\theta=\eps \partial_{x_1,\eps}$. 

\begin{prop}\label{e24}
Suppose $m>\frac{d+1}{2}+1$.
(a) When $d_\eps^2=\partial_{x_1,\eps}^2$ or $\partial_{x_2}\partial_{x_1,\eps}$ we have for $0\leq k\leq m$
\begin{align}\label{e24az}
|\Lambda^{\frac{1}{2}}[a(v^s)d_\eps^2,\partial^k]v|_{0,0,\gamma}\lesssim |\Lambda^{\frac{3}{2}} v|_{0,m,\gamma}\left|\frac{v^s}{\eps}\right|_{\infty,m}h(|v^s|_{\infty,m})+|\Lambda v|_{\infty,m,\gamma}\left|\frac{\Lambda^{\frac{1}{2}}_1v^s}{\eps}\right|_{0,m}h(|v^s|_{\infty,m}).
\end{align}

(b) When $d_\eps^2=\partial_{x_2}^2$,  we have  the same estimate with  $|\Lambda^{\frac{3}{2}} v|_{0,m,\gamma}$ and $|\Lambda v|_{\infty,m,\gamma}$ replaced respectively by
\begin{align}
 |\Lambda^{\frac{1}{2}}\eps \partial_{x_2}^2 v|_{0,m-1,\gamma} \text{ and } |\eps\partial_{x_2}^2 v|_{\infty,m-1,\gamma}.
\end{align}

\end{prop}

\begin{proof}
(a) Writing  $a(v^s)=a(0)+b(v^s)v^s$ and taking $k=m$ (the``main" case), we must estimate a sum of terms of the form
\begin{align} \label{e24a}
|\Lambda^{\frac{1}{2}}[(\partial^{m_1}(b(v^s)v^s))\cdot (\partial^{m_2}d_\eps^2 v)]|_{0,0,\gamma}=\frac{1}{\eps}|\Lambda^{\frac{1}{2}}[(\partial^{m_1}(b(v^s)v^s))\cdot (\partial^{m_2}d_\eps\partial v)]|_{0,0,\gamma},
\end{align}
where $m_1+m_2=m$, $m_1\geq 1$.   Treating $x_2$ as a parameter we apply Proposition \ref{f12} ``in the $(x',\theta)$ variables" with $\sigma=0$, $s=m-m_2-1$ and $t=m-m_1$.   Finish by taking the $L^2(x_2)$ norm of the inequality thus obtained.  The proof of (b) is essentially the same.\footnote{The formulas in (a) and (b) agree if one regards $\eps\partial_{x_2}$ (resp. $\partial_{x_2}$) as ``equivalent" to a  tangential derivative $\partial$ (resp. $\Lambda_D$).}

\end{proof}

\begin{rem}
Observe that the right side of \eqref{e24az} is $\lesssim E_{m,\gamma}(v)Q(E_{m,T}(v^s))$.  As we show below and in the next section, the same statement applies to \emph{all} interior and boundary commutators.
\end{rem}

\begin{prop}\label{e25}
Suppose $m>\frac{d+1}{2}+1$.
(a) When $d_\eps^2=\partial_{x_1,\eps}^2$ or $\partial_{x_2}\partial_{x_1,\eps}$ we have for $0\leq k\leq m+1$
\begin{align}
|[a(v^s)d_\eps^2,\partial^k]v|_{0,0,\gamma}\lesssim |\Lambda v|_{0,m+1,\gamma}\left|\frac{v^s}{\eps}\right|_{\infty,m}h(|v^s|_{\infty,m})+|\Lambda v|_{\infty,m,\gamma}\left|\frac{v^s}{\eps}\right|_{0,m+1}h(|v^s|_{\infty,m+1}).
\end{align}

(b) When $d_\eps^2=\partial_{x_2}^2$, we have  the same estimate with  $|\Lambda v|_{0,m+1,\gamma}$ and $|\Lambda v|_{\infty,m,\gamma}$ replaced respectively by
\begin{align}
|\eps \partial_{x_2}^2 v|_{0,m,\gamma} \text{ and } |\eps\partial_{x_2}^2 v|_{\infty,m-1,\gamma}.
\end{align}

\end{prop}

\begin{proof}
(a) Taking $k=m+1$, we must estimate a sum of terms of the form
\begin{align} \label{e25a}
|[(\partial^{m_1}(b(v^s)v^s))\cdot (\partial^{m_2}d_\eps^2 v)]|_{0,0,\gamma}=\frac{1}{\eps}|[(\partial^{m_1}(b(v^s)v^s))\cdot (\partial^{m_2}d_\eps\partial v)]|_{0,0,\gamma}:= A,
\end{align}
where $m_1+m_2=m+1$, $m_1\geq 1$.  In the case $m_1=m+1$ we obtain easily
\begin{align}\label{e25bz}
A\lesssim |v^sb(v^s)/\eps|_{0,m+1}|e^{-\gamma t}d_\eps\partial v|_{L^\infty}\lesssim |\Lambda v|_{\infty,m,\gamma}|v^s/\eps|_{0,m+1}h(|v^s|_{\infty,m+1}).
\end{align}
In the remaining case, we have $m_1\geq 1$, $m_2\geq 1$.  Treating $x_2$ as a parameter,  we estimate $A$ by  applying  Proposition  \ref{f12} with $r=0$, $\sigma=0$, $s=(m+1)-(m_2+1)$ and $t=m-m_1$, and taking the $L^2(x_2)$ norm of the inequality thus obtained.   Again, the proof of (b) is the similar.
\end{proof}

The next estimate is needed to control terms in the second line of \eqref{c0a}.

\begin{prop}\label{e26}
Suppose $m>\frac{d+1}{2}+2$.
(a) When $d_\eps^2=\partial_{x_1,\eps}^2$ or $\partial_{x_2}\partial_{x_1,\eps}$ we have for $0\leq k\leq m+1$
\begin{align}
\begin{split}
&\sqrt{\eps}|\Lambda^{\frac{1}{2}}[a(v^s)d_\eps^2,\partial^k]v|_{0,0,\gamma}\lesssim \left(\left|\frac{\Lambda_1^{\frac{1}{2}}v^s}{\sqrt{\eps}}\right|_{0,m+1}+\left|\frac{v^s}{\eps}\right|_{0,m+1}\right)|\Lambda v|_{\infty,m,\gamma}h(|v^s|_{\infty,m+1})+\\
&\qquad \left(|\sqrt{\eps}\Lambda^{\frac{3}{2}} v|_{0,m+1,\gamma}\;+|\Lambda v|_{0,m+1,\gamma}\right)\left|\frac{v^s}{\eps}\right|_{\infty,m}h(|v^s|_{\infty,m}).
\end{split}
\end{align}

(b) When $d_\eps^2=\partial_{x_2}^2$, we have the same estimate with $|\Lambda v|_{\infty,m,\gamma}$, $|\sqrt{\eps}\Lambda^{\frac{3}{2}} v|_{0,m+1,\gamma}$, and $|\Lambda v|_{0,m+1,\gamma}$ replaced by 
\begin{align}
|\eps\partial_{x_2}^2 v|_{\infty,m-1,\gamma},\; |\eps^{\frac{3}{2}}\Lambda^{\frac{1}{2}}\partial_{x_2}^2 v|_{0,m,\gamma}, \text{ and }|\eps \partial_{x_2}^2v|_{0,m,\gamma}.
\end{align}
respectively.

\end{prop}

\begin{proof}
Parallel to \eqref{e25a} we must estimate a sum of terms of the form 
\begin{align} \label{e26a}
\sqrt{\eps}|\Lambda^{\frac{1}{2}}[(\partial^{m_1}(b(v^s)v^s))\cdot (\partial^{m_2}d_\eps^2 v)]|_{0,0,\gamma}=\frac{1}{\sqrt{\eps}}|\Lambda^{\frac{1}{2}}[(\partial^{m_1}(b(v^s)v^s))\cdot (\partial^{m_2}d_\eps\partial v)]|_{0,0,\gamma}:= A,
\end{align}
where $m_1+m_2=m+1$, $m_1\geq 1$.  In the case $m_1=m+1$ we apply Proposition \ref{f12}(b) (in the $(x',\theta)$ variables)  with $m-\frac{3}{2}$ playing the role of $``s"$ and $``u"=d_\eps\partial v$ to obtain
\begin{align}\label{e26b}
A\lesssim \left(\left|\frac{\Lambda_1^{\frac{1}{2}}v^s}{\sqrt{\eps}}\right|_{0,m+1}+\left|\frac{v^s}{\eps}\right|_{0,m+1}\right)|\Lambda v|_{\infty,m,\gamma}h(|v^s|_{\infty,m+1}).
\end{align}
Here we have used $m-\frac{3}{2}>\frac{d+1}{2}$ to estimate 
\begin{align}
|\sqrt{\eps}\Lambda^{\frac{1}{2}}d_\eps\partial v|_{\infty,m-\frac{3}{2},\gamma}\leq |\Lambda v|_{\infty,m,\gamma}.
\end{align}

In the case $m_1=m$ we similarly apply Proposition \ref{f12}(b) with $``s"=m-\frac{5}{2}$  and $``t"=1$ to obtain \eqref{e26b} again.

In the  case $m_1\leq m-1$ we apply Proposition \ref{f12} with $\sigma=0$, $t=m-1-m_1$ and $s=(m+1)-(m_2+1)$ to obtain
\begin{align}
A\lesssim \left(|\sqrt{\eps}\Lambda^{\frac{3}{2}} v|_{0,m+1,\gamma}\;\left|\frac{v^s}{\eps}\right|_{\infty,m-1}+|\Lambda v|_{0,m+1,\gamma}\;\left|\frac{\Lambda^{\frac{1}{2}}_1v^s}{\sqrt{\eps}}\right|_{\infty,m-1}\right)h(|v^s|_{\infty,m}).   
\end{align}
To finish we observe
$|\frac{\Lambda^{\frac{1}{2}}_1v^s}{\sqrt{\eps}}|_{\infty,m-1}\lesssim |\frac{v^s}{\eps}|_{\infty,m}$.

\end{proof}

To control the terms \eqref{e26bb}  we must estimate the interior commutators
\begin{align}
\frac{1}{\eps}|\Lambda^{\frac{1}{2}}[a(v^s)d_\eps^2,\partial^m]u|_{0,0,\gamma}, \frac{1}{\eps}|[a(v^s)d_\eps^2,\partial^{m+1}]u|_{0,0,\gamma}, \text{ and }\;\frac{1}{\eps^{1/2}}|\Lambda^{1/2}[a(v^s)d_\eps^2,\partial^{m+1}]u|_{0,0,\gamma}.
\end{align}
In these estimates we sometimes write $v$ in place of $d_\eps u$ when a factor of $v^s$ is present.

\begin{prop}\label{e26c}
Suppose $m>\frac{d+1}{2}$.   For $0\leq k\leq m$ we have
\begin{align}
\begin{split}
&\;\frac{1}{\eps}|\Lambda^{\frac{1}{2}}[a(v^s)d_\eps^2,\partial^k]u|_{0,0,\gamma}\lesssim |\Lambda^{\frac{3}{2}} v|_{0,m,\gamma}\left|\frac{v^s}{\eps}\right|_{\infty,m}h(|v^s|_{\infty,m})+|\Lambda v|_{\infty,m,\gamma}\left|\frac{\Lambda^{\frac{1}{2}}_1v^s}{\eps}\right|_{0,m}h(|v^s|_{\infty,m})\\
\end{split}
\end{align}
\end{prop}

\begin{proof}
Parallel to \eqref{e24a}, to prove (a) we must estimate a sum of terms of the form
\begin{align} 
\frac{1}{\eps}|\Lambda^{\frac{1}{2}}[(\partial^{m_1}(b(v^s)v^s))\cdot (\partial^{m_2}d_\eps^2 u)]|_{0,0,\gamma}=\frac{1}{\eps}|\Lambda^{\frac{1}{2}}[(\partial^{m_1}(b(v^s)v^s))\cdot (\partial^{m_2}d_\eps v)]|_{0,0,\gamma},
\end{align}
where $m=m_1+m_2$, $m_1\geq 1$. To finish apply Proposition \ref{f12} with $s=m-m_2$, $t=m-m_1$.  
\end{proof}

\begin{prop}\label{e26d}
Suppose $m>\frac{d+1}{2}$.    For $0\leq k\leq m+1$ we have
\begin{align}
\frac{1}{\eps}|[a(v^s)d_\eps^2,\partial^k]u|_{0,0,\gamma}\lesssim |\Lambda v|_{0,m+1,\gamma}\left|\frac{v^s}{\eps}\right|_{\infty,m}h(|v^s|_{\infty,m})+|\Lambda v|_{\infty,m,\gamma}\left|\frac{v^s}{\eps}\right|_{0,m+1}h(|v^s|_{\infty,m+1}).
\end{align}

\end{prop}

\begin{proof}
(a) Parallel to \eqref{e25a} we must estimate a sum of terms of the form
\begin{align} 
\frac{1}{\eps}|[(\partial^{m_1}(b(v^s)v^s))\cdot (\partial^{m_2}d_\eps^2 u)]|_{0,0,\gamma}=\frac{1}{\eps}|[(\partial^{m_1}(b(v^s)v^s))\cdot (\partial^{m_2}d_\eps v)]|_{0,0,\gamma}:= A,
\end{align}
where $m_1+m_2=m+1$, $m_1\geq 1$.  In the case $m_1=m+1$ we obtain 
\begin{align}\label{e25b}
A\lesssim |v^sb(v^s)/\eps|_{0,m+1}|e^{-\gamma t}d_\eps v|_{L^\infty}\lesssim |\Lambda v|_{\infty,m,\gamma}|v^s/\eps|_{0,m+1}h(|v^s|_{\infty,m+1}).
\end{align}
In the remaining case, we have $m_1\geq 1$, $m_2\geq 1$, and we estimate $A$ by  applying  Proposition  \ref{f12} with $r=0$, $\sigma=0$, $s=m+1-m_2$ and $t=m-m_1$, and taking the $L^2(x_2)$ norm of the inequality thus obtained.  
\end{proof}

\begin{prop}\label{e26h}
Suppose $m>\frac{d+1}{2}+1$.
For $0\leq k\leq m+1$ we have
\begin{align}
\begin{split}
&\frac{1}{\sqrt{\eps}}|\Lambda^{\frac{1}{2}}[a(v^s)d_\eps^2,\partial^k]u|_{0,0,\gamma}\lesssim \left(\left|\frac{\Lambda_1^{\frac{1}{2}}v^s}{\sqrt{\eps}}\right|_{0,m+1}+\left|\frac{v^s}{\eps}\right|_{0,m+1}\right)|\Lambda v|_{\infty,m,\gamma}h(|v^s|_{\infty,m+1})+\\
&\qquad \qquad \qquad |\sqrt{\eps}\Lambda^{\frac{3}{2}} v|_{0,m,\gamma}\left|\frac{v^s}{\eps}\right|_{\infty,m}h(|v^s|_{\infty,m}).
\end{split}
\end{align}

\end{prop}

\begin{proof}
Parallel to \eqref{e26a} we must estimate a sum of terms of the form 
\begin{align} \label{e26i}
\frac{1}{\sqrt{\eps}}|\Lambda^{\frac{1}{2}}[(\partial^{m_1}(b(v^s)v^s))\cdot (\partial^{m_2}d_\eps^2 u)]|_{0,0,\gamma}=\frac{1}{\sqrt{\eps}}|\Lambda^{\frac{1}{2}}[(\partial^{m_1}(b(v^s)v^s))\cdot (\partial^{m_2}d_\eps v)]|_{0,0,\gamma}:=A,
\end{align}
where $m_1+m_2=m+1$ and $m_1\geq 1$.
We can now repeat the proof of Proposition \ref{e26} in the case $m_1=m+1$ with $``u"=d_\eps v$ and using $``s"=m-\frac{1}{2}$ in place of $m-\frac{3}{2}$.   In the case $m_1\leq m$ we apply Proposition \ref{f12}  with $t=m-m_1$ and $s=m-m_2$  to obtain
\begin{align}
A\lesssim  \left(|\sqrt{\eps}\Lambda^{\frac{3}{2}} v|_{0,m,\gamma} \left|\frac{v^s}{\eps}\right|_{\infty,m}+|\Lambda v|_{\infty,m,\gamma}\left|\frac{\Lambda_1^{\frac{1}{2}}v^s}{\sqrt{\eps}}\right|_{0,m}\right)h(|v^s|_{\infty,m}).
\end{align}

\end{proof}

To control the terms \eqref{e14dd} we must estimate the commutator $\frac{1}{\sqrt{\eps}}|[a(v^s)d_\eps^2,\partial ^k]v|_{0,0,\gamma}$, $k\leq m$.

\begin{prop}\label{e26e}
Suppose $m>\frac{d+1}{2}+1$.
(a) When $d_\eps^2=\partial_{x_1,\eps}^2$ or $\partial_{x_2}\partial_{x_1,\eps}$ we have for $0\leq k\leq m$
\begin{align}
\frac{1}{\sqrt{\eps}}|[a(v^s)d_\eps^2,\partial^k]v|_{0,0,\gamma}\lesssim \left|\frac{\Lambda v}{\sqrt{\eps}}\right|_{0,m,\gamma}\left|\frac{v^s}{\eps}\right|_{\infty,m}h(|v^s|_{\infty,m}).
\end{align}

(b) When $d_\eps^2=\partial_{x_2}^2$, we have  the same estimate with  $\left|\frac{\Lambda v}{\sqrt{\eps}}\right|_{0,m,\gamma}$ replaced by $\left|\frac{\eps\partial^2_{x_2} v}{\sqrt{\eps}}\right|_{0,m-1,\gamma}$.

\end{prop}

\begin{proof}
(a) Taking $k=m$, we must estimate a sum of terms of the form
\begin{align} \label{e26f}
\frac{1}{\sqrt{\eps}}|[(\partial^{m_1}(b(v^s)v^s))\cdot (\partial^{m_2}d_\eps^2 v)]|_{0,0,\gamma}=\frac{1}{\eps^{3/2}}|[(\partial^{m_1}(b(v^s)v^s))\cdot (\partial^{m_2}d_\eps\partial v)]|_{0,0,\gamma}:= A,
\end{align}
where $m_1+m_2=m$, $m_1\geq 1$.  Treating $x_2$ as a parameter,  we estimate $A$ by  applying  Proposition  \ref{f12} with $r=0$, $\sigma=0$, $s=m-(m_2+1)$ and $t=m-m_1$, and taking the $L^2(x_2)$ norm of the inequality thus obtained.   Again, the proof of (b) is  similar.
\end{proof}


\subsection{Boundary  commutators}\label{bc}

The boundary commutators vanish for the Dirichlet problem \eqref{c2}, but not for the Neumann-type problem \eqref{c1}.  To estimate terms in the first lines of \eqref{c0a},  we must estimate  norms involving boundary commutators of the form
\begin{align}\label{e27}
\langle\Lambda [a(v^s)d_\eps,\partial^k]v\rangle_{0,\gamma}\text{ and }\langle \Lambda^{\frac{1}{2}}[a(v^s)d_\eps,\partial^{k+1}]v\rangle_{0,\gamma},
\end{align}
where $k\leq m$, $a(v^s)=d_{v_1}H(v^s_1,h(v^s))$, and $v=v_1$. 

\begin{prop}\label{e28}
Suppose $m>\frac{d+1}{2}+1$.     We have for $0\leq k\leq m$
\begin{align}
\langle\Lambda[a(v^s)d_\eps,\partial^k]v\rangle_{0,\gamma}\lesssim \langle\Lambda v\rangle_{m,\gamma}\left\langle\frac{v^s}{\eps}\right\rangle_{m}h(\langle v^s\rangle_{m})+\left\langle \frac{\nabla_\eps u}{\eps}\right \rangle_{m,\gamma}\langle\Lambda_1v^s \rangle_{m}h(\langle v^s\rangle_{m}).
\end{align}
\end{prop}

\begin{proof}
 Writing  $a(v^s)=a(0)+b(v^s)v^s$ and taking $k=m$, we must estimate a sum of terms of the form
\begin{align} \label{e28a}
\langle\Lambda [(\partial^{m_1}(b(v^s)v^s))\cdot (\partial^{m_2}d_\eps v)]\rangle_{0,\gamma}=\frac{1}{\eps}\langle\Lambda [(\partial^{m_1}(b(v^s)v^s))\cdot (\partial^{m_2}\partial v)]\rangle_{0,\gamma},
\end{align}
where $m_1+m_2=m$, $m_1\geq 1$.  To finish apply Proposition \ref{f12} with $r=1$, $\sigma=0$, $s=m-(m_2+1)$, and $t=m-m_1$. 
\end{proof}


\begin{prop}\label{e29}
Suppose $m>\frac{d+1}{2}+1$.     We have for $0\leq k\leq m+1$
\begin{align}
\begin{split}
&\langle\Lambda^{\frac{1}{2}}[a(v^s)d_\eps,\partial^k]v\rangle_{0,\gamma}\lesssim \left(\left\langle\frac{\nabla_\eps u}{\eps}\right\rangle_{m,\gamma}\;
\langle\Lambda^{\frac{1}{2}}_1 v^s\rangle_{m+1} +\left\langle\frac{\Lambda^{\frac{1}{2}}v}{\sqrt{\eps}}\right\rangle_{m,\gamma}\;\left\langle\frac{v^s}{\sqrt{\eps}}\right\rangle_{m+1}\right) h(\langle v^s\rangle_{m+1})+\\
&\qquad \left(\langle\Lambda^{\frac{1}{2}} v\rangle_{m+1,\gamma}\;\left\langle\frac{v^s}{\eps}\right\rangle_m 
 +\left\langle\frac{\nabla_\eps u}{\sqrt{\eps}}\right\rangle_{m+1,\gamma}\;\left\langle\frac{\Lambda^{\frac{1}{2}}_1v^s}{\sqrt{\eps}}\right\rangle_{m}\right) h(\langle v^s\rangle_{m}).
\end{split}
\end{align}
\end{prop}

\begin{proof}
 Taking $k=m+1$, we must estimate a sum of terms of the form
\begin{align} \label{e29a}
\langle\Lambda^{\frac{1}{2}} [(\partial^{m_1}(b(v^s)v^s))\cdot (\partial^{m_2}d_\eps v)]\rangle_{0,\gamma}=\frac{1}{\eps}\langle\Lambda^{\frac{1}{2}} [(\partial^{m_1}(b(v^s)v^s))\cdot (\partial^{m_2}\partial v)]\rangle_{0,\gamma}:=A,
\end{align}
where $m_1+m_2=m+1$, $m_1\geq 1$.  In the case $m_1=m+1$ we apply Proposition \ref{f12} with $t=0$, $s=m-1$ to obtain
\begin{align}\label{e30}
A\lesssim  \left(\left\langle\frac{\Lambda^{\frac{1}{2}}v}{\sqrt{\eps}}\right\rangle_{m,\gamma}\;\left\langle\frac{v^s}{\sqrt{\eps}}\right\rangle_{m+1}+\left\langle\frac{\nabla_\eps u}{\eps}\right\rangle_{m,\gamma}\;
\langle\Lambda^{\frac{1}{2}}_1 v^s\rangle_{m+1}\right) h(\langle v^s\rangle_{m+1}).
\end{align}
In the  case $m_1\leq m$ we apply Proposition \ref{f12} with $r=\frac{1}{2}$, $\sigma=0$, $s=m+1-(m_2+1)$, and $t=m-m_1$ to obtain
\begin{align}
A\lesssim \left(\langle\Lambda^{\frac{1}{2}} v\rangle_{m+1,\gamma}\;\left\langle\frac{v^s}{\eps}\right\rangle_m 
 +\left\langle\frac{\nabla_\eps u}{\sqrt{\eps}}\right\rangle_{m+1,\gamma}\;\left\langle\frac{\Lambda^{\frac{1}{2}}_1v^s}{\sqrt{\eps}}\right\rangle_{m}\right) h(\langle v^s\rangle_{m}).
\end{align}
\end{proof}


The next estimate is needed to control terms in the second line of \eqref{c0a}.

\begin{prop}\label{e31}
Suppose $m>\frac{d+1}{2}+1$.     We have for $0\leq k\leq m+1$
\begin{align}
\begin{split}
&\sqrt{\eps}\langle\Lambda[a(v^s)d_\eps,\partial^k]v\rangle_{0,\gamma}\lesssim \left(\left\langle\frac{\nabla_\eps u}{\eps}\right\rangle_{m,\gamma}\;
\langle \sqrt{\eps}\Lambda^{\frac{1}{2}}_1 v^s\rangle_{m+1} +\langle\Lambda v\rangle_{m,\gamma}\;\left\langle\frac{v^s}{\sqrt{\eps}}\right\rangle_{m+1}\right) h(\langle v^s\rangle_{m+1})+\\
&\qquad \left(\langle \sqrt{\eps}\Lambda v\rangle_{m+1,\gamma}\;\left\langle\frac{v^s}{\eps}\right\rangle_m 
 +\left\langle\frac{\nabla_\eps u}{\sqrt{\eps}}\right\rangle_{m+1,\gamma}\;\langle\Lambda_1 v^s\rangle_{m}\right) h(\langle v^s\rangle_{m})
\end{split}
\end{align}
\end{prop}

\begin{proof}
 Taking $k=m+1$, we must estimate a sum of terms of the form
\begin{align} \label{e31a}
\sqrt{\eps}\langle\Lambda [(\partial^{m_1}(b(v^s)v^s))\cdot (\partial^{m_2}d_\eps v)]\rangle_{0,\gamma}=\frac{1}{\sqrt{\eps}}\langle\Lambda[(\partial^{m_1}(b(v^s)v^s))\cdot (\partial^{m_2}\partial v)]\rangle_{0,\gamma}:=A,
\end{align}
where $m_1+m_2=m+1$, $m_1\geq 1$.  In the case $m_1=m+1$ we apply Proposition \ref{f12} with $t=0$, $s=m-1$ to obtain
\begin{align}
A\lesssim  \left(\left\langle\frac{\nabla_\eps u}{\eps}\right\rangle_{m,\gamma}\;
\langle \sqrt{\eps}\Lambda^{\frac{1}{2}}_1 v^s\rangle_{m+1} +\langle\Lambda v\rangle_{m,\gamma}\;\left\langle\frac{v^s}{\sqrt{\eps}}\right\rangle_{m+1}\right) h(\langle v^s\rangle_{m+1}).
\end{align}
In the  case $m_1\leq m$ we apply Proposition \ref{f12} with $s=m+1-(m_2+1)$, and $t=m-m_1$ to obtain
\begin{align}
A\lesssim \left(\langle \sqrt{\eps}\Lambda v\rangle_{m+1,\gamma}\;\left\langle\frac{v^s}{\eps}\right\rangle_m 
 +\left\langle\frac{\nabla_\eps u}{\sqrt{\eps}}\right\rangle_{m+1,\gamma}\;\langle\Lambda_1 v^s\rangle_{m}\right) h(\langle v^s\rangle_{m}).
\end{align}
\end{proof}

To control the terms \eqref{e26bb} we must estimate for $u$ in \eqref{c3} the boundary commutators
\begin{align}
\frac{1}{\eps}\langle\Lambda [a(v^s)d_\eps,\partial^m] u\rangle_{0,\gamma},\;\frac{1}{\eps}\langle\Lambda^{\frac{1}{2}} [a(v^s)d_\eps,\partial^{m+1}] u\rangle_{0,\gamma},\text{ and }\; \frac{1}{\sqrt{\eps}}\langle\Lambda [a(v^s)d_\eps,\partial^{m+1}] u\rangle_{0,\gamma}.
\end{align}

\begin{prop}\label{e32}
Suppose $m>\frac{d+1}{2}$.     We have for $0\leq k\leq m$
\begin{align}
\begin{split}
&\;\frac{1}{\eps}\langle\Lambda[a(v^s)d_\eps,\partial^k]u\rangle_{0,\gamma}\lesssim \langle\Lambda v\rangle_{m,\gamma}\left\langle\frac{v^s}{\eps}\right\rangle_{m}h(\langle v^s\rangle_{m})+\left\langle \frac{\nabla_\eps u}{\eps}\right \rangle_{m,\gamma}\langle\Lambda_1v^s \rangle_{m}h(\langle v^s\rangle_{m})\\
\end{split}
\end{align}
\end{prop}

\begin{proof}
 Parallel to \eqref{e28a}, to prove (a) we must estimate a sum of terms of the form
\begin{align} 
\frac{1}{\eps}\langle\Lambda [(\partial^{m_1}(b(v^s)v^s))\cdot (\partial^{m_2}d_\eps u)]\rangle_{0,\gamma}=\frac{1}{\eps}\langle\Lambda [(\partial^{m_1}(b(v^s)v^s))\cdot (\partial^{m_2}v)]\rangle_{0,\gamma},
\end{align}
where $m_1+m_2=m$, $m_1\geq 1$.  To finish apply Proposition \ref{f12} with $r=1$, $\sigma=0$, $s=m-m_2$, and $t=m-m_1$.    

\end{proof}

\begin{prop}\label{e33}
Suppose $m>\frac{d+1}{2}$.     We have for $0\leq k\leq m+1$
\begin{align}
\begin{split}
&\frac{1}{\eps}\langle\Lambda^{\frac{1}{2}}[a(v^s)d_\eps,\partial^k]u\rangle_{0,\gamma}\lesssim \left(\left\langle\frac{\nabla_\eps u}{\eps}\right\rangle_{m,\gamma}\;
\langle\Lambda^{\frac{1}{2}}_1 v^s\rangle_{m+1} +\left\langle\frac{\Lambda^{\frac{1}{2}}v}{\sqrt{\eps}}\right\rangle_{m,\gamma}\;\left\langle\frac{v^s}{\sqrt{\eps}}\right\rangle_{m+1}\right) h(\langle v^s\rangle_{m+1})+\\
&\qquad \left(\langle\Lambda^{\frac{1}{2}} v\rangle_{m+1,\gamma}\;\left\langle\frac{v^s}{\eps}\right\rangle_m 
 +\left\langle\frac{\nabla_\eps u}{\sqrt{\eps}}\right\rangle_{m+1,\gamma}\;\left\langle\frac{\Lambda^{\frac{1}{2}}_1v^s}{\sqrt{\eps}}\right\rangle_{m}\right) h(\langle v^s\rangle_{m}).
\end{split}
\end{align}
\end{prop}

\begin{proof}
 Parallel to \eqref{e29a} we must estimate a sum of terms of the form
\begin{align} 
\frac{1}{\eps}\langle\Lambda^{\frac{1}{2}} [(\partial^{m_1}(b(v^s)v^s))\cdot (\partial^{m_2}d_\eps u)]\rangle_{0,\gamma}=\frac{1}{\eps}\langle\Lambda^{\frac{1}{2}} [(\partial^{m_1}(b(v^s)v^s))\cdot (\partial^{m_2} v)]\rangle_{0,\gamma}:=A,
\end{align}
where $m_1+m_2=m+1$, $m_1\geq 1$.  In the case $m_1=m+1$ we apply Proposition \ref{f12} to obtain
\begin{align}
A\lesssim  \left(\left\langle\frac{\Lambda^{\frac{1}{2}}v}{\sqrt{\eps}}\right\rangle_{m,\gamma}\;\left\langle\frac{v^s}{\sqrt{\eps}}\right\rangle_{m+1}+\left\langle\frac{\nabla_\eps u}{\eps}\right\rangle_{m,\gamma}\;
\langle\Lambda^{\frac{1}{2}}_1 v^s\rangle_{m+1}\right) h(\langle v^s\rangle_{m+1}).
\end{align}
In the  case $m_1\leq m$ we apply Proposition \ref{f12} with $r=\frac{1}{2}$, $\sigma=0$, $s=m+1-m_2$, and $t=m-m_1$ to obtain
\begin{align}
A\lesssim \left(\langle\Lambda^{\frac{1}{2}} v\rangle_{m+1,\gamma}\;\left\langle\frac{v^s}{\eps}\right\rangle_m 
 +\left\langle\frac{\nabla_\eps u}{\sqrt{\eps}}\right\rangle_{m+1,\gamma}\;\left\langle\frac{\Lambda^{\frac{1}{2}}_1v^s}{\sqrt{\eps}}\right\rangle_{m}\right) h(\langle v^s\rangle_{m}).
\end{align}
\end{proof}

\begin{prop}\label{e36}
Suppose $m>\frac{d+1}{2}$.     We have for $0\leq k\leq m+1$
\begin{align}
\begin{split}
&\frac{1}{\sqrt{\eps}}\langle\Lambda[a(v^s)d_\eps,\partial^k]u\rangle_{0,\gamma}\lesssim \left(\left\langle\frac{\nabla_\eps u}{\eps}\right\rangle_{m,\gamma}\;
\langle \sqrt{\eps}\Lambda^{\frac{1}{2}}_1 v^s\rangle_{m+1} +\langle\Lambda v\rangle_{m,\gamma}\;\left\langle\frac{v^s}{\sqrt{\eps}}\right\rangle_{m+1}\right) h(\langle v^s\rangle_{m+1})+\\
&\qquad \left(\langle \sqrt{\eps}\Lambda v\rangle_{m+1,\gamma}\;\left\langle\frac{v^s}{\eps}\right\rangle_m 
 +\left\langle\frac{\nabla_\eps u}{\sqrt{\eps}}\right\rangle_{m+1,\gamma}\;\langle\Lambda_1 v^s\rangle_{m}\right) h(\langle v^s\rangle_{m})
\end{split}
\end{align}
\end{prop}

\begin{proof}
 Taking $k=m+1$, we must estimate a sum of terms of the form
\begin{align} 
\frac{1}{\sqrt{\eps}}\langle\Lambda [(\partial^{m_1}(b(v^s)v^s))\cdot (\partial^{m_2}d_\eps u)]\rangle_{0,\gamma}=\frac{1}{\sqrt{\eps}}\langle\Lambda[(\partial^{m_1}(b(v^s)v^s))\cdot (\partial^{m_2} v)]\rangle_{0,\gamma}:=A,
\end{align}
where $m_1+m_2=m+1$, $m_1\geq 1$.  Comparison with \eqref{e31a} shows that we can repeat the proof of Proposition \ref{e31}, now applying Proposition \ref{f12} with $t=0$, $s=m$ in the case $m_1=m+1$, and with $t=m-m_1$, $s=m+1-m_2$ in the case $m_1\leq m$.

\end{proof}

To control the terms \eqref{e14dd} with $j=1$ we must estimate $\frac{1}{\sqrt{\eps}}\langle \Lambda^{\frac{1}{2}}[a(v^s)d_\eps,\partial^k]v\rangle_{0,\gamma}$, $k\leq m$ and $v=v_1$.

\begin{prop}\label{e34}
Suppose $m>\frac{d+1}{2}+1$.     We have for $0\leq k\leq m$ and $v=v_1$
\begin{align}
\frac{1}{\sqrt{\eps}}\langle\Lambda^{\frac{1}{2}}[a(v^s)d_\eps,\partial^k]v\rangle_{0,\gamma}\lesssim \left\langle\frac{\Lambda^{\frac{1}{2}} v}{\sqrt{\eps}}\right\rangle_{m,\gamma}\left\langle\frac{v^s}{\eps}\right\rangle_{m}h(\langle v^s\rangle_{m})+\left\langle \frac{\nabla_\eps u}{\eps}\right \rangle_{m,\gamma}\left\langle\frac{\Lambda^{\frac{1}{2}}_1 v}{\sqrt{\eps}}\right\rangle_{m}h(\langle v^s\rangle_{m}).
\end{align}
\end{prop}

\begin{proof}
 Parallel to \eqref{e28a} we must estimate a sum of terms of the form
\begin{align} \label{e35}
\frac{1}{\sqrt{\eps}}\langle\Lambda^{\frac{1}{2}} [(\partial^{m_1}(b(v^s)v^s))\cdot (\partial^{m_2}d_\eps v)]\rangle_{0,\gamma}=\frac{1}{\eps}\langle\Lambda^{\frac{1}{2}} [(\partial^{m_1}(b(v^s)v^s))\cdot (\partial^{m_2}\partial v)]\rangle_{0,\gamma},
\end{align}
where $m_1+m_2=m$, $m_1\geq 1$.  To finish apply Proposition \ref{f12} with $r=\frac{1}{2}$, $\sigma=0$, $s=m-(m_2+1)$, and $t=m-m_1$. 
\end{proof}


\subsection{Boundary forcing II}\label{bf2}

\emph{\quad}Here we estimate the forcing given by $\cG_2$ in the problem \eqref{c2}.   For this we require 

\begin{prop}\label{g1}
Suppose $m>3d+4+\frac{d+1}{2}$.  There exist positive constants $\eps_0$,   $\gamma_0$
such that for $j\in J_h\cup J_e$, $\eps\in (0,\eps_0]$, and each $T$ with $0<T\leq T^*_\eps$, 
\begin{align}\label{g2}
\begin{split}
&E_{m,\gamma}(v_1)+\langle\phi_j\Lambda^{\frac{3}{2}}v_1\rangle_{m,\gamma}+\langle\phi_j\Lambda v_1\rangle_{m+1,\gamma}+\sqrt{\eps}\langle\phi_j\Lambda^{\frac{3}{2}}v_1\rangle_{m+1,\gamma}+\frac{1}{\sqrt{\eps}}\langle\phi_j\Lambda v_1\rangle_{m,\gamma}\lesssim\\
&\frac{1}{\sqrt{\gamma}}\{Q_1(E_{m,T}(v^s))\cdot E_{m,\gamma}(v_1)+Q^o_2(E_{m,T}(v^s))+(M_G+\sqrt{\eps})Q_3(E_{m,T}(v^s))\}\text{ for }\gamma\geq \gamma_0.
\end{split}
\end{align}
 
\end{prop}

\begin{proof}

For all terms in the definition of $E_{m,\gamma}(v_1)$, 
domination by the right side of \eqref{g2} follows directly from the linearized estimates of section \ref{b1a} and the propositions  of sections \ref{if}, \ref{bf}, \ref{ic}, and \ref{bc}. 
The ``right side" of every  commutator proposition needed to estimate these terms is dominated by $Q_1(E_{m,T}(v^s))\cdot E_{m,\gamma}(v_1)$,
and the right side of every forcing proposition  is dominated by $Q^o_2(E_{m,T}(v^s))+(M_G+\sqrt{\eps})Q_3(E_{m,T}(v^s))$. The estimates of section \ref{b1a} show that $1/\sqrt{\gamma}$ is the correct factor in \eqref{g2}.

Consider for example the first term in the first line of \eqref{c0a} when $j=1$.  By \eqref{b3} we have
\begin{align}\label{g2a}
\left|\begin{pmatrix}\Lambda^{\frac{3}{2}} v_1\\D_{x_2}\Lambda^{\frac{1}{2}} v_1\end{pmatrix}\right|_{0,m,\gamma}\lesssim \gamma^{-1}\left(|\Lambda^{\frac{1}{2}}\cF_1|_{0,m,\gamma}+\langle\Lambda \cG_1\rangle_{m,\gamma}+|\Lambda^{\frac{1}{2}}\cF_{1,com}|_{0,m,\gamma}+\langle\Lambda \cG_{1,com}\rangle_{m,\gamma}+\cE_{m,\gamma},\right),
\end{align}
where $\cF_{1,com}$ and $\cG_{1,com}$ denote the interior and boundary commutators estimated in Propositions \ref{e24} and \ref{e28}, and $\cE_{m,\gamma}$ denotes the error term
\begin{align}
\cE_{m,\gamma}:=\left|\begin{pmatrix}\Lambda v_1/\sqrt{\eps}\\D_{x_2} v_1/\sqrt{\eps}\end{pmatrix}\right|_{0,m,\gamma}.
\end{align}
 These propositions, along with the interior and boundary forcing estimates in Propositions \ref{e12} and \ref{e15}, show that the right side of \eqref{g2a} is dominated by the right side of \eqref{g2}.   By the estimate \eqref{b3}, a term involving $\phi_j$ like $\langle\phi_j\Lambda^{\frac{3}{2}}v_1\rangle_{m,\gamma}$ is controlled as soon as the left side of \eqref{g2a} is.   However, for this $\phi_j$ term we require the factor $\frac{1}{\sqrt{\gamma}}$ in place  of $\gamma^{-1}$ on the right side of \eqref{g2a}.



\end{proof}

The next Corollary is immediate.
\begin{cor}\label{g3}
Let $m>3d+4+\frac{d+1}{2}$.  
There exist positive constants $\eps_1=\eps_1(M_0)$,   $\gamma_m=\gamma_m(M_0)$ 
such that for $\eps\in (0,\eps_1]$ and each $T$ with $0<T\leq T^*_\eps$, 
\begin{align}\label{g4}
E_{m,\gamma}(v_1)\leq \frac{1}{\sqrt{\gamma}}[Q_2^o(E_{m,T}(v^s))+(M_G+\sqrt{\eps})Q_3(E_{m,T}(v^s))]\text{ for }\gamma\geq \gamma_m.
\end{align}
The function $\gamma(M_0)$ increases with $M_0$. 

\end{cor}

\begin{rem}\label{e17}
1) When we apply Proposition \ref{f7} to estimate $\langle\phi_j\Lambda^{\frac{3}{2}}\mathcal{G}_2\rangle_{m,\gamma}$ for $j\in J_h$, we must 
control norms like $\langle \Lambda_1 (\chi^*(t)v_1)\rangle_m$ (also $\langle\chi^* v_1\rangle_m$), where $v_1$ and not $v_1^s$ occurs, and $\mathrm{supp}\;\chi^*\subset (-1,2)$.  Using Proposition \ref{c0e} we have

\begin{align}
E_m(\chi^*(t)v_1)\lesssim E_{m,T=2}(v_1)\lesssim e^{\gamma_m(M_0)2}E_{m,\gamma_m(M_0)}(v_1)
\end{align}
Applying Corollary \ref{g3} with $\gamma=\gamma_m(M_0)$ we obtain
\begin{align}\label{e17a}
\begin{split}
&E_m(\chi^*(t)v_1)\lesssim e^{\gamma_m(M_0)2}[Q_2^o(E_{m,T}(v^s))+(M_G+\sqrt{\eps})Q_3(E_m(v^s))]\lesssim \\
&\qquad\qquad e^{\gamma_m(M_0)2}[Q_2^o(M_0)+(M_G+\sqrt{\eps})Q_3(M_0)].
\end{split}
\end{align}

2) The right side of \eqref{e17a} is small when $M_G$, $M_0$, and $\eps_1$ are small.  The constant $M_G$ is small when $T_0$ as in \eqref{c0h} is small, since $G(x',\theta)=0$ in $t<0$.   
\end{rem}

To control terms in the first line of \eqref{c0a} when $j=2$, we must estimate $\langle\Lambda \cG_2\rangle_{m,\gamma}$ and $\langle\Lambda^{\frac{1}{2}}\cG_2\rangle_{m+1,\gamma}$; recall $\cG_2=\chi_0(t)H(v_1,g_\eps)$. 
\begin{prop}\label{e18}
Let $m>\frac{d+1}{2}$. For small enough positive constants  $\eps_1$, $M_G$ \eqref{MG},   and $M_0$ \eqref{c0h} 
we have for $\eps\in (0,\eps_1]$
\begin{align}\label{e19}
\begin{split}
&(a)\;\langle\Lambda \cG_2\rangle_{m,\gamma}\lesssim (\langle \Lambda v_1\rangle_{m,\gamma}+\eps)[Q_1(E_m(v^s))+Q_2(M_G)]\\
&(b)\;\langle\Lambda^{\frac{1}{2}}\cG_2\rangle_{m+1,\gamma}\lesssim (\langle \Lambda^{\frac{1}{2}} v_1\rangle_{m+1,\gamma}+\eps^{\frac{3}{2}})[Q_1(E_m(v^s))+Q_2(M_G)].
\end{split}
\end{align}
\end{prop}

\begin{proof}
Choose $\chi^*(t)$ supported in $(-1,2)$ such that $\chi_0\chi^*=\chi_0$ and write $\chi^*(t)v_1=v_1^*$. 
By Proposition \ref{f3} we have 
\begin{align}\label{e20}
\langle\Lambda \cG_2\rangle_{m,\gamma}\leq \langle \Lambda (v_1^*,g_\eps)\rangle_{m,\gamma} h(\langle v^*_1,g_\eps\rangle_m)+\langle v^*_1,g_\eps\rangle_{m,\gamma}\langle\Lambda_1(v^*_1,g_\eps)\rangle_mh(\langle v^*_1,g_\eps\rangle_m).
\end{align}
Then \eqref{e19}(a) follows since
\begin{align}
\langle \Lambda (v_1^*,g_\eps)\rangle_{m,\gamma}\lesssim  \langle \Lambda v_1\rangle_{m,\gamma}+\eps M_G,
\end{align}
and the  factors involving $\langle\cdot\rangle_m$ norms  on the right in \eqref{e20} are controlled using \eqref{e17a}.  In particular, $\eps_1$, $M_0$, and $M_G$ must be taken small enough so that $\langle v^*_1,g_\eps\rangle_m$ lies in the region of convergence of $h$.   The commutator $\langle [\Lambda,\chi^*]v_1\rangle_{m,\gamma}$ is estimated using Proposition \ref{commutator1}.
The inequality \eqref{e19}(b) is proved similarly. 

\end{proof}

\begin{prop}\label{e21}
Let $m>\frac{d+1}{2}$ and let $\phi_j$ and $\psi_j$, $j\in J_h$, be singular symbols related as in Notation \ref{f4b}. For small enough positive constants  $\eps_1$, $M_G$,   and $M_0$ 
we have for $\eps\in (0,\eps_1]$
\begin{align}\label{e22}
\begin{split}
&(a)\;\langle\phi_j\Lambda^{\frac{3}{2}} \cG_2\rangle_{m,\gamma}\lesssim \left(\langle \psi_j\Lambda^{\frac{3}{2}} v_1\rangle_{m,\gamma}+\langle \Lambda v_1\rangle_{m,\gamma}+\sqrt{\eps}\right)[Q_1(E_m(v^s))+Q_2(M_G)]\\
&(b)\;\langle\phi_j\Lambda\cG_2\rangle_{m+1,\gamma}\lesssim  \left(\langle \psi_j\Lambda v_1\rangle_{m+1,\gamma}+\langle \Lambda^{\frac{1}{2}} v_1\rangle_{m+1,\gamma}+\eps\right)[Q_1(E_m(v^s))+Q_2(M_G)]\\
\end{split}
\end{align}
\end{prop}

\begin{proof}
The argument is parallel to the proof of Proposition \ref{e18}, except that Proposition \ref{f7} is used in place of Proposition \ref{f3}.   Commutators like $\langle[\psi_j\Lambda^{\frac{3}{2}},\chi^*(t)]v_1\rangle_{m,\gamma}$ are estimated using Proposition \ref{commutator1}.
\end{proof}

To control the terms in the second line of \eqref{c0a} when $j=2$,  we must estimate
\begin{align}
\sqrt{\eps}\langle\Lambda\cG_2\rangle_{m+1,\gamma},\; \sqrt{\eps}\langle\phi_j\Lambda^{\frac{3}{2}}\cG_2\rangle_{m+1,\gamma},\; \frac{1}{\sqrt{\eps}}\langle\Lambda^{\frac{1}{2}}\cG_2\rangle_{m,\gamma}, \text{ and }\frac{1}{\sqrt{\eps}}\langle\phi_j\Lambda\cG_2\rangle_{m,\gamma}, \text{ where }j\in J_h.
\end{align}

The proofs of the next two propositions are essentially repetitions of the proofs of Propositions \ref{e18} and \ref{e21}. 

\begin{prop}\label{e22b}
Let $m>\frac{d+1}{2}$. For small enough positive constants  $\eps_1$, $M_G$,   and $M_0$ 
we have for $\eps\in (0,\eps_1]$
\begin{align}\label{e22c}
\begin{split}
&(a)\;\sqrt{\eps}\langle\Lambda\cG_2\rangle_{m+1,\gamma}\lesssim (\sqrt{\eps}\langle \Lambda v_1\rangle_{m+1,\gamma}+\eps^{\frac{3}{2}})[Q_1(E_m(v^s))+Q_2(M_G)]\\
&(b)\;\frac{1}{\sqrt{\eps}}\langle\Lambda^{\frac{1}{2}}\cG_2\rangle_{m,\gamma}\lesssim (\frac{1}{\sqrt{\eps}}\langle \Lambda^{\frac{1}{2}} v_1\rangle_{m,\gamma}+\eps)[Q_1(E_m(v^s))+Q_2(M_G)].
\end{split}
\end{align}
\end{prop}

\begin{prop}\label{e23d}
Let $m>\frac{d+1}{2}$ and let $\phi_j$ and $\psi_j$, $j\in J_h$, be singular symbols related as in Notation \ref{f4b}. For small enough positive constants  $\eps_1$, $M_G$,   and $M_0$
we have for $\eps\in (0,\eps_1]$
\begin{align}\label{e23e}
\begin{split}
&(a)\;\sqrt{\eps}\langle\phi_j\Lambda^{\frac{3}{2}} \cG_2\rangle_{m+1,\gamma}\lesssim \left(\sqrt{\eps}\langle \psi_j\Lambda^{\frac{3}{2}} v_1\rangle_{m+1,\gamma}+\sqrt{\eps}\langle \Lambda v_1\rangle_{m+1,\gamma}+\eps\right)[Q_1(E_m(v^s))+Q_2(M_G)]\\
&(b)\;\frac{1}{\sqrt{\eps}}\langle\phi_j\Lambda\cG_2\rangle_{m,\gamma}\lesssim  \left(\frac{1}{\sqrt{\eps}}\langle \psi_j\Lambda v_1\rangle_{m,\gamma}+\frac{1}{\sqrt{\eps}}\langle \Lambda^{\frac{1}{2}} v_1\rangle_{m,\gamma}+\sqrt{\eps}\right)[Q_1(E_m(v^s))+Q_2(M_G)]\\
\end{split}
\end{align}
\end{prop}

The next proposition is the last step in the proof of Proposition \ref{c5}.
\begin{prop}\label{h1}
Suppose $m>3d+4+\frac{d+1}{2}$.  There exist positive constants $\eps_0$,   $\gamma_0$
such that for $\eps\in (0,\eps_0]$ and each $T$ with $0<T\leq T^*_\eps$, 
\begin{align}\label{h2}
\begin{split}
&E_{m,\gamma}(v_2)\lesssim \gamma^{-1}E_{m,\gamma}(v)Q(E_{m,T}(v^s))+(\gamma^{-\frac{1}{2}}+\sqrt{\eps})Q(E_{m,T}(v^s))\text{ for }\gamma\geq \gamma_0.
\end{split}
\end{align}
 
\end{prop}

\begin{proof}

\textbf{1. }For all terms in the definition of $E_{m,\gamma}(v_2)$,
domination by the right side of \eqref{h2} follows directly from the linearized estimates of section \ref{b1a} and the propositions  of sections \ref{if}, \ref{bf}, \ref{ic}, and \ref{bf2}. 
The ``right side" of every  interior or boundary commutator proposition needed to estimate these terms is dominated by $Q(E_{m,T}(v^s))\cdot E_{m,\gamma}(v)$,\footnote{Although the boundary commutators for the $v_2$ problem \eqref{c2} are zero, the boundary commutators for \eqref{c3} are not; such commutators are needed to control terms in the third row of \eqref{c0a} when $j=1$ or $j=2$.}
 the right side of every interior forcing proposition  is dominated by $ Q^o(E_{m,T}(v^s))$, and the boundary forcing terms involving $\cG$ \eqref{e22bb} are dominated by
\begin{align}
Q^o(E_{m,T}(v^s))+(M_G+\sqrt{\eps})Q(E_{m,T}(v^s))\leq Q(E_{m,T}(v^s))\text{ since }M_G\leq 1.
 \end{align}
  The boundary forcing terms involving $\cG_2$ have been estimated in this section. Those that do not involve $\phi_{j,D}$, $j\in J_h$  are dominated by 
  \begin{align}
  (E_{m,\gamma}(v_1)+\sqrt{\eps})[Q(E_{m,T}(v^s))+Q(M_G)]\leq (E_{m,\gamma}(v_1)+\sqrt{\eps})Q(E_{m,T}(v^s)).
  \end{align}
  Using Proposition \ref{g1},  we see that the boundary forcing terms involving $\phi_{j,D}$ on the left in Propositions \ref{e21} and \ref{e23d} are dominated by 
  \begin{align}
  \begin{split}
  &\frac{1}{\sqrt{\gamma}}\{Q_1(E_{m,T}(v^s))\cdot E_{m,\gamma}(v_1)+Q^o_2(E_{m,T}(v^s))+(M_G+\sqrt{\eps})Q_3(E_{m,T}(v^s))\}[Q(E_{m,T}(v^s))+Q(M_G)]\leq\\
&\qquad \qquad \frac{1}{\sqrt{\gamma}}[E_{m,\gamma}(v_1)Q(E_{m,T}(v^s))+Q(E_{m,T}(v^s))].
\end{split}
\end{align}
  The coefficients on the right in \eqref{h2} are explained below.
  
 \textbf{2. }Consider for example the second term in the first line of \eqref{c0a} when $j=2$.  By \eqref{b7} we have
\begin{align}\label{h3}
\begin{split}
&\left|\begin{pmatrix}\Lambda v_2\\D_{x_2}v_2\end{pmatrix}\right|_{\infty,m,\gamma}\lesssim \gamma^{-1}|\cF_2|_{0,m+1,\gamma}+\gamma^{-1}|\cF_{2,com}|_{0,m+1,\gamma}+\\
&\qquad \qquad (\langle\Lambda \cG_2\rangle_{m,\gamma}+\langle\Lambda^{\frac{1}{2}} \cG_{2}\rangle_{m+1,\gamma})+\gamma^{-\frac{1}{2}}\sum_{j\in J_h}\langle\phi_j\Lambda \cG_2\rangle_{m+1,\gamma},
\end{split}
\end{align}
 The discussion in step \textbf{1} shows that the right side of \eqref{h3} is dominated by
 \begin{align}\label{h4}
 \begin{split}
 &\gamma^{-1}Q(E_{m,T}(v^s))+\gamma^{-1}E_{m,\gamma}(v)Q(E_{m,T}(v^s))+(E_{m,\gamma}(v_1)+\sqrt{\eps})Q(E_{m,T}(v^s))+\\
&\qquad\qquad \gamma^{-1}[E_{m,\gamma}(v_1)Q(E_{m,T}(v^s))+Q(E_{m,T}(v^s))].
 \end{split}
 \end{align}
 By Corollary \ref{g3} $E_{m,\gamma}(v_1)\leq \gamma^{-\frac{1}{2}}Q(E_{m,T}(v^s))$, so we see that the right side of \eqref{h4} is dominated by the right side of \eqref{h2}.
 

\end{proof}

\begin{proof}[Proof of Proposition \ref{c5}]
The estimate of Proposition \ref{c5}  follows immediately from Corollary \ref{g3} and Proposition \ref{h1}.

\end{proof}

\section{Local existence and continuation for the singular problems with $\eps$ fixed.}\label{local}

\emph{\quad}For fixed $\eps\in (0,1]$  we prove here  existence and continuation  theorems for solutions $v^\eps_1$, $v^\eps_2$, $u^\eps$ to the triple of coupled systems \eqref{a7}, \eqref{a8}, \eqref{a9}.
 We also establish  the relation $v^\eps=\nabla_\eps u^\eps$ on $\Omega_{T_\eps}$.  These results are used in the proof of Proposition \ref{mainprop},  which establishes a uniform time of existence with respect to $\eps$ for solutions of the singular nonlinear problems.

 Recall that we describe these systems as singular both because of the factors of $\frac{1}{\eps}$ that appear and because of the fact that $\partial_{x'}$ and $\partial_\theta$ derivatives occur in the combination $\partial_{x'}+\beta\frac{\partial_\theta}{\eps}$.  Even when we fix $\eps$, the second singular feature is still present.  Thus, we still need to use the singular calculus to prove estimates.   


\begin{defn}
For $m\geq 0$ and $T\in \mathbb{R}$ set $\cN^m(\Omega_T)=\{w:|w|^*_{m,T}<\infty\}$, where 
\begin{align}\label{defnN}
|w|^*_{m,T}=\left|\begin{pmatrix}\Lambda_1^{\frac{3}{2}}w\\D_{x_2}\Lambda_1^{\frac{1}{2}}w\end{pmatrix}\right|_{m,T}+\left <  \begin{pmatrix}\Lambda_1 w\\D_{x_2} w\end{pmatrix} \right>_{m,T}.
\end{align}
The norms $|w|^*_{m,\gamma}$, $|w|^*_{m,\gamma,T}$, and $|w|^*_m$ are defined by substituting $(m,\gamma)$, etc. for $(m,T)$ in \eqref{defnN}; recall Definition \ref{normal}.

\end{defn}

\begin{rem}\label{N}
1.  Let $m>\frac{d+1}{2}+1$. Observe that if $(v^\eps,u^\eps)$ is a solution of \eqref{a7}, \eqref{a8}, \eqref{a9} such that $v^\eps=\nabla_\eps u^\eps$ on $\Omega_T$ and  $E_{m,T}(v^\eps)<\infty$, then $v^\eps\in\cN^{m+1}(\Omega_T)$ and we have
\begin{align}
|v^\eps|^*_{m+1,T}\leq C_\eps E_{m,T}(v^\eps).
\end{align}
Indeed, the control of tangential derivatives is immediate, and since the boundary is noncharacteristic, the equation can be used in the standard way for $\eps$ fixed to control $\partial_{x_2}$ derivatives.  (Note:  We need estimates as in section \ref{nonlinear} involving higher $\partial_{x_2}$ derivatives for this.) 
On the other hand for any $(v^\eps,u^\eps)\in\cN^{m+1}(\Omega_T)$ with $v^\eps=\nabla_\eps u^\eps$ we have
\begin{align}
E_{m,T}(v^\eps)\leq C_\eps |v^\eps|^*_{m+1,T}.
\end{align}

2.  Let $\delta>0$ and  $\eps>0$ be fixed.   Suppose $(v^\eps,u^\eps)\in\cN^{m+1}(\Omega_T)$ for some $T>0$ with $\nabla_\eps u^\eps=v^\eps$  and vanishes in $t<0$.   For sufficiently small $T_\eps$ with $0<T_\eps< T$ we have
$E_{m,T_\eps}(v)<\delta$.   This is easily seen by first approximating $v$ in the $\cN^{m+1}(\Omega_T)$ norm by a $C^\infty$ function with compact support in $\overline{\Omega_T}$.

\end{rem}

\subsection{Local existence.}

\begin{prop}[Local existence for fixed $\eps$.]\label{localex}
Let $m>3d+4+\frac{d+1}{2}$ and 
consider the problems \eqref{a7}, \eqref{a8}, \eqref{a9} for $\eps$ fixed, where $G(x',\theta)\in H^{m+3}(b\Omega)$ (see section 4) and vanishes in $t<0$.  There exist $T>0$ and  unique solutions $v=(v_1,v_2)\in\cN^{m+1}(\Omega_T)$ to \eqref{a7}, \eqref{a8} and $u\in\cN^{m+1}(\Omega_T)$ to \eqref{a9} vanishing in $t<0$. Moreover, we have $v=\nabla_\eps u$ on $\Omega_T$.

\end{prop}

\begin{proof}

\textbf{1.  A priori estimates and existence for the linearized problems. }Consider the linearized problems for $\eps$ fixed
\begin{align}\label{k-1}
\begin{split}
&\partial_{t,\eps}^2 v_1+\sum_{|\alpha|=2} A_\alpha(w)\partial_{x,\eps}^\alpha v_1=\mathcal{F}_1\text{ on }\Omega\\
&\partial _{x_2} v_1-d_{v_1}H(w_1,h(w))\partial_{x_1,\eps}v_1=\mathcal{G}_1\text{ on }x_2=0
\end{split}
\end{align}
and
\begin{align}\label{k0}
\begin{split}
&\partial_{t,\eps}^2 v_2+\sum_{|\alpha|=2} A_\alpha(w)\partial_{x,\eps}^\alpha v_2=\mathcal{F}_2\text{ on }\Omega\\
&v_2=\mathcal{G}_2\text{ on }x_2=0.
\end{split}
\end{align}
Here the functions  $v_i$, $w$, $\mathcal{F}_i$, $\mathcal{G}_i$ all vanish in $t<0$.

Let $k\in\{0,1,\dots,m+1\}$.  For the problem \ref{k-1}  we have:  for $|w|^*_{m+1}$ sufficiently small,  there exist constants $\gamma_{m+1}(k)$, $C=C_{m+1}(k)$ (depending on $|w|^*_{m+1}$, $k$, and $\eps$) such that for $\gamma\geq \gamma_{m+1}(k)$

\begin{align}\label{k1}
\begin{split}
&\gamma \left|\begin{pmatrix}\Lambda^{\frac{3}{2}}v_1\\D_{x_2}\Lambda^{\frac{1}{2}}v_1\end{pmatrix}\right|^2_{k,\gamma}+\gamma\left <  \begin{pmatrix}\Lambda v_1\\D_{x_2} v_1\end{pmatrix} \right>^2_{k,\gamma}+\sum_{J_h\cup J_e }\left <  \phi_j \begin{pmatrix}\Lambda^{\frac{3}{2}} v_1\\D_{x_2}\Lambda^{\frac{1}{2}} v_1\end{pmatrix} \right>^2_{k,\gamma}\leq \\
&\qquad\quad \frac{1}{\gamma}C\left(|\Lambda^{\frac{1}{2}}\mathcal{F}_1|^2_{k,\gamma}+\langle\Lambda \mathcal{G}_1\rangle^2_{k,\gamma}\right).
\end{split}
\end{align}

For the problem \ref{k0} we have:  for $|w|^*_{m+1}$ sufficiently small, there exist constants $\gamma_{m+1}(k)$, $C=C_{m+1}(k)$ (depending on $|w|^*_{m+1}$, $k$, and $\eps$) such that for $\gamma\geq \gamma_{m+1}(k)$ 
\begin{align}\label{k2}
\begin{split}
&\gamma \left|\begin{pmatrix}\Lambda^{\frac{3}{2}}v_2\\D_{x_2}\Lambda^{\frac{1}{2}}v_2\end{pmatrix}\right|^2_{k,\gamma}+\gamma\left <  \begin{pmatrix}\Lambda v_2\\D_{x_2} v_2\end{pmatrix} \right>^2_{k,\gamma}\leq \\
&\qquad\quad C\left(\frac{1}{\gamma}|\Lambda^{\frac{1}{2}}\mathcal{F}_2|^2_{k,\gamma}+ \gamma\langle\Lambda \mathcal{G}_2\rangle^2_{k,\gamma}+\sum_{j\in J_h}\langle\phi_j\Lambda^{\frac{3}{2}} \mathcal{G}_2\rangle^2_{k,\gamma}\right).
\end{split}
\end{align}

The estimates \eqref{k1} and \eqref{k2} are derived under the assumption of having sufficiently regular $v_i$.   When $k=0$ they follow directly from \eqref{b3} and \eqref{b6}.  Estimates of $|\cdot|_{0,k,\gamma}$ and $\langle\cdot\rangle_{k,\gamma}$ norms follow by applying the $k=0$ estimate to the tangentially differentiated problems $\partial^k\eqref{k-1}$, $\partial^k\eqref{k0}$.   As in section \ref{mainestimate} the singular norm estimates of section \ref{nontame} are needed here; in particular, Proposition \ref{f12}(b) is again used to estimate commutators.\footnote{As in section \ref{mainestimate}, interior norms are estimated by first doing a tangential estimate in $(x',\theta)$ for fixed $x_2$, and then taking an $L^2(x_2)$ norm.}  Estimates of $\left|\begin{pmatrix}D_{x_2}^j\left(\Lambda^{\frac{3}{2}}v_i\right)\\D_{x_2}^j\left(D_{x_2}\Lambda^{\frac{1}{2}}v_i\right)\end{pmatrix}\right|_{0,k-j,\gamma}$ for $j>0$ are proved as usual by induction on $j$ using the equation and the noncharacteristic boundary assumption.  Proposition \ref{normalest} is needed here.



We have the following existence theorem for the singular linearized problems with $\eps$ fixed:

\begin{theo}\label{exist}
Consider the singular linear problems \eqref{k-1}, \eqref{k0} for $\eps>0$ fixed and assume $d=2$.
Let $m>3d+4+\frac{d+1}{2}$ and suppose $|w|^*_{m+1}<\infty$.     Suppose the data $\cF_i$, $\cG_i$, $i=1,2$ vanish in   $t<0$ and are such that for $k\in\{0,1,\dots,m+1\}$ the right sides of \eqref{k1}, \eqref{k2} are finite.  
For $|w|^*_{m+1}$ sufficiently small,  there exist constants $\gamma_{m+1}(k)$, $C=C_{m+1}(k)$ (depending on $|w|^*_{m+1}$, $k$, and $\eps$) such that for $\gamma\geq \gamma_{m+1}(k)$
the problems \eqref{k-1}, \eqref{k0} admit unique solutions $v_i$, $i=1,2$ vanishing in $t<0$  satisfying the estimates \eqref{k1}, \eqref{k2}.

\end{theo}

The theorem follows from the a priori estimates by a classical duality argument using the properties (P6) and (P7)  described in section \ref{assumptions}.   We refer to \cite{S-T}, p. 279 or to \cite{CP}, Chapter 7 for this kind of argument.   

\begin{rem}\label{k2y}
An important consequence of Theorem \ref{exist} that we use in the proof of Proposition \ref{c5} and later in the error analysis 
 is that the linearized singular problems exhibit \emph{causality} \cite{CP};  solutions in $t<T$ are unaffected by changing the forcing terms and coefficients in $t>T$. 
\end{rem}

\textbf{2.  Iteration schemes. }For $T_0>0$ to be chosen,  as before we let  $\chi_0(t)\geq 0$  be a  $C^\infty$ function that  is equal to 1 on a neighborhood of $[0,1]$ and compactly supported in $(-1,2)$.    We use  iteration schemes similar to those in \cite{S-T}, initialized by $v^0=0$.  In $t<0$ we have $G=0$ and $v_k=0$, $u_k=0$ for all $k$.

\begin{align}\label{k0a}
\begin{split}
&\partial_{t,\eps}^2 v_1^{k+1}+\sum_{|\alpha|=2} A_\alpha(v^k)\partial_{x,\eps}^\alpha v_1^{k+1}=-\left[\sum_{|\alpha|=2,\alpha_1\geq 1}\partial_{x_1,\eps}(A_\alpha(v^k))\partial_{x_1,\eps}^{\alpha_1-1}\partial_{x_2}^{\alpha_2}v_1^k-\partial_{x_1,\eps}(A_{(0,2)}(v^k))\partial_{x_2}v_2^k\right]\\
&\partial _{x_2} v_1^{k+1}-d_{v_1}H(v^k_1,h(v^k))\partial_{x_1,\eps}v_1^{k+1}=d_gH(v^k_1,\eps^2G)\partial_{x_1,\eps}(\eps^2G)\text{ on }x_2=0.
\end{split}
\end{align}

\begin{align}\label{k0b}
\begin{split}
&\partial_{t,\eps}^2 v_2^{k+1}+\sum_{|\alpha|=2} A_\alpha(v^k)\partial_{x,\eps}^\alpha v_2^{k+1}=-\left[\sum_{|\alpha|=2,\alpha_1\geq 1}\partial_{x_2}(A_\alpha(v^k))\partial_{x_1,\eps}^{\alpha_1-1}\partial_{x_2}^{\alpha_2}v^k_1-\partial_{x_2}(A_{(0,2)}(v^k))\partial_{x_2}v^k_2\right]\\
&v_2^{k+1}=\chi_0(t)H(v_1^{k+1},\eps^2 G)\text{ on }x_2=0.
\end{split}
\end{align}

\begin{align}\label{k0c}
\begin{split}
&\partial_{t,\eps}^2 u^{k+1}+\sum_{|\alpha|=2} A_\alpha(v^k)\partial_{x,\eps}^\alpha u^{k+1}=0\\\
&\partial _{x_2} u^{k+1}-d_{v_1}H(v^k_1,h(v^k))\partial_{x_1,\eps}u^{k+1}=\left[H(v^k_1,\eps^2 G(x',\theta))-d_{v_1}H(v^k_1,\eps^2 G)v^k_1\right]\text{ on }x_2=0.
\end{split}
\end{align}

This is actually a family of  problems parametrized by $T$, where $0<T\leq T_0$. Each problem is solved on $\Omega$.
The function $v^k$ appearing in the coefficients of  \eqref{k0a}-\eqref{k0c} is really $v^k_T$ (with the $T$ suppressed), where $v^k_T$ is a Seeley extension of $v_k|_{\Omega_T}$ to $\Omega$ chosen as in Proposition \ref{c0e}.   Similarly, $G$ here denotes a Seeley extension of $G|_{b\Omega_{T_0}}$ to $H^{m+3}(b\Omega)$.  We can suppose $T_0$ is small enough so that $\langle G\rangle_{m+3,T_0}\leq 1$.

The function $v^{k+1}_1$ appearing in the boundary condition of \eqref{k0b} is taken to be the \emph{same} as that in \eqref{k0a}, not a Seeley extension of $v^{k+1}_1|_{\Omega_T}$, for the reasons explained in Remark \ref{c4a}.

\textbf{3. High norm boundedness. }We make the following induction assumption:

$\bullet$ there exists $T>0$ independent of $k$ such that for $j=0,\dots,k$ we have
$|v_j|^*_{m+1,T}<\delta$
for $\delta$ to be chosen.\footnote{This choice is made precise later in this step; for now we claim that the estimates in this step are valid for $\delta$ and $T_0$ small enough.}

Using the estimates \eqref{k1}, \eqref{k2} together with Proposition \ref{c0e} (which relates $|\cdot|^*_{m,\gamma}$ and $|\cdot|^*_{m,T}$ norms) and the estimates of nonlinear functions in section \ref{nontame}, we can complete the induction step by an argument modeled on  \cite{S-T}, pages 286-289.\footnote{We mainly use our Corollary \ref{f2},  the singular Rauch-type lemma Proposition \ref{f7}, and Proposition \ref{normalest}.    The use of Proposition \ref{f7} permits a simplification of the argument in \cite{S-T}; for example, dyadic decompositions are not needed here. }

We denote the interior and boundary forcing terms in \eqref{k0a}, \eqref{k0b} by $\cF_1$, $\cF_2$ and $\cG_1$, $\cG_2$, respectively. By 
Proposition \ref{normalest} and Corollary \ref{f2} we have
\begin{align}\label{ka0}
|\Lambda^{\frac{1}{2}}\cF_i|_{m+1,\gamma}\lesssim C(\delta)\text{ and }\langle\Lambda \cG_1\rangle_{m+1,\gamma}\leq C(\delta,\langle\Lambda^2_1G\rangle_{m+1,T_0}).
\end{align}
Thus, estimate \eqref{k1} applied to the problem \eqref{k0a},  yields for $\gamma\geq \gamma_{m+1}(m+1):={\gamma}^*_{m+1}$ and $j\in J_h\cup J_e$:
\begin{align}\label{ka1}
\gamma |v^{k+1}_1|^*_{m+1,\gamma}+\langle\phi_j\Lambda^{\frac{3}{2}}v^{k+1}_1\rangle_{m+1,\gamma}\leq K_1(\delta, \langle\Lambda^2_1G\rangle_{m+1,T_0}), \text{ where }\langle\Lambda^2_1G\rangle_{m+1,T_0}\lesssim \langle G\rangle_{m+3,T_0}.
\end{align} 
The estimates giving \eqref{ka0} show that  $K_1$ can be taken to be a continuous function such that $K_1(0,0)=0.$  

To estimate $v^{k+1}_2$ with \eqref{k2} we write
\begin{align}\label{ka2}
\chi_0(t)H(v_1^{k+1},\eps^2 G)=h(\chi^*v^{k+1}_1, \eps^2\chi^* G)(\chi^*v^{k+1}_1, \eps^2\chi^* G),
\end{align}
where $\chi^*(t)$ is a smooth cutoff, compactly supported in $(-1,2)$, and  equal to one on the support of $\chi_0$. Using \eqref{ka1} and 
\begin{align}
\langle\Lambda_1v^{k+1}_1\rangle_{m+1,T=2} \lesssim e^{\gamma^*_{m+1}2}\langle \Lambda v^{k+1}_1\rangle_{m+1,\gamma^*_{m+1}},
\end{align}
we obtain
\begin{align}\label{kaa2}
\langle\Lambda_1(\chi^* v^{k+1}_1)\rangle_{m+1}\lesssim e^{\gamma^*_{m+1}2}K_1\text{ and } \langle \Lambda(\chi^* v^{k+1}_1)\rangle_{m+1,\gamma}\lesssim \frac{1}{\gamma}K_1.
\end{align}
Applying  Corollary \ref{f2} to \eqref{ka2}, this gives
\begin{align}\label{ka3}
\langle\Lambda(\chi_0 H(v^{k+1}_1,\eps^2 G))\rangle_{m+1,\gamma}\lesssim C(K_1,\gamma^*_{m+1})(\gamma^{-1}+\langle\Lambda G\rangle_{m+1,T_0}).
\end{align}

Similarly, applying Proposition \ref{f7} and using the simplified notation explained in Notation \ref{e10a}, we obtain with $z=(\chi^*v^{k+1}_1, \eps^2\chi^* G)$:
\begin{align}\label{ka4}
\begin{split}
&\langle\phi_j\Lambda^{\frac{3}{2}}(\chi_0 H(v^{k+1}_1,\eps^2 G))\rangle_{m+1,\gamma}\lesssim \\
&\quad\left(\langle\Lambda z\rangle_{m+1,\gamma}\langle\Lambda_1 z\rangle_{m+1}h(\langle z\rangle_{m+1})+ \langle\psi\Lambda z\rangle_{m+1,\gamma}\langle\Lambda^{\frac{1}{2}}_1z\rangle_{m+1}h(\langle z\rangle_{m+1})\right)+\langle\psi\Lambda^{\frac{3}{2}}z\rangle_{m+1,\gamma}h(\langle z\rangle_{m+1})\lesssim\\
&\quad\quad (\gamma^{-1}+\langle\Lambda G\rangle_{m+1,T_0}\rangle)C(K_1,\gamma^*_{m+1})+ \left(\gamma^{-1/2}K_1+\langle\Lambda^{\frac{3}{2}}G\rangle_{m+1,T_0}\right).
\end{split}
\end{align}

In view of \eqref{ka0}, \eqref{ka3}, and \eqref{ka4}  the estimate \eqref{k2} applied to the problem \eqref{k0b} gives
\begin{align}\label{ka5}
|v^{k+1}_2|^*_{m+1,\gamma}\leq \gamma^{-1}K_2+K_3\langle\Lambda G\rangle_{m+1,T_0}+\gamma^{-\frac{1}{2}}K_4.
\end{align}

\textbf{Choice of $\delta$ and $T_0$.} We first fix $\delta$ sufficiently small so that the estimates \eqref{k1}, \eqref{k2} apply to the problems \eqref{k0a}, \eqref{k0b} when
$|v^k|^*_{m+1,T}<\delta$.   We also need $\delta$ and $T_0$ small enough so that $|v^k|_{m+1}$, $\langle \chi^* v^{k+1}_1,\chi^*\eps^2G\rangle_{m+1}$ and $\langle v^k_1, \eps^2 G\rangle_{m+1}$ lie in the domain of convergence of the analytic functions (like $h$ in \eqref{ka4}, for example) that appear on the right in the estimates of section \ref{nonlinear} that we use.  To see that $\langle\chi^* v^{k+1}_1\rangle_{m+1}$ can be made small by taking $\delta$ and $T_0$ small, one uses the estimate \eqref{kaa2} and the fact that $K_1(0,0)=0$.   A further possible reduction of $T_0$ occurs at the end of this  step.

Taking $\gamma=\frac{1}{T}$ in \eqref{ka5} and using Proposition \ref{c0e}, we find
\begin{align}\label{k6a}
|v^{k+1}_2|^*_{m+1,T}\lesssim   TK_2+K_3\langle\Lambda G\rangle_{m+1,T_0}+ T^{\frac{1}{2}}K_4.
\end{align}
After reducing $T_0$ if necessary, the right side of \eqref{k6a} will be $<\delta$ for small enough $T>0$. 

       Uniform boundedness of the iterates $|u^k|^*_{m+1,T}$  with respect to $k$ now follows by applying the estimate \eqref{k1} to the problem \eqref{k0c}.

\textbf{4. Low norm contraction. }  Having the uniform boundedness of the iterates $v^k$ and $u^k$ in the norm $|\cdot|^*_{m+1,T}$ allows us to repeat the 
low norm contraction argument of \cite{S-T}, p. 290, to show that the sequence $v^k$ is Cauchy in the $\left|\begin{matrix}\Lambda v^k\\D_{x_2}v^k\end{matrix}\right|_{0,T}$ norm for a possibly smaller $T>0$.   The same applies to the $u^k$.   This argument  uses the estimates 
of Proposition \ref{bb1} and, of course, some of the nonlinear estimates of section \ref{nonlinear}.
 A standard interpolation argument implies $v^k\to v$ and $u^k\to u$ in $\cN^m(\Omega_T)$, where $v$ and $u$ both lie in 
$\cN^{m+1}(\Omega_T)$ and satisfy  \eqref{a7}, \eqref{a8}, \eqref{a9} on $\Omega_T$.

\textbf{5. Show $v^\eps=\nabla_\eps u^\eps$ on $\Omega_T$. }By a computation similar to the one  on p. 291 of \cite{S-T}, we see that the differences $v_1-\partial_{x_1,\eps}u$ and $v_2-\partial_{x_2}u$ satisfy a coupled system of singular, linear equations on $\Omega_T$ with boundary and interior forcing terms that are identically zero.  One can therefore apply the estimates  of Proposition \ref{bb1} 
to see that these differences are both zero on $\Omega_T$.
\end{proof}

\subsection{Continuation. }

We will prove the continuation theorem using estimates of the following form that are tame with respect to the norm $|\cdot|^*_{k,\gamma}$.

\begin{prop}\label{tame2}
Let $m_0>3d+4+\frac{d+1}{2}$ and suppose $k\in\{0,1,\dots,m_0+\left[\frac{m_0+1}{2}\right]\}$, where $[r]$ denotes the largest integer $\leq r$.  
Suppose that the function $w$ occurring in the coefficients of the problem \ref{k-1} has compact $t-$support contained in $[0,2)$ and let $\chi^*(t)$ be a $C^\infty$ function equal to one on $\mathrm{supp}\; w$ with support in $(-1,2)$. 

For the problem \ref{k-1}  we have:   there exist $\gamma_0>0$ and an increasing function $K:\mathbb{R}^+\to\mathbb{R}^+$ such that for $\gamma\geq \gamma_0$, 

\begin{align}\label{k3}
\begin{split}
&\gamma \left|\begin{pmatrix}\Lambda^{\frac{3}{2}}v_1\\D_{x_2}\Lambda^{\frac{1}{2}}v_1\end{pmatrix}\right|^2_{k,\gamma}+\gamma\left <  \begin{pmatrix}\Lambda v_1\\D_{x_2} v_1\end{pmatrix} \right>^2_{k,\gamma}+\sum_{J_h\cup J_e }\left <  \phi_j \begin{pmatrix}\Lambda^{\frac{3}{2}} v_1\\D_{x_2}\Lambda^{\frac{1}{2}} v_1\end{pmatrix} \right>^2_{k,\gamma}\leq \\
&\quad \frac{1}{\gamma}K^2(|w|^*_{m_0})\left(\left|\begin{pmatrix}\Lambda^{\frac{3}{2}}v_1\\D_{x_2}\Lambda^{\frac{1}{2}}v_1\end{pmatrix}\right|^2_{k,\gamma}+(|\chi^* v_1|^*_{m_0})^2\left|\begin{pmatrix}\Lambda^{\frac{3}{2}}w\\D_{x_2}\Lambda^{\frac{1}{2}}w\end{pmatrix}\right|^2_{k,\gamma}+|\Lambda^{\frac{1}{2}}\mathcal{F}_1|^2_{k,\gamma}\right)+\\
&\quad \quad \frac{1}{\gamma}K^2(|w|^*_{m_0})\left(\left\langle\begin{pmatrix}\Lambda v_1\\D_{x_2} v_1\end{pmatrix}\right\rangle^2_{k,\gamma}+(|\chi^* v_1|^*_{m_0})^2\left\langle\begin{pmatrix}\Lambda w\\D_{x_2}w\end{pmatrix}\right\rangle^2_{k,\gamma}+\langle\Lambda \mathcal{G}_1\rangle^2_{k,\gamma}\right).
\end{split}
\end{align}
For the problem \ref{k0} we have:  there exist $\gamma_0>0$ and an increasing function $K:\mathbb{R}^+\to\mathbb{R}^+$ such that for $\gamma\geq \gamma_0$, 
\begin{align}\label{k4}
\begin{split}
&\gamma \left|\begin{pmatrix}\Lambda^{\frac{3}{2}}v_2\\D_{x_2}\Lambda^{\frac{1}{2}}v_2\end{pmatrix}\right|^2_{k,\gamma}+\gamma\left <  \begin{pmatrix}\Lambda v_2\\D_{x_2} v_2\end{pmatrix} \right>^2_{k,\gamma}\leq \\
&\quad \frac{1}{\gamma}K^2(|w|^*_{m_0})\left(\left|\begin{pmatrix}\Lambda^{\frac{3}{2}}v_2\\D_{x_2}\Lambda^{\frac{1}{2}}v_2\end{pmatrix}\right|^2_{k,\gamma}+(|\chi^* v_2|^*_{m_0})^2\left|\begin{pmatrix}\Lambda^{\frac{3}{2}}w\\D_{x_2}\Lambda^{\frac{1}{2}}w\end{pmatrix}\right|^2_{k,\gamma}+|\Lambda^{\frac{1}{2}}\mathcal{F}_2|^2_{k,\gamma}\right)+\\
&\quad \quad K^2(|w|^*_{m_0})\left(\gamma\langle\Lambda \mathcal{G}_2\rangle^2_{k,\gamma}+\sum_{j\in J_h}\langle\phi_j\Lambda^{\frac{3}{2}} \mathcal{G}_2\rangle^2_{k,\gamma}\right).
\end{split}
\end{align}
In these estimates  $|w|_{L^\infty(\Omega)}$ must be small enough, $\gamma_0$ depends on
\begin{align}\label{k4z}
\left|\frac{w}{\eps}\right|_{C^{0,n}}+\left|\frac{w}{\eps}\right|_{CH^{s_0}}+|\partial_{x_2}w|_{L^\infty(\Omega)},\text{ where }n\geq 3d+4, \;s_0>\frac{d+1}{2}+2,
\end{align}
and $K(|w|^*_{m_0})$  depends also on $\eps$ and $k$.

\end{prop}

\begin{proof}
\textbf{1. }One first proves the same estimates, but where the cutoffs $\chi^*$ are absent in the factors $|\chi^*v_j|^*_{m_0}$.

\textbf{2. }The proof is similar to  that of estimates \ref{k1} and \ref{k2}.  Again, the starting point is the $k=0$ estimate 
given by \eqref{b3} and \eqref{b6}.  Estimates of $|\cdot|_{0,k,\gamma}$ and $\langle\cdot\rangle_{k,\gamma}$ norms follow by applying the $k=0$ estimate to the tangentially differentiated problems $\partial^k\eqref{k-1}$, $\partial^k\eqref{k0}$.   Estimates of commutators are now done using the tame estimate 
of Proposition \ref{j7}(b). 

\textbf{3. }As before, estimates of $\left|\begin{pmatrix}D_{x_2}^j\left(\Lambda^{\frac{3}{2}}v_i\right)\\D_{x_2}^j\left(D_{x_2}\Lambda^{\frac{1}{2}}v_i\right)\end{pmatrix}\right|_{0,k-j,\gamma}$ for $j>0$ are proved by induction on $j$ using the equation and the noncharacteristic boundary assumption. 
We set $U=\begin{pmatrix}e^{\gamma t}\Lambda_De^{-\gamma t} v_i\\D_{x_2}v_i\end{pmatrix}$ and    estimate 
$|\Lambda^{\frac{1}{2}}D_{x_2}^jU|_{0,k-j,\gamma}$ using the first order singular equation \eqref{b10d}.   For $j$ such that $j\leq k$, $1\leq j\leq \left[\frac{m_0+1}{2}\right]$   (range I),  we
apply the tame estimate of Proposition \ref{j7}(b), while for $j$ such that $j\leq k$, $\left[\frac{m_0+1}{2}\right]\leq j\leq m_0+\left[\frac{m_0+1}{2}\right]$ (range II), we apply Proposition \ref{f12}(b).

For example, in the estimate of $|\Lambda^{\frac{1}{2}}D_{x_2}^jU|_{0,k-j,\gamma}$ we must estimate terms of the form
\begin{align}
|\Lambda^{\frac{1}{2}}\left(a(w)d_\eps D_{x_2}^{j-1}U\right)|_{0,k-j,\gamma}.
\end{align}
Since $\eps$ is fixed, we treat $d_\eps$ like a nonsingular tangential derivative.  Thus, we are led to consider a sum of terms of form
\begin{align}\label{k4y}
|\Lambda^{\frac{1}{2}}\left((\partial^{k_1}z)(\partial^{k_2} D_{x_2}^{j-1}U)\right)|_{0,0,\gamma},
\end{align}
where, with $a(w)=a(0)+b(w)w$,
\begin{align}
z:=b(w)w,\;k_1+k_2=k-(j-1), k_2\geq 1.
\end{align}
 We first do a tangential estimate for $x_2$ fixed.   For $j$ in range I we apply  Proposition \ref{j7}(b) with ``$m"=k-(j-1)$ and ``$m_0"=m_0-(j-1)$; the factors involving $z$ are estimated using Remark \ref{j9}.     For $j$ in range II we apply Proposition \ref{f12}(b) with $``u"=D_{x_2}^{j-1}U$, $``v"=z$, $``s"=k-k_2-(j-1)$, $``t"=m_0-k_1$; factors involving $z$ are estimated using Proposition \ref{normalest}.  In both cases  the remaining $L^2(x_2)$ norm
is easily controlled by arguments like those in section \ref{mainestimate}

\textbf{4. }The cutoff functions $\chi^*$ appearing in the estimates \eqref{k3}, \eqref{k4}
 can be inserted because the factors $|v_j|^*_{m_0}$ only arise from terms like \eqref{k4y} in which ``free" factors of $\partial^lw$ appear for some $l$.  The support assumption on $w$ thus permits the insertion of $\chi^*$.
\end{proof}


\begin{prop}[Continuation of solutions for fixed $\eps$] \label{continuation}
(a) Let $m>3d+4+\frac{d+1}{2}$.   For $\eps$ fixed, suppose we have a solution $(v^\eps,u^\eps)\in \cN^{m+1}(\Omega_{T_1})$  of \eqref{a7}, \eqref{a8}, \eqref{a9} with $v^\eps=\nabla_\eps u^\eps$  for some $0<T_1<T_0$, where $v^\eps=0$ in $t<0$ and is such that $E_{m,T_1}(v^\eps)\leq \frac{M_0}{N}$ 
for $M_0$ as in \eqref{c0h}.   For $N$ large enough there exists $T_2>T_1$ and an extension of $(v^\eps,u^\eps)$ to a solution on $\Omega_{T_2}$ with $v^\eps\in \cN^{m+1}(\Omega_{T_2})$ and $v^\eps=\nabla_\eps u^\eps$ on $\Omega_{T_2}$, where $T_2$ depends on $\eps$ and $M_0$. \footnote{The size of $N$ depends on the norm of the Seeley extension operator used in step \textbf{2} of the proof.} 

(b)  There exists $T_3$ with $T_1<T_3\leq T_2$ such that $E_{m,T_3}(v^\eps)<M_0$.   

\end{prop}

\begin{proof}

As usual, we often suppress superscripts $\eps$ below.   We note that part (b) follows easily from part (a) by taking $T_3$ close enough to $T_1$.

\textbf{1. Translation.} By downward translation in time we reduce to the case where $T_1=0$ and where the given solution $v=(v_1,v_2)\in \cN^{m+1}(\Omega_{0})$  vanishes in $t<-T_1$ and satisfies $E_{m,0}(v)\leq \frac{M_0}{N}$.

\textbf{2. Iteration scheme.}       Setting $g=\eps^2 G$ (which now vanishes in $t<-T_1$), let us write the translates of \eqref{k0a}, \eqref{k0b}, and \eqref{k0c} as 
\begin{align}\label{k5}
\begin{split}
&\cL(v^k)v^{k+1}_1=\cF_1(v^k,\nabla_\eps v^k)\text{ on }\Omega\\
&\cB_1(v^k)v^{k+1}_1=\cG_1(v^k_1,g)\text{ on }x_2=0,\\
&v^{k_1+1}=v_1\text{ in }t<0.
\end{split}
\end{align}

\begin{align}\label{k6}
\begin{split}
&\cL(v^k)v^{k+1}_2=\cF_2(v^k,\nabla_\eps v^k)\text{ on }\Omega\\
&v^{k+1}_2=\cG_2(v^{k+1}_1,g)\text{ on }x_2=0,\\
&v^{k+1}_2=v_2\text{ in }t<0.
\end{split}
\end{align}

\begin{align}\label{k6aa}
\begin{split}
&\cL(v^k)u^{k+1}=0\text{ on }\Omega\\
&\cB_1(v^k)u^{k+1}=\cG(v^k_1,g)\text{ on }x_2=0,\\
&u^{k+1}=u\text{ in }t<0.
\end{split}
\end{align}
We focus now on the continuation of $v$. The treatment of $u$ is similar.

Let $(v^e,u^e)\in \cN^{m+1}(\Omega)$ be an extension of $(v,u)|_{\Omega_0}$ to $\Omega$ supported in $[-T_1,2-T_1)$ and such that $v^e=\nabla_\eps u^e$ and $E_{m}(v^e)<\frac{M_0}{2}$.   If we initialize the schemes \eqref{k5}, \eqref{k6} with $v^0=v^e$, the  iterates may be written
 $v^k=v^e+z^k$, where the $z^k$ satisfy
 \begin{align}\label{k7}
 \begin{split}
&\cL(v^e+z^k)z^{k+1}_1=\cF_1(v^e+z^k,\nabla_\eps (v^e+z^k))-\cL(v^e+z^k)v^e_1:=\F_1.\\
&\cB_1(v^e+z^k)z^{k+1}_1=\cG_1(v^e_1+z^k_1,g)-\cB_1(v^e+z^k)v^e_1:=\G_1\text{ on }x_2=0,\\
&z^{k+1}_1=0\text{ in }t<0.
\end{split}
\end{align}

\begin{align}\label{k8}
\begin{split}
&\cL(v^e+z^k)z^{k+1}_2=\cF_2(v^e+z^k,\nabla_\eps (v^e+z^k))-\cL(v^e+z^k)v^e_2:=\F_2\\
&z^{k+1}_2=\cG_2(v^e_1+z^{k+1}_1,g)-v^e_2:=\G_2\text{ on }x_2=0\\
&z^{k+1}_2=0\text{ in }t<0.
\end{split}
\end{align}

\textbf{3. Lower regularity continuation.} In the forcing terms on the right sides of \eqref{k7} and \eqref{k8} derivatives of $v^e$ occur, so we first obtain a lower regularity continuation.  
Just as in the proof of Proposition \ref{localex}, we introduce Seeley extensions $z^k_T$ with support in $[0,2-T_1)$ and regard \eqref{k7}, \eqref{k8} as a family of problems on $\Omega$ parametrized by $T$, where now $0<T\leq T_0-T_1$; thus, every $z^k$ (as opposed to $z^{k+1}$) inside a coefficient or forcing term is really $z^k_T$. 
We make the induction assumption:

$\bullet$ there exists $T>0$ independent of $k$ such that for $j=0,\dots,k$ we have (dropping the subscript on $z^j$)
$|z^j|^*_{m,T}<\delta$ and thus $|z^j_T|^*_{m}<C_1\delta$ for a small enough $\delta$.\footnote{Here, as in the proof of Proposition \ref{localex}, we require $\delta$ and $M_0$ to be small enough so that the estimates \eqref{k1}, \eqref{k2} apply.}

From this assumption and the regularity of $v^e$ we obtain (arguing as in the earlier proof)
\begin{align}
|\Lambda^{\frac{1}{2}}\F_i|_{m,\gamma}<\infty \text{ and }\langle\Lambda\G_1\rangle_{m,\gamma}<\infty;
\end{align}
moreover, $\F_i$ and $\G_i$ vanish in $t<0$.

Applying the argument used to prove Proposition \ref{localex} to the scheme \eqref{k7}, \eqref{k8}, we  obtain a solution $v=v^e+z \in \cN^{m}(\Omega_{T_2'})$ to \eqref{a7}, \eqref{a8} for some $T_2'>0$. 
Henceforth, let $v^e$, $z^j$, $z^j_T$, and $z$ denote the upward translations by $T_1$ of the similarly denoted functions we have just defined.  
We now have a continuation $v:=v^e+z\in  \cN^{m}(\Omega_{T_2})$, where $T_2=T_1+T_2'>T_1$ and $v=0$ in $t<0$, and 
the functions $v^j=v^e+z^j$ satisfy
\begin{align}\label{k8a}
\begin{split}
&\cL(v^k)v^{k+1}_1=\cF_1(v^k,\nabla_\eps v^k)\text{ on }\Omega_{}\\
&\cB_1(v^k)v^{k+1}_1=\cG_1(v^k_1,g)\text{ on }x_2=0,\\
&v^{k_1+1}=v_1\text{ in }t<T_1.
\end{split}
\end{align}

\begin{align}\label{k8b}
\begin{split}
&\cL(v^k)v^{k+1}_2=\cF_2(v^k,\nabla_\eps v^k)\text{ on }\Omega_{}\\
&v^{k+1}_2=\cG_2(v^{k+1}_1,g)\text{ on }x_2=0,\\
&v^{k+1}_2=v_2\text{ in }t<T_1,
\end{split}
\end{align}
where as before every $v^k$ (as opposed to $v^{k+1}$) appearing inside  a coefficient or forcing term is really $v^k_T$.    Recall that $v_1$ and $v_2$ vanish in $t<0$.

In the rest of the continuation argument the norm $|\cdot|^*_{m,T_2}$ will have the role usually played by a Lipschitz-type norm in such arguments.

\begin{rem}\label{k9}
 
 1)  From our construction there exists a constant $R>0$ such that $v^k_{T_2}=v^e+z^k_{T_2}$ has support in a fixed compact subset of $[0,2)$ and satisfies
 \begin{align}\label{k10}
 |v^k_{T_2}|^*_{m,T_2}\leq R \text{ and }|v^k_{T_2}|^*_{m}\leq C R\text{ for all }k.
 \end{align}
 
  2)  Since $E_m(v^e)\leq \frac{M_0}{2}$, 
  and $|z^k_{T_2}|^*_{m}< C\delta$ for all $k$,   we have in particular that 
 \begin{align}\label{k11}
|v^k_{T_2}|_{\infty,m}\leq \frac{M_0}{2}+C\delta\text{ for all }k,
\end{align}
a fact that will be used in the next step.

3) Let $\chi^*(t)$ as usual denote a cutoff  support in $(-1,2)$ and equal to one on the compact subset of $[0,2)$ in which the $v^k_{T_2}$ have support, and let $v^{k+1}_1$ be the solution of \eqref{k8a} (not a Seeley extension). By the argument that gave \eqref{kaa2}, but  applied in the construction of $z^{k+1}_1$ satisfying   \eqref{k7},  we obtain that there exists a constant $\cK$ independent of $k$ such that 
\begin{align}\label{k11a}
|\chi^* v^{k+1}_1|^*_m\leq \cK.
\end{align}

\end{rem}

\textbf{4. Higher regularity.} 
To show $v\in\cN^{m+1}(\Omega_{T_2})$, it is enough to show that the iterates $v^j$ in \eqref{k8a}, \eqref{k8b}  are uniformly bounded in $\cN^{m+1}(\Omega_{T_2})$, since we already have the iterates converging to $v$ in a lower norm.  We make the following induction assumption:

$\bullet$  There exist positive constants  $P>0$ and $\gamma_c$ independent of $k$ such that for $j=1,\dots, k$ we have
\begin{align}
|v^k|^*_{m+1,\gamma}\leq P\text{ for }\gamma\geq \gamma_c. 
\end{align}
Once this is shown to hold for all $k$, it then follows from Proposition \ref{c0e} that 
\begin{align}
|v^k|^*_{m+1,T_2}\leq Ce^{\gamma_cT_2}|v^k|^*_{m+1,\gamma_c}\leq Ce^{\gamma_cT_2}P\text{ for all }k.
\end{align}
Part (b) of the Proposition follows as in part (2) of Remark \ref{N}.  

\textbf{5. Induction step.} Let $a^{k+1}_{j,\gamma}=|v^{k+1}_j|^*_{m+1,\gamma}$ for $j=1,2$.    We will complete the induction step by applying the estimates \eqref{k3}, \eqref{k4} with $``k"=m+1$ and $``m_0"=m$, together with the  estimates of sections \ref{nontame} and \ref{tames}, to the problems \eqref{k8a}, \eqref{k8b}.
In these estimates $K$, $K_j$ denote nonnegative continuous functions of one or more arguments which increase as any one argument increases.  Also,  $v^k=v^k_{T_2}$, but $v^{k+1}$ is not a Seeley extension.

We first apply Corollary \ref{j3} (with Remark \ref{j9}) to estimate $\cF_i$ in \eqref{k5}, \eqref{k6}, obtaining
\begin{align}\label{k11aa}
|\Lambda^{\frac{1}{2}}\cF_i(v^k,\nabla_\eps v^k)|_{m+1,\gamma}\leq K_1(R)P \text{ for }j=1,2.
\end{align}
Here and below we use \eqref{k11} to bound the analytic functions like $h(\langle v^k,  \eps^2 G \rangle_m$ that appear  by constants $C$;  we suppose that $M_0$,  $\delta$, and $M_G$ have been chosen small enough so that the arguments of these functions lie in their domains of convergence.  Since,
\begin{align}
\langle \Lambda (v^k_1,\eps^2 G)\rangle_{m+1,\gamma}\lesssim P+M_G, 
\end{align}
Corollary \ref{j3} yields
\begin{align}
\langle\Lambda\cG_1(v^k_1,\eps^2 G)\rangle_{m+1,\gamma}\leq C_1(P+PR+M_G). 
\end{align}
Applying the estimate \ref{k3} to \eqref{k8a} and using \eqref{k11a}, we obtain for $j\in J_h$ and $\gamma\geq \gamma_0$ 
\begin{align}\label{k12}
\begin{split}
&a^{k+1}_{1,\gamma}+\gamma^{-\frac{1}{2}}\langle \phi_j\Lambda^{\frac{3}{2}}v^{k+1}_1\rangle_{m+1,\gamma}\leq \frac{1}{\gamma}K(R)\left(a^{k+1}_{1,\gamma}+\cK P+K_1(R)P\right)+\\
&\qquad \frac{1}{\gamma}K(R)\left(a^{k+1}_{1,\gamma}+\cK P+C_1(P+PR+M_G)\right)\leq \frac{1}{\gamma}(2K(R)a^{k+1}_{1,\gamma}+K_2(R)P+K_3(R)M_G).\\
\end{split}
\end{align}
Thus, 
\begin{align}\label{k12a}
\begin{split}
&a^{k+1}_{1,\gamma}+\gamma^{-\frac{1}{2}}\langle \phi_j\Lambda^{\frac{3}{2}}v^{k+1}_1\rangle_{m+1,\gamma}\leq \gamma^{-1}K_4(R)(P+M_G):=\gamma^{-1}K_5(R,P,M_G)\\&\qquad \qquad \text{ for }\gamma\geq \max {(4K(R),\gamma_0)}:=\gamma_a,  \text{ where }K_5(R,0,0)=0.  
\end{split}
\end{align}
 This  implies
\begin{align}\label{k13}
 |\Lambda(\chi^* v^{k+1}_1)|_{m+1,\gamma}\lesssim \frac{1}{\gamma}K_5(R,P,M_G)\text{ for }\gamma \geq \gamma_a,
\end{align}
and since $\mathrm{supp}\;\chi^*\subset (-1,2)$, this gives
\begin{align}\label{k13b}
 |\Lambda_1(\chi^* v^{k+1}_1)|_{m+1}\lesssim e^{\gamma_a 2}K_5(R,P,M_G):=K_6.
 \end{align}

Application of the (nontame) estimate of Proposition \ref{f3}(c) yields with $z=(\chi^*v^{k+1}_1, \eps^2\chi^* G)$:
\begin{align}\label{k13a}
\begin{split}
&\langle\Lambda\cG_2\rangle_{m+1,\gamma}=\langle\Lambda(\chi_0 H(v^{k+1}_1,\eps^2 G))\rangle_{m+1,\gamma}\lesssim \left(\langle\Lambda z\rangle_{m+1,\gamma}+\langle z\rangle_{m+1,\gamma}\langle\Lambda_1 z\rangle_{m+1}\right)\;h(\langle z\rangle_{m+1})\lesssim \\
&\qquad \left(\frac{K_5}{\gamma}+ M_G\right)+\left(\frac{K_5}{\gamma}+M_G\right)(K_6+ M_G)\lesssim \frac{K_7}{\gamma}+M_G.
\end{split}
\end{align}
In the above estimate we have used \eqref{k13b},  noting that since $K_5(R,0,0)=0$ we will have $\langle z\rangle_{m+1}$ in the domain of convergence of $h$ provided $P$ and $M_G$ are small enough.

Similarly, applying Proposition \ref{f7}  and using \eqref{k12a} to estimate $\langle \psi\Lambda^{\frac{3}{2}}v^{k+1}_1\rangle_{m+1,\gamma}$, we obtain
\begin{align}\label{k14}
\begin{split}
&\langle\phi_j\Lambda^{\frac{3}{2}}(\chi_0 H(v^{k+1}_1,\eps^2 G))\rangle_{m+1,\gamma}\lesssim \\
&\quad \langle\Lambda z\rangle_{m+1,\gamma}\langle\Lambda_1 z\rangle_{m+1}h(\langle z\rangle_{m+1})+ \langle\psi\Lambda z\rangle_{m+1,\gamma}\langle\Lambda^{\frac{1}{2}}_1z\rangle_{m+1}h(\langle z\rangle_{m+1})+\langle\psi\Lambda^{\frac{3}{2}}z\rangle_{m+1,\gamma}h(\langle z\rangle_{m+1})\lesssim\\
&\qquad  \left(\frac{K_5}{\gamma}+ M_G\right)(K_6+ M_G)+ \left(\frac{K_5}{\gamma}+ M_G\right)(K_6+\ M_G)+ \left(\frac{K_5}{\gamma^{\frac{1}{2}}}+ M_G\right)\lesssim \frac{K_8}{\gamma^{\frac{1}{2}}}+1.
\end{split}
\end{align}
With \eqref{k11a}, \eqref{k13a} and \eqref{k14} we can now apply the estimate \eqref{k4} to \eqref{k8b} to obtain
\begin{align}
\begin{split}
&a^{k+1}_{2,\gamma}\lesssim \frac{1}{\gamma}K(R)\left(a^{k+1}_{2,\gamma}+\cK P+K_1(R)P\right)+\\
&\qquad K(R)\left[\left(\frac{K_7}{\gamma}+M_G\right)+\frac{1}{\gamma^{\frac{1}{2}}}\left(\frac{K_8}{\gamma^{\frac{1}{2}}}+1\right)\right].
\end{split}
\end{align}
Thus, for large enough $\gamma$
\begin{align}
a^{k+1}_{2,\gamma}\lesssim \frac{1}{\gamma}K_9P+K_{10}M_G+\frac{1}{\gamma^\frac{1}{2}}K_{11}<\frac{P}{2}, 
\end{align}
 if $\gamma$ is large enough and $M_G$ small enough. 
 Since \eqref{k12a} implies
 $a^{k+1}_{1,\gamma}\leq \frac{P}{2}$ for $\gamma$ large enough, so this completes the induction step.

\end{proof}


\chapter{Approximate solutions}
\label{chapter4}

\emph{\quad}  This chapter is devoted to the construction of an approximate solution (leading term and corrector) to the coupled, singular nonlinear problems \eqref{a7}-\eqref{a9} on a short time interval.   We will use the result of Proposition \ref{propwellposed}, but otherwise  this chapter can be read independently of chapter \ref{chapter2} and all but the first few pages of chapter \ref{chapter3}.   
We introduce some notations here  that differ from those used in chapter \ref{chapter2}, but which we have found more suitable for estimating the nonlinear interactions that appear in the first corrector.

=
\section{Introduction}
We seek initially an approximate solution to the original nonlinear, nonsingular problem \eqref{a2} of the form
\begin{align}
U^\eps_a(t,x)=u_a(t,x,\theta,z)|_{\theta=\frac{\beta\cdot (t,x_1)}{\eps}, z=\frac{x_2}{\eps}} \text{ on }\Omega_{T_1}
\end{align}
for some $T_1>0$, where\footnote{In chapter \ref{chapter2} $u_\sigma$ and $u_\tau$ were denoted $u^{(1)}$ and $u^{(2)}$ respectively.} 
\begin{align}\label{o2a}
\begin{split}
&u_a(t,x,\theta,z)=\eps^2u_\sigma(t,x,\theta,z)+\eps^3 u_\tau(t,x,\theta,z) \text{ with }\\
&u_\sigma(t,x,\theta,z)=\sum_{j=1}^4\sigma_j(t,x,\theta+\omega_j z) r_j\text{ and }u_\tau(t,x,\theta,z)=   \sum_{j=1}^4 \left(\chi_\eps(D_\theta)\tau_j(t,x,\theta,z)\right) r_j.
\end{split}
\end{align}
Here the $\sigma_j$ and $\tau_j$ are scalar profiles constructed in section \ref{moreonc}, and the $\omega_j$, $r_j$ are characteristic roots and vectors satisfying 
\begin{align}
\det L(\beta,\omega_j)=0  \text{ and }  L(\beta,\omega_j)r_j=0,
\end{align}
where $L(\xi',\xi_2)$ is the $2\times 2$ matrix symbol defined below in \eqref{oo1}.   The operator $\chi_\eps(D_\theta)$ is a low-frequency cutoff operator that is discussed below.    The functions 
\begin{align}
u^\eps_a(t,x,\theta):=u_a(t,x,\theta,\frac{x_2}{\eps})\text{ and }v^\eps_a=\nabla_\eps u^\eps_a
\end{align}
will then turn out to be approximate solutions of the singular problems \eqref{a7}-\eqref{a9} on the same time interval.

When $U^\eps_a(t,x)$ is plugged into the system \eqref{a2},   we obtain
\begin{align}\label{o2}
\begin{split}
&\partial_{t}^2 U_a^\eps+\sum_{|\alpha|=2} A_\alpha(\nabla U^\eps_a)\partial_{x}^\alpha U_a^\eps=F_a(t,x,\theta,z)_{\theta=\frac{\beta\cdot (t,x_1)}{\eps}, z=\frac{x_2}{\eps}}\text{ on }\Omega_{T_1}\\
&h(\nabla U^\eps_a)=g_a(t,x_1,\theta)|_{\theta=\frac{\beta\cdot (t,x_1)}{\eps}}\text{ on }x_2=0,
\end{split}
\end{align}
for interior and boundary profiles $F_a$, $g_a$  described below. All functions here are zero in $t<0$.  

The coefficients $A_\alpha(v)$ are  polynomials in $v$ of order two, so we write
\begin{align}\label{Asub}
A_\alpha(v)=A_\alpha(0)+L_\alpha(v)+Q_\alpha(v), 
\end{align}
a sum of constant,  linear, and quadratic parts.  Let us set
\begin{align}\label{o2aa}
D_\eps=(\partial_{x_1}+\frac{\beta_1}{\eps}\partial_\theta, \partial_{x_2}+\frac{1}{\eps}\partial_z)
\end{align}
and define\footnote{Here, for example, we write $\partial_{x_1\theta}$ for $\partial_{x_1}\partial_\theta$.} 
\begin{align}
\partial_{ss}^\alpha=\partial_x^\alpha,\;\partial^\alpha_{fs}=\begin{cases}2\beta_1\partial_{x_1\theta}, \;\alpha=(2,0)\\\beta_1\partial_{x_2\theta}+\partial_{x_1z},\;\alpha=(1,1)\\2\partial_{x_2z},\;\alpha=(0,2)\end{cases},\;\partial^\alpha_{ff}=\begin{cases}\beta^2_1\partial_{\theta\theta}, \;\alpha=(2,0)\\\beta_1\partial_{\theta z},\;\alpha=(1,1)\\\partial_{zz},\;\alpha=(0,2)\end{cases}.
\end{align}
Next set
\begin{align}\label{o2b}
\begin{split}
&L_{ss}=\partial^2_t+\sum_{|\alpha|=2}A_\alpha(0)\partial^\alpha_{ss},\;\;      L_{fs}=2\beta_0\partial_{t\theta}+\sum_{|\alpha|=2}A_\alpha(0)\partial^\alpha_{fs},\;\;L_{ff}=\beta^2_0\partial_{\theta\theta}+\sum_{|\alpha|=2}A_\alpha(0)\partial^\alpha_{ff}\\
&N(u_\sigma)=\sum_{|\alpha|=2}L_\alpha(\beta_1\partial_\theta u_\sigma,\partial_z u_\sigma)\partial^\alpha_{ff}u_\sigma.
\end{split}
\end{align}

The function $h(v)$ in the boundary condition of \eqref{o2} is a cubic polynomial in $v$ satisfying $h(0)=0$, so we write
\begin{align}\label{hofv}
h(v)=\ell (v)+q(v)+c(v),
\end{align}
a sum of linear, quadratic, and cubic parts, and  define
\begin{align}
\ell_s(u_\sigma)=\ell(\partial_{x_1}u_\sigma,\partial_{x_2}u_\sigma),\;\ell_f(u_\tau)=\ell(\beta_1\partial_\theta u_\tau,\partial_z u_\tau),\;n(u_\sigma)=q(\beta_1\partial_\theta u_\sigma,\partial_z u_\sigma).
\end{align}

One attempts to construct $u_\sigma$ and $u_\tau$ so that the  terms in $F_a$ of orders $O(\eps^0)$ and $O(\eps)$ vanish, and so that the terms in $g_a-g$, where $g=\eps^2 G$,  of orders $O(\eps)$ and $O(\eps^2)$ vanish.   This leads directly to interior and boundary equations that one \emph{would like} $u_\sigma$ and $u_\tau$ to satisfy:
\begin{align}\label{o3}
\begin{split}
&(a)\;L_{ff}u_\sigma=0\\
&(b)\;L_{ff}u_\tau+L_{fs}u_\sigma+N(u_\sigma)=0 \text{ in }x_2\geq 0, z\geq 0
\end{split}
\end{align}
and
\begin{align}\label{o4}
\begin{split}
&(a)\;\ell_f(u_\sigma)=0\\
&(b)\;\ell_f u_\tau+\ell_su_\sigma+n(u_\sigma)=G(t,x_1,\theta)\text{ on }x_2=0,z=0.
\end{split}
\end{align}

We now introduce a new unknown $\cU_\tau(t,x,\theta,z)$ that will be related to $u_\tau$ by $u_\tau=\chi_\eps(D_\theta)\cU_\tau$.  
The amplitude equation of Proposition \ref{propelas} is a solvability condition for the following problem (which is the same as \eqref{bkwordre2}):
\begin{align}\label{o5}
\begin{split}
&(a)\;L_{ff}\cU_\tau=-(L_{fs}u_\sigma+N(u_\sigma))\text{ on }x_2=0,z\geq 0\\
&(b)\;\ell_f \cU_\tau+\ell_su_\sigma+n(u_\sigma)=G(t,x_1,\theta)\text{ on }x_2=0,z=0.
\end{split}
\end{align}
\emph{More precisely,} it is a solvability condition for the Fourier transform of this problem with respect to $\theta$.     The traces $\hat\sigma_j(t,x_1,0,k)$  turn out to be constant multiples of $\hat w(t,x_1,k)$, where  $w(t,x_1,\theta)$ is the solution of the amplitude equation \ref{ampeqn}.

For each $k$  one is able to construct $\hat \cU_\tau(t,x,k,z)$ so that the transform of \eqref{o5}(b) holds and the transform of 
\eqref{o5}(a) holds in $\{x_2\geq 0, z\geq 0\}$. 
  A difficulty is that $\hat \cU_\tau(t,x,k,z)$ is so singular at $k=0$ that one cannot take the inverse transform to define $\cU_\tau$; in fact, \eqref{o23a} shows that 
  $\hat \cU_\tau(t,x,k,z)=O(\frac{1}{k^2})$.    The division by $k^2$ in \eqref{o23a} reflects the two integrations needed in this second order problem.  Since  $\cU_\tau(t,x,\theta,z)$ is not well-defined, we can not actually solve \eqref{o5}.   Moreover, $\hat \cU_\tau(t,x,k,z)$ is too large to be of any direct use in the error analysis.

This is the reason for $\chi_\eps(D_\theta)$ in the definition of $u_\tau$ \eqref{o2a}.    For some $b>0$ to be chosen later and a $C^\infty$ cutoff $\chi(s)$ vanishing near $s=0$ with $\chi=1$ on $|s|\geq 1$, we define
\begin{align}
\hat u_\tau(t,x,k,z)=\chi\left(\frac{k}{\eps^b}\right)\hat \cU_\tau(t,x,k,z):=\chi_\eps(k)\hat \cU_\tau(t,x,k,z),
\end{align}
and observe that $u_\tau$, which is well-defined, satisfies
\begin{align}\label{o6}
\begin{split}
&(a)\;L_{ff}u_\tau+\chi_\eps(D_\theta)(L_{fs}u_\sigma+N(u_\sigma))=0\text{ on }x_2\geq 0,z\geq 0\\
&(b)\;\ell_f u_\tau+\chi_\eps(D_\theta)(\ell_su_\sigma+n(u_\sigma))=\chi_\eps(D_\theta)G(t,x_1,\theta)\text{ on }x_2=0,z=0.
\end{split}
\end{align}
The use of $\chi_\eps(D_\theta)$ introduces new errors of course, but we show in chapter \ref{chapter5} that if the exponent $\emph{}{}{}{b}$ is chosen correctly, the errors will converge to zero at a computable rate as $\eps\to0$, \emph{}{}{}{because of the presence of the factor of $\eps^3$ on $u_\tau$.}
This kind of \emph{}{}{}{low-frequency cutoff} was first used in the rigorous study of pulses propagating in the interior in work of Alterman-Rauch \cite{AR}.


Using \eqref{o3}(a), \eqref{o4}(a), and \eqref{o6}, we can now write the interior error profile in \eqref{o2} as 
\begin{align}\label{fa}
\begin{split}
&F_a(t,x,\theta,z):=L_{ss}u_a+\eps^2 L_{fs}u_\tau+\sum_{|\alpha|=2}Q_\alpha(D_\eps u_a)D_\eps^\alpha u_a+\left(\sum_{|\alpha|=2}L_\alpha(D_\eps u_a)D^\alpha_\eps u_a-\eps N(u_\sigma)\right)+\\
&\qquad \eps(1-\chi_\eps(D_\theta))(L_{fs}u_\sigma+N(u_\sigma)).
\end{split}
\end{align}

Similarly, we can write the boundary profile in \eqref{o2} as $g_a=g+h_a$, where the boundary error profile is 
\begin{align}\label{ba}
\begin{split}
&h_a(t,x_1,\theta):=\eps^3 \ell_s(u_\tau)+c(D_\eps u_a)+(q(D_\eps u_a)-\eps^2n(u_\sigma))+\\
&\qquad  \eps^2\left[(\chi_\eps(D_\theta)-1)G+(1-\chi_\eps(D_\theta))(\ell_s(u_\sigma)+n(u_\sigma))\right]\text{ at } x_2=0,z=0.
\end{split}
\end{align}

\begin{rem}\label{baz}
1. The terms in the first line of \eqref{fa} would have been the only terms to appear if there were no need to introduce $\chi_\eps(D_\theta)$.
 These terms are all (formally) $O(\eps^2)$ or smaller.

2.  The second line represents the ``low frequency cutoff error" incurred by introducing $\chi_\eps(D_\theta)$.


3.  Remarks analogous to  (1) and (2) apply to the first and second lines of \eqref{ba}.

4. The approximate solution exhibits \emph{}{}{}{amplification}, as often happens in weakly stable problems, in the following sense.    The approximate solution is $O(\eps^2)$, and its gradient is of size $O(\eps)$, given boundary data of size $O(\eps^2)$.  If the uniform Lopatinskii condition were satisfied, we would expect the solution to be of size $O(\eps^3)$ and its gradient of size $O(\eps^2)$  in this second-order problem.

5. Even if one assumes $G(t,x_1,\theta)$ decays \emph{}{}{}{exponentially} as $|\theta|\to \infty$, the profiles $\sigma_j(t,x,\theta)$ defining $u_\sigma$
generally exhibit \emph{}{}{}{no better} than $H^s(t,x,\theta)$ ``decay" in $\theta$.  This reflects that fact that $\hat\sigma_j(t,x,k)$ may be \emph{}{}{}{discontinuous} at $k=0$. This loss of $\theta$-decay  from data to solution is \emph{}{}{}{typical of evanescent pulses}, and occurs even in problems where the uniform Lopatinskii condition is satisfied [Willig].

\end{rem}

\section{Construction of the leading term and corrector}\label{moreonc}

\textbf{1. Some notation. }   
As in Lardner \cite{Lardner1},  the operator $L_{ss}$ can be written 
\begin{align}\label{form}
L_{ss}=\partial_{tt}-\begin{pmatrix}r&0\\0&1\end{pmatrix}\partial_{x_1x_1}-\begin{pmatrix}0&r-1\\r-1&0\end{pmatrix}\partial_{x_1x_2}-\begin{pmatrix}1&0\\0&r\end{pmatrix}\partial_{x_2x_2},
\end{align}
where $r>1$  is the ratio of the squares of pressure $c_d$ and shear $c_s$ velocities.\footnote{We have $c_s^2=\mu$ and  $c_d^2=(\lambda+2\mu)$, where  $\lambda$, $\mu$ are the Lam\'e constants.  The form \eqref{form} is obtained by taking units of time so that $c_s=1$.  Observe $r=c_d^2/c_s^2>1$ since $\lambda+\mu>0$.}   
The boundary frequency $\beta$ has the form $(-c,1)$ for a $c$ to be chosen, so the operators $L_{fs}$ and $L_{ff}$ are 
\begin{align}
\begin{split}
&L_{fs}=-2c\partial_{t\theta}-\begin{pmatrix}2r&0\\0&2\end{pmatrix}\partial_{x_1\theta}-\begin{pmatrix}0&r-1\\r-1&0\end{pmatrix}(\partial_{x_1z}+\partial_{x_2\theta})-\begin{pmatrix}2&0\\0&2r\end{pmatrix}\partial_{x_2z},\\
&L_{ff}=\begin{pmatrix}c^2-r&0\\0&c^2-1\end{pmatrix}\partial_{\theta\theta}-\begin{pmatrix}0&r-1\\r-1&0\end{pmatrix}\partial_{\theta z}-\begin{pmatrix}1&0\\0&r\end{pmatrix}\partial_{zz},
\end{split}
\end{align}
and
\begin{align}
\begin{split}
&\ell_s=\begin{pmatrix}0&1\\r-2&0\end{pmatrix}\partial_{x_1}+\begin{pmatrix}1&0\\0&r\end{pmatrix}\partial_{x_2}\\
&\ell_f=\begin{pmatrix}0&1\\r-2&0\end{pmatrix}\partial_{\theta}+\begin{pmatrix}1&0\\0&r\end{pmatrix}\partial_{z}.
\end{split}
\end{align}

With $\xi'=(\sigma,\xi_1)=(\xi_0,\xi_1)$ the symbol of $L_{ss}$ is 
\begin{align}\label{oo1}
\begin{split}
&L(\xi',\xi_2)=\begin{pmatrix}\sigma^2-(r-1)\xi^2_1-|\xi|^2 &-(r-1)\xi_1\xi_2\\-(r-1)\xi_1\xi_2&\sigma^2-(r-1)\xi_2^2-|\xi|^2\end{pmatrix}
\end{split}
\end{align}
with characteristic roots $\omega_j$ and vectors $r_j$ satisfying
\begin{align}\label{oo2}
\det L(\beta,\omega_j)=0\text{ and }L(\beta,\omega_j)r_j=0, j=1,\dots 4.
\end{align}
The numbers $\omega_1$ and $\omega_2$  are pure imaginary with positive imaginary part.     We have   $\omega_3=\overline{\omega_1}$,  $\omega_4=\overline{\omega_2}$
and 
\begin{align}\label{ray1}
\omega_1^2=c^2-1, \; \omega_2^2= \frac{c^2}{r}-1,  
\end{align}
the numbers on the right in \eqref{ray1} being negative since $\beta$ lies in the elliptic region. 
 If we define $q=q(c)>0$ by 
\begin{align}\label{ray2}
q^2=-\omega_1\omega_2,
\end{align}
then the condition for $\beta=(-c,1)$ to be a Rayleigh frequency is that\footnote{This is the condition for the $2\times 2$ matrix $\cB_{Lop}$ defined \eqref{o10a} to have vanishing determinant.   Equation \ref{ray}  is equivalent to equation \eqref{defcR}.    The $\omega_j$ appearing here are obtained from the $\omega_j$ in \eqref{emodes} by multiplying the latter by $i$.}
\begin{align}\label{ray}
(2-c^2)^2=4q^2(c),    \text{ or equivalently } 2-c^2=2q.
\end{align}
For the existence of $\beta=(-c,1)$ in the elliptic region satisfying \eqref{ray} we refer, for example, to \cite{T2}.    We take 
\begin{align}\label{oo3}
r_1=\begin{pmatrix}-\omega_1\\1\end{pmatrix},\;r_2=\begin{pmatrix}1\\ \omega_2\end{pmatrix},\;r_3=\overline{r}_1, r_4=\overline{r}_2.
\end{align}

\textbf{2.  First-order system for the $\sigma_j$. }
Consider a problem of the form
\begin{align}\label{o6a}
L_{ff}u=f(t,x,\theta,z)\text{ in }\{x_2\geq 0, z\geq 0\},\;\;\; \ell_f u= g(t,x_1,\theta) \text{ on }x_2=z=0.
\end{align} 
Taking the Fourier transform with respect to $\theta$ and  setting $\hat U=\begin{pmatrix}\hat u\\ \partial_z\hat u\end{pmatrix}$, we obtain the $4 \times 4$ first order system
\begin{align}\label{o6b}
\partial_z \hat U-G(\beta,k)\hat U=\hat \cF,   \; C(\beta,k)\hat U=\hat g,
\end{align}
where $G(\beta,k)=\begin{pmatrix}0&1\\D&B\end{pmatrix}$ with 
\begin{align}\label{o6c}
D=k^2\begin{pmatrix}r-c^2&0\\0&\frac{1-c^2}{r}\end{pmatrix},\; B=ik\begin{pmatrix}0&1-r\\\frac{1}{r}-1&0\end{pmatrix}\text{ and }\cF=\begin{pmatrix}0\\-\begin{pmatrix}1&0\\0&r\end{pmatrix}^{-1} f\end{pmatrix}
\end{align}
and 
\begin{align}\label{o6d}
C(\beta,k)=\begin{pmatrix}ik\begin{pmatrix}0&1\\r-2&0\end{pmatrix}&\begin{pmatrix}1&0\\0&r\end{pmatrix}\end{pmatrix}.
\end{align}
The matrix $G(\beta,k)$ has eigenvalues $ik\omega_j$, $j=1,\dots,4$
corresponding respectively to the right eigenvectors
\begin{align}
R_1=\begin{pmatrix}-\omega_1\\1\\-ik\omega_1^2\\ik\omega_1\end{pmatrix}, R_2=\begin{pmatrix}1\\\omega_2\\ik\omega_2\\ik\omega_2^2\end{pmatrix}, R_3=\begin{pmatrix}\omega_1\\1\\-ik\omega_1^2\\-ik\omega_1\end{pmatrix},  R_4=\begin{pmatrix}1\\-\omega_2\\-ik\omega_2\\ik\omega_2^2\end{pmatrix}.
\end{align}

Using the $R_k$ to diagonalize $G(\beta,k)$ we see that solutions of  $\partial_z \hat U-G(\beta,k)\hat U=0$ that decay as $z\to +\infty$  must have the form
\begin{align}\label{o10}
\hat U_j(t,x,k,z)=e^{ik\omega_j z}\hat \sigma_j(t,x,k)R_j, 
\end{align}
where the $\hat \sigma_j$ are scalar functions to be determined such that 
\begin{align}
 k\mathrm{Im}\; \omega_j\geq 0 \text{ on supp }\hat \sigma_j.
\end{align}
This explains the form of $u_\sigma$ in \eqref{o2a}, where $r_j$ is the vector given by the first two components of $R_j$. 

A short calculation using the relations \eqref{ray1}-\eqref{ray} shows that 
\begin{align}\label{o10a}
[C(\beta,k)R_1\;\; C(\beta,k)R_2]=ik\begin{pmatrix}2-c^2&2\omega_2\\2\omega_1&c^2-2\end{pmatrix}:=ik\cB_{Lop},
\end{align}
and that 
\begin{align}\label{o11}
\ker \cB_{Lop}=\mathrm{span}\;\begin{pmatrix}\omega_2\\-q\end{pmatrix},\;\;\mathrm{coker}\; \cB_{Lop}=\mathrm{span}\; \begin{pmatrix} q&\omega_2\end{pmatrix}.
\end{align}

In  order for  (the Fourier transform in $\theta$ of ) \eqref{o4}(a) to hold we must have, when $k>0$,
\begin{align}
C(\beta,k)(\hat U_1+\hat U_2)=0\text{ on }x_2=0, z=0 \text{ for }\hat U_i \text{ as in }\eqref{o10},
\end{align}
and thus in view of \eqref{o11}
\begin{align}\label{o12}
\begin{pmatrix}\hat\sigma_1(t,x_1,0,k)\\\hat\sigma_2 (t,x_1,0,k)\end{pmatrix}=\alpha(t,x_1,k)\begin{pmatrix}\omega_2\\-q\end{pmatrix},
\end{align}
for some scalar amplitude $\alpha$.   We take $\alpha$ to be $-2i \hat w(t,x_1,k)$, where $w$ is the function constructed using  Proposition \ref{propwellposed}  to satisfy the solvability condition \eqref{eqw}.\footnote{The factor of $-2i$ could be replaced by one if we multiplied $r_1$ and $r_2$ in \eqref{oo3} by $-2i$.}      With  the traces of the $\hat\sigma_j$  thereby determined, we extend the $\hat\sigma_j$ into $x_2\geq 0$ by setting
\begin{align}\label{o19a}
\hat\sigma_j(t,x_1,x_2,k):=\psi(x_2)\hat\sigma_j(t,x_1,0,k),
\end{align}
where $\psi\in C^\infty$ is compactly supported and equal to $1$ near $x_2=0$.

\textbf{3.  First-order system for the $\tau_j$. }
To proceed further we must now consider the inhomogenous problem \eqref{o5} for $\cU_\tau$.  We look for the solution as a sum of a homogeneous solution $\cU^h_\tau$ and a ``particular" solution $\cU^p_\tau$ as 
\begin{align}\label{o12a}
\begin{split}
&\cU_\tau=\cU_\tau^h+\cU^p_\tau, \text{ where }\\
&\cU_\tau=\sum^4_{j=1}\tau_j(t,x,\theta,z) r_j,\;\;\cU^h_\tau=\sum_{j=1}^4\tau_j^h(t,x,\theta,z)r_j,\;\;\cU^p_\tau=\sum_{j=1}^4\tau_j^p(t,x,\theta,z)r_j.
\end{split}
\end{align}
The Fourier transform of \eqref{o5} can be written as the first order system \eqref{o6b}-\eqref{o6d}, where now
 \begin{align}\label{o13}
\begin{split}
&(a)\;f=f_\sigma:=-(L_{fs}u_\sigma+N(u_\sigma))\\
&(b)\;g=-[\ell_su_\sigma+n(u_\sigma)]+G(t,x_1,\theta)\text{ on }x_2=0,z=0.
\end{split}
\end{align}
The interior forcing term $\hat{\cF}$ can be written
\begin{align}
\hat{\cF}=\sum^4_{j=1}\hat F_jR_j, \text{ with }\hat F_j=L_j\hat\cF,
\end{align}
where the $L_j$ are left eigenvectors of $G(\beta,k)$ associated to $\omega_j$ chosen so that $L_mR_n=\delta_{mn}$.   The $L_j$ are given by
\begin{align}\label{o14}
\begin{split}
&L_1(k)=(-ik(r-c^2),-ik\omega_1,\omega_1,-r)/(-2i\omega_1c^2k),\,\,L_2(k)=(ik\omega_2r, ik(c^2-1),1,r\omega_2)/2i\omega_2c^2k,\\
&L_3(k)=\overline{L}_1(-k), \;L_4(k)= \overline{L}_2(-k).
\end{split}
\end{align}
The decoupled interior system for the $\hat \tau^p_j$ is then
\begin{align}\label{o15}
(\partial_z-ik\omega_j)\hat \tau^p_j=\hat F_j, \;j=1,\dots,4.
\end{align}

From  \eqref{o14} and the form of $\cF$ in \eqref{o6c}, we see that 
\begin{align}\label{o15a}
\begin{split}
&\hat F_j=\ell_j(k) \hat f_\sigma \text{ for }f_\sigma\text{ as in \eqref{o13} and row vectors }\ell_j(k) \text{ given by: }\\
&\ell_1(k)= (\omega_1, -1)/(-2i\omega_1 c^2 k),\; \ell_2(k)=(1,\omega_2)/(2i\omega_2 c^2 k),\; \ell_3(k)=\overline{\ell}_1(-k),\; \ell_4(k)=\overline{\ell}_2(-k).
\end{split}
\end{align}

\textbf{4. Formulas for the $\hat\tau_j$. } Decaying solutions of \eqref{o15} are given by
\begin{align}\label{o19aa}
\hat\tau^p_j(t,x,k,z)=\begin{cases}\int^z_0 e^{ik\omega_j(z-s)}\hat F_j(t,x,k,s)ds,\; k>0\\ \int^z_{+\infty}e^{ik\omega_j(z-s)}\hat F_j(t,x,k,s)ds,\; k<0\end{cases}\text{ for }j=1,2,
\end{align}
with the same formulas for $j=3,4$, except that $k>0$ (resp. $k<0$) is now associated with $\int_{+\infty}^z$ (resp. $\int^z_0$).  For the homogeneous parts we have 
\begin{align}
\hat\tau^h_j(t,x,k,z)= e^{ik\omega_jz}\hat \tau_j^*(t,x,k), \text{ where }\mathrm{supp}\;\hat\tau^*_j\subset\begin{cases} \{k\geq 0\},\;j=1,2\\ \{k\leq 0\},\;j=3,4\end{cases}
\end{align}
for functions $\hat\tau^*_j$ to be determined.

To complete the determination of the   $\tau_j$ we now examine the Fourier transform in $\theta$ of the boundary equation in \eqref{o6a}.
As in \eqref{o6b} this can be written
\begin{align}\label{o20}
C(\beta,k)\hat U_\tau=\hat g, \text{ where }g=-[\ell_su_\sigma+n(u_\sigma)]+G(t,x_1,\theta)\text{ and }\hat U_\tau:=\begin{pmatrix}\hat\cU_\tau\\\partial_z\hat\cU_\tau\end{pmatrix}
\end{align}
on $x_2=0$, $z=0$.  We have
\begin{align}
\hat U_\tau=\hat U^h_\tau+\hat U^p_\tau, \text{ where }\hat U^h_\tau=\sum^4_{j=1}\hat\tau^*_j R_j,\;\;\hat U^p_\tau=\sum^4_{j=1}\hat \tau^p_j R_j\text{ on }x_2=0, z=0,
\end{align}
so \eqref{o20} becomes for $k>0$
\begin{align}\label{o21}
C(\beta,k)\hat U^h_\tau=ik\cB_{Lop}\begin{pmatrix}\hat\tau^*_1\\ \hat\tau^*_2\end{pmatrix}=\hat g- C(\beta,k)\hat U^p_\tau\text{ on }x_2=0,z=0.
\end{align}
The matrix $\cB_{Lop}$ is singular, but the function $\hat w(t,x_1,k)$, which determines the traces of the $\hat\sigma_j$ on $x_2=0$ \eqref{o12}, was  chosen precisely so that the right side of 
\eqref{o21} lies in the range of $\cB_{Lop}$.  Thus, we can solve for the $\hat \tau^*_j(t,x_1,0,k)$.  We extend the $\hat\tau^*_j$ to $x_2\geq 0$ by setting
\begin{align}\label{o21a}
\hat\tau^*_j(t,x_1,x_2,k)=\psi(x_2)\hat\tau^*_j(t,x_1,0,k) 
\end{align}
for $\psi$ as in \eqref{o19a}.

This completes the construction of the $\hat\sigma_j$, $\hat\tau_j$.  It remains to examine the regularity of these functions.

\textbf{5. Regularity of the $\hat\sigma_j$. }
We introduce the notation
\begin{align}
\hat H^s(t,x_1,k)=\{\hat u(t,x_1,k), \text{ where }u\in H^s(t,x_1,\theta)\},\;|\hat u|_{\hat H^s(t,x_1,k)}:=|\langle k\rangle^s\hat u(t,x_1,k)|_{L^2(k,H^s(t,x_1))}.
\end{align}
The space $\hat H^s(t,x,k)$ is defined similarly.

It follows from Proposition \ref{propwellposed} that $\hat w(t,x_1,k)\in \hat H^s$ when $G\in H^{s}(t,x_1,\theta)$, provided $s>d+3$.  
Thus, we have from \eqref{o12} and \eqref{o19a},
\begin{prop}
Assume $s>d+3$ and $G(t,x_1,\theta)\in H^{s}(t,x_1,\theta)$.  Then
\begin{align}
\hat\sigma_j(t,x_1,0,k)\in \hat H^s(t,x_1,k)\text{ and }\hat\sigma_j(t,x_1,x_2,k)\in \hat H^s(t,x,k).
\end{align}
\end{prop}

\textbf{6. Regularity of the $\hat \tau_j$. }

 \begin{prop}\label{o23}
 Assume $s>d+3$ and $G(t,x_1,\theta)\in H^{s}(t,x_1,\theta)$. We have
 \begin{align}\label{o23a}
 \hat\tau_j(t,x,k,z)=\frac{1}{k^2}T_j(t,x,k,z)
\end{align}
where
\begin{align}
T_j(t,x,k,z)\in C_c(x_2;L^\infty(z,\hat H^{s-2}(t,x_1,k)))\cap C(x_2,z;\hat H^{s-2}(t,x_1,k)).
\end{align}
\end{prop}
 
\begin{proof}
\textbf{1. }We take $j=1$ and using \eqref{o19aa} consider first
\begin{align}\label{o23b}
\hat\tau^p_1(t,x,k,z)=\int^z_{+\infty}e^{ik\omega_1(z-s)}\hat F_1(t,x,k,s)ds,\; k<0
\end{align}
for $\hat F_1$ as in \eqref{o15a}.   From \eqref{o15a} we see that $\hat F_1$ is a linear combination  of terms of the form
\begin{align}\label{o24}
\begin{split}
&(a)\frac{1}{k}\widehat{\partial\partial_\theta \sigma_j}, \\
&(b)\frac{1}{k}\widehat{\partial_\theta\sigma_m\partial_{\theta\theta}\sigma_n}, 
\end{split}
\end{align}
where $\sigma_j=\sigma_j(t,x,\theta+\omega_jz)$.  

Let us consider first the case  (b) when $m\neq n$.   We then obtain for the corresponding term in $\hat\tau^p_1$ a scalar multiple of:
\begin{align}\label{o24a}
\begin{split}
&\hat\tau^p_{mn}(t,x,k,z):= \frac{1}{k^2}T^p_{mn}(t,x,k,z)\text{ where }\\
&T^p_{mn}(t,x,k,z)=ke^{ik\omega_1z}\int_{+\infty}^z\int e^{-ik\omega_1s+i(k-k')\omega_ms+ik'\omega_n s}\widehat{\partial_\theta\sigma_m}(t,x,k-k')\widehat{\partial_{\theta\theta}\sigma_n}(t,x,k')dk'ds.\\
\end{split}
\end{align}
For a given $(x_2,z)$ the  norm  $|T^p_{mn}(t,x,k,z)|_{\hat H^{s-2}(t,x_1,k)}$ can be estimated by considering just the $L^2(k,H^{s-2}(t,x_1))$ and $L^2(t,x_1,\hat H^{s-2}(k))$ norms.   Since $\mathrm{Im}\;(k-k')\omega_m\geq 0$ (resp. $\mathrm{Im}\;k'\omega_n\geq 0$) on $\mathrm{supp}\;\hat\sigma_m$ (resp. $\hat\sigma_n$), we obtain for the $H^{s-2}(t,x_1)$ norm:
\begin{align}\label{o24aa}
\begin{split}
&|T^p_{mn}(t,x,k,z)|_{H^{s-2}(t,x_1)}\leq |k|\int^{+\infty}_z\int e^{ik\omega_1(z-s)}|\widehat{\partial_\theta\sigma_m}(t,x,k-k')\widehat{\partial_{\theta\theta}\sigma_n}(t,x,k')|_{H^{s-2}(t,x_1)}dk'ds=\\
&\qquad\qquad \frac{|k|}{ik\omega_1}\int |\widehat{\partial_\theta\sigma_m}(t,x,k-k')\widehat{\partial_{\theta\theta}\sigma_n}(t,x,k')|_{H^{s-2}(t,x_1)}dk',
\end{split}
\end{align}
which yields
\begin{align}\label{o25}
|T^p_{mn}(t,x,k,z)|_{L^2(k,H^{s-2}(t,x_1))}\lesssim |\hat\sigma_m|_{\hat H^{s}(t,x_1,k)}|\hat\sigma_n|_{\hat H^{s}(t,x_1,k)}.
\end{align}
Here we have used a Moser estimate in the $(t,x_1)$ variables, and observed that, for example,
\begin{align}\label{o25aa}
\begin{split}
&\left||\widehat{\partial_{\theta\theta}\sigma_n}(t,x_1,x_2,k)|_{L^\infty(t,x_1)}\right|_{L^1(k)}\lesssim \left| |\langle k\rangle^{\frac{1}{2}+\delta}\widehat{\partial_{\theta\theta}\sigma_n}(t,x_1,x_2,k)|_{H^{\frac{d}{2}+\delta}(t,x_1)}\right|_{L^2(k)}\lesssim\\
&\qquad\qquad  |\hat\sigma_n|_{\hat H^{\frac{d}{2}+2+\frac{1}{2}+2\delta}(t,x_1,k)}.
\end{split}
\end{align}
The $L^2$ norm of the convolution in $k'$ can then be estimated by Young's inequality.  The $L^2(t,x_1,\hat H^{s-2}(k))$ norm is estimated similarly, using 
\begin{align}
\langle k\rangle^{s-2}\lesssim \langle k-k'\rangle^{s-2}+\langle k'\rangle^{s-2}
\end{align}
in place of the Moser estimate.

Continuity of $T^p_{mn}$ in $x_2$ is evident from the special $x_2$ dependence of the $\hat\sigma_j$ \eqref{o19a}, while continuity in $z$ follows from the dominated convergence theorem.

\textbf{2. }In the case $k>0$  we define $T^p_{mn}$ with $\int^z_{+\infty}...ds$ replaced by $\int^z_0...ds$, and find that $|T^p_{mn}(t,x,k,z)|_{H^{s-2}(t,x_1)}$ is again dominated by the right side of \eqref{o24aa}. The rest of the estimate goes as before.  The estimates of the contributions to $\hat\tau^p_1$ in the case (b) when $m=n$ or case (a) are similar (or easier).   

We observe that in the case \eqref{o24}(b) when $m=n=1$, the integral \eqref{o23b} is zero since $\hat\sigma_1(t,x,k)$ is supported in $k\geq 0$, so one just needs to estimate the 
$\int^z_0\dots ds$ integral that defines $\hat\tau^p_1(t,x,k,z)$ when $k>0$.  In this case the $ds$ integral produces a factor of $z$ (``secular growth").  
Using the fact that 
\begin{align}\label{o25a}
\sup_{z\geq 0} ze^{-|k\omega_j|z}\thickapprox \frac{1}{|k|},
\end{align}
we see that this factor of $z$ has the same effect (introducing an extra factor of $1/|k|$) as the $ds$ integral in \eqref{o24aa}.   Hence one obtains the same estimate for the contribution of \eqref{o24}(b) to $\hat\tau^p_1$ when $m=n=1$  as when $m\neq n$ or when $m=n\neq 1$.

\textbf{3. }From \eqref{o21} and the expression for $g$ \eqref{o20}, we see that the regularity of $\hat\tau^*_j(t,x_1,0,k)$, $j=1,2$ for $k>0$ is determined by the regularity of 
\begin{align}
\frac{\widehat{\ell_s u_\sigma}}{k},\;\frac{\widehat{n(u_\sigma)}}{k},\; \frac{\widehat G}{k},\; \text{ and } \widehat U^p_\tau \text{ on }x_2=0,z=0.
\end{align}
The first three terms have the form $\frac{1}{k}h(t,x_1,k)$, where $h$ lies respectively in $\hat H^{s-1}$, $\hat H^{s-1}$, and $H^{s+1}$, while \eqref{o23a} for $\hat\tau^p_j$ implies
$k^2\hat U^p_\tau\in\hat H^{s-2}$ at $x_2=0,z=0$. With \eqref{o21a} we see that 
\begin{align}
\hat\tau^h_j(t,x,k,z)=\hat\tau^*_j(t,x,k)e^{ik\omega_j z}
\end{align}
also satisfies \eqref{o23a}.
\end{proof}

\chapter{Error Analysis and proof of Theorem \ref{approxthm}}\label{chapter5}

\emph{\quad}In this chapter we show  that the approximate solution is close in a precise sense to the exact solution constructed in Theorem \ref{uniformexistence}. 
We will use the notation and estimates of chapter \ref{chapter3}, especially sections \ref{b1a}, \ref{uniform}, \ref{nonlinear}, and \ref{mainestimate}. 

\section{Introduction}

The error analysis is performed on functions of $(t,x,\theta)$.   We write
\begin{align}\label{p0}
\begin{split}
&u^\eps_a(t,x,\theta):=\eps^2 u^\eps_\sigma(t,x,\theta)+\eps^3u^\eps_\tau(t,x,\theta),\text{ where }\\
&u^\eps_\sigma(t,x,\theta):=u_\sigma(t,x,\theta,z)|_{z=\frac{x_2}{\eps}}=\sum_{j=1}^4\left(\sigma_j(t,x,\theta+\omega_j z)|_{ z=\frac{x_2}{\eps}}\right) r_j\text{ and }\\
&u^\eps_\tau(t,x,\theta):=u_\tau(t,x,\theta,z)|_{z=\frac{x_2}{\eps}} = \sum_{j=1}^4 \left(\chi_\eps(D_\theta)\tau_j(t,x,\theta,z)|_{ z=\frac{x_2}{\eps}}\right) r_j.
\end{split}
\end{align}

When $u_a^\eps$ is plugged into the system \eqref{a5} we obtain 
\begin{align}\label{p0a}
\begin{split}
&\partial_{t,\eps}^2 u_a^\eps+\sum_{|\alpha|=2} A_\alpha(\nabla_\eps u^\eps_a)\partial_{x,\eps}^\alpha u_a^\eps=F_a^\eps:=F_a(t,x,\theta,\frac{x_2}{\eps})\text{ on }\Omega_{T_1}\\
&h(\nabla_\eps u^\eps_a)=g_a\text{ or }\partial _{x_2} u_a^\eps=H(\partial_{x_1,\eps} u_a^\eps, g_a)\text{ on }x_2=0\\
\end{split}
\end{align}
for some $T_1>0$ and $F_a$, $g_a$ as in \eqref{fa}, \eqref{ba}.  Here $u^\eps_a$ is  zero in $t<0$.  

\begin{rem}\label{real}
The functions $u^\eps_\sigma$ and $u^\eps_\tau$ must be constructed to be real-valued.  The analysis of Chapter \ref{chapter2} shows that the amplitude  $w(t,x_1,\theta)$ is real-valued.  The fact that the $\omega_j$ and $r_j$ occur in complex conjugate pairs permits one to construct $u^\eps_\sigma$ as a real-valued function.   Since the function $\chi(s)$ giving the low frequency cutoff in \eqref{p0} can be chosen to be an even function,  one can similarly construct $u^\eps_\tau$ to be real-valued for the same reason.  This remark is used in solving the systems \eqref{p1}-\eqref{p3}.

\end{rem}

\subsection{Extension of approximate solutions to the whole space}

\emph{\quad}   Recall that the estimates of $(v^\eps,\nabla u^\eps)$ leading to the proof of Theorem \ref{uniformexistence} had to be performed initially on the whole half-space $\Omega$.   We estimated functions $(v^\eps,\nabla u^\eps)$ on $\Omega$  that were solutions to the modified, singular \emph{linear} systems \eqref{c1}-\eqref{c3}, whose restrictions to $\Omega_T$ for $0<T<T_\eps$ coincided (by causality, Remark \eqref{k2y}) with solutions provided by Theorem \ref{localex} to the nonlinear singular problems \eqref{a7}-\eqref{a9}.

Similarly,  the estimates in the error analysis must be done on the whole space.  However, the approximate solutions $u^\eps_a$ and $v^\eps_a:=\nabla_\eps u^\eps_a$ constructed in Chapter \ref{chapter4} are defined just on $(-\infty, T_1]$ for some $ T_1>0$. Thus, we must extend  these functions to the whole half-space in order to compare them to the functions $(v^\eps, u^\eps)$.   The form of the modified systems \eqref{c1}-\eqref{c3}  suggests that we will need two kinds of extensions.  First, we will need Seeley extensions $v^s_a$ to $\Omega$ of $v^\eps_a|_{\Omega_T}$ for $0<T\leq T_1$,  that we can hope to prove are close in the $E_{m,T}$ norm to the Seeley extensions $v^s$ appearing in the coefficients of the systems \eqref{c1}-\eqref{c3}.  These Seeley extensions are good approximate solutions only for $t\leq T$;  for later times they are useless as approximate solutions to the modified systems \eqref{c1}-\eqref{c3}.  Thus, we also need extensions to $\Omega$ of $v^\eps_a|_{\Omega_T}$ $u^\eps_a|_{\Omega_T}$,  that we can hope to prove are close in the $E_{m,\gamma}$ norm  to the solutions on $\Omega$ of the modified systems   \eqref{c1}-\eqref{c3}.   These extensions will be constructed as solutions on $\Omega$ of the systems \eqref{p1}-\eqref{p3} below, which should be viewed as approximations to the systems \eqref{c1}-\eqref{c3}.   We now explain how to obtain these extensions in more detail.

 The first step is to extend the error terms.    With $g(t,x_1,\theta)=\eps^2G$ as before,
we  have
\begin{align}\label{p0b}
\begin{split}
&F_a(t,x,\theta)=\cF(u_\sigma(t,x,\theta,z),u_\tau(t,x,\theta,z)|_{z=\frac{x_2}{\eps}}\\
&g_a(t,x_1,\theta)=g(t,x_1,\theta)+h_a(t,x_1,\theta), \text{ where }h_a(t,x_1,\theta)=\cH(u_\sigma(t,x_1,0,\theta,0), u_\tau(t,x_1,0,\theta,0)),
\end{split}
\end{align}
where $\cF$ and $\cH$ are functions that may be read off from the formulas \eqref{fa}, \eqref{ba}.\footnote{The arguments of $\cF$ and $\cG$ should also involve derivatives of $(u_\sigma,u_\tau)$, but we have suppressed these in the notation.}   

The functions $u_\sigma(t,x,\theta,z)$, $u_\tau(t,x,\theta,z)$, and $u_a(t,x,\theta,z)$  are built out of the component functions $\sigma_j(t,x,\theta)$, $j=1,\dots,4$.  
For any $T$ satisfying $0<T\leq T_1$ we let $\sigma^s_{j,T}$ denote a Seeley extension of $\sigma_j|_{\Omega_T}$ to $\Omega$  defined as in \eqref{c0gg}, and we denote by 
$u_\sigma^s$, the function of $(t,x,\theta,z)$ built out of the extended $\sigma_j$.  This is the same as the Seeley extension to $\Omega$  of $u_\sigma|_{\Omega_T}$.   Next we define
$u_\tau^s(t,x,\theta,z)$ to be the Seeley extension of $u_\tau|_{\Omega_T}$.  This is \emph{not}  the same as the extension obtained by replacing $\sigma_j$ by $\sigma_{j,T}^s$ in the definition of $u_\tau$.\footnote{The Seeley extension of a product is not the same as the product of the Seeley extensions.}   We set
\begin{align}
u^s_a=\eps^2u^s_\sigma+\eps^3u^s_\tau \text{ and }v_a^s(t,x,\theta,z)=D_\eps u^s_a(t,x,\theta,z)
\end{align}
for $D_\eps$ as in \eqref{o2aa} and observe that 
\begin{align}
\nabla_\eps u_a^{s,\eps} (t,x,\theta)=v^{s,\eps}_a(t,x,\theta).
\end{align}


Recalling that $g$ has already been extended, we now define (using superscript $e$ for ``extension")
\begin{align}\label{p0bb}
\begin{split}
&F^e_a(t,x,\theta)=\cF(u^s_\sigma(t,x,\theta,z),u^s_\tau(t,x,\theta,z)|_{z=\frac{x_2}{\eps}}\\
&g^e_a(t,x_1,\theta)=g(t,x_1,\theta)+h^e_a(t,x_1,\theta), \text{ where }h^e_a(t,x_1,\theta)=\cH(u^s_\sigma(t,x_1,0,\theta,0), u^s_\tau(t,x_1,0,\theta,0)).
\end{split}
\end{align}


We need  additional extensions of $v_a^\eps|_{\Omega_T}$ and $u^\eps_a|_{\Omega_T}$ to $\Omega$ that are obtained  by solving on $\Omega$ the following three {linear} systems for the respective unknowns $v^\eps_{1a}$, $v^\eps_{2a}$, and $u^\eps_a$.  In these systems all functions have arguments $(t,x,\theta)$ and we suppress superscripts $\epsilon$; as usual, subscripts $T$ are suppressed on Seeley extensions.
\begin{align}\label{p1}
\begin{split}
&(a)\partial_{t,\eps}^2 v_{1a}+\sum_{|\alpha|=2} A_\alpha(v^s_a)\partial_{x,\eps}^\alpha v_{1a}=\\
&\qquad -\left[\sum_{|\alpha|=2,\alpha_1\geq 1}\partial_{x_1,\eps}(A_\alpha(v^s_a))\partial_{x_1,\eps}^{\alpha_1-1}\partial_{x_2}^{\alpha_2}v_{1a}^s-\partial_{x_1,\eps}(A_{(0,2)}(v_a^s))\partial_{x_2}v_{2a}^s\right]+\partial_{x_1,\eps}F^e_a,\\
&(b)\partial _{x_2} v_{1a}-d_{v_1}H(v^s_{1a},h(v^s_a))\partial_{x_1,\eps}v_{1a}=d_gH(v^s_{1a},g^e_a)\partial_{x_1,\eps}(g^e_a)\text{ on }x_2=0.
\end{split}
\end{align}

\begin{align}\label{p2}
\begin{split}
&(a)\partial_{t,\eps}^2 v_{2a}+\sum_{|\alpha|=2} A_\alpha(v^s_a)\partial_{x,\eps}^\alpha v_{2a}=\\
&\qquad -\left[\sum_{|\alpha|=2,\alpha_1\geq 1}\partial_{x_2}(A_\alpha(v^s_a))\partial_{x_1,\eps}^{\alpha_1-1}\partial_{x_2}^{\alpha_2}v^s_{1a}-\partial_{x_2}(A_{(0,2)}(v^s_a))\partial_{x_2}v^s_{2a}\right]+\partial_{x_2}F^e_a,\\
&(b)v_{2a}=\chi_0(t)H(v_{1a},g^e_a)\text{ on }x_2=0.
\end{split}
\end{align}

\begin{align}\label{p3}
\begin{split}
&(a)\partial_{t,\eps}^2 u_a+\sum_{|\alpha|=2} A_\alpha(v^s_a)\partial_{x,\eps}^\alpha u_a=F_a^e\\\
&(b)\partial _{x_2} u_a-d_{v_1}H(v^s_{1a},h(v^s_a))\partial_{x_1,\eps}u_a=H(v^s_{1a},g_a^e)-d_{v_1}H(v^s_{1a},g^e_a)v^s_{1a}\text{ on }x_2=0.
\end{split}
\end{align}
In \eqref{p2} $\chi_0(t)\geq 0$ is the \emph{same} $C^\infty$ function (that  is equal to 1 on a neighborhood of $[0,1]$ and supported in $(-1,2)$) as in \eqref{c2}.


The systems \eqref{p1}-\eqref{p3} for $(v_a,u_a)$ should be compared to the systems \eqref{c1}-\eqref{c3} for $(v,u)$.  In view of the smallness of $F_a^e$ and $g^e_a-g$ and the expected, but still unproved, smallness of $v^s-v^s_a$, we can hope to show that $(v_a,u_a)$ is close to $(v,u)$ in the $E_{m,\gamma}$ norm on $\Omega$.   We carry out this strategy in the remainder of this chapter by studying the error equations computed below.

\begin{rem}\label{p3a}
1.  The above three systems are solved on the full domain $\Omega$.  Each system depends on the parameters $\eps$ and $T$, which are usually suppressed in our notation.

2.   Since the underlying linearized problem corresponding to both \eqref{p1} and \eqref{p3} is just weakly stable, it is important that the functions $v^s_a$ appearing in the coefficients here be real-valued.  That is so as a consequence of remark \ref{real}.

3.  The definitions of $v^s_a=v^s_{a,T}$, $F^e_a$, $g^e_a$ and causality  (see remark \eqref{k2y}) imply that the solutions $v_{1a}$, $v_{2a}$,  and $u_a$ of \eqref{p1}, \eqref{p2}, and \eqref{p3} are equal to the already constructed, similarly denoted functions on $\Omega_T$.      The various Seeley extensions depend on $T$, so the solutions 
$v_a$ and $u_a$ change as $T$ changes. 

\end{rem}

\subsection{Error equations}

\emph{\quad}By subtracting the equations \eqref{p1}, \eqref{p2}, \eqref{p3} from the equations \eqref{c1}, \eqref{c2}, \eqref{c3} we obtain the following equations on $\Omega$ for the error functions
\begin{align}\label{p3b}
w=(w_1,w_2)=v-v_a \text{ and }   z=u-u_a.
\end{align}
With slight abuse let us write the terms in brackets on the right sides of \eqref{p1}(a) and \eqref{p2}(a) respectively as  $b_j(v^s_a)d_\eps v^s_ad_\eps v^s_a$, $j=1,2$ and similarly write the corresponding terms in \eqref{c1} and \eqref{c2} as  $b_j(v^s)d_\eps v^sd_\eps v^s$, $j=1,2$.

\begin{align}\label{p4}
\begin{split}
&(a)\partial_{t,\eps}^2 w_{1}+\sum_{|\alpha|=2} A_\alpha(v^s)\partial_{x,\eps}^\alpha w_{1}=\\
&-\sum_{|\alpha|=2}(A_\alpha(v^s)-A_\alpha(v^s_a))\partial_{x,\eps}^\alpha v_{1a}+b_1(v^s)d_\eps v^s d_\eps v^s -b_1(v^s_a)d_\eps v^s_a d_\eps v^s_a -\partial_{x_1,\eps}F^e_a,\\
&(b)\partial _{x_2} w_{1}-d_{v_1}H(v^s_{1},h(v^s))\partial_{x_1,\eps}w_{1}=-[d_{v_1}H(v^s_1,h(v^s))-d_{v_1}H(v^s_{1a},h(v^s_a))]\partial_{x_1,\eps}v_{1a}+\\
&  \qquad \qquad\qquad d_gH(v^s_{1},g)\partial_{x_1,\eps}g - d_gH(v^s_{1a},g^e_a)\partial_{x_1,\eps} g^e_a \text{ on }x_2=0.
\end{split}
\end{align}

\begin{align}\label{p5}
\begin{split}
&(a)\partial_{t,\eps}^2 w_{2}+\sum_{|\alpha|=2} A_\alpha(v^s)\partial_{x,\eps}^\alpha w_{2}=\\
&-\sum_{|\alpha|=2}(A_\alpha(v^s)-A_\alpha(v^s_a))\partial_{x,\eps}^\alpha v_{2a}+b_2(v^s)d_\eps v^s d_\eps v^s -b_2(v^s_a)d_\eps v^s_a d_\eps v^s_a -\partial_{x_2}F^e_a,\\
&(b)w_2=\chi_0(t)[H(v_{1},g)     -  H(v_{1a},g^e_a)]\text{ on }x_2=0.
\end{split}
\end{align}

\begin{align}\label{p6}
\begin{split}
&(a)\partial_{t,\eps}^2 z+\sum_{|\alpha|=2} A_\alpha(v^s)\partial_{x,\eps}^\alpha z=
-\sum_{|\alpha|=2}(A_\alpha(v^s)-A_\alpha(v^s_a))\partial_{x,\eps}^\alpha u_{a}-F^e_a,\\
&(b)\partial _{x_2} z-d_{v_1}H(v^s_{1},h(v^s))\partial_{x_1,\eps}z=H(v^s_{1},g)-d_{v_1}H(v^s_{1},g)v^s_{1}\\
&-[H(v^s_{1a},g_a^e)-d_{v_1}H(v^s_{1a},g^e_a)v^s_{1a}]+[d_{v_1}H(v^s_1,h(v^s))-d_{v_1}H(v^s_{1a},h(v^s_a))]v_{1a}\text{ on }x_2=0.
\end{split}
\end{align}

\begin{rem}
1. We note again that these equations depend on $\eps$ and $T$ as parameters, and so the same is true of the solutions.   Just like the trios \eqref{c1}-\eqref{c3} and \eqref{p1}-\eqref{p3}, the equations \eqref{p4}-\eqref{p6} must be estimated simultaneously. 

2.  Smallness of $F^e_a$ and $h^e_a$ in appropriate norms on $\Omega$ will follow, provided one has smallness of $F_a$ and $h_a$ in the corresponding time-localized norms on $\Omega_T$,  from continuity properties of Seeley extensions, and the rules of section \ref{nonlinear} for computing norms of nonlinear functions of $u$ in terms of norms of  $u$.

3.  The smallness of $F^e_a$ and $h^e_a$ will imply the smallness of $(w,\nabla_\eps z)$ on $\Omega_{T_2}$ for $T_2$ small enough, but independent of $\eps$.

\end{rem}

The error analysis will be accomplished by estimating the error systems \eqref{p4}-\eqref{p6} on $\Omega$ in the $E_{m,\gamma}$ norm.  
Since the right sides of these systems include terms that contain factors given by derivatives of $(v_a,u_a)$, we shall also need to estimate the approximate solution systems \eqref{p1}-\eqref{p3} in the $E_{m,\gamma}$ norm.

The next two sections are devoted to estimating the functions $F^e_a$ and $h^e_a=g^e_a-g$ that determine the size of the forcing in the approximate solution and error systems.    Each of these functions is constructed from the building blocks $u^s_\sigma$ and $u^s_\tau$, so the first step is to estimate these components in a variety of singular norms;   the estimates of $F^e_a$ and $h^e_a$ then follow by applying the results of section \ref{nonlinear}.
Having estimates of $F^e_a$ and $h^e_a$, we estimate  $(v_a,u_a)$ in section \ref{estapprox}, and finally estimate the error $(w,z)$ in section \ref{end}.   The estimates of section \ref{end} have much in common with the estimates of section \ref{mainestimate},  the main difference being that new kinds of forcing terms such as, for example,  
$$
-\sum_{|\alpha|=2}(A_\alpha(v^s)-A_\alpha(v^s_a))\partial_{x,\eps}^\alpha v_{1a} \text{ in }\eqref{p4}
$$
are encountered.

\section{Building block estimates}\label{bblock}

\emph{\quad} In these estimates the functions being estimated are evaluated at $(t,x_1,x_2,\theta,\frac{x_2}{\eps})$ \emph{after} the indicated derivatives are taken.     
Recall that 
\begin{align}
u_\sigma(t,x,\theta,z)=\sum_{j=1}^4\sigma_j(t,x,\theta+\omega_j z) r_j\text{ and }u_\tau(t,x,\theta,z)=   \sum_{j=1}^4 \left(\chi_\eps(D_\theta)\tau_j(t,x,\theta,z)\right) r_j.
\end{align}
The functions $u_\sigma$  (resp.,  $u_\tau$) are built out of the constituent functions $\sigma_j(t,x,\theta)$, $j=1,...,4$  (resp.,  the $\sigma_j$ and $G(t,x_1,\theta)$).  
The notation $C(m+r)$ indicates a constant that depends on norms
$$
\sup_{x_2\geq 0}|\sigma_j(t,x_1,x_2,\theta)|_{H^{m+r}(t,x_1,\theta)}\text{ or }|G(t,x_1,\theta)|_{H^{m+r}(t,x_1,\theta)}
$$
of the constituent functions.\footnote{The index $m+r$ can often be reduced in the case of $G$ norms, but we wish to lighten the notation by not indicating this.}  In view of the special way  $x_2$ dependence enters into the definitions of $u_\sigma$ and $u_\tau$ (recall \eqref{o19a}, the estimates take exactly the same form if the $\langle \cdot\rangle_{m,\gamma}$ norms on the left are replaced by 
$|\cdot |_{0,m,\gamma}$ or $|\cdot|_{\infty,m,\gamma}$ norms.

The constant $p>0$ is the one that appears in the low-frequency cutoff $\chi_\eps(k)=\chi(\frac{k}{p})|_{p=\eps^b}$, for $b>0$ to be chosen.   We recall that $\chi\geq 0$ and 
$\mathrm{supp}\;\chi(k/p)\subset \{k:|k|\geq p\}$.

We will sometimes use $\partial_f$ to represent $\partial_\theta$ or $\partial_z$ and $\partial_s$ to represent $\partial_t$, $\partial_{x_1}$, or $\partial_{x_2}$.

\begin{rem}\label{qq1}
In the estimates of this section we write $u_\sigma$, $u_\tau$ to indicate the Seeley extensions $u^s_{\sigma,T}$, $u^s_{\tau,T}$.  Recall that the Seeley extension of $u_{\tau,T}$ is not the same as the extension of $u_{\tau,T}$ obtained by replacing the constituent functions $\sigma_{j,T}$ by their Seeley extensions.  Nevertheless, we need to estimate $u^s_{\tau,T}$ in terms of the $\sigma^s_{j,T}$.  This is justified by noting that, for example,
\begin{align}
\langle\Lambda^{\frac{1}{2}}\partial_f u^s_{\tau,T}\rangle_{m,\gamma}\leq C \langle\Lambda^{\frac{1}{2}}\partial_f u^s_{\tau,T}\rangle_{m,T,\gamma}
\end{align}
and the  time-localized norm on the right can be realized as an \emph{infimum} over norms of extensions of $u_{\tau,T}$  to the full-half space.

\end{rem}


We begin with estimates of $u_\sigma$.    

\begin{prop}\label{qq0}
Let $m\geq 0$.  
\begin{align}
\begin{split}
&\langle\Lambda^{\frac{k}{2}}\partial_{f } u_\sigma\rangle_{m,\gamma}\lesssim \frac{C(m+1+\frac{k}{2})}{\eps^{k/2}},\;  k=0,1,2,3,4\\
&\langle\Lambda^{\frac{k}{2}}\partial_{ff} u_\sigma\rangle_{m,\gamma}\lesssim \frac{C(m+2+\frac{k}{2})}{\eps^{k/2}},\;k=0,1,2,3\\
&\langle\Lambda^{\frac{k}{2}}\partial_{s } u_\sigma\rangle_{m,\gamma}\lesssim \frac{C(m+1+\frac{k}{2})}{\eps^{k/2}},\;  k=0,1,2,3,4.
\end{split}
\end{align}
\end{prop}

\begin{proof}

Since  $|X,\gamma |\lesssim |\xi',\gamma |+|\frac{k}{\eps}|$,    we can for $r\geq 0$ estimate norms of $\Lambda_D^r w$ by estimating pieces corresponding to 
$ |\xi',\gamma |^r$ and $|\frac{k}{\eps}|^r$.   Let us write the operators associated to these pieces as $\Lambda_{x',D}^r$ and $\Lambda_{\theta,D}^r$ respectively.

As in the proof of Proposition \ref{o23}  we will use 
 \begin{align}\label{qq2}
 |u(t,x_1,k)|_{\hat H^m(t,x_1,k)}\lesssim   |u(t,x_1,k)|_{L^2(k, H^m(t,x_1))}+|u(t,x_1,k)|_{L^2(t,x_1,\hat H^m(k))}.
\end{align}
Given the form \eqref{p0} of $u_\sigma$ and recalling that 
\begin{align}\label{qqq2}
\mathrm{supp}\;\hat\sigma_j(t,x,k)\subset\{k: k\mathrm{Im}\;\omega_j\geq 0\},  
\end{align}
the  estimates are  immediate.

\end{proof}


The next proposition gives  estimates of $\partial_f u_\tau$. 
\begin{prop}Let $m>\frac{d+1}{2}$.  
\begin{align}
\begin{split}
&\langle\partial_f u_\tau\rangle_{m,\gamma}\lesssim C(m+2)/p\\
&\langle\Lambda^{\frac{1}{2}}\partial_f u_\tau\rangle_{m,\gamma}\lesssim \frac{C(m+2)}{\eps^{1/2}p^{1/2}}+\frac{C(m+2+\frac{1}{2})}{p}\\
&\langle\Lambda\partial_f u_\tau\rangle_{m,\gamma}\lesssim \frac{C(m+2)}{\eps}+\frac{C(m+3)}{p}\\
&\langle\Lambda^{\frac{3}{2}}\partial_f u_\tau\rangle_{m,\gamma}\lesssim \frac{C(m+2+\frac{1}{2})}{\eps^{3/2}}+\frac{C(m+3+\frac{1}{2})}{p}\\
&\langle\Lambda^{2}\partial_f u_\tau\rangle_{m,\gamma}\lesssim \frac{C(m+3)}{\eps^2}+\frac{C(m+4)}{p}.
\end{split}
\end{align}
\end{prop}

\begin{proof}

\textbf{1. }We give the proof of the second estimate, which is typical of the rest. 
Consider, for example,  $\langle\Lambda^{1/2}\partial_\theta u_\tau\rangle_{m,\gamma}$ and in particular $\langle\Lambda^{1/2}\partial_\theta \;\chi_\eps(D_\theta)\tau^p_1\rangle_{m,\gamma}$.   
We estimate first the part of $\tau^p_1$ considered in \eqref{o24a}, where $k<0$ and $m\neq n$:
Recall 
\begin{align}\label{qq4}
\begin{split}
&\hat\tau^p_{mn}(t,x,k,z)=e^{ik\omega_1z}\int_{+\infty}^z\int e^{-ik\omega_1s+i(k-k')\omega_ms+ik'\omega_n s}\frac{1}{k}\widehat{\partial_\theta\sigma_m}(t,x,k-k')\widehat{\partial_{\theta\theta}\sigma_n}(t,x,k')dk'ds:=\\
&\qquad \int^z_{+\infty}e^{ik\omega_1(z-s)}\hat F_{mn}(t,x,k,s)ds.
\end{split}
\end{align}
Here we will suppress some expressions involving $\gamma$ such as factors of $e^{-\gamma t}$. 
Using \eqref{qq2} and \eqref{qqq2} we estimate $\langle\Lambda^{1/2}_\theta\partial_\theta \;\chi_\eps(D_\theta) \tau^p_{mn}\rangle_{m,\gamma}$ as follows:
\begin{align}\label{qq3}
\begin{split}
 &\left|\frac{|k|^{3/2}}{\eps^{1/2}}\;\chi(k/p)\int^{z}_{+\infty}\int e^{ik\omega_1(z-s)}| \frac{1}{k}\widehat{\partial_\theta\sigma_m}(t,x,k-k')\widehat{\partial_{\theta\theta}\sigma_n}(t,x,k')|_{H^{m}(t,x_1)}dk'ds\right|_{L^2(k)}\lesssim\\
&\left|\frac{1}{\eps^{1/2}p^{1/2}}\;\int |\widehat{\partial_\theta\sigma_m}(t,x,k-k')\widehat{\partial_{\theta\theta}\sigma_n}(t,x,k')|_{H^{m}(t,x_1)}dk'\right|_{L^2(k)}\lesssim \frac{C(m+2)}{\eps^{1/2}p^{1/2}}.
\end{split}
\end{align}
Here we have taken account of the factor of $1/k$ introduced by the $ds$ integral, and have
estimated the convolution in $k$ using the observation \eqref{o25aa} and Young's inequality.
When $\Lambda^{1/2}_{\theta,D}$ is replaced by $\Lambda^{1/2}_{x',D}$  a parallel argument shows that the right side of \eqref{qq3} is replaced by $\frac{C(m+2+\frac{1}{2})}{p}$.
The $L^2(t,x_1,\hat H^m(k))$ norm is estimated similarly after writing $\langle k\rangle^{1/2}\lesssim \langle k-k'\rangle^{1/2}+\langle k'\rangle^{1/2}$.

When $k>0$ the integral $\int^z_{+\infty}\dots ds$ in \eqref{qq3} is replaced by $\int^z_0\dots ds$, but the resulting estimate is the same.  The cases $m=n$ also yield the same estimate\footnote{One actually gets a better estimate in this case, since $\partial_\theta\sigma_m\partial_{\theta\theta}\sigma_m=\frac{1}{2}\partial_\theta(\partial_\theta\sigma_m)^2$.}; this is true even in the case $m=n=1$ leading to secular growth for the reason given in step 2 of the proof of Proposition \ref{o23}.

As in step 3 of the proof of Proposition \ref{o23}, we find that $\tau^h_1$ satisfies the same estimate.  

\textbf{2. }To estimate $\langle\Lambda^{1/2}\partial_z\; \chi_\eps(D_\theta)\tau^p_1\rangle_{m,\gamma}$  we first use \eqref{qq4} to write
\begin{align}\label{qq5}
\partial_z\hat\tau^p_{mn}(t,x,k,z)=\hat F_{mn}(t,x,k,z)+\int^z_{+\infty}ik\omega_1 e^{ik\omega_1(z-s)}\hat F_{mn}(t,x,k,z)ds.
\end{align}
The contribution of the second term can be estimated exactly as the corresponding term in step 1.  Moreover, since the $ds$ integral introduces a factor of $1/k$ it is clear that the contribution of the first term satisfies the same estimate.  The remaining details are  straightforward.

\end{proof}

Next we consider  estimates of $\partial_{ff}u_\tau$.
\begin{prop}\label{qq6} Let $m>\frac{d+1}{2}$.   
\begin{align}
\begin{split}
&\langle\Lambda^{k/2}\partial_{\theta\theta} u_\tau\rangle_{m,\gamma}\lesssim \frac{C(m+2+\frac{k}{2})}{\eps^{k/2}}, k=0,1,\dots,3\\
&\langle\partial_{zz} u_\tau\rangle_{m,\gamma}\lesssim C(m+3)/p\\
&\langle\Lambda^{\frac{1}{2}}\partial_{zz} u_\tau\rangle_{m,\gamma}\lesssim \frac{C(m+3)}{\eps^{1/2}p^{1/2}}+\frac{C(m+3+\frac{1}{2})}{p}\\
&\langle\Lambda^{\frac{3}{2}}\partial_{zz} u_\tau\rangle_{m,\gamma}\lesssim \frac{C(m+3+\frac{1}{2})}{\eps^{3/2}}+\frac{C(m+4+\frac{1}{2})}{p}\\
&\langle\partial_{\theta z}u_\tau\rangle_{m,\gamma}\lesssim C(m+2)\\
&\langle\Lambda^{\frac{1}{2}}\partial_{\theta z} u_\tau\rangle_{m,\gamma}\lesssim \frac{C(m+2+\frac{1}{2})}{\eps^{1/2}}\\
&\langle\Lambda^{\frac{3}{2}}\partial_{\theta z} u_\tau\rangle_{m,\gamma}\lesssim \frac{C(m+3+\frac{1}{2})}{\eps^{3/2}}.
\end{split}
\end{align}

\end{prop}

\begin{proof}
The $\partial_{\theta\theta}$ estimates are obvious.   For the estimates involving $\partial_{zz}\hat\tau^p_{mn}$ for $\hat\tau^p_{mn}$ as in \eqref{qq4}, we use \eqref{qq5} to obtain
\begin{align}
\begin{split}
&\partial_{zz}\hat\tau^p_{mn}=\int [i(k-k')\omega_m+ik'\omega_n ]e^{i(k-k')\omega_mz+ik'\omega_n z}\frac{1}{k}\widehat{\partial_\theta\sigma_m}(t,x,k-k')\widehat{\partial_{\theta\theta}\sigma_n}(t,x,k')dk+\\
&ik\omega_1\hat F_{mn}(t,x,k,z)+\int^z_{+\infty}(ik\omega_1)^2e^{ik\omega_1(z-s)}\hat F_{mn}(t,x,k,s)ds=A+B+C.
\end{split}
\end{align}
The  term $C$ is estimated exactly like $\partial_{\theta\theta}\hat\tau^p_{mn}$, and $B$ satisfies the same estimates (since the $ds$ integral in $C$ introduces a factor of $1/k$).  One shows by  arguments already used in this section that in each $\partial_{zz}$ estimate the contribution of $A$ is dominated by the terms that appear on the right in  Proposition \ref{qq6}.   The $\partial_{\theta z}$ estimates are similar but simpler.

\end{proof}

Finally, we have estimates of $\partial_s u_\tau$.


\begin{prop}Let $m>\frac{d+1}{2}$.  
\begin{align}\label{q6}
\begin{split}
&\langle\partial_s u_\tau\rangle_{m,\gamma}\lesssim C(m+3)/p^2\\
&\langle\Lambda^{\frac{1}{2}}\partial_s u_\tau\rangle_{m,\gamma}\lesssim \frac{C(m+3)}{\eps^{1/2}p^{3/2}}+\frac{C(m+3+\frac{1}{2})}{p^2}\\
&\langle\Lambda\partial_{s} u_\tau\rangle_{m,\gamma}\lesssim \frac{C(m+3)}{\eps p}+\frac{C(m+4)}{p^2}\\
&\langle\Lambda^{\frac{3}{2}}\partial_s u_\tau\rangle_{m,\gamma}\lesssim \frac{C(m+3)}{\eps^{3/2}p^{1/2}}+\frac{C(m+4+\frac{1}{2})}{p^2}\\
&\langle\Lambda^2\partial_{s} u_\tau\rangle_{m,\gamma}\lesssim \frac{C(m+3)}{\eps^2}+\frac{C(m+5)}{p^2}
\end{split}
\end{align}
\end{prop}

\begin{proof}
We omit the details of the proof, since it uses only the arguments given earlier in this section.
\end{proof}

\begin{rem}\label{qq7}
Observe that $\partial_{sf}$ (resp. $\partial_{ss}$) estimates of $u_\sigma$ or $u_\tau$ can be obtained immediately from $\partial_f$ (resp. $\partial_s$) estimates by increasing the argument of $C(\cdot)$ by $1$ for each $C(\cdot)$ that appears.  For example,
\begin{align}
\langle\Lambda^{\frac{1}{2}}\partial_f u_\tau\rangle_{m,\gamma}\lesssim \frac{C(m+2)}{\eps^{1/2}p^{1/2}}+\frac{C(m+2+\frac{1}{2})}{p}\text{ implies }\langle\Lambda^{\frac{1}{2}}\partial_{sf} u_\tau\rangle_{m,\gamma}\lesssim \frac{C(m+3)}{\eps^{1/2}p^{1/2}}+\frac{C(m+3+\frac{1}{2})}{p}.
\end{align}
Similarly, estimates of $u_\sigma$ or $u_\tau$ (without the $\partial_s$) can be obtained by decreasing the argument of $C(\cdot)$ in $\partial_s$ estimates. 
\end{rem}

\section{Forcing estimates}\label{forcing}

 \emph{\qquad}\textbf{Interior forcing.  }Here we give estimates of the terms appearing in  $F_a^e(t,x,\theta)$ \eqref{p0bb}.  The estimates starting in Proposition \ref{r3} are based on the assumption that $1\geq p\geq \eps^{\frac{1}{2}-\delta}$  for some $\delta>0$.  In that case the building block estimates show that any given norm of $u_\sigma$  is always  bigger than (i.e., satisfies an estimate with a bigger right hand side) than  the same norm of $\eps u_\tau$.\footnote{When constructing profiles it is not uncommon, but more importantly not circular, to assume at one stage that a profile has a given form, and then later to construct a profile with that form.}   For example,
 \begin{align}
 \begin{split}
 &\langle\Lambda^{3/2}\partial_{zz}u_\sigma\rangle_{m,\gamma}\lesssim \frac{1}{\eps^{3/2}},\text{ while }\langle\Lambda^{3/2}\eps\partial_{zz}u_\tau\rangle_{m,\gamma}\lesssim \eps\left(\frac{1}{\eps^{3/2}}+\frac{1}{p}\right)\\
 &\langle \partial_s u_\sigma\rangle_{m,\gamma}\lesssim 1, \text{ while }\langle\partial_s \;\eps u_\tau\rangle_{m,\gamma}\lesssim \frac{\eps}{p^2}.
 \end{split}
 \end{align}
This observation allows us to greatly reduce the number of estimates that have to be done in estimating $F^e_a$ or $g^e_a$.

We have the following estimates for the ``low frequency cutoff" error in \eqref{fa}.
\begin{prop}\label{r1}
Let $m>\frac{d+1}{2}$. 
\begin{align}\label{rr0}
\begin{split}
&(a)\;|\frac{1}{\eps}\Lambda^{1/2}[\eps(1-\chi_\eps(D_\theta))(L_{fs}u_\sigma+N(u_\sigma))]|_{0,m,\gamma}\leq C(m+1)\frac{p^3}{\eps^{1/2}}+ C(m+2+\frac{1}{2})p^{1/2}\\
&(b)\;|\Lambda^{3/2}[\eps(1-\chi_\eps(D_\theta))(L_{fs}u_\sigma+N(u_\sigma))]|_{0,m,\gamma}\leq C(m+2)\frac{p^2}{\eps^{1/2}}+ C(m+3+\frac{1}{2})\eps.
\end{split}
\end{align}
\end{prop}

\begin{proof}
\textbf{1. Preliminaries. }Since  $|X,\gamma |\lesssim |\xi',\gamma |+|\frac{k}{\eps}|$,    we can for $r\geq 0$ estimate norms of $\Lambda_D^r w$ by estimating pieces corresponding to 
$ |\xi',\gamma |^r$ and $|\frac{k}{\eps}|^r$, thereby taking advantage of the helpful factor $|k|^r$ which is $\lesssim p^r$ on $\mathrm{supp }(1-\chi(k/p))\subset [-p,p]$.

We will again use \eqref{qq2} and we 
recall from \eqref{o19a} that
\begin{align}\label{neglect}
\hat\sigma_j(t,x_1,x_2,k)=\psi(x_2)\hat\sigma_j(t,x_1,0,k), \text{ where }\psi\in C^\infty_c \text{ and }\psi=1 \text{ near }x_2=0.
\end{align}
Below it will be convenient to suppress some arguments $(t,x_1,x_2)$ as well as some expressions involving $\gamma$ such as factors of $e^{-\gamma t}$.   

\textbf{2.  Part (a).} To estimate the $L^2(x_2,L^2(k, H^m(t,x_1)))$ norm of the term involving $N(u_\sigma)$ we consider first
\begin{align}\label{rr2}
\left|\frac{\sqrt{|k|}}{\sqrt{\eps}}(1-\chi(k/p))\int |\widehat{\partial_\theta\sigma}_m(k-k')|_{H^m(t,x_1)}e^{i(k-k')\omega_m\frac{x_2}{\eps}}|\widehat{\partial_{\theta\theta}\sigma}_n(k')|_{H^m(t,x_1)}e^{ik'\omega_n\frac{x_2}{\eps}}dk'\right|_{L^2(k,x_2)}:= A,
\end{align}
and recall that the support properties of the $\hat\sigma_j$ imply both $i(k-k')\omega_m\frac{x_2}{\eps}\leq 0$ and $ik'\omega_n\frac{x_2}{\eps}\leq 0$ on the support of the integrand in \eqref{rr2}.  We 
break up the $k'$ integral into pieces on $|k'|\leq p$, $|k'|\geq p$ that we call $A_1$, $A_2$ respectively.  To estimate $A_1$ we ignore the exponentials (which are $\leq 1$), 
write 
\begin{align}
|\widehat{\partial_{\theta\theta}\sigma_n}(k')|=|k'^2\hat\sigma_n(k')|\leq p^2|\hat\sigma_n(k')|,
\end{align}
and use Cauchy-Schwarz to estimate the $dk'$ integral to obtain
\begin{align}
A_1\lesssim \frac{\sqrt{p}}{\sqrt{\eps}}\cdot p^2\cdot C(m+1)\cdot \sqrt{p}.
\end{align}
Here the last $\sqrt{p}$ is $|1|_{L^2(\{|k|\leq p\})}$.

To estimate $A_2$ we ignore the first exponential and integrate the (square of the) second one to obtain
\begin{align}
|e^{ik'\omega_n\frac{x_2}{\eps}}|_{L^2(x_2)}\lesssim \frac{\sqrt{\eps}}{\sqrt{p}},
\end{align}
after using \eqref{neglect} to neglect the $x_2$ dependence in the $\hat\sigma_j$.  Thus we find
\begin{align}
A_2\lesssim   \frac{\sqrt{p}}{\sqrt{\eps}} \cdot \frac{\sqrt{\eps}}{\sqrt{p}} \cdot C(m+2)\cdot \sqrt{p},
\end{align}
where the final $\sqrt{p}$ arises as before.

For the analogue of $A$ obtained when $\Lambda^{1/2}_{\theta,D}$  is replaced by $\Lambda^{1/2}_{x',D}$, we clearly obtain $$A\leq C(m+2+\frac{1}{2})\sqrt{p},$$ so this completes the $L^2(x_2,L^2(k, H^m(t,x_1)))$ estimate in part (a) for the term that involves $N(u_\sigma)$.    Writing 
\begin{align}
\langle k\rangle^m\lesssim \langle k-k'\rangle^m+\langle k'\rangle^m,
\end{align}
we obtain the same estimate for the $L^2(x_2,L^2(t,x_1,\hat H^m(k)))$ norm of the term involving $N(u_\sigma)$ in part (a) by very similar arguments.

\textbf{3. } To estimate  the $L^2(x_2,L^2(k, H^m(t,x_1)))$ norm of the term involving $L_{fs}u_\sigma$ in \eqref{rr0}(a), instead of \eqref{rr2} we first consider
\begin{align}\label{rr3}
\left|\frac{\sqrt{|k|}}{\sqrt{\eps}}(1-\chi(k/p))k|\widehat{\partial_{s}\sigma}_m(k)|_{H^m(t,x_1)}e^{ik\omega_m\frac{x_2}{\eps}}\right|_{L^2(k,x_2)}:= A,
\end{align}
whence
\begin{align}
A\lesssim C(m+1)p.
\end{align}
Here again we have used \eqref{neglect} and computed $|e^{ik\omega_m\frac{x_2}{\eps}}|_{L^2(x_d)}\lesssim \frac{\sqrt{\eps}}{\sqrt{|k|}}$.

For the analogue of $A$ obtained when $\Lambda^{1/2}_{\theta,D}$  is replaced by $\Lambda^{1/2}_{x',D}$, we have $$A\leq C(m+1+\frac{1}{2})p,$$ so this completes the $L^2(x_2,L^2(k, H^m(t,x_1)))$ estimate in part (a) for the term that involves $L_{fs}u_\sigma$. 
Again, we obtain the same estimate for the $L^2(x_2,L^2(t,x_1,\hat H^m(k)))$ norm.

\textbf{4. Part (b)}.  This estimate is simpler than the one in part (a), since one can take advantage of the helpful factor $|k|^{3/2}$ in $\frac{|k|^{3/2}}{\eps^{3/2}}$; in particular, there is no need to break up the $dk'$ integral as in step 2 or compute $L^2(x_d)$ norms of exponentials.   Thus, we  omit the details.

\end{proof}

We now omit some of the constants $C(m+r)$.  Eventually we will take $r$ to be the largest such index appearing in the building block estimates.  
\begin{prop}\label{r3}
Let $m>\frac{d+1}{2}$.  For $k=1,3$ we have
\begin{align}
\begin{split}
&(a)\;\langle \Lambda^{k/2} L_{ss} u_a\rangle_{m,\gamma}\lesssim \eps^{2-\frac{k}{2}}\\
&(b)\;\langle \Lambda^{k/2}(\eps^2 L_{fs}u_\tau)\rangle_{m,\gamma} \lesssim \eps^{2-\frac{k}{2}}\\
&(c)\;\langle \Lambda^{k/2} [\sum_{|\alpha|=2}Q_\alpha(D_\eps u_a)D_\eps^\alpha u_a  ]\rangle_{m,\gamma}\lesssim \eps^{2-\frac{k}{2}}\\
&(d)\;\langle\Lambda^{k/2}[\sum_{|\alpha|=2}L_\alpha(D_\eps u_a)D_\eps^\alpha u_a  -\eps N(u_\sigma)]\rangle_{m,\gamma}\lesssim \eps^{2-\frac{k}{2}}
\end{split}
\end{align}
\end{prop}

\begin{proof}
\textbf{1. Preliminaries. }Recall 
\begin{align}
\begin{split}
&u_a=\eps^2 u_\sigma+\eps^3 u_\tau\\
&D_\eps=(\partial_{x_1}+\frac{\beta_1}{\eps}\partial_\theta,\partial_{x_2}+\frac{1}{\eps}\partial_z).
\end{split}
\end{align}
With some abuse we will write $D_\eps=\partial_s+\frac{\partial_f}{\eps}$ and will, for example, write products of derivatives of possibly different components of $u_\sigma$  as 
``$\partial_f u_\sigma \partial_f u_\sigma$".  

We will use the observation explained at the beginning of this section that $u_\sigma$ is always bigger than $\eps u_{\tau}$ in any given building block norm as long as 
$0<\eps^{\frac{1}{2}-\delta}\leq p\leq 1$.   Below we take $k=3$ to illustrate the proofs.

\textbf{2. (a). }  The biggest contribution is by Proposition \ref{qq0} and Remark \ref{qq7}:
\begin{align}
\eps^2\langle\Lambda^{3/2}\partial_{ss}u_\sigma\rangle_{m,\gamma}\lesssim \eps^2\cdot \frac{1}{\eps^{3/2}}=\eps^{1/2}.
\end{align}

\textbf{3. (c). } For a given $\alpha$ the largest contribution comes from products $\eps^6 \left(\frac{\partial_f u_\sigma}{\eps} \frac{\partial_f u_\sigma}{\eps} \frac{\partial_{ff} u_\sigma}{\eps^2} \right)$.  Using Corollary \ref{f2}(a) and Proposition \ref{qq0} we obtain
\begin{align}
\eps^2\langle\Lambda^{3/2}(\partial_f u_\sigma\partial_f u_\sigma \partial_{ff}u_\sigma)\rangle_{m,\gamma}\lesssim \eps^2 \cdot \frac{1}{\eps^{3/2}}=\eps^{1/2}.
\end{align}

\textbf{4. (d). }After $\eps N(u_\sigma)$ is subtracted away and after examining the size of terms involving products $\partial_fu_\tau\partial_{ff}u_\tau$,  one sees that the worst terms for a given $\alpha$ are $\eps^4\left(\partial_s u_\sigma \frac{\partial_{ff}u_\sigma}{\eps^2}\right)$ and $\eps^4\left(\frac{\partial_f u_\sigma}{\eps} \frac{\partial_{sf}u_\sigma}{\eps}\right)$.    For these we obtain the desired estimate as above.

\end{proof}

The next corollary is an immediate consequence of the previous two Propositions.

\begin{cor}\label{r3a}
Let $m>\frac{d+1}{2}$. 
\begin{align}
\begin{split}
&(a)\;\Lambda^{1/2}\partial_{x_1,\eps}F_a^e|_{0,m,\gamma}\lesssim \sqrt{\eps}+ (\frac{p^2}{\sqrt{\eps}}+\eps)+\sqrt{\eps}\lesssim \sqrt{\eps}+ \frac{p^2}{\sqrt{\eps}}\\
&(b)\;\frac{1}{\eps}|\Lambda^{1/2}F^e_a|_{0,m,\gamma}\leq \frac{p^3}{\eps^{1/2}}+p^{1/2}+\sqrt{\eps}.
\end{split}
\end{align}

\end{cor}



\begin{prop}\label{r3b}
There exists a constant $T_2>0$ such that $v^s_a=\nabla_\eps u^s_a$ satisfies
\begin{align}
E_{m,T}(v_a^s)\lesssim 1 
\end{align}
uniformly for $\eps\in (0,1]$ and $0\leq T\leq T_2$. \footnote{Recall that $v^s_a$ depends on $T$ as a parameter.}
\end{prop}

\begin{proof}
To prove this one must show that the quantities 
$|\Lambda^{3/2}v^s_a|_{0,m,\gamma}$,  $|\Lambda v^s_a|_{\infty,m,\gamma}$,  $|\frac{\nabla_\eps u^s_a}{\eps}|_{\infty,m,\gamma}$, etc., appearing in the definition of $E_{m,T}(v_a^s)$, are all $\lesssim 1$.   

Consider for example $|\Lambda^{3/2} v^s_a|_{0,m,\gamma}$.   The worst terms are $\eps^2 |\Lambda^{3/2}_\theta D_\eps u_\sigma|_{0,m,\gamma}$, for which we perform the $L^2(x_2,k,H^m(t,x_1))$ estimate as follows.   Writing $D_\eps=\partial_s+\frac{\partial_f}{\eps}$ we obtain for the worst ($\partial_f/\eps$) part:
\begin{align}
\begin{split}
&\eps\left|\frac{|k|^{3/2}}{\eps^{3/2}}k\psi(x_2)|\hat\sigma_m(t,x_1,0,k)|_{H^m(t,x_1)}e^{ik\omega_m\frac{x_2}{\eps}}\right|_{L^2(x_2,k)}\lesssim \\
&\qquad\qquad \left|k^2 |\hat\sigma_m(t,x_1,0,k)\right|_{H^m(t,x_1)}|_{L^2(k)}\lesssim C(m+2),
\end{split}
\end{align}
since $|e^{ik\omega_m\frac{x_2}{\eps}}|_{L^2(x_2)}\lesssim \frac{\eps^{1/2}}{|k|^{1/2}}$. 
The remaining estimates are similar or simpler.
\end{proof}

\textbf{Boundary forcing estimates.  }
Here we carry out the estimation of  $g_a^e$ \eqref{p0bb}.  The proof of the next proposition is similar to that of propositions \ref{r1} and \ref{r3} but simpler.

\begin{prop}\label{r4}
Let $m>\frac{d+1}{2}$.  For $k=0,\dots,4$ we have
\begin{align}
\begin{split}
&\langle \Lambda^{k/2}\eps^2 G\rangle_{m,\gamma}\lesssim \eps^{2-\frac{k}{2}}\\
&\langle \Lambda^{k/2} \eps^3 \ell_s(u_\tau)\rangle_{m,\gamma}\lesssim \frac{\eps^{3-\frac{k}{2}}}{p^{2-\frac{k}{2}}}\\
&\langle \Lambda^{k/2}c(D_\eps u_a)\rangle_{m,\gamma}\lesssim \eps^{3-\frac{k}{2}}\\
&\langle\Lambda^{k/2}[q(D_\eps u_a)-\eps^2 n(u_\sigma)]\rangle_{m,\gamma}\lesssim   \eps^{3-\frac{k}{2}}\\
&\langle\Lambda^{k/2}[\eps^2(\chi_\eps(D_\theta)-1)G]\rangle_{m,\gamma}\leq p^{k/2}\eps^{2-\frac{k}{2}}C(m)+\eps^2C(m+\frac{k}{2})\\
&\langle\Lambda^{k/2}[\eps^2(1-\chi_\eps(D_\theta))(\ell_s(u_\sigma)+n(u_\sigma))]\rangle_{m,\gamma}\leq p^{k/2}\eps^{2-\frac{k}{2}}C(m+1)+\eps^2C(m+1+\frac{k}{2}).
\end{split}
\end{align}

\end{prop}

This yields the immediate corollary:

\begin{cor}\label{r4a}
Let $m>\frac{d+1}{2}$. 
\begin{align}\label{s9}
\begin{split}
&(a)\;\langle g-g^e_a\rangle_{m,\gamma}\lesssim \frac{\eps^3}{p^2}+\eps^3+\eps^2\lesssim \frac{\eps^3}{p^2}+\eps^3+\eps^2\lesssim \eps^2\\
&(b)\;\langle \Lambda^{1/2}(g-g^e_a)\rangle_{m,\gamma}\lesssim \frac{\eps^{5/2}}{p^{3/2}}+\eps^{5/2}+p^{1/2}\eps^{3/2}+\eps^2 \lesssim \frac{\eps^{5/2}}{p^{3/2}}+p^{1/2}\eps^{3/2}\\
&(c)\;\langle \Lambda(g-g^e_a)\rangle_{m,\gamma}\lesssim \frac{\eps^2}{p}+\eps^2+p\eps \lesssim p\eps \\
&(d)\;\langle \Lambda^{3/2}(g-g^e_a)\rangle_{m,\gamma}\lesssim \frac{\eps^{3/2}}{p^{1/2}}+\eps^{3/2}+p^{3/2}\eps^{1/2}+\eps^2 \lesssim\frac{\eps^{3/2}}{p^{1/2}}+p^{3/2}\eps^{1/2}\\
&(e)\;\langle \Lambda^2(g-g^e_a)\rangle_{m,\gamma}\lesssim \eps+p^2+\eps^2 \lesssim p^2.
\end{split}
\end{align}

\end{cor}


\section{Estimates of the extended approximate solution}\label{estapprox}

\emph{\quad}In order to estimate solutions of the error equations \eqref{p4}-\eqref{p6} we first need estimates uniform in $\eps$ for solutions of \eqref{p1}-\eqref{p3} in the $E_{m,\gamma}(v_a)$ norm.    Except for the boundary term in \eqref{p2} the coefficients of \eqref{p1}-\eqref{p3} depend on $v_a^s$ and the forcing is given by $F^e_a$ and $g^e_a$. 

The next proposition and its proof are very similar to Proposition \ref{c5}.    Given the forcing estimates of section \ref{forcing}, the proof is somewhat simpler, as we explain below.
\begin{prop}\label{s1}
Let $m> 3d+4+\frac{d+1}{2}$ and $G\in H^{m+3}(b\Omega)$. There exist positive constants $M_G$ as in  \eqref{MG} and $T_2$ such that the following is true. 
There exist positive constants $\eps_0$,   $\gamma_0$, and there exist increasing functions $Q_i:\mathbb{R}_+\to\mathbb{R}_+$, $i=1,2$, with 
$Q_i(z)\geq z$ such that for $\eps\in (0,\eps_0]$ and each $T$ with $0<T\leq T_2$, the solution to \eqref{p1}-\eqref{p3} satisfies
\begin{align}\label{s2}
E_{m,\gamma}(v_a)\leq \gamma^{-1}E_{m,\gamma}(v_a)Q_1(E_{m,T}(v^s_a))+(\gamma^{-\frac{1}{2}}+\sqrt{\eps})Q_2(E_{m,T}(v^s_a))\text{ for }\gamma\geq \gamma_0.   
\end{align}

\end{prop}

\begin{proof}

\textbf{1. }Contrary to the situation in Proposition \ref{c5},  we already know (Prop. \ref{r3b}) that there exist positive constants $T_2$, $M_0$ such that 
\begin{align}\label{s3}
E_{m,T}(v_a^s)\leq M_0 \text{ for } T\leq T_2,\; \eps\in (0,1].
\end{align}
Moreover, since the constants $C(m+r)$ appearing in the estimates of section \ref{bblock} can be made as small as desired by taking $T_2$ small, the same applies to $M_0$. 

\textbf{2. } The two terms on the right side of \eqref{s2} correspond to the terms on the right in \eqref{c6}; the first comes from commutators, and the second from forcing.  
The commutator analysis here is identical to that given for Prop. \ref{c5}, since it depends only on the control of the coefficients given by \eqref{s3}.  
Moreover, the forcing estimates of section \ref{forcing} show that 
\begin{align}
F_a^e=o(1) \text{ and }g^e_a= \eps^2G+o(1) \text{ as } \eps \to 0
\end{align}
in every one of the norms that is used in the forcing estimate.
Thus, the second term of \eqref{s2} has the same form as the second term of \eqref{c6}.

\end{proof}

\begin{cor}\label{s4a}
Under the assumptions of Prop. \ref{s1} we have, after enlarging $\gamma_0$ if necessary,
\begin{align}\label{s4}
E_{m,\gamma}(v_a)\leq (\gamma^{-\frac{1}{2}}+\sqrt{\eps})Q_2(E_{m,T}(v^s_a))\text{ for }\gamma\geq \gamma_0.   
\end{align}

\end{cor}

\section{Endgame}\label{end}
\emph{\quad} The next Proposition is the analogue of Proposition \ref{g1} for the trio of error equations \eqref{p4}-\eqref{p6}.    The main difference in the proof is that there are now new kinds of forcing terms to deal with.

Henceforth we will denote functions $Q_i(z)$ like those in Proposition \ref{s1} simply by $Q(z)$, and we will allow $Q$ to change from term to term, or even from factor to factor within a given term.   In the next proposition $T_1$ and $T_2$ are the positive constants appearing in Theorem \ref{uniformexistence} and Proposition \ref{s1} respectively.

\begin{prop}\label{s5}
Suppose $m>3d+4+\frac{d+1}{2}$, $G\in H^{m+3}(b\Omega)$,  and $0<\eps^{\frac{1}{2}-\delta}\leq p\leq 1$ for some $\delta>0$.  There exist positive constants $T_3\leq\min\{T_1,T_2\}$, $\eps_0$,   $\gamma_0$
such that for $j\in J_h\cup J_e$, $\eps\in (0,\eps_0]$, $\gamma\geq \gamma_0$, and $T\leq T_3$ the solution to the error system \eqref{p4}-\eqref{p6} satisfies
\begin{align}\label{s6}
\begin{split}
&E_{m,\gamma}(w_1)+\langle\phi_j\Lambda^{\frac{3}{2}}w_1\rangle_{m,\gamma}+\langle\phi_j\Lambda w_1\rangle_{m+1,\gamma}+\sqrt{\eps}\langle\phi_j\Lambda^{\frac{3}{2}}w_1\rangle_{m+1,\gamma}+\frac{1}{\sqrt{\eps}}\langle\phi_j\Lambda w_1\rangle_{m,\gamma}\lesssim\\
&\frac{1}{\sqrt{\gamma}}\{Q_1(E_{m,T}(v^s))\cdot E_{m,\gamma}(w_1)+E_{m,T}(w^s)E_{m+1,\gamma}(v_a)Q_2(E_{m,T}(v^s_a,v^s))+\\
&\qquad E_{m,T}(w^s)Q_3(E_{m,T}(v^s_a,v^s))+(\frac{p^2}{\sqrt{\eps}}+\sqrt{p}+\frac{\eps}{p})\}.
\end{split}
\end{align}
 
\end{prop}

\begin{proof}

To prove \eqref{s6} we use \eqref{b3}-\eqref{b5} to  estimate the solutions of \eqref{p4} and \eqref{p6}.  The arguments have much in common with those of section \ref{mainestimate}, so here we will emphasize what is different.

\textbf{1. }The first term on the right in \eqref{s6} arises by exactly the same commutator analysis as given in section \ref{mainestimate}. See Proposition \ref{g1}.

Let us denote the right sides of the interior equations \eqref{p4}(a)-\eqref{p6}(a) by $\cF_1$, $\cF_2$, $\cF$ respectively, and the right sides of the corresponding boundary equations
by $\cG_1$, $\cG_2$, $\cG$ respectively.   These terms differ from the forcing terms estimated in section 5, and the error analysis depends on their careful estimation.
In particular, to estimate the terms which define $E_{m,\gamma}(w_1)$  in the first, second, and third lines of \eqref{c0a} respectively, one must respectively estimate
\begin{align}\label{s6a}
\begin{split}
&(a)\;|\Lambda^{1/2}\cF_1|_{0,m,\gamma}, |\cF_1|_{0,m+1,\gamma}, \langle\Lambda\cG_1\rangle_{m,\gamma},  \langle\Lambda^{1/2}\cG_1\rangle_{m+1,\gamma}\\
&(b)\;\sqrt{\eps}|\Lambda^{1/2}\cF_1|_{0,m+1,\gamma}, \frac{1}{\sqrt{\eps}}|\cF_1|_{0,m,\gamma}, \sqrt{\eps}\langle\Lambda\cG_1\rangle_{m+1,\gamma},  \frac{1}{\sqrt{\eps}}\langle\Lambda^{1/2}\cG_1\rangle_{m,\gamma}\\
&(c)\;\frac{1}{\eps}|\Lambda^{1/2}\cF|_{0,m,\gamma}, \frac{1}{\eps}|\cF|_{0,m+1,\gamma}, \langle\Lambda\cG\rangle_{m,\gamma}, \frac{1}{\eps} \langle\Lambda^{1/2}\cG\rangle_{m+1,\gamma}.
\end{split}
\end{align}
Below we will focus on estimating the terms in \eqref{s6a}(a) and (c). We will omit the (similar) details for (b).

\textbf{2.  Interior forcing in \eqref{p4}. }We show that for $|\alpha|=2$ 
\begin{align}\label{s7}
\begin{split}
&(a) \;|\Lambda^{1/2}[(A_\alpha(v^s)-A_\alpha(v^s_a))\partial^\alpha_{x,\eps} v_{1a}]|_{0,m,\gamma}\leq E_{m,T}(w^s)E_{m+1,\gamma}(v_{1a})Q(E_{m,T}(v^s_a,v^s))\\
&(b)\;|[(A_\alpha(v^s)-A_\alpha(v^s_a))\partial^\alpha_{x,\eps} v_{1a}]|_{0,m+1,\gamma}\leq E_{m,T}(w^s)E_{m+1,\gamma}(v_{1a})Q(E_{m,T}(v^s_a,v^s)).
\end{split}
\end{align}
We write $A_\alpha(v^s)-A_\alpha(v^s_a))=f(v^s_a,v^s)w^s$, and  reduce to considering the case $\alpha=(2,0)$ by using the equation and the noncharacteristic boundary assumption to treat $\alpha=(1,1)$ or $(0,2)$ .  Applying Corollary \ref{f2}(a) with $x_2$ fixed, we have
\begin{align}
\begin{split}
&\langle \Lambda^{1/2}[f(v^s_a,v^s)w^s\partial^\alpha_{x,\eps}v_{1a}]\rangle_{m,\gamma}\leq\\
&\quad \langle\Lambda^{1/2}\eps\partial^\alpha_{x,\eps}v_{1a}\rangle_{m,\gamma}\langle f(v^s_a,v^s)\frac{w^s}{\eps}\rangle_m+\langle\eps\partial^\alpha_{x,\eps}v_{1a}\rangle_{m,\gamma}\langle\Lambda^{1/2}_1[f(v^s_a,v^s)\frac{w^s}{\eps}]\rangle_m\leq\\
&\qquad \langle\Lambda^{3/2}v_{1a}\rangle_{m+1,\gamma}h(\langle v^s_a,v^s\rangle_m)\langle\frac{w^s}{\eps}\rangle_m +\langle\Lambda v_{1a}\rangle_{m+1,\gamma}\langle\Lambda^{1/2}_1(v^s_a,v^s)\rangle_m h(\langle v^s_a,v^s\rangle_m)\langle\frac{w^s}{\eps}\rangle_m+\\
&\qquad\qquad\langle\Lambda v_{1a}\rangle_{m+1,\gamma} h(\langle v^s_a,v^s\rangle_m) \langle \Lambda^{1/2} \frac{w^s}{\eps}\rangle_m = A+B+C. 
\end{split}
\end{align}
Next take the $L^2(x_d)$ norm of both sides.  In the final estimate the factors $|\Lambda^{3/2}v_{1a}|_{0,m+1,\gamma}$, $|\Lambda v_{1a}|_{0,m+1,\gamma}$, and $| \Lambda^{1/2} \frac{w^s}{\eps}|_{0,m}$ should respectively appear in each of the three terms on the right; the remaining factors in those terms all involve $L^\infty(x_d)$ norms.   This gives \eqref{s7}(a).
   
   After again writing $w^s\partial^\alpha_{x,\eps}v_{1a}=\frac{w^s}{\eps} \eps\partial^\alpha_{x,\eps}v_{1a}$, one can prove \eqref{s7}(b) by an argument similar to the proof of Proposition \ref{e14}.

\textbf{3. }An argument similar to that in step 2 but more straightforward yields
\begin{align}
\begin{split}
&|\Lambda^{1/2}[b(v^s)d_\eps v^sd_\eps v^s-b(v^s_a)d_\eps v^s_a d_\eps v^s_a]|_{0,m,\gamma}\leq E_{m,T}(w^s)Q(E_{m,T}(v^s_a,v^s))\\
&|b(v^s)d_\eps v^sd_\eps v^s-b(v^s_a)d_\eps v^s_a d_\eps v^s_a]|_{0,m+1,\gamma}\leq E_{m,T}(w^s)Q(E_{m,T}(v^s_a,v^s))\\
\end{split}
\end{align}

\textbf{4. }Corollary \ref{r3a}(a) gives
\begin{align}
|\Lambda^{1/2}\partial_{x_1,\eps}F_a^e|_{0,m,\gamma}\lesssim \sqrt{\eps}+ (\frac{p^2}{\sqrt{\eps}}+\eps)+\sqrt{\eps}\lesssim \sqrt{\eps}+ \frac{p^2}{\sqrt{\eps}}.
\end{align}
The same estimate clearly holds for $|\partial_{x_1,\eps}F_a^e|_{0,m+1,\gamma}$.

\textbf{5. Boundary forcing in \eqref{p4}. }
An argument  like that in step 2 yields
\begin{align}\label{ss8}
\begin{split}
&\langle\Lambda[d_{v_1}H(v^s_1,h(v^s))-d_{v_1}H(v^s_{1a},h(v^s_a))]\partial_{x_1,\eps}v_{1a}]\rangle_{m,\gamma}\leq E_{m,T}(w^s)E_{m+1,\gamma}(v_{1a})Q_2(E_{m,T}(v^s_a,v^s))\\
&\langle\Lambda^{1/2}[d_{v_1}H(v^s_1,h(v^s))-d_{v_1}H(v^s_{1a},h(v^s_a))]\partial_{x_1,\eps}v_{1a}]\rangle_{m+1,\gamma}\leq E_{m,T}(w^s)E_{m+1,\gamma}(v_{1a})Q_2(E_{m,T}(v^s_a,v^s)).
\end{split}
\end{align}

We claim
\begin{align}\label{s8}
\langle\Lambda [d_gH(v^s_{1},g)\partial_{x_1,\eps}g - d_gH(v^s_{1a},g^e_a)\partial_{x_1,\eps} g^e_a]\rangle_{m,\gamma}\lesssim E_{m,T}(w^s)Q(E_{m,T}(v^s_1,v^s_{1a}))+\eps+p^2,
\end{align}
and the same estimate holds for $\langle\Lambda^{1/2} [d_gH(v^s_{1},g)\partial_{x_1,\eps}g - d_gH(v^s_{1a},g^e_a)\partial_{x_1,\eps} g^e_a]\rangle_{m+1,\gamma}$.

Writing $d_gH(v^s_{1},g)\partial_{x_1,\eps}g - d_gH(v^s_{1a},g^e_a)\partial_{x_1,\eps} g^e_a$ in the obvious way as a sum of three differences, we see that  Corollary \ref{r4a}(a),(c),(e) and Corollary \ref{f2}(a) imply \eqref{s8}.

\textbf{6. Interior forcing in \ref{p6}. }The above estimates show that the terms arising from the application of \eqref{b3}-\eqref{b5} to the system \eqref{p4} are dominated by the right side of 
\eqref{s6}, so we now turn to the $z$-system, \eqref{p6}, which is needed to estimate terms like $|\nabla_\eps z/\eps|_{\infty,m,\gamma}$.\footnote{We will use the fact that a remark analogous to 
Remark \ref{e1z} applies to $w$ and $\nabla_\eps z$. }

Corollary \ref{r3a}(b) gives
\begin{align}
\frac{1}{\eps}|\Lambda^{1/2}F^e_a|_{0,m,\gamma}\leq \frac{p^3}{\eps^{1/2}}+p^{1/2}+\sqrt{\eps}.
\end{align}
Next we show that for $|\alpha|=2$:
\begin{align}\label{s10}
\begin{split}
&(a)\;\frac{1}{\eps}|\Lambda^{1/2}[(A_\alpha(v^s)-A_\alpha(v^s_a))\partial^\alpha_{x,\eps} u_{a}]|_{0,m,\gamma}\leq E_{m,T}(w^s)E_{m,\gamma}(v_{a})Q(E_{m,T}(v^s_a,v^s))\\
&(b)\;\frac{1}{\eps}|[(A_\alpha(v^s)-A_\alpha(v^s_a))\partial^\alpha_{x,\eps} u_{a}]|_{0,m+1,\gamma}\leq E_{m,T}(w^s)E_{m,\gamma}(v_{a})Q(E_{m,T}(v^s_a,v^s))\\
\end{split}
\end{align}
For fixed $x_2$ we estimate for $|\alpha|=1$
\begin{align}
\begin{split}
&\frac{1}{\eps}\langle \Lambda^{1/2}[f(v^s_a,v^s)w^s\partial^\alpha_{x,\eps}v_{a}]\rangle_{m,\gamma}\leq\\
&\quad \langle\Lambda^{1/2}\partial^\alpha_{x,\eps}v_{a}\rangle_{m,\gamma}\langle f(v^s_a,v^s)\frac{w^s}{\eps}\rangle_m+\langle \partial^\alpha_{x,\eps}v_{a}\rangle_{m,\gamma}\langle\Lambda^{1/2}_1[f(v^s_a,v^s)\frac{w^s}{\eps}]\rangle_m,
\end{split}
\end{align}
so applying rules of section \ref{nonlinear} and  taking the $L^2(x_2)$ norm of both sides gives \eqref{s10}(a).   The proof of \eqref{s10}(b) is parallel to that of \eqref{s7}(b).

\textbf{7. Boundary forcing in \ref{p6}. }As in \eqref{e6}, \eqref{e9} let us write
\begin{align}
H(v^s_1,g)-d_{v_1}H(v^s_1,g)v^s_1=Cg+b(v^s_1,g)(v^s_1,g)^2:=F(v^s_1,g).
\end{align}
We express the  boundary forcing term in \eqref{p6} as 
\begin{align}
\begin{split}
&\cG=[F(v^s_1,g)-F(v^s_{1a},g^e_a)]+[d_{v_1}H(v^s_1,h(v^s))-d_{v_1}H(v^s_{1a},h(v^s_a))]v_{1a}=A+B\\
\end{split}
\end{align}
By arguments similar to step 2 we obtain 
\begin{align}
\frac{1}{\eps}\langle\Lambda B\rangle_{m,\gamma}\leq E_{m,\gamma}(v_{1a})E_{m,T}(w^s) Q(E_{m,T}(v^s,v^s_a)).
\end{align}
Writing $A$ as a sum of two differences, using the special form of $F$, and using \eqref{s9}(a),(c), we obtain
\begin{align}
\begin{split}
&\frac{1}{\eps}\langle\Lambda A\rangle_{m,\gamma}\lesssim E_{m,T}(w^s) Q(E_{m,T}(v^s,v^s_a))+(\frac{\eps^2}{p^2}+\eps^2+\eps)+(\frac{\eps}{p}+\eps+p) \lesssim \\
&\qquad\qquad E_{m,T}(w^s) Q(E_{m,T}(v^s,v^s_a))+\frac{\eps}{p}+p.
\end{split}
\end{align}
This gives
\begin{align}\notag
\frac{1}{\eps}\langle \Lambda(A+B)\rangle_{m,\gamma}\lesssim E_{m,\gamma}(v_{1a})E_{m,T}(w^s) Q(E_{m,T}(v^s,v^s_a))+E_{m,T}(w^s) Q(E_{m,T}(v^s,v^s_a))+\frac{\eps}{p}+p
\end{align}
and similarly
\begin{align}\notag
\frac{1}{\eps}\langle \Lambda^{1/2}(A+B)\rangle_{m+1,\gamma}\lesssim E_{m,\gamma}(v_{1a})E_{m,T}(w^s) Q(E_{m,T}(v^s,v^s_a))+E_{m,T}(w^s) Q(E_{m,T}(v^s,v^s_a))+\frac{\eps}{p}+p.
\end{align}

\textbf{8. }The same interior and boundary forcing estimates are satisfied by the terms in \eqref{s6a}(b).  
Combining these results  we obtain \eqref{s6}.

\end{proof}

By enlarging $\gamma_0$ if necessary and using Corollary \ref{s4a}, we obtain
\begin{cor}\label{s11}
Suppose $m>3d+4+\frac{d+1}{2}$ and $0<\eps^{\frac{1}{2}-\delta}\leq p\leq 1$ for some $\delta>0$.  There exist positive constants $T_3$, $\eps_0$,  and  $\gamma_0$
such that for $j\in J_h\cup J_e$, $\eps\in (0,\eps_0]$, $\gamma\geq \gamma_0$, and $T\leq T_3$:
\begin{align}\label{s12}
\begin{split}
&E_{m,\gamma}(w_1)+\langle\phi_j\Lambda^{\frac{3}{2}}w_1\rangle_{m,\gamma}+\langle\phi_j\Lambda w_1\rangle_{m+1,\gamma}+\sqrt{\eps}\langle\phi_j\Lambda^{\frac{3}{2}}w_1\rangle_{m+1,\gamma}+\frac{1}{\sqrt{\eps}}\langle\phi_j\Lambda w_1\rangle_{m,\gamma}\lesssim\\
&\frac{1}{\sqrt{\gamma}}\{E_{m,T}(w^s)Q(E_{m+1,T}(v^s_{a}))Q(E_{m,T}(v^s_a,v^s))+(\frac{p^2}{\sqrt{\eps}}+\sqrt{p}+\frac{\eps}{p})\}.
\end{split}
\end{align}

\end{cor}

To complete the error analysis we will apply \eqref{b6}-\eqref{b8} to the Dirichlet problem \eqref{p5} in order to estimate $E_{m,\gamma}(w_2)$.
This requires us to estimate the terms in \eqref{s6a}(a),(b) where $\cF_1$, $\cG_1$ are replaced by $\cF_2$, $\cG_2$.   In addition we must estimate norms of $\cG_2$ involving the singular pseudodifferential operators $\phi_{jD}$, $\psi_{jD}$ for $j\in J_h$ (recall Prop. \ref{e21}, for example).  The estimates of $\cF_2$ are identical to those already given of $\cF_1$.  In the next two Propositions we give the required estimates of $\cG_2$.   These estimates are quite similar to those of Propositions \ref{e18}, \ref{e21}, \ref{e22b}, and \ref{e23d} of section \ref{mainestimate}, and are proved in the same way.

\begin{prop}\label{s13}
Suppose $m>3d+4+\frac{d+1}{2}$ and $0<\eps^{\frac{1}{2}-\delta}\leq p\leq 1$ for some $\delta>0$.   Let $\phi_j$ and $\psi_j$, $j\in J_h$ be singular symbols related as in Notation \ref{f4b}.  There exist positive constants $\eps_0$,   $\gamma_0$, $T_3$ such that 
for $\eps\in (0,\eps_0]$, $\gamma\geq \gamma_0$, $T\leq T_3$  we have
\begin{align}\label{s14}
\begin{split}
&(a)\;\langle\Lambda \cG_2\rangle_{m,\gamma}\lesssim \langle \Lambda w_1\rangle_{m,\gamma}Q(E_{m,T}(v^s,v^s_a))+\eps\\
&(b)\;\langle\Lambda^{\frac{1}{2}}\cG_2\rangle_{m+1,\gamma}\lesssim \langle \Lambda^{1/2} w_1\rangle_{m+1,\gamma}Q(E_{m,T}(v^s,v^s_a))+\eps\\
&(c)\;\langle\phi_j\Lambda^{\frac{3}{2}} \cG_2\rangle_{m,\gamma}\lesssim \left(\langle \psi_j\Lambda^{\frac{3}{2}} w_1\rangle_{m,\gamma}+\langle \Lambda w_1\rangle_{m,\gamma}\right)Q(E_{m,T}(v^s,v^s_a))+\sqrt{\eps}\\
&(d)\;\langle\phi_j\Lambda\cG_2\rangle_{m+1,\gamma}\lesssim  \left(\langle \psi_j\Lambda w_1\rangle_{m+1,\gamma}+\langle \Lambda^{\frac{1}{2}} w_1\rangle_{m+1,\gamma}\right)Q(E_{m,T}(v^s,v^s_a))+\eps.
\end{split}
\end{align}
\end{prop}

\begin{prop}\label{s17}
With notation as in the previous proposition we have for $\eps\in (0,\eps_0]$, $\gamma\geq \gamma_0$, $T\leq T_3$: 
\begin{align}\label{s18}
\begin{split}
&(a)\;\sqrt{\eps}\langle\Lambda\cG_2\rangle_{m+1,\gamma}\lesssim \sqrt{\eps}\langle \Lambda w_1\rangle_{m+1,\gamma} Q(E_{m,T}(v^s,v^s_a))+\eps^{3/2}\\
&(b)\;\frac{1}{\sqrt{\eps}}\langle\Lambda^{\frac{1}{2}}\cG_2\rangle_{m,\gamma}\lesssim \frac{1}{\sqrt{\eps}}\langle \Lambda^{\frac{1}{2}} w_1\rangle_{m,\gamma}Q(E_{m,T}(v^s,v^s_a))+\sqrt{\eps}\\
&(c)\;\sqrt{\eps}\langle\phi_j\Lambda^{\frac{3}{2}} \cG_2\rangle_{m+1,\gamma}\lesssim \left(\sqrt{\eps}\langle \psi_j\Lambda^{\frac{3}{2}} w_1\rangle_{m+1,\gamma}+\sqrt{\eps}\langle \Lambda w_1\rangle_{m+1,\gamma}\right)Q(E_{m,T}(v^s,v^s_a))+\eps\\
&(d)\;\frac{1}{\sqrt{\eps}}\langle\phi_j\Lambda\cG_2\rangle_{m,\gamma}\lesssim  \left(\frac{1}{\sqrt{\eps}}\langle \psi_j\Lambda w_1\rangle_{m,\gamma}+\frac{1}{\sqrt{\eps}}\langle \Lambda^{\frac{1}{2}} w_1\rangle_{m,\gamma}\right)Q(E_{m,T}(v^s, v^s_a))+\sqrt{\eps}.
\end{split}
\end{align}
\end{prop}

The following Proposition is the analogue for the error system of Proposition \ref{h1}.

\begin{prop}\label{endgame}
Suppose $m>3d+4+\frac{d+1}{2}$, $G\in H^{m+3}(b\Omega)$,  and $0<\eps^{\frac{1}{2}-\delta}\leq p\leq 1$ for some $\delta>0$.   There exist positive constants $\eps_0$,   $\gamma_0$, $T_3$ such that 
for $\eps\in (0,\eps_0]$, $\gamma\geq \gamma_0$, and $T\leq T_3$  the solution to the error system \eqref{p4}-\eqref{p6} satisfies
\begin{align}\label{s20}
\begin{split}
&E_{m,\gamma}(w_2)\lesssim \frac{1}{\gamma}E_{m,\gamma}(w_2)Q(E_{m,T}(v^s))+\\
&\qquad\qquad \frac{1}{\sqrt{\gamma}}\left[E_{m,T}(w^s)Q(E_{m+1,T}(v^s_a))Q(E_{m,T}(v^s_a,v^s))+(\frac{p^2}{\sqrt{\eps}}+\sqrt{p}+\frac{\eps}{p})\right]+\sqrt{\eps}.
\end{split}
\end{align}

\end{prop}

\begin{proof}
Given the earlier results of this section, the proof is quite similar to that of Proposition \ref{h1}.   Parallel to step 2 of the earlier proof, consider for example the term $\left|\begin{pmatrix}\Lambda w_2\\D_{x_2}w_2\end{pmatrix}\right|_{\infty,m,\gamma}$.   By \eqref{b7} applied to \eqref{p5} we have 
\begin{align}\label{s21}
\begin{split}
&\left|\begin{pmatrix}\Lambda w_2\\D_{x_2}w_2\end{pmatrix}\right|_{\infty,m,\gamma}\lesssim \gamma^{-1}|\cF_{2,com}|_{0,m+1,\gamma}+ \gamma^{-1}|\cF_2|_{0,m+1,\gamma}+\\
&\qquad \qquad (\langle\Lambda \cG_2\rangle_{m,\gamma}+\langle\Lambda^{\frac{1}{2}} \cG_{2}\rangle_{m+1,\gamma})+\gamma^{-\frac{1}{2}}\sum_{j\in J_h}\langle\phi_j\Lambda \cG_2\rangle_{m+1,\gamma}:=A+B+C+D,
\end{split}
\end{align}
where $\cF_{2,com}$ is the interior commutator.   By the interior commutator arguments of section \ref{mainestimate} we have
\begin{align}
A\lesssim \frac{1}{\gamma}E_{m,\gamma}(w_2)Q(E_{m,T}(v^s)).
\end{align}
Corollary \ref{s4a} and the estimates in steps 2, 3, and 4 of the proof of Proposition \ref{s5} apply to $\cF_2$ to show
\begin{align}
B\lesssim \frac{1}{\gamma}\left[E_{m,T}(w^s)Q(E_{m+1,T}(v^s_a))Q(E_{m,T}(v^s_a,v^s))+\sqrt{p}+\frac{p^2}{\sqrt{\eps}}\right].
\end{align}
From Proposition \ref{s13}(a),(b)  we have
\begin{align}
C\lesssim E_{m,\gamma}(w_1)Q(E_{m,T}(v^s_a,v^s))+\sqrt{\eps},
\end{align}
while Proposition \ref{s13}(d) gives
\begin{align}
D\lesssim \frac{1}{\sqrt{\gamma}}\left[ \sum_{j\in J_h} \left(\langle \psi_j\Lambda w_1\rangle_{m+1,\gamma}+\langle \Lambda^{\frac{1}{2}} w_1\rangle_{m+1,\gamma}\right)Q(E_{m,T}(v^s,v^s_a))+\sqrt{\eps}\right].
\end{align}
Next apply Corollary \ref{s11} to obtain
\begin{align}
C+D\lesssim \frac{1}{\sqrt{\gamma}}\left[E_{m,T}(w^s)Q(E_{m+1,T}(v^s_a))Q(E_{m,T}(v^s_a,v^s))+(\frac{p^2}{\sqrt{\eps}}+\sqrt{p}+\frac{\eps}{p})\right]+\sqrt{\eps}.
\end{align}
Thus, we see that $A+B+C+D$ is dominated by the right side of \eqref{s20}.   The other terms making up $E_{m,\gamma}(w_2)$ are estimated in the same way using the earlier results of this section.

\end{proof}

Combining Corollary \ref{s11} with Proposition \ref{endgame} we immediately obtain the following corollary by enlarging $\gamma_0$ if necessary:

\begin{cor}\label{s22}
Suppose $m>3d+4+\frac{d+1}{2}$ and $0<\eps^{\frac{1}{2}-\delta}\leq p\leq 1$ for some $\delta>0$.   There exist positive constants $\eps_0$,   $\gamma_0$, $T_3$ such that 
for $\eps\in (0,\eps_0]$, $\gamma\geq \gamma_0$, and $T\leq T_3$  we have
\begin{align}\label{s23}
\begin{split}
&E_{m,\gamma}(w)\lesssim \frac{1}{\sqrt{\gamma}}\left[E_{m,T}(w^s)Q(E_{m+1,T}(v^s_a))Q(E_{m,T}(v^s_a,v^s))+(\frac{p^2}{\sqrt{\eps}}+\sqrt{p}+\frac{\eps}{p})\right]+\sqrt{\eps}.
\end{split}
\end{align}

\end{cor}

\begin{cor}\label{s24a}

Suppose $m>3d+4+\frac{d+1}{2}$ and $\delta>0$.   There exist positive constants $T_4\leq T_3$, $\eps_0$, and $C_\delta$ such that 
for $\eps\in (0,\eps_0]$,  $T\leq T_4$  the solution to the error system \eqref{p4}-\eqref{p6} satisfies
\begin{align}\label{s24}
\begin{split}
&E_{m,T}(w)\leq C_\delta \eps^{\frac{1}{4}-\delta}.
\end{split}
\end{align}

\end{cor}

\begin{proof}
Since $E_{m,\gamma}(w)\gtrsim e^{-\gamma T}E_{m,T}(w)$,   we obtain from \eqref{s23}
 \begin{align}\label{s25}
\begin{split}
E_{m,T}(w)\lesssim e^{\gamma T} \left(\frac{1}{\sqrt{\gamma}}\left[E_{m,T}(w^s)Q(E_{m+1,T}(v^s_a))Q(E_{m,T}(v^s_a,v^s))+(\frac{p^2}{\sqrt{\eps}}+\sqrt{p}+\frac{\eps}{p})\right]+\sqrt{\eps}\right)
\end{split}
\end{align}
for $\gamma\geq \gamma_0$ as in Corollary \ref{s22}.   Thus, for $T\leq  1/\gamma_0$ we have
 \begin{align}\label{s26}
\begin{split}
E_{m,T}(w)\lesssim \sqrt{T}\left[E_{m,T}(w^s)Q(E_{m+1,T}(v^s_a))Q(E_{m,T}(v^s_a,v^s))+(\frac{p^2}{\sqrt{\eps}}+\sqrt{p}+\frac{\eps}{p})\right]+\sqrt{\eps}
\end{split}
\end{align}
Since $w=w^s_T$ on $\Omega_T$ for $T\leq T_3$, by 
choosing $0<T_4\leq 1/\gamma_0$ small enough and setting $p=\eps^{\frac{1}{2}-2\delta}$, we deduce \eqref{s24} for $T\leq T_4$ from \eqref{s26}.

\end{proof}

\chapter{Some extensions}\label{chapter6}


\section{Extension to general isotropic hyperelastic materials.}\label{generalisotropic}

\emph{\quad} We now describe  the minor changes in the analysis that are needed when assumption (A1) in chapter \ref{chapter3} is replaced by assumption (A1g), that is, when we pass from a Saint Venant-Kirchhoff system to general isotropic hyperelastic materials.     Theorems \ref{uniformexistence}  and \ref{approxthm} both remain true as stated under assumption (A1g).  

The   amplitude equation in Chapter \ref{chapter2} was already treated in  this more general context; in particular,  Proposition \ref{propwellposed} applies to these more general materials.  


 The construction and estimation of the approximate solution in chapter \ref{chapter4} goes through almost without change.  In the definition of the coefficients $A_\alpha$ as in \eqref{Asub},
 \begin{align}
 A_\alpha(v)=A_\alpha(0)+L_\alpha(v)+Q_\alpha(v),
 \end{align}
only the term $Q_\alpha(v)$ changes; it must now be defined as the sum of all terms in the expansion of $A_\alpha(v)$ that are quadratic or of higher order.  Similarly, in the definition of the boundary function $h(v)$ as in \eqref{hofv},
\begin{align}
h(v)=\ell(v)+q(v)+c(v),
\end{align}
only the term $c(v)$ changes; it must now be defined as the sum of all terms in the expansion of $h(v)$ that are cubic or of higher order.   We observe that these changes have no effect on the interior and boundary profile equations \eqref{o3} and \eqref{o4}.  Thus, the construction of the $\sigma_j$ and $\tau_j$  and  the building block estimates are unaffected.

However, these changes do affect the expressions for the interior and boundary error profiles 
$F_a(t,x,\theta,z)$ \eqref{fa} and $h_a(t,x_1,\theta)$ \eqref{ba}, which respectively contain the terms
$$
\sum_{|\alpha|=2}Q_\alpha(D_\eps u_a)D_\eps^\alpha u_a \text{ and } c(D_\eps u_a).
$$
Under assumption (A1) these terms were estimated in Proposition \ref{r3}(c) and in the third estimate of Proposition \ref{r4}.     The estimates remain true as stated under assumption (A1g), but a small change is needed in the proofs.  For example, in the proof of Proposition \ref{r3}(c), instead of using Corollary \ref{f2}(a) for products, one should use Proposition \ref{f3}(c) for analytic functions $f(u)$.\footnote{Here we use the analyticity assumption on $W(E)$ in (A1g).}  The cubic and higher order terms in $Q_\alpha(v)$ are readily seen to make contributions to the estimate that are negligible compared to the quadratic terms.

It remains to check that properties (P1)-(P7) of section \ref{assumptions}  continue to hold under the assumption (A1g). This is clear for (P1)-(P4), since the problem $(P^0,B^0)$ is unchanged.\footnote{Properties (P1)-(P7) were verified for the SVK system in \cite{S-T} under the assumption $\mu>0$, $\lambda>0$.  The discussion in chapter  
\ref{chapter2} and an inspection of \cite{S-T} shows that her arguments apply just as well under the assumption $\mu>0$, $\lambda+\mu>0$.}  The verification of (P5), given in \cite{S-T}, p. 283 for the SVK system, is based on showing that the Lopatinskii determinant $\det b^+(v,z)$ as in \eqref{lop} is real for $\gamma=0$.  The argument is based on symmetry properties of the coefficient matrices $A_\alpha(v)$ defining $P^v$ \eqref{d3}, which arise as consequences of the fact that these matrices are derived from a stored energy function $W(\nabla u)$.  For example, the argument uses the fact that the $(\alpha,\beta)$ component of $A_{(1,1)}(v)$ is given by 
\begin{align}
c_{\alpha1 \beta 2}(v)+c_{\alpha 2 \beta 1}(v), \text{ where } c_{\alpha j \beta \ell}(v)=\frac{\partial^2 W(v)}{\partial u_{\alpha,j}\partial u_{\beta,\ell}},
\end{align}
and thus we have the matrix equalities
\begin{align}
(c_{\alpha 2 \beta 1}(v))=(c_{\beta 1 \alpha 2}(v))=(c_{\alpha 1 \beta 2}(v))^t.
\end{align}
This property is clearly unaffected by the passage from (A1) to (A1g).   Similarly, the verification of (P6) and (P7) on p. 286  of \cite{S-T}  for the SVK system continues to hold under assumption (A1g) because the matrices $A_\alpha(v)$  are constructed from derivatives of $W(v)$.

\section{Extension to wavetrains. }\label{wavetrains}
\emph{\quad} Next we summarize how the method used in this paper to treat surface pulses extends to treat surface wavetrains.   
We explain in Remark \ref{optimal} that another method is available  for wavetrains that should yield more detailed qualitative information about the solutions. 

We restrict attention to the 2D case.   The boundary data \eqref{a3} is now given by $\eps^2 G(t,x_1,\theta)$, where $\theta\in\mathbb{T}=\mathbb{R}/2\pi\mathbb{Z}$.  
    Writing $G(t,x_1,\theta)=\sum_{n\in\mathbb{Z}} G^n(t,x_1)e^{in\theta}$, we assume $G^0=0$.

    We work with exactly the same Sobolev spaces and singular operators as before, except now $\Omega=\{(t,x_1,x_2,\theta)\in\mathbb{R}^3\times \mathbb{T}:x_2>0\}$, 
    $k\in\mathbb{Z}$, where $k$ is the variable dual to $\theta$, and integrals $\int \dots dk$ are replaced by sums over $k\in\mathbb{Z}$.    A version of the  singular calculus for wavetrains  parallel to the calculus for pulses used here was developed in \cite{CGW}; in fact the wavetrain calculus is ``better behaved", since $\theta$ now lies in a compact set.
    
        Except for the changes just mentioned, the results of chapter \ref{chapter3} on the existence of exact solutions hold with identical statements and proofs.  The main changes in the treatment of wavetrains occur in the construction of approximate solutions.  Unlike a pulse a wavetrain has a well-defined mean, and in the construction of approximate solutions one has to separate out the initial boundary value problems satisfied by the means of profiles.   As in \eqref{o2a} we look for the profile of the approximate solution in the form $u_a=\eps^2 u_\sigma+\eps^3 u_\tau$, where now
\begin{align}
\begin{split}
&u_\sigma(t,x,\theta,z)=\sum_{n\in\mathbb{Z}}u_\sigma^n(t,x,z)e^{in\theta}\\
&u_\sigma^0(t,x,z)=\underline{u}^0_\sigma(t,x)+u^{0,*}_\sigma(t,x,z) \text{ with }\lim_{z\to\infty}u^{0,*}_\sigma =0,
\end{split}
\end{align}
and a similar expansion holds for $u_\tau$.  We define $\underline{u}^0_\sigma(t,x)$ to be the \emph{mean} of $u_\sigma$.  

We seek   $u_\sigma$ and $u_\tau$ as solutions of the interior and boundary profile equations \eqref{o3}, \eqref{o4} of the form
\begin{align}
u_\sigma(t,x,\theta,z)=\sum^4_{j=1}\sigma_j(t,x,\theta+\omega_jz)r_j \text{ and }u_\tau(t,x,\theta,z)=\sum^4_{j=1}\tau_j(t,x,\theta,z)r_j,
\end{align}
where  the low frequency cutoff $\chi_\eps(D_\theta)$ acting on the $\tau_j$ in \eqref{o2a} is now absent.

The determination of $u_\sigma-u_\sigma^0$ and of $u_\tau-u_\tau^0$ follows the procedure used in chapter \ref{chapter4} to determine the $\hat\sigma_j(t,x,k)$ and the $\hat\tau_j(t,x,k,z)$ for $k\neq 0$.   In particular, the traces $\hat\sigma_j(t,x_1,0,k)$, $k\neq 0$ are constant multiples of $\hat w(t,x_1,k)$, where $\hat w(t,x_1,k)$ is the solution of the amplitude equation \eqref{propelas}.
The building block estimates for  $u_\sigma-u_\sigma^0$ and  $u_\tau-u_\tau^0$ are readily seen to be the same as the estimates for $u_\sigma$ and $u_\tau$ given in section \ref{bblock}, except for the change that every occurrence of $p$ is replaced by the number $1$.   For example,  parallel to the third  estimate in Proposition \ref{qq6} we now have instead
\begin{align}\label{t1}
\langle\Lambda^{\frac{1}{2}}\partial_{zz} (u_\tau-u_\tau^0)\rangle_{m,\gamma}\lesssim \frac{C(m+3)}{\eps^{1/2}}+C(m+3+\frac{1}{2}).
\end{align}

Additional work beyond what is given in chapter \ref{chapter4} is needed to determine $u_\sigma^0$ and $u_\tau^0$, and this work has essentially been carried out in Chapter 2 of the thesis of Marcou \cite{Mar}.  In that chapter Marcou considers a simplified version of the elasticity equations in which a number of the nonlinear terms, including quadratic terms, have been thrown away.  Nevertheless, her method does apply directly to the full equations and shows that $u_\sigma^0=0$ and that $u_\tau^{0,*}$ is given by a simple integral formula.\footnote{This formula is provided in equations (2.6.5), (2.6.6) of \cite{Mar}, where $k=3$ and $H^0_{k-1}$, $K^0_{k-1}$ should be replaced by the analogous (quadratic) terms which appear in the full elasticity equations.}   Since we are only concerned with solving the profile equations \eqref{o3}, \eqref{o4}, we have no need to construct the mean $\underline u^0_\tau(t,x)$; we can set it equal to $0$.  To complete the building block estimates one needs to combine estimates like \eqref{t1} with corresponding estimates of $u^0_\tau=u_\tau^{0,*}$.  Such estimates follow readily from the integral formulas given in \cite{Mar}. One finds in each case that the estimate already obtained (or a better one) continues to hold for $u^0_\tau$.   Thus, for example, \eqref{t1} continues to hold when $u_\tau-u^0_\tau$ is replaced by $u_\tau$.

The forcing estimates of section \ref{forcing} are adapted to the wavetrain case by using the building block estimates, modified in the way we have just described, as before,  and observing that terms involving positive powers of $p$ on the right in the estimates of section \ref{forcing} always arise from forcing terms in which the low frequency cutoff $\chi_\eps(D_\theta)$ appears.  Such $p$-terms are therefore absent now.  Thus, for example,  in the estimates of Corollary \ref{r3a}, only the term $\sqrt{\eps}$ should appear on  the right.  In estimates (c) and (d) of Corollary \ref{r4a} the right sides are now $\eps^2$ and $\eps^{3/2}$, respectively.   Similarly, the last term on the right in the estimate of Proposition \ref{s5} is now $\eps$, and on the right side of Proposition 
\ref{endgame}, 
\begin{align}
\frac{1}{\sqrt{\gamma}}(\frac{p^2}{\sqrt{\eps}}+\sqrt{p}+\frac{\eps}{p})+\sqrt{\eps}
\end{align}
should be replaced by $\sqrt{\eps}$.   The same substitution should be made in the estimate of Corollary \ref{s22}, and this yields the rate of convergence 
\begin{align}
\begin{split}
&E_{m,T}(w)\lesssim \sqrt{\eps}
\end{split}
\end{align}
in Corollary \ref{s24a}.

\begin{rem}\label{optimal}
We believe it should be possible to construct arbitrarily high order approximate solutions to the elasticity equations in the wavetrain case, and then to show that these approximate solutions are close to exact solutions on an $\eps$-independent time interval by some variant of the method introduced by \cite{Gues}.   With high order approximate solutions it should not be necessary to consider singular problems; one attempts to solve directly the error equation satisfied by the difference between exact and approximate solutions.   

Such a result would yield more precise information than the result for wavetrains described above.  For example, the second chapter of \cite{Mar} indicates that one should expect ``internal rectification" to occur in the leading corrector; in other words one expects $u_\tau$ to have in general a nonzero mean $\underline u_\tau^0(t,x)$.  That property cannot be detected by our result for wavetrains, but could be detectable using this alternative method.   As we have already mentioned, there is no hope of constructing high order approximate solutions in the case of pulses.

\end{rem}

\section{The case of dimensions $d\geq 3$. }\label{higherD}
\emph{\quad}In this section we restrict attention to the Saint Venant-Kirchoff model \eqref{a0}.    For $d\geq 3$ the solution of the amplitude equation and the construction of the approximate solution goes through as in $d=2$.  However, there is a serious difficulty in $d\geq 3$ in constructing Kreiss symmetrizers for the linearized problem.  
As we explain below, the linearized problem has characteristics of variable multiplicity;  moreover, these include characteristics which fail to be \emph{algebraically regular} in the sense of \cite{MZ}, and which are at the same time \emph{glancing}.

We now let $\Omega=\{(t,x):x=(x_1,\dots,x_d), x_d>0\}$ and denote dual variables by $(\sigma,\xi)=(\sigma,\xi_1,\dots,\xi_d)$.   We regard $\xi$ as a column vector, so $\xi\xi^t$ is the $d\times d$ matrix $(a_{ij}) = (\xi_i\xi_j)$.   Writing $\phi(t,x)=x+U(t,x)$ as in \eqref{a2} and setting $\theta(t,x)=\nabla U$, one obtains for the $d\times d$ principal matrix symbol of the  interior equation in \eqref{a0} linearized at $\nabla\phi=I+\theta$:\footnote{Recall that $E(I+\theta)=\frac{1}{2}(\theta^t\theta+\theta+\theta^t)$.}
\begin{align}\label{u0}
\begin{split}
&P(\theta,\sigma,\xi)=-\sigma^2I +(\lambda+\mu)(I+\theta)\xi\xi^t (I+\theta)^t+\mu|\xi|^2(I+\theta)(I+\theta^t)+\\
&\qquad \qquad \left[(\lambda \mathrm{tr}\;E(I+\theta)-\mu)|\xi|^2+\mu\xi^t(I+\theta)^t(I+\theta)\xi\right]I.
\end{split}
\end{align}
When $\theta=0$, we have $E(I)=0$ and this reduces to $-\sigma^2I+(\lambda+\mu)\xi\xi^t+\mu |\xi|^2I$, a system for which the characteristic equation is 
\begin{align}\label{u1}
(\sigma^2-\mu |\xi|^2)^{d-1}\;(\sigma^2-(\lambda+2\mu)|\xi|^2)=0.
\end{align}
The roots $\sigma=\pm\sqrt{\mu}|\xi|$, $\sigma=\pm\sqrt{\lambda+2\mu}\;|\xi|$ have multiplicities $d-1$ and $1$ respectively.   Here the factors $\sigma-c|\xi|$ occurring in 
\eqref{u1} have double roots in $\xi_d$ when $\sigma=c|(\xi_1,\dots,\xi_{d-1},0)|$.  We then refer to $(\theta=0,\sigma,\xi_1,\dots,\xi_{d-1},0)$ as a \emph{glancing mode}.

\begin{prop}\label{u2}
Assume $d=3$.  Let $\uxi=(1,0,0)$ and 
consider the glancing mode $(\theta=0,\sqrt{\mu},\uxi)$, where $\sigma=\sqrt{\mu}$ is a root of multiplicity two of $P(0,\sigma,\uxi)=0$.  
For $\theta$ near $0$ the double root splits into roots $\sigma_j(\theta)$ satisfying  $P(\theta,\sigma_j(\theta),\uxi)=0$, where 
\begin{align}
\sigma_j(\theta)=\sqrt{\mu}+n_j(\theta), \;j=1,2
\end{align}
and the $n_j(\theta)$ are distinct, continuous functions of $\theta$ which fail to be $C^1$ at $\theta=0$.   

\end{prop}

\begin{proof}
The only things to check are that the $n_j$ are distinct and fail to be $C^1$ at $\theta=0$.  We choose $\theta$ so that $\theta_{ij}=0$, except for the 
$\theta_{23}$ and $\theta_{33}$ components which vary near $0$.  Write
\begin{align}
P(\theta,\sigma,\xi)+\sigma^2I=M(\theta,\xi)+S(\theta,\xi)I,
\end{align}
where $S(\theta,\xi)I$ is given by the second line of \eqref{u0}.   For this choice of $\theta$, the matrix $M(\theta,\uxi)+S(\theta,\uxi)I$ has the form
\begin{align}
\begin{pmatrix}\lambda+2\mu+S&0&0\\0&\mu+S+\mu\theta_{23}^2&\mu\theta_{23}\\0&\mu\theta_{23}&\mu+S+\mu(\theta_{33}^2+\theta_{33})\end{pmatrix}.
\end{align}
The eigenvalues of the lower $2\times 2$ block are  computed to be
\begin{align}
\beta_{\pm}=(\mu+S(\theta,\uxi))+\frac{\mu}{2}(\theta_{23}^2+\theta_{33}^2+\theta_{33})\mp\frac{\mu}{2}\sqrt{[\theta_{23}^2-(\theta_{33}^2+\theta_{33})]^2+4\theta_{23}^2}.
\end{align}
The argument of the square root is $\theta_{33}^2+4\theta_{23}^2+O(|\theta|^3)$, so the square root term fails to be $C^1$ in $\theta$ near $\theta=0$.

\end{proof}

\begin{rem}
1.   We refer to \cite{MZ} for the precise definition of algebraically regular multiple characteristics.  Roughly, a characteristic mode of multiplicity $m\geq 2$ is algebraically regular when it splits smoothly with respect to small changes of parameters.  Thus, the glancing mode in Proposition \ref{u2} fails to be algebraically regular.

2. The state of the art in Kreiss symmetrizers is represented by the papers \cite{Met} and \cite{MZ}.    The first paper constructs smooth Kreiss symmetrizers for  symmetric hyperbolic systems, including high order systems, but only for systems with characteristics of constant multiplicity.   Since we are dealing here with characteristics of variable multiplicity, the first paper does not apply.

The second paper treats only first order hyperbolic systems,   and constructs smooth Kreiss symmetrizers in certain situations where characteristics of variable multiplicity are present.  
The extension of the results of \cite{MZ} to higher order systems appears  to be nontrivial, and has not yet been done as far as we know.   
Moreover,  the results of \cite{MZ} do not appear to cover first order problems with characteristic modes of multiplicity $m\geq 2$ that fail to be algebraically regular and are at the same time glancing. 
Thus,  Proposition \ref{u2} indicates that even if we had an extension of \cite{MZ} to higher order systems,  it might not apply to the  kinds of variable multiplicity characteristics that we encounter in this problem.     If Theorems  \ref{uniformexistence} and \ref{approxthm}  do actually extend to $d\geq 3$,  it appears that further development of the theory of smooth Kreiss symmetrizers will be necessary to prove such an extension by our methods.  

\end{rem}

\appendix
\chapter{Singular pseudodifferential calculus for pulses}
\label{calculus}

In this Appendix, we summarize the parts of the singular pulse calculus constructed in \cite{CGW} that 
are needed in Chapter \ref{chapter3}. The calculus of \cite{CGW} was constructed with applications to 
first-order hyperbolic systems in mind. The  systems of elasticity equations considered in this work are 
second-order, so some extensions of the calculus are needed; these are given in section \ref{commutator} 
below.

First we define the singular Sobolev spaces used to describe mapping properties. The variable in $\R^{d+1}$ 
is denoted $(x,\theta)$, $x \in \R^d$, $\theta \in \R$, and the associated frequency is denoted $(\xi,k)$. We 
consider a fixed vector $\beta \in \R^d \setminus \{ 0\}$. Then for $s \in \R$ and $\eps \in \, (0,1]$, the 
anisotropic Sobolev space $H^{s,\eps} (\R^{d+1})$ is defined by
\begin{equation*}
H^{s,\eps}(\R^{d+1}) := \Big\{ u \in {\mathcal S}'(\R^{d+1}) \, / \, \widehat{u} \in L^2_{\rm loc}(\R^{d+1}) \, 
\text{\rm and} \, \int_{\R^{d+1}} \left( 1+\left| \xi+\dfrac{k \, \beta}{\eps} \right|^2 \right)^s 
\, \big| \widehat{u}(\xi,k) \big|^2 \, {\rm d}\xi \, {\rm d}k <+\infty \Big\} \, .
\end{equation*}
Here $\widehat{u}$ denotes the Fourier transform of $u$ on $\R^{d+1}$. The space $H^{s,\eps}(\R^{d+1})$ is 
equipped with the family of norms
\begin{equation*}
\forall \, \gamma \ge 1 \, ,\quad \forall \, u \in H^{s,\eps}(\R^{d+1}) \, ,\quad 
\| u \|_{H^{s,\eps},\gamma}^2 := \dfrac{1}{(2\, \pi)^{d+1}} \, \int_{\R^{d+1}} 
\left( \gamma^2 +\left| \xi+\dfrac{k \, \beta}{\eps} \right|^2 \right)^s 
\, \big| \widehat{u}(\xi,k) \big|^2 \, {\rm d}\xi \, {\rm d}k \, .
\end{equation*}
When $m$ is an integer, the space $H^{m,\eps} (\R^{d+1})$ coincides with the space of functions $u \in L^2 
(\R^{d+1})$ such that the derivatives, in the sense of distributions,
\begin{equation*}
\left( \partial_{x_1} +\dfrac{\beta_1}{\eps} \, \partial_\theta \right)^{\alpha_1} \dots
\left( \partial_{x_d} +\dfrac{\beta_d}{\eps} \, \partial_\theta \right)^{\alpha_d} \, u \, ,\quad
\alpha_1+\dots+\alpha_d \le m \, ,
\end{equation*}
belong to $L^2 (\R^{d+1})$. In the definition of the norm $\| \cdot \|_{H^{m,\eps},\gamma}$, one power of 
$\gamma$ counts as much as one derivative.

\section{Symbols}

Our singular symbols are built from the following sets of classical symbols.

\begin{defn}\label{n1}
Let $\cO\subset \R^N$ be an open subset that contains the origin.  For $m\in\R$ we let $\bfS^m(\cO)$ denote
the class of all functions $\sigma:\cO\times \R^d\times [1,\infty)\to \C^{N \times N}$, $N \ge 1$, such that 
$\sigma$ is $\cC^\infty$ on $\cO \times \R^d$ and for all compact sets $K\subset \cO$:
\begin{equation*}
\sup_{v\in K} \, \sup_{\xi'\in\R^d} \, \sup_{\gamma\geq 1} \, (\gamma^2+|\xi|^2)^{-(m-|\nu|)/2} \,
|\partial^\alpha_v\partial_{\xi'}^\nu \sigma(v,\xi,\gamma)| \leq C_{\alpha,\nu,K}.
\end{equation*}
\end{defn}

Let ${\mathcal C}^k_b(\R^{d+1})$, $k \in \N$, denote the space of continuous and bounded functions 
on $\R^{d+1}$, whose derivatives up to order $k$ are continuous and bounded. Let us first define the 
singular symbols.

\begin{defn}[Singular symbols]
\label{def4}
Fix $\beta\in\R^d\setminus 0$, let $m \in \R$, and let $n \in \N$. Then we let $S^m_n$ denote the set of 
families of functions $(a^{\eps,\gamma})_{\eps \in (0,1],\gamma \ge 1}$ that are constructed as follows:
\begin{equation}
\label{singularsymbolp}
\forall \, (x,\theta,\xi,k) \in \R^{d+1} \times \R^{d+1} \, ,\quad a^{\eps,\gamma} (x,\theta,\xi,k) =
\sigma \left( \eps \, V(x,\theta),\xi+\dfrac{k \, \beta}{\eps},\gamma \right) \, ,
\end{equation}
where $\sigma \in {\bf S}^m({\mathcal O})$, $ V$ belongs to the space ${\mathcal C}^n_b (\R^{d+1})$ and 
where furthermore $V$ takes its values in a convex compact subset $K$ of ${\mathcal O}$ that contains 
the origin (for instance $K$ can be a closed ball centered round the origin).
\end{defn}

All results below extend to the case where in place of a function $V$ that is independent of $\eps$, the 
representation \eqref{singularsymbolp} is considered with a function $V_\eps$ that is indexed by $\eps$, 
provided that we assume that all functions $\eps \, V_\eps$ take values in a {\it fixed} convex compact 
subset $K$ of ${\mathcal O}$ that contains the origin, and $(V_\eps)_{\eps \in (0,1]}$ is a bounded family 
of ${\mathcal C}^n_b (\R^{d+1})$.

\section{Definition of operators and action on Sobolev spaces}
\label{sect8}

To each symbol $a = (a^{\eps,\gamma})_{\eps \in (0,1],\gamma \ge 1} \in S^m_n$ given by the formula 
\eqref{singularsymbolp}, we associate a singular pseudodifferential operator $a^{\eps,\gamma}_D$, with 
$\eps \in (0,1]$ and $\gamma \ge 1$, whose action on a function $u \in {\mathcal S} (\R^{d+1};\C^N)$ is 
defined by
\begin{equation}
\label{singularpseudop}
a^{\eps,\gamma}_D \, u \, (x,\theta) := \dfrac{1}{(2\, \pi)^{d+1}} \, \int_{\R^{d+1}} {\rm e}^{i\, (\xi \cdot x +k \, \theta)} 
\, \sigma \left( \eps \, V(x,\theta),\xi+\dfrac{k \, \beta}{\eps},\gamma \right) \, \widehat{u} (\xi,k)
\, {\rm d}\xi \, {\rm d}k \, .
\end{equation}
Let us briefly note that for the Fourier multiplier $\sigma (v,\xi,\gamma) =i\, \xi_1$, the corresponding 
singular operator is $\partial_{x_1} +(\beta_1/\eps) \, \partial_\theta$. We now describe the action of 
singular pseudodifferential operators on Sobolev spaces. Detailed proofs of all results stated below 
can be found in \cite{CGW}.

\begin{rem}
\label{8a}
\textup{We will usually write $a_D \, u$ instead of $a^{\eps,\gamma}_D \, u$ for the function defined in 
\eqref{singularpseudop}. Also we often write $X$ for $\xi+k \, \beta/\eps$.}
\end{rem}

\begin{prop}
\label{prop13}
Let $n \ge d+1$, and let $a \in S^m_n$ with $m \le 0$. Then $a_D$ in \eqref{singularpseudop} defines
a bounded operator on $L^2 (\R^{d+1})$: there exists a constant $C>0$, that only depends on $\sigma$
and $V$ in the representation \eqref{singularsymbolp}, such that for all $\eps \in (0,1]$ and for all
$\gamma \ge 1$, there holds
\begin{equation*}
\forall \, u \in {\mathcal S} (\R^{d+1}) \, ,\quad \left\| a_D \, u \right\|_0 \le \dfrac{C}{\gamma^{|m|}} \, \| u \|_0 \, .
\end{equation*}
\end{prop}

\noindent The constant $C$ in Proposition \ref{prop13} depends uniformly on the compact set in which $V$ 
takes its values and on the norm of $V$ in ${\mathcal C}^{d+1}_b$. For operators defined by symbols of order 
$m>0$, we have:

\begin{prop}
\label{prop14}
Let $n \ge d+1$, and let $a \in S^m_n$ with $m>0$. Then $a_D$ in \eqref{singularpseudop} defines
a bounded operator from $H^{m,\eps}(\R^{d+1})$ to $L^2 (\R^{d+1})$: there exists a constant $C>0$, that
only depends on $\sigma$ and $V$ in the representation \eqref{singularsymbolp}, such that for all $\eps \in
(0,1]$ and for all $\gamma \ge 1$, there holds
\begin{equation*}
\forall \, u \in {\mathcal S} (\R^{d+1}) \, ,\quad \left\| a_D \, u \right\|_0 \le C \, \| u \|_{H^{m,\eps},\gamma} \, .
\end{equation*}
\end{prop}

\noindent The next proposition describes the smoothing effect of operators of order $-1$.

\begin{prop}
\label{prop15}
Let $n \ge d+2$, and let $a \in S^{-1}_n$. Then $a_D$ in \eqref{singularpseudop} defines a bounded
operator from $L^2 (\R^{d+1})$ to $H^{1,\eps}(\R^{d+1})$: there exists a constant $C>0$, that only depends
on $\sigma$ and $V$ in the representation \eqref{singularsymbolp}, such that for all $\eps \in (0,1]$ and for
all $\gamma \ge 1$, there holds
\begin{equation*}
\forall \, u \in {\mathcal S} (\R^{d+1}) \, ,\quad \left\| a_D \, u \right\|_{H^{1,\eps},\gamma} \le C \, \| u \|_0 \, .
\end{equation*}
\end{prop}

\begin{rem}
\label{a4}
\textup{In applications of the pulse calculus, we verify the hypothesis that for $V$ as in \eqref{singularsymbolp}, 
$V \in \mathcal{C}^n_b(\R^{d+1})$, by showing $V\in H^s(\R^{d+1})$ for some $s>\frac{d+1}{2}+n$.}
\end{rem}

\section{Adjoints and products}
\label{sect9}

For proofs of the following results we refer again to \cite{CGW}. The two first results deal with adjoints of singular 
pseudodifferential operators while the last two deal with products.

\begin{prop}
\label{prop18}
Let $a=\sigma(\eps V,X,\gamma) \in S_n^0$, $n \ge 2\, (d+1)$, where $V\in H^{s_0}(\R^{d+1})$ for some 
$s_0>\frac{d+1}{2}+1$, and let $a^*$ denote the conjugate transpose of the symbol $a$. Then $a_D$ 
and $(a^*)_D$ act boundedly on $L^2$ and there exists a constant $C \ge 0$ such that for all $\eps \in 
(0,1]$ and for all $\gamma \ge 1$, there holds
\begin{equation*}
\forall \, u \in {\mathcal S} (\R^{d+1}) \, ,\quad
\left\| (a_D)^* \, u-(a^*)_D \, u \right\|_0 \le \dfrac{C}{\gamma} \, \| u \|_0 \, .
\end{equation*}
If $n \ge 3\, d +3$, then for another constant $C$, there holds
\begin{equation*}
\forall \, u \in {\mathcal S} (\R^{d+1}) \, ,\quad
\left\| (a_D)^* \, u-(a^*)_D \, u \right\|_{H^{1,\eps},\gamma} \le C \, \| u \|_0 \, ,
\end{equation*}
uniformly in $\eps$ and $\gamma$.
\end{prop}

\begin{prop}
\label{prop19}
Let $a=\sigma(\eps V,X,\gamma) \in S_n^1$, $n \ge 3\, d +4$, where $V\in H^{s_0}(\R^{d+1})$ for some 
$s_0>\frac{d+1}{2}+1$, and let $a^*$ denote the conjugate transpose of the symbol $a$. Then $a_D$ 
and $(a^*)_D$ map $H^{1,\eps}$ into $L^2$ and there exists a family of operators $R^{\eps,\gamma}$ 
that satisfies
\begin{itemize}
 \item there exists a constant $C \ge 0$ such that for all $\eps \in (0,1]$ and for all $\gamma \ge 1$, there holds
\begin{equation*}
\forall \, u \in {\mathcal S} (\R^{d+1}) \, ,\quad \left\| R^{\eps,\gamma} \, u \right\|_0 \le C \, \| u \|_0 \, ,
\end{equation*}

 \item the following duality property holds
\begin{equation*}
\forall \, u,v \in {\mathcal S} (\R^{d+1}) \, ,\quad
\langle a_D \, u,v \rangle_{L^2} -\langle u, (a^*)_D \, v \rangle_{L^2} 
=\langle R^{\eps,\gamma} \, u,v \rangle_{L^2} \, .
\end{equation*}
In particular, the adjoint $(a_D)^*$ for the $L^2$ scalar product maps $H^{1,\eps}$ into $L^2$.
\end{itemize}
\end{prop}

\begin{prop}
\label{prop20}
(a)\; Let $a,b \in S_n^0$, $n \ge 2\, (d+1)$, and suppose $b=\sigma(\eps V,X,\gamma)$ where $V \in 
H^{s_0}(\R^{d+1})$ for some $s_0>\frac{d+1}{2}+1$. Then there exists a constant $C \ge 0$ such that 
for all $\eps \in (0,1]$ and for all $\gamma \ge 1$, there holds
\begin{equation*}
\forall \, u \in {\mathcal S} (\R^{d+1}) \, ,\quad 
\left\| a_D \, b_D \, u -(a \, b)_D \, u \right\|_0 \le \dfrac{C}{\gamma} \, \| u \|_0 \, .
\end{equation*}
If $n \ge 3\, d +3$, then for another constant $C$, there holds
\begin{equation}\label{prop20z}
\begin{split}
\forall \, u \in {\mathcal S} (\R^{d+1}) \, ,\quad
& \left\| a_D \, b_D \, u -(a \, b)_D \, u \right\|_{H^{1,\eps},\gamma} \le C \, \| u \|_0 \, ,\\
& \left\| a_D \, b_D \, u -(a \, b)_D \, u \right\|_0 \le C \, \| u \|_{H^{-1,\eps},\gamma} \, ,
\end{split}
\end{equation}
uniformly in $\eps$ and $\gamma$.

(b)\; Let $a \in S_n^1,b \in S_n^0$ or $a \in S_n^0,b \in S_n^1$, $n \ge 3\, d +4$, and in each case suppose 
$b=\sigma(\eps V,X,\gamma)$ where $V\in H^{s_0}(\R^{d+1})$ for some $s_0>\frac{d+1}{2}+1$. Then there 
exists a constant $C \ge 0$ such that for all $\eps \in (0,1]$ and for all $\gamma \ge 1$, there holds
\begin{equation*}
\forall \, u \in {\mathcal S} (\R^{d+1}) \, ,\quad \left\| a_D \, b_D \, u -(a\, b)_D \, u \right\|_0 \le C \, \| u \|_0 \, .
\end{equation*}
\end{prop}

\noindent The proof of the first estimate of \eqref{prop20z} in \cite{CGW} gives an explicit amplitude for 
the remainder $a_D \, b_D \, u -(a\, b)_D \, u$, and from this it is clear that the adjoint of $a_D \, b_D \, u 
-(a\, b)_D \, u$ has the same mapping property. Duality therefore implies the second estimate in 
\eqref{prop20z}.

\begin{prop}
\label{prop21}
Let $a \in S_n^{-1},b \in S_n^1$, $n \ge 3\, d +4$, and suppose $b=\sigma(\eps V,X,\gamma)$ where $V 
\in H^{s_0}(\R^{d+1})$ for some $s_0>\frac{d+1}{2}+1$. Then $a_D \, b_D$ defines a bounded operator 
on $H^{1,\eps}$ and there exists a constant $C \ge 0$ such that for all $\eps \in (0,1]$ and for all $\gamma 
\ge 1$, there holds
\begin{equation*}
\forall \, u \in {\mathcal S} (\R^{d+1}) \, ,\quad 
\left\| a_D \, b_D \, u -(a\, b)_D \, u \right\|_{H^{1,\eps},\gamma} \le C \, \| u \|_0 \, .
\end{equation*}
\end{prop}

\noindent Our next result is G{\aa}rding's inequality, either for symbols of degree $0$ or $1$.

\begin{theo}
\label{thm11}
(a) Let $\sigma \in {\bf S}^0$ satisfy $\text{\rm Re} \, \sigma (v,\xi,\gamma) \ge C_K>0$ for all $v$ in a compact 
subset $K$ of ${\mathcal O}$. Let now $a \in S^0_n$, $n \ge 2\, d+2$ be given by $a=\sigma(\eps V,X,\gamma)$, 
where $V\in H^{s_0}(\R^{d+1})$ for some $s_0>\frac{d+1}{2}+1$ and is valued in a convex compact subset 
$K$. Then for all $\delta >0$, there exists $\gamma_0$ which depends uniformly on $V$, the constant $C_K$ 
and $\delta$, such that for all $\gamma \ge \gamma_0$ and all $u \in {\mathcal S}(\R^{d+1})$, there holds
\begin{equation*}
\text{\rm Re } \langle a_D \, u ;u \rangle_{L^2} \ge (C_K-\delta) \, \| u \|_0^2 \, .
\end{equation*}

(b) Let $\sigma \in {\bf S}^1$ satisfy $\text{\rm Re} \, \sigma (v,\xi,\gamma) \ge C_K \, \langle\xi,\gamma\rangle$ 
for all $v$ in a compact subset $K$ of ${\mathcal O}$. Let now $a \in S^1_n$, $n \ge 3d+4$ be given by 
$a=\sigma(\eps V,X,\gamma)$, where $V\in H^{s_0}(\R^{d+1})$ for some $s_0>\frac{d+1}{2}+1$ and is valued 
in a convex compact subset $K$. Then for all $\delta >0$, there exists $\gamma_0$ which depends uniformly 
on $V$, the constant $C_K$ and $\delta$, such that for all $\gamma \ge \gamma_0$ and all $u \in {\mathcal S} 
(\R^{d+1})$, there holds
\begin{equation*}
\text{\rm Re } \langle a_D \, u ;u \rangle_{L^2} \ge (C_K-\delta) \, \| u \|_{H^{1/2,\eps},\gamma}^2 \, .
\end{equation*}
\end{theo}

\section{Extended calculus}
\label{extended}

At times we use a slight extension of the singular calculus. For given parameters $0<\delta_1<\delta_2<1$, we 
choose a cutoff $\chi^e (\xi,k\, \beta/\eps,\gamma)$ such that
\begin{align}\label{n31}
\begin{split}
&0\leq \chi^e \leq 1\, ,\\
&\chi^e \left( \xi,\dfrac{k\, \beta}{\eps},\gamma \right) =1 \text{ on } \left\{
(\gamma^2 +|\xi|^2)^{1/2} \leq \delta_1 \, \left| \dfrac{k\, \beta}{\eps} \right| \right\} \, ,\\
&\mathrm{supp }\chi^e \subset \left\{ (\gamma^2 +|\xi|^2)^{1/2} \leq \delta_2 \, \left| \dfrac{k\, \beta}{\eps} \right|
\right\} \, ,
\end{split}
\end{align}
and define a corresponding Fourier multiplier $\chi_D$ in the extended calculus by the formula 
\eqref{singularpseudop} with $\chi^e (\xi,k\, \beta/\eps,\gamma)$ in place of $\sigma(\eps V,X,\gamma)$. 
Composition laws involving such operators are proved in \cite{CGW}, but here we need only the fact that 
part {\it (a)} of Proposition \ref{prop20} (composition of two zero order singular operators) holds when either 
$a$ or $b$ is replaced by an extended cutoff $\chi^e$.

\section{Commutator estimates}\label{commutator}

\emph{\quad} In the proofs of this section we ignore some constant factors such as powers of $2\pi$.
For a given amplitude $c(x,\theta,y,\omega,X,\gamma)$, $\widetilde{Op}(c)$ denotes the operator defined by
\begin{align}
\widetilde{Op}(c) u (x,\theta)=\int e^{i(x-y)\xi+i(\theta-\omega)k} c(x,\theta,y,\omega,X,\gamma)u(y,\omega)dyd\omega d\xi dk.
\end{align}

\begin{lem}\label{n40}
Let $r\geq 0$ and let $M_{e^{iy\eta+i\omega l}}$ denote the operator that multiplies $u(y,\omega)$ by $e^{iy\eta+i\omega l}$.   Then
\begin{align}
\|\Lambda_D^rM_{e^{iy\eta+i\omega l}}u\|_0\lesssim \|\Lambda^r_Du\|_0+ \langle\eta\rangle^r\|u\|_0+\left\langle\frac{l}{\eps}\right\rangle^r\|u\|_0.
\end{align}
\end{lem}

\begin{proof}
This follows directly from the definition of the norm on the left and 
\begin{align}
\left\langle \xi+\eta+\frac{(k+l)\beta}{\eps},\gamma \right\rangle^r\lesssim \langle X,\gamma\rangle^r+\langle\eta\rangle^r+\left\langle \frac{l}{\eps}\right\rangle^r.
\end{align}

\end{proof}

\begin{prop}\label{commutator1}
i.)  Set  $x=(x_0,x'')=(t,x'')\in \R^{d+1}$ and $y=(y_0,y'')\in\R^{d+1}$,   let $a(X,\gamma)$ be a singular symbol of order $r\geq 1$, and let $\chi(x_0)\in H^\infty(\R)$.  Then for $m\in \{0,1,2,\dots\}$ we have
\begin{align}
\begin{split}
&(a)\; |[a_D,\chi] u|_{H^m_\gamma}\lesssim |\Lambda^{r-1}_D u|_{H^m_\gamma}\\
&(b)\;|[a_D,\chi] u|_{m,\gamma}\lesssim |\Lambda^{r-1}_D u|_{m,\gamma}.
\end{split}
\end{align}

ii.) The same estimates hold when $\chi(x)\in H^\infty(\R^{d+1})$.

iii.) When $\chi(x,\theta)\in H^\infty(\R^{d+1}\times\R)$, we have\footnote{It is easy to give a version of this proposition for $\chi\in H^s$ for $s$ large enough.}
\begin{align}
|[a_D,\chi] u|_{m,\gamma}\lesssim |\Lambda^{r-1}_D u|_{m,\gamma}+\frac{|u|_{m,\gamma}}{\eps^{r-1}}.
\end{align}
\end{prop}

\begin{proof}
\textbf{1. }It suffices to prove (a) for $m=0$, since derivatives with respect to $(x'',\theta)$ commute with the commutator and $\partial^k_t ([a_D,\chi]u)$ is a linear combination  of terms of the form
\begin{align}
[a_D,\chi^{(k_1)}]\partial_t^{k_2}u, \text{ where }k_1+k_2=k.
\end{align}
Replacing $u$ by $e^{-\gamma t}u$ then yields (b).

\textbf{2. }The function $[a_D,\chi]u$ may be written as (some constant factors are ignored here and below)  
\begin{align}\label{n32}
\begin{split}
&\int e^{i(x-y)\xi+i(\theta-\omega)k}a(X,\gamma)[\chi(x_0)-\chi(y_0)]u(y,\omega)dyd\omega d\xi dk=\\
&\qquad i\int e^{i(x-y)\xi+i(\theta-\omega)k}\partial_{\xi_0}a(X,\gamma)\left(\int^1_0\chi'(y_0+s(x_0-y_0))ds\right)u(y,\omega)dyd\omega d\xi dk.
\end{split}
\end{align}
Write $\chi'(y_0+s(x_0-y_0))=\int\widehat{\chi'}(\eta_0)e^{i\eta_0(y_0+s(x_0-y_0))}d\eta_0$, let $\widetilde{Op}(A_{s,\eta_0})$ be the operator associated 
to the amplitude 
of order $r-1$ given by 
\begin{align}
e^{i\eta_0s x_0}\cdot \partial_{\xi_0}a(X,\gamma)\cdot e^{i\eta_0(1-s)y_0},
\end{align}    
and observe that 
\begin{align}\label{n33}
[a_D,\chi]=\int^1_0\int \widetilde{Op}(A_{s,\eta_0})\;\widehat{\chi'}(\eta_0)d\eta_0 ds.
\end{align}
We have 
\begin{align}
\widetilde{Op}(A_{s,\eta_0})=M_{e^{i\eta_0s x_0}} \circ (\partial_{\xi_0}a)_D \circ M_{e^{i\eta_0(1-s)y_0}},
\end{align}
where $M$ denotes multiplication.   Thus,   by Lemma \ref{n40}
\begin{align}\label{n34}
|\widetilde{Op}(A_{s,\eta_0}) u|_{L^2}\lesssim |\Lambda^{r-1}_D( e^{i\eta_0(1-s)y_0}u)|_{L^2}\lesssim |\Lambda_D^{r-1}u|_{L^2}+|\eta_0|^{(r-1)}|u|_{L^2}.
\end{align}
Together with \eqref{n33}  this gives the estimate (a) when $m=0$.

\textbf{3. }The proofs of (ii) and (iii) involve only minor changes.

\end{proof}

\begin{prop}\label{commutator3}
For $j=1,0$ let $a^j=\sigma^j(\eps V,X,\gamma) \in S_n^j$, $n \ge 3\, d +4$, where $V\in H^{s_0}(\R^{d+1})$ for some 
$s_0>\frac{d+1}{2}+1$, and let $\phi(X,\gamma)$ be a singular symbol of order $0$.    We have  
\begin{align}
\begin{split}
&(a)\;\| [a^1_D,\phi_D] u \|_0 \lesssim \|  u\|_0\\
&(b)\; \|[a^0_D,\phi_D] u \|_0 \lesssim \| \Lambda^{-1}_D u\|_0.
\end{split}
\end{align}
\end{prop}

\begin{proof}
Part (a) follows directly from Proposition \ref{prop20}(b) and part (b) follows from Proposition \ref{prop20}(a).   

\end{proof}

\begin{prop}\label{commutator4}
Let the singular symbols $a^j$ have the product form $a^j=\sigma(\eps V)b^j(X,\gamma) \in S^j_n$, $n \ge 3\, d +4$, $j=0,1$, where $V\in H^{s_0}(\R^{d+1})$ for some 
$s_0>\frac{d+1}{2}+2$.    
\begin{align}
\begin{split}
&(a)\;\| [a^1_D,\Lambda_D^{\frac{1}{2}}] u \|_0 \lesssim \| \Lambda^{\frac{1}{2}} u\|_0+\|u/\sqrt{\eps}\|_0\\
&(b)\|\Lambda_D^{\frac{1}{2}} [a^1_D,\Lambda_D^{\frac{1}{2}}] u \|_0 \lesssim \| \Lambda_Du\|_0+\|\Lambda^{\frac{1}{2}}_Du/\sqrt{\eps}\|_0+\|u/\eps\|_0\\
&(c)\;\| [a^1_D,\Lambda_D^{-\frac{1}{2}}] u \|_0 \lesssim \gamma^{-\frac{1}{2}}\|  u\|_0\\
&(d)\; \|[a^0_D,\Lambda^{\frac{1}{2}}] u \|_0 \lesssim  \gamma^{-\frac{1}{2}}\|  u\|_0.
\end{split}
\end{align}
\end{prop}
\begin{proof}

\textbf{1. Proof of (a). }We write\footnote{Here $(a^*)_D^*$ denotes the adjoint of $(a^*)_D$.}
\begin{align}
\Lambda^{\frac{1}{2}}_Da_D=\Lambda^{\frac{1}{2}}_D[(a^*)_D^*+R],
\end{align}
 where the operator $R$ is given by the amplitude
 \begin{align}
\begin{split}
 &(\sigma(\eps V(x,\theta)-\sigma(\eps V(y,\omega)))b(X,\gamma)=r_1+r_2,\text{ with }\\
 &\quad r_1=\left(\int^1_0 d_v\sigma(\eps V(y+s(x-y),\theta)))\; \;\eps\partial_yV(y+s(x-y),\theta)ds\right) (x-y)b(X,\gamma)\\
 &\quad r_2=\left(\int^1_0d_v\sigma(\eps V(y, \omega+s(\theta-\omega)))\;\;\eps \partial_\omega V(y,\omega+s(\theta-\omega))ds\right) (\theta-\omega)b(X,\gamma).
 \end{split}
 \end{align}
 The operator $a_D\Lambda^{\frac{1}{2}}_D-\Lambda^{\frac{1}{2}}_D(a^*)_D^*$ is given by the amplitude 
 \begin{align}
 [(\sigma(\eps V(x,\theta))-\sigma(\eps V(y,\omega)))b(X,\gamma)]\Lambda^{\frac{1}{2}}=(r_1+r_2)\Lambda^{\frac{1}{2}}.
\end{align}
Thus, we have
\begin{align}\label{n35}
[a_D,\Lambda^{\frac{1}{2}}_D]= \widetilde{Op}\left((r_1+r_2)\Lambda^{\frac{1}{2}}\right)-\Lambda^{\frac{1}{2}}_D \circ \widetilde{Op}(r_1+r_2).
\end{align}

\textbf{2. }We focus on the ``worst" terms in \eqref{n35}, those involving $r_2$. Integration by parts with respect to $k$ using 
$\partial_k(e^{i(x-y)\xi+i(\theta-\omega)k})=i(\theta-\omega)e^{i(x-y)\xi+i(\theta-\omega)k}$ shows that 
\begin{align}\label{n36}
\widetilde{Op}\left(r_2\Lambda^{\frac{1}{2}}\right)=\widetilde{Op}(f),\text{ where }f(x,\theta,y,\omega,X,\gamma)=\int^1_0 F(y,\omega+s(\theta-\omega))ds\cdot  c(X,\gamma),
\end{align}
with $c$ of order $\frac{1}{2}$ and  $F(y,\omega)$  satisfying 
\begin{align}\label{n37}
\int \langle\eta,l\rangle |\hat{F}(\eta,l)|d\eta dl\lesssim |F(y,\omega)|_{H^{s_0-1}(y,\omega)}.
\end{align}
We have
\begin{align}
F(y,\omega+s(\theta-\omega))=\int e^{ils\theta}\hat{F}(\eta,l) e^{i\eta y+il\omega(1-s)}d\eta dl.
\end{align}
Letting $\widetilde{Op}(A_{s,\eta,l})$ be the operator associated to the amplitude $e^{ils\theta}\hat{F}(\eta,l) c(X,\gamma) e^{i\eta y+il\omega(1-s)}$, we have with obvious notation
\begin{align}\label{n38}
\widetilde{Op}(A_{s,\eta,l})=M_{e^{ils\theta}}\hat{F}(\eta,l)c_D M_{ e^{i\eta y+il\omega(1-s)}}.
\end{align}
We obtain
\begin{align}\label{n38a}
\begin{split}
&\|\widetilde{Op}(A_{s,\eta,l}) u\|_{0}=\|\hat{F}(\eta,l)c_D\Lambda^{-\frac{1}{2}}_D\Lambda^{\frac{1}{2}}_DM_{ e^{i\eta y+il\omega(1-s)}}u\|_0\lesssim\\
& \qquad |\hat{F}(\eta,l)|\|\Lambda^{\frac{1}{2}}_DM_{ e^{i\eta y+il\omega(1-s)}}u\|_0\lesssim |\hat{F}(\eta,l)|\left(\|\Lambda^{\frac{1}{2}}_D u\|_0+\langle\eta\rangle^{\frac{1}{2}}\|u\|_0+\langle l\rangle^{\frac{1}{2}}\frac{\|u\|_0}{\sqrt{\eps}}\right),
\end{split}
\end{align}
where the last inequality follows from Lemma \ref{n40}. 
With \eqref{n36} and \eqref{n37} this implies
\begin{align}\label{n38b}
\|\widetilde{Op}\left(r_2\Lambda^{\frac{1}{2}}\right)u\|_0\lesssim \left(\|\Lambda^{\frac{1}{2}}_D u\|_0+\frac{\|u\|_0}{\sqrt{\eps}}\right).
\end{align}

To see that the operator $\Lambda^{\frac{1}{2}}_D\widetilde{Op}(r_2)$ in \eqref{n35} satisfies
\begin{align}\label{n39}
\|\Lambda^{\frac{1}{2}}_D\widetilde{Op}(r_2) u\|_0\lesssim \left(\|\Lambda^{\frac{1}{2}}_D u\|_0+\frac{\|u\|_0}{\sqrt{\eps}}\right),
\end{align}
we first decompose $\widetilde{Op}(r_2)$ into operators like those in \eqref{n38}, where now $c_D$ in \eqref{n38} has order zero.  The estimate \eqref{n39} then follows readily.    

The terms involving $r_1$ in \eqref{n35} are treated similarly, but satisfy slightly better estimates since, for example, one integrates by parts with respect to $\xi$ instead of $k$ to define the analogue of $f$ in \eqref{n36}.

\textbf{3. Proof of (b). }With $\widetilde{Op}(A_{s,\eta,l})$ as in \eqref{n38}, two applications of Lemma \ref{n40} yield

\begin{align}\label{n41}
\begin{split}
&\|\Lambda^{\frac{1}{2}}\widetilde{Op}(A_{s,\eta,l}) u\|_{0}\lesssim\\
&\quad |\hat{F}(\eta,l)|\left(\|\Lambda_D u\|_0+\langle l\rangle^{\frac{1}{2}}\frac{\|\Lambda_D^{\frac{1}{2}}u\|_0}{\sqrt{\eps}}+\langle\eta\rangle\|u\|_0+\langle\eta\rangle^{\frac{1}{2}}\langle l\rangle^{\frac{1}{2}}\frac{\|u\|_0}{\sqrt{\eps}}+\left\langle\frac{l}{\eps}\right\rangle\|u\|_0\right),
\end{split}
\end{align}
so   \textbf{(b)} follows from \eqref{n37}.

\textbf{4. } The proof of (c) is similar to the proof of (b) but simpler.  For example, the operators $c_D$ that appear as in \eqref{n38} but now in the decomposition of
$\widetilde{Op}\left(r_2\Lambda^{-\frac{1}{2}}\right)$ have order $-\frac{1}{2}$, so the right side of \eqref{n38b} is replaced by $\gamma^{-\frac{1}{2}}\|u\|_0$.  The proof of (d) is almost the same as that of (c).

\end{proof}

\begin{prop}\label{commutator5}
Let the singular symbols $a^j$ have the product form $a^j=\sigma^j(\eps V)b^j(X,\gamma) \in S^j_n$, $n \ge 3\, d +4$, $j=0,1$, where $V\in H^{s_0}(\R^{d+1})$ for some 
$s_0>\frac{d+1}{2}+2$.    
\begin{align}\label{comm5}
\begin{split}
&\|\Lambda_D^{\frac{1}{2}} [a^1_D,a^0_D] u \|_0 \lesssim  \| \Lambda^{\frac{1}{2}} u\|_0+\|u/\sqrt{\eps}\|_0.
\end{split}
\end{align}
\end{prop}
\begin{proof}
Using a classical ``**-argument" to analyze compositions\footnote{See the proof of Proposition 11 of \cite{CGW}, for example.}, we obtain
\begin{align}
\begin{split}
&[a^1_D,a^0_D] =\widetilde{Op} \;[a^1(\eps V(x,\theta),X,\gamma)(a^0(\eps V(x,\theta),X,\gamma)-a^0(\eps V(y,\omega),X,\gamma))]-\\
&\quad \qquad \widetilde{Op} \;[a^0(\eps V(x,\theta),X,\gamma)(a^1(\eps V(x,\theta),X,\gamma)-a^1(\eps V(y,\omega),X,\gamma))]+\\
&a^1_D\widetilde{Op} \;[a^0(\eps V(x,\theta),X,\gamma)-a^0(\eps V(y,\omega),X,\gamma)]-a^0_D\widetilde{Op} \;[a^1(\eps V(x,\theta),X,\gamma)-a^1(\eps V(y,\omega),X,\gamma)].
\end{split}
\end{align}
The product form of the symbols allows us to analyze these operators  using  Fourier decompositions as in step \textbf{2} of the proof of Proposition \ref{commutator4}. The result then follows by application of Lemma \ref{n40}. 

\end{proof}

We also need an estimate like \eqref{comm5} for symbols that are \emph{not} of product form.

\begin{cor}
Let the singular symbols $a^j$ have the form $a^j=\sigma^j(\eps V,X,\gamma) \in S^j_n$, $n \ge 3\, d +4$, $j=0,1$, where $V\in H^{s_0}(\R^{d+1})$ for some 
$s_0>\frac{d+1}{2}+2$, and where $\sigma^j(w,\xi,\gamma)$ is \emph{homogeneous} of degree $j$ in $(\xi,\gamma)$.  Then    
\begin{align}
\begin{split}
&\|\Lambda_D^{\frac{1}{2}} [a^1_D,a^0_D] u \|_0 \lesssim  \| \Lambda^{\frac{1}{2}} u\|_0+\|u/\sqrt{\eps}\|_0.
\end{split}
\end{align}

\end{cor}

\begin{proof}
One can reduce to the product case considered above by expanding the symbols $\sigma^j(w,\xi,\gamma)$ in terms of spherical harmonics.
\end{proof}

\bibliographystyle{alpha}
\bibliography{Rayleigh}

\begin{thebibliography}{CGW14b}

\bibitem[AH03]{AliHunter}
G.~Al{\`{\i}} and J.~K. Hunter.
\newblock Nonlinear surface waves on a tangential discontinuity in
  magnetohydrodynamics.
\newblock {\em Quart. Appl. Math.}, 61(3):451--474, 2003.

\bibitem[AH13]{AustriaHunter}
L.~Austria and J.~K. Hunter.
\newblock Nonlinear variational surface waves.
\newblock {\em Commun. Inf. Syst.}, 13(1):3--43, 2013.

\bibitem[AHP02]{AliHunterParker}
G.~Al{\`{\i}}, J.~K. Hunter, and D.~F. Parker.
\newblock Hamiltonian equations for scale-invariant waves.
\newblock {\em Stud. Appl. Math.}, 108(3):305--321, 2002.

\bibitem[AR03]{AR}
D.~Alterman and J.~Rauch.
\newblock Diffractive nonlinear geometric optics for short pulses.
\newblock {\em Siam J. Math. Analysis}, 34(6):1477--1502, 2003.

\bibitem[BG98]{Benzoni1998}
S.~Benzoni-Gavage.
\newblock Stability of multi-dimensional phase transitions in a van der {W}aals
  fluid.
\newblock {\em Nonlinear Anal.}, 31(1-2):243--263, 1998.

\bibitem[BG09]{B}
S.~Benzoni-Gavage.
\newblock Local well-posedness of nonlocal {B}urgers equations.
\newblock {\em Differential and Integral Equations}, 22(43-4):303--320, 2009.

\bibitem[BGC12]{BC}
S.~Benzoni-Gavage and J.-F. Coulombel.
\newblock On the amplitude equations for weakly nonlinear surface waves.
\newblock {\em Arch. Ration. Mech. Anal.}, 205(3):871--925, 2012.

\bibitem[BGC15]{BCproc}
S.~Benzoni-Gavage and J.-F. Coulombel.
\newblock Amplitude equations for weakly nonlinear surface waves in variational
  problems.
\newblock arXiv:1510.01119, 2015.

\bibitem[BGS07]{BS}
S.~Benzoni-Gavage and D.~Serre.
\newblock {\em Multidimensional hyperbolic partial differential equations}.
\newblock Oxford University Press, 2007.

\bibitem[CC05]{CC}
P.G. Ciarlet and P.~Ciarlet.
\newblock Another approach to linearized elasticity and a new proof of korn's
  inequality.
\newblock {\em Math. Models Methods Appl. Sci.}, 15(2):259--271, 2005.

\bibitem[CGW11]{CGW1}
J.-F. Coulombel, O.~Gu{\`e}s, and M.~Williams.
\newblock Resonant leading order geometric optics expansions for quasilinear
  hyperbolic fixed and free boundary problems.
\newblock {\em Comm. Partial Differential Equations}, 36(10):1797--1859, 2011.

\bibitem[CGW14a]{CGW2}
J.-F. Coulombel, O.~Gu{\`e}s, and M.~Williams.
\newblock Semilinear geometric optics with boundary amplification.
\newblock {\em Analysis and PDE}, 7(3):551--625, 2014.

\bibitem[CGW14b]{CGW}
J.-F. Coulombel, O.~Gu{\`e}s, and M.~Williams.
\newblock Singular pseudodifferential calculus for wavetrains and pulses.
\newblock {\em Bull. Soc. Math. France}, 12:719--776, 2014.

\bibitem[Cia83]{ciarlet}
P.~G. Ciarlet.
\newblock {\em Lectures on three-dimensional elasticity}.
\newblock Published for the Tata Institute of Fundamental Research, Bombay; by
  Springer-Verlag, Berlin, 1983.
\newblock Notes by S. Kesavan.

\bibitem[Cia88]{Ci}
P.G. Ciarlet.
\newblock {\em Mathematical elasticity, volume I: three dimensional
  elasticity}.
\newblock Elsevier, North Holland, 1988.

\bibitem[CP82]{CP}
J.~Chazarain and A.~Piriou.
\newblock {\em Introduction to the theory of linear partial differential
  equations}.
\newblock North-Holland Publishing Co., 1982.

\bibitem[CW13]{CW1}
J.-F. Coulombel and M.~Williams.
\newblock Nonlinear geometric optics for reflecting uniformly stable pulses.
\newblock {\em J. Differential Equations}, 255(7):1939--1987, 2013.

\bibitem[CW14]{CW2}
J.-F. Coulombel and M.~Williams.
\newblock Amplification of pulses in nonlinear geometric optics.
\newblock {\em Journal of Hyperbolic Differential Equations}, 11:749--793,
  2014.

\bibitem[CW16]{CW4}
J.-F. Coulombel and M.~Williams.
\newblock The {M}ach stem equation and amplification in strongly nonlinear
  geometric optics.
\newblock {\em To appear in American J. Math.}, 2016.

\bibitem[Gu{\`e}93]{Gues}
O.~Gu{\`e}s.
\newblock D\'eveloppement asymptotique de solutions exactes de syst\`emes
  hyperboliques quasilin\'eaires.
\newblock {\em Asymptotic Anal.}, 6(3):241--269, 1993.

\bibitem[Her15]{Hernandez}
M.~Hernandez.
\newblock Resonant leading term geometric optics expansions with boundary
  layers for quasilinear hyperbolic boundary problems.
\newblock {\em Comm. Partial Differential Equations}, 40(3):387--437, 2015.

\bibitem[HIZ95]{HamiltonIlinskyZabolotskaya}
M.~F. Hamilton, Yu.~A. Il'insky, and E.~A. Zabolotskaya.
\newblock Evolution equations for nonlinear {R}ayleigh waves.
\newblock {\em J. Acoust. Soc. of Amer.}, 97(2):891--897, 1995.

\bibitem[HM78]{HM}
T.~Hughes and J.~Marsden.
\newblock Classical elastodynamics as a linear symmetric hyperbolic system.
\newblock {\em J. of Elasticity}, 8:97--110, 1978.

\bibitem[Hun89]{H}
J.~Hunter.
\newblock Nonlinear surface waves.
\newblock {\em Contemporary Mathematics}, 100:185--202, 1989.

\bibitem[Hun06]{Hunter2006}
J.~K. Hunter.
\newblock Short-time existence for scale-invariant {H}amiltonian waves.
\newblock {\em J. Hyperbolic Differ. Equ.}, 3(2):247--267, 2006.

\bibitem[JMR93]{JMR1}
J.-L. Joly, G.~M{\'e}tivier, and J.~Rauch.
\newblock Generic rigorous asymptotic expansions for weakly nonlinear
  multidimensional oscillatory waves.
\newblock {\em Duke Math. J.}, 70(2):373--404, 1993.

\bibitem[JMR95]{JMR}
J.-L. Joly, G.~M{\'e}tivier, and J.~Rauch.
\newblock Coherent and focusing multidimensional nonlinear geometric optics.
\newblock {\em Ann. Sci. \'Ecole Norm. Sup. (4)}, 28(1):51--113, 1995.

\bibitem[Kat85]{Kato}
T.~Kato.
\newblock {\em Abstract differential equations and nonlinear mixed problems}.
\newblock Lezioni Fermiane. [Fermi Lectures]. Scuola Normale Superiore, Pisa;
  Accademia Nazionale dei Lincei, Rome, 1985.

\bibitem[Kre70]{K}
H.~O. Kreiss.
\newblock Initial boundary value problems for hyperbolic systems.
\newblock {\em Comm. Pure Appl. Math.}, 23:277--298, 1970.

\bibitem[Lar83]{Lardner1}
R.~W. Lardner.
\newblock Nonlinear surface waves on an elastic solid.
\newblock {\em Int. J. Engng Sci.}, 21(11):1331--1342, 1983.

\bibitem[Lar86]{Lardner2}
R.~W. Lardner.
\newblock Nonlinear surface acoustic waves on an elastic solid of general
  anisotropy.
\newblock {\em J. Elasticity}, 16(1):63--73, 1986.

\bibitem[Les07]{Le}
V.~Lescarret.
\newblock Wave transmission in dispersive media.
\newblock {\em Math. Models Methods Appl. Sci.}, 17(4):485--535, 2007.

\bibitem[MA88]{Ma}
A.~Majda and M.~Artola.
\newblock Nonlinear geometric optics for hyperbolic mixed problems.
\newblock In {\em Analyse math\'ematique et applications}, pages 319--356.
  Gauthier-Villars, 1988.

\bibitem[Mar10]{Mar}
A.~Marcou.
\newblock Rigorous weakly nonlinear geometric optics for surface waves.
\newblock {\em Asymptot. Anal.}, 69(3-4):125--174, 2010.

\bibitem[M{\'e}t00]{Met}
G.~M{\'e}tivier.
\newblock The block structure condition for symmetric hyperbolic systems.
\newblock {\em Bull. London Math. Soc.}, 32(6):689--702, 2000.

\bibitem[M{\'e}t04]{MetivierCL}
G.~M{\'e}tivier.
\newblock {\em Small viscosity and boundary layer methods}.
\newblock Birkh\"auser, 2004.
\newblock Theory, stability analysis, and applications.

\bibitem[MZ05]{MZ}
G.~M{\'e}tivier and K.~Zumbrun.
\newblock Hyperbolic boundary value problems for symmetric systems with
  variable multiplicities.
\newblock {\em J. Differential Equations}, 211(1):61--134, 2005.

\bibitem[Pad14]{Pa}
M.~Paddick.
\newblock {\em Stability of boundary layers and solitary waves in fluid
  mechanics}.
\newblock Thesis, Univ. of Rennes I, 2014.

\bibitem[Par88]{Parker}
D.~F. Parker.
\newblock Waveform evolution for nonlinear surface acoustic waves.
\newblock {\em Int. J. Engng Sci.}, 26(1):59--75, 1988.

\bibitem[PT85]{ParkerTalbot}
D.~F. Parker and F.~M. Talbot.
\newblock Analysis and computation for nonlinear elastic surface waves of
  permanent form.
\newblock {\em J. Elasticity}, 15(4):389--426, 1985.

\bibitem[Rau79]{R}
J.~Rauch.
\newblock Singularities of solutions to semilinear wave equations.
\newblock {\em J. Math. Pures Appl. (9)}, 58(3):299--308, 1979.

\bibitem[RR82]{RR}
J.~Rauch and M.~Reed.
\newblock Nonlinear microlocal analysis of semilinear hyperbolic systems in one
  space dimension.
\newblock {\em Duke Math. J.}, 49(2):397--475, 1982.

\bibitem[Sak82]{sakamoto}
R.~Sakamoto.
\newblock {\em Hyperbolic boundary value problems}.
\newblock Cambridge University Press, Cambridge, 1982.

\bibitem[Sax89]{Saxton}
R.~A. Saxton.
\newblock Dynamic instability of the liquid crystal director.
\newblock In {\em Current progress in hyperbolic systems: {R}iemann problems
  and computations ({B}runswick, {ME}, 1988)}, volume 100 of {\em Contemp.
  Math.}, pages 325--330. Amer. Math. Soc., 1989.

\bibitem[Sec15]{Secchi}
P.~Secchi.
\newblock Nonlinear surface waves on the plasma-vacuum interface.
\newblock {\em Quart. Appl. Math.}, 73(4):711--737, 2015.

\bibitem[Ser06]{SerreJFA}
D.~Serre.
\newblock Second order initial boundary-value problems of variational type.
\newblock {\em J. Funct. Anal.}, 236(2):409--446, 2006.

\bibitem[SN89]{SN}
Y.~Shibata and G.~Nakamura.
\newblock On a local existence theorem of {N}eumann problem for some
  quasilinear hyperbolic systems of 2nd order.
\newblock {\em Math. Zeitschrift}, 202(1):1--64, 1989.

\bibitem[ST88]{S-T}
M.~Sabl{\'e}-Tougeron.
\newblock Existence pour un prob{l'e}me d'elastodynamique {N}eumann non
  lineaire en dimension 2.
\newblock {\em Arch. Rational Mech. Anal.}, 101(3):261--292, 1988.

\bibitem[Tay77]{T2}
M.~Taylor.
\newblock Rayleigh waves in linear elasticity as a propagation of singularities
  phenomenon.
\newblock In {\em Partial Differential Equations and Geometry, Proc. Conf.,
  Park City, Utah, 1977}, pages 273--291. Dekker, New York, 1977.

\bibitem[Tay11a]{T}
M.~Taylor.
\newblock {\em Partial differential equations I. Basic theory. Second edition.
  Applied Mathematical Sciences, 115.}
\newblock Springer, New York, 2011.

\bibitem[Tay11b]{TaylorIII}
M.~E. Taylor.
\newblock {\em Partial differential equations {III}. {N}onlinear equations},
  volume 117 of {\em Applied Mathematical Sciences}.
\newblock Springer, 2011.

\bibitem[Wil02]{W}
M.~Williams.
\newblock Singular pseudodifferential operators, symmetrizers, and oscillatory
  multidimensional shocks.
\newblock {\em J. Funct. Anal.}, 191(1):132--209, 2002.

\bibitem[Wil15]{Willig}
C.~Willig.
\newblock Nonlinear geometric optics for reflecting and evanescent pulses.
\newblock {\em PhD Thesis, UNC Chapel Hill}, 2015.

\end{thebibliography}
\end{document}